\documentclass{book}
\usepackage[mems, lang=british]{ems-book} 
\usepackage{mathtools}
\usepackage{ulem}
\usepackage{comment}

\usepackage{pgfplots}
\pgfplotsset{compat=1.17}

\newcommand{\A}{\ensuremath{\mathcal{A}}}
\newcommand{\B}{\mathrm{B}}
\newcommand{\C}{\mathbb{C}}

\newcommand{\E}{\ensuremath{\mathbb{E}}}

\let\H\relax 
\newcommand{\H}{\mathrm{H}}

\newcommand{\I}{\mathrm{I}}
\newcommand{\g}{\mathfrak{g}}
\newcommand{\W}{\mathrm{W}}
\let\L\relax
\newcommand{\L}{\mathrm{L}}

\newcommand{\M}{\mathrm{M}}

\newcommand{\N}{\ensuremath{\mathbb{N}}}

\newcommand{\R}{\ensuremath{\mathbb{R}}}
\newcommand{\T}{\ensuremath{\mathbb{T}}}
\newcommand{\Z}{\ensuremath{\mathbb{Z}}}

\newcommand{\QH}{\mathbb{H}}

\renewcommand{\leq}{\ensuremath{\leqslant}}
\renewcommand{\geq}{\ensuremath{\geqslant}}
\newcommand{\bnorm}[1]{ \big\| #1  \big\|}

\newcommand{\Bgnorm}[1]{ \Bigg\| #1  \Bigg\|}
\newcommand{\norm}[1]{\left\Vert#1\right\Vert}
\newcommand{\frak}{\mathfrak}
\newcommand{\scr}{\mathscr}
\newcommand{\xra}{\xrightarrow}

\newcommand{\ot}{\otimes}
\newcommand{\epsi}{\varepsilon}
\newcommand{\ovl}{\overline}
\newcommand{\otvn}{\ovl\ot}

\newcommand{\la}{\langle}
\newcommand{\ra}{\rangle}
\newcommand{\co}{\colon}
\renewcommand{\d}{\mathop{}\mathopen{}\mathrm{d}} 
\let\i\relax 
\newcommand{\i}{\mathrm{i}}

\newcommand{\ov}{\overset}
\newcommand{\Mult}{\mathrm{Mult}}

\newcommand{\reg}{\mathrm{reg}}

\newcommand{\Ad}{\mathrm{Ad}}

\newcommand{\Closed}{\mathrm{Closed}}

\newcommand{\QWEP}{\mathrm{QWEP}}

\newcommand{\Str}{\mathrm{Str}}

\newcommand{\Id}{\mathrm{Id}}
\newcommand{\HP}{\mathrm{HP}}
\newcommand{\VN}{\mathrm{VN}}

\newcommand{\UMD}{\mathrm{UMD}}

\newcommand{\e}{e} 
\let\ker\relax 
\DeclareMathOperator{\ker}{Ker} 
\DeclareMathOperator{\Ran}{Ran} 
\DeclareMathOperator{\Ker}{Ker}

\DeclareMathOperator{\dom}{dom} 

\DeclareMathOperator{\tr}{Tr}

\DeclareMathOperator{\supp}{supp} 
\DeclareMathOperator{\type}{type}
\DeclareMathOperator{\cotype}{cotype}
\DeclareMathOperator{\Span}{span} 
\newcommand{\cb}{\mathrm{cb}} 
\newcommand{\HS}{\mathrm{HS}} 

\newcommand{\univ}{\mathrm{univ}}
\let\cal\relax
\newcommand{\cal}{\mathcal}
\newcommand{\Hor}{\textnormal{Hör}}
\newcommand{\HI}{\mathrm{H}^\infty}

\newcommand{\TC}{\mathcal{C}}
\newcommand{\h}{h}
\newtheorem{thm}{Theorem}[section]
\newtheorem{defi}[thm]{Definition}

\newtheorem{prop}[thm]{Proposition}

\newtheorem{cor}[thm]{Corollary}
\newtheorem{lemma}[thm]{Lemma}

\newtheorem{example}[thm]{Example}
\newtheorem{remark}[thm]{Remark}

\begin{document}

\cleardoublepage
\chapter*{}
\thispagestyle{empty}

\begin{flushleft}
\textbf{The harmonic oscillator on the Moyal-Groenewold plane:}\\
\textit{An approach via Lie groups and twisted Weyl tuples}

\vspace{1em}
\textbf{Cédric Arhancet, Lukas Hagedorn, Christoph Kriegler, and Pierre Portal}
\end{flushleft}

\vspace{2em}

\title{The harmonic oscillator on the Moyal-Groenewold plane: an approach via Lie groups and twisted Weyl tuples} 
\author{C\'edric Arhancet, Lukas Hagedorn, Christoph Kriegler and Pierre Portal}
\frontmatter

\begin{abstract}
This paper investigates the functional calculus of the harmonic oscillator on each Moyal-Groenewold plane, the noncommutative phase space which is a fundamental object in quantum mechanics. Specifically, we show that the harmonic oscillator admits a bounded $\mathrm{H}^\infty(\Sigma_\omega)$ functional calculus for any angle $0 < \omega < \frac{\pi}{2}$ and even a bounded H\"ormander functional calculus on the associated noncommutative $\mathrm{L}^p$-spaces, where $\Sigma_\omega=\{ z \in \mathbb{C}^*: |\arg z| <\omega \}$. To achieve these results, we develop a connection with the theory of 2-step nilpotent Lie groups by introducing a notion of twisted Weyl tuple and connecting it to some semigroups of operators previously investigated by Robinson via group representations. Along the way, we demonstrate that $\mathrm{L}^p$-square-max decompositions lead to new insights between noncommutative ergodic theory and $R$-boundedness, and we prove a twisted transference principle, which is of independent interest. Our approach accommodates the presence of a constant magnetic field and they are indeed new even in the framework of magnetic Weyl calculus on classical $\mathrm{L}^p$-spaces. Our results contribute to the understanding of functional calculi on noncommutative spaces and have implications for the maximal regularity of the most basic evolution equations associated to the harmonic oscillator.
\end{abstract}

\dedication{In memory of Alan McIntosh (1942–2016) and Jos\'e Enrique Moyal (1910–1998).}


\classification[46L51, 47D03, 22E25]{47A60}

\begin{ack}
We are grateful to Brian Hall, Boris Khesin, Orr Shalit, and Victor Gayral for brief discussions, and to Serge Nicaise for pointing us to the reference \cite{KMR97}.
\end{ack}

\begin{funding}
The authors gratefully acknowledge the support from the French National Research Agency grant ANR-18-CE40-0021 (project HASCON). 
\end{funding}

\tableofcontents
\mainmatter


\chapter{Introduction}

\section{Introduction}
\label{sec-introduction}

A fundamental consequence of the spectral theory developed by Hilbert and von Neumann is the ability to provide spectral decompositions and functional calculus of differential operators (or variants) acting on $\L^2$-spaces. Such an unbounded operator $A$ is typically a geometric Laplacian or a Hamiltonian, and one needs to analyse their behaviour on $\L^p$-spaces. When $A$ is a positive self-adjoint operator, one can define a bounded operator $f(A)$ for an essentially bounded measurable function $f \co \R^+ \to \mathbb{C}$ by functional calculus. In harmonic analysis, the spectral multiplier problem then consists in establishing sharp estimates on the norm of the operator $f(A)$ acting on the subspace $\L^2 \cap \L^p$ equipped with the $\L^p$-norm. Even for the usual Laplacian $-\Delta$ on $\R^{d}$, this problem is not completely solved, and questions with deep consequences such as the Bochner-Riesz conjecture remain open. See e.g.~\cite[Chapter 5]{Gra14b}, \cite{GOWWZ21} and \cite{LuY13} and references therein. 

For Hamiltonians, the problem is not simpler, as can be seen in the fundamental example of the Hermite operator  
\begin{equation}
\label{Hermite-operator}
-\Delta+|x|^2
=-\sum_{k=1}^d \partial_{k}^2+\sum_{k=1}^d x_k^2
\end{equation}
on the Euclidean space $\R^d$, investigated in \cite{LeR22}, \cite{Tha89a}, \cite{Tha89b} and \cite{Tha91}. 
By \cite[p.~492]{Nee22}, this positive self-adjoint operator admits the discrete spectrum 
$$
2\mathbb{N}+d=\{d,2+d,4+d,\ldots\},
$$ 
in sharp contrast to the classical Laplacian $-\Delta$. In particular, this operator has a spectral gap. The Hermite operator and its eigenfunctions, the so-called Hermite functions, respectively represent the Hamiltonian and quantum states of the particle for the quantum-mechanical analogue of the classical harmonic oscillator.

Going beyond Euclidean geometry or simple Hamiltonians, the spectral multiplier problem rapidly gains in complexity and optimal estimates are currently out of reach. However, three types of results are available. First of all, there are abstract theories that reduce spectral multiplier estimates to heat kernel information (e.g.~\cite{COSY16} and \cite{KrW18}). Secondly, in concrete situations (primarily on Lie groups), strong results can be obtained through a subtle understanding of the relevant wave equation (the literature is vast, see \cite{MMN23} and references therein). Finally, one can start from the algebraic structures underlying the context, and use operator algebraic approaches, often with a probabilistic flavour (see \cite{JMPX21} and references therein). This is required in the context of noncommutative geometry and useful even in commutative situations, but quite difficult to implement in practice.

In this paper, we work at the intersection of these three approaches to derive concrete spectral multiplier estimates for the quantum harmonic oscillator in the noncommutative geometric context of quantum Euclidean spaces, but also in the commutative context of the <<deformed>> Hermite operator with the presence of a magnetic field. Actually, we establish a result for the <<sum of the squares>> $A_1^2+\cdots+A_n^2$ of a \textit{suitable} family $(A_1,\ldots,A_n)$ of unbounded operators acting on a (noncommutative) $\L^p$-space, which encompasses both situations. We call these families $\Theta$-Weyl tuples, see Definition \ref{defi-Weyl-tuple} for a precise definition, which depends on a $d \times d$ real skew-symmetric matrix $\Theta$. Moreover, in general the <<sum of the squares>> admits a spectral gap and we investigate the shifted operator $A_1^2+\cdots+A_n^2-\alpha$, where $\alpha \geq 0$ is the value of the spectral gap.

Since the spectral multiplier problem becomes significantly more challenging in noncommutative settings, it is natural to seek transference methods that allow us to leverage well-established techniques from classical harmonic analysis. A key aspect of our approach is the development of an abstract framework that enables the use of a suitable form of the <<transference technique>> of Coifman and Weis \cite{CoW76}. Roughly speaking, a  generalization of the formulation of \cite[Theorem 3.15 p.~19]{CoW76} says that if a linear operator $T \co X \to X$ can be written as a suitable integral
\begin{equation}
\label{trans-intro}
T
=\int _{G} b(s) \pi_s \d s
\end{equation}
for a continuous uniformly bounded representation $\pi$ of an amenable locally compact group $G$ on a Banach space $X$ and some suitable function $b \co G \to \mathbb{C}$ 
then the norm of the operator $T$ can be estimated by that of the convolution operator $C_b \co \L^{p}(G,X) \to \L^{p}(G,X)$ by the function $b$ acting on the Bochner space $\L^{p}(G,X)$, where $1 \leq p < \infty$. In other words, a harmonic analysis problem in a challenging setting (e.g.~in a noncommutative context) can be turned into a vector-valued harmonic analysis problem in a commutative context (e.g.~the Euclidean setting). For instance, instead of tackling spectral estimates for an operator in a noncommutative setting directly, one can transfer the problem to a classical setting where harmonic analysis problems may be easier to solve, before transferring the results back. Vector-valued harmonic analysis is thoroughly investigated in the books \cite{HvNVW16}, \cite{HvNVW18} and \cite{HvNVW23}. To be precise, note that, in our approach, we use a projective representation instead of an ordinary representation, and a convolution twisted  by a 2-cocycle instead of the classical convolution.

Although transference is a very natural idea from the perspective of representation theory, one major difficulty is to obtain an explicit decomposition as in \eqref{trans-intro}. In this paper, we construct a projective representation for each $\Theta$-Weyl tuple $(A_1,\ldots,A_n)$. Moreover, we define a universal family of operators, called universal $\Theta$-Weyl tuple, which allows us to reduce to this case via the transference method. Incidentally, the vector-valued harmonic analysis problems arising in this setting motivate the development of new techniques related to $R$-boundedness and square-max decompositions, allowing us, in the process, to even refine in Corollary \ref{cor-heat-R-bounded} and Example \ref{example-Lap-new} classical estimates for the Laplacian on $\R^n$ from the standard literature.

A transference approach was recently developed in \cite{NPS23} for the quintessential example of the Hermite operator defined in \eqref{Hermite-operator}. In this setting, the sum of squares associated with the universal Weyl tuple corresponds to the twisted Laplacian, while the group $G$ is identified with the Heisenberg group. The present paper can be seen as a far-reaching generalization of this particular case, extending the use of transference methods to non-trivial noncommutative geometry contexts. Specifically, we move beyond the Heisenberg group by replacing it with 2-step stratified Lie groups when $\Theta \not=0$, thus broadening the applicability of these techniques, including the context of quantum Euclidean spaces.

Quantum Euclidean spaces $\R^d_\Theta$, also known as Moyal-Groenewold planes \cite{Gro46}, \cite{Moy49} or simply Moyal planes (although the name Moyal-Groenewold $d$-spaces is undeniably more suitable if $d>2$), provide fundamental examples of \textit{noncompact} noncommutative (spin) manifolds as observed initially in \cite{GGISV04}. Their spectral action was computed in \cite{GaI05}. These spaces are also used as leading examples of quantum locally compact metric spaces as first noticed in \cite{Lat13}. They are locally compact counterparts of quantum tori and admit several equivalent definitions. Initially, in \cite{Gro46}, \cite{Moy49}, these entities were interpreted as noncommutative deformations of classical Euclidean spaces $\R^d$. From the more general point of view of Rieffel's deformation quantization theory, as described in the memoir \cite{Rie93}, one can describe $\R^d_\Theta$ using the algebra of classical smooth functions equipped with the Moyal star product 
\begin{equation}
\label{star-product}
(f \star_\Theta g)(x)
\ov{\mathrm{def}}{=} \frac{1}{(2\pi)^d} \int_{\R^d}\int_{\R^d} f(x+\tfrac{1}{2}\Theta u)g(x+s) e^{-\i \langle u,s \rangle} \d s \d u, \quad f,g \in \cal{S}(\R^d),
\end{equation}
which is a deformation of the classical product of functions. Here $\Theta$ denotes a $d\times d$ real skew-symmetric matrix. The relation between this deformation quantization and the operator-theoretic approach to quantum mechanics is given by the Weyl calculus, which we generalize in this paper in Definition \ref{defi-Weyl-calculus}.  

It is worth noting that the Moyal star product is associative and noncommutative when $\Theta \not=0$. For further details, we refer to \cite{GrV88} and \cite{VaG88}, where this product was extended beyond Schwartz functions to suitable tempered distributions, thus broadening its applicability.
This product enabled Moyal, in \cite{Moy49}, to develop an alternative formulation of quantum mechanics in phase space for some classes of phase spaces. In this formulation, the states of a quantum system are represented by their Wigner functions, which encode quantum information in terms of phase-space distributions.

A key aspect of this approach is that quantum observables retain their classical form as functions on phase space, but their multiplication follows the Moyal star product instead of the usual commutative product. This deformed product introduces quantum effects while preserving the overall structure of classical mechanics in phase space.

The expectation value of an observable in a given state is computed as an integral over phase space, involving the product of the Wigner function and the phase-space function representing the observable. This integral formulation provides an alternative to the traditional Hilbert space formalism of quantum mechanics, offering insights into quantum-classical correspondence and quantum decoherence.

For a more detailed presentation of this approach, we refer to \cite{CFZ14}, \cite{Gad95} and \cite{Mar21}.  For a historical perspective, see \cite{Hil21}. Additionally, this phase-space formulation has found applications in various domains, including quantum optics and  semiclassical approximations.

Several analytic and geometric aspects of quantum Euclidean spaces were investigated in \cite{SuZ18} (Connes integration formula), \cite{LSZ20} (Cwikel estimates), \cite{MSX20} (quantum differentiability), \cite{GJP21} \cite{MSZ19} (pseudo-differential calculus) and \cite{HWW22} (Fourier restriction estimates). We also refer to \cite{BrS11} and \cite{NeS98} for the construction of instantons in the case $d=4$ and to \cite{CaW10} and \cite{CDMW11} for investigations of some properties of the spectral distance. Finally, sigma-Model Solitons on these noncommutative spaces were studied in \cite{DLL15} using time-frequency analysis and Gabor analysis.

On a dual side, quantum Euclidean spaces can be viewed  as \textit{twisted} group von Neumann algebras deformed by 2-cocycles. More precisely, each of these spaces can be identified with the von Neumann algebra $\VN(\R^d,\sigma_\Theta)$ where $\sigma_\Theta \co \R^d \times \R^d \to \T$ is a 2-cocycle defined by
\begin{equation}
\label{def-cocycle-intro}
\sigma_\Theta(s,t) 
\ov{\mathrm{def}}{=} \e^{\frac{1}{2}\i\langle s, \Theta t\rangle} \quad s,t \in \R^d
\end{equation}
for a $d \times d$ real skew-symmetric matrix $\Theta$. It is a more algebraic way to define these structures, which allows one to introduce the associated noncommutative $\L^p$-spaces $\L^p(\R^d_\Theta)$. In particular, we can see the algebra $\R^d_\Theta$ as generated by $d$ unbounded self-adjoint operators $x_1,\ldots,x_d$ (<<noncommutative spatial variables>>) satisfying the commutation relations
\begin{equation}
\label{non-com-var}
\i[x_j, x_k] 
=\Theta_{jk}, \quad j,k \in \{1,\ldots,d\}. 
\end{equation}

We show here that an advantage of this approach is that it connects with the famous work \cite{Mac58} of Mackey, allowing us to exploit the <<Mackey group>> $\mathbb{H}_\Theta$, which is a nilpotent Lie group defined by $\mathbb{H}_\Theta \ov{\mathrm{def}}{=} \R^d \times \T$ and a \textit{twisted} group law. To the best of our knowledge, it is the first time where this group is used in this context. This opens the door to more extensive use of the heavy machinery of analysis on Lie groups in the study of the Moyal-Groenewold planes. 

Finally, it should be noted that exploring Moyal-Groenewold planes can be considered as a subset of the investigation of quantum phase spaces constructed from locally compact abelian groups equipped with a sufficiently regular 2-cocycle, see the recent \cite{BCLPY23}, \cite{FulG23} and references therein for more information. 

The interest for the analysis of these fuzzy structures is strengthened by their links with quantum field theory, quantum gravity, string theory and quantum Hall effect. Indeed, it is widely accepted that when we reach the Planck length scale, the continuous nature of space-time collapses, leading to a somewhat indistinct or noncommutative space-time. We refer to the surveys \cite{ABJ08}, \cite{BaQ06}, \cite{DoN01}, \cite{Riv07}, \cite{Sza03}, \cite{Wul06} and \cite{Wul19} for a presentation of these topics. 

In particular, Grosse and Wulkenhaar introduced in \cite{GrW05} a remarkable quantum field model $\phi^4$ on the four-dimensional Moyal plane $\R^4_\Theta$ that is renormalizable to all orders. Here, we suppose that the real skew-symmetric matrix $\Theta$ is invertible. The associated action
\begin{align}
\MoveEqLeft
\label{action-GW}S(\phi,m,\lambda,\Omega)
S(\phi,m,\lambda,\Omega)  \\
&\ov{\mathrm{def}}{=} \int \Big( \tfrac{1}{2} [\partial_i \phi  \star
\partial^i \phi] + \tfrac{\Omega^2}{2} (\tilde{x}_i \phi) \star (\tilde{x}^i \phi) + \tfrac{m^2}{2} \phi\star\phi
+ \tfrac{\lambda}{4!} \phi \star \phi\star \phi \star \phi\Big)(x) \d^4x  \nonumber       
\end{align}
is obtained by adding a <<harmonic term>> $\tfrac{\Omega^2}{2} (\tilde{x}_i \phi) \star (\tilde{x}^i \phi)$ to the naive action (this last one is obtained by replacing the usual commutative product by the Moyal star product introduced in \eqref{star-product} in the classical action), curing the UV/IR mixing curse of the naive model on $\R^4_\Theta$. Here the <<frequency parameter>> $\Omega$ is a positive dimensionless parameter, $m$ is the mass, $\lambda$ is the coupling constant and $\tilde{x}_i \ov{\mathrm{def}}{=} 2\sum_{j=1}^4[\Theta^{-1}]_{ij}x_j$. A particularly striking aspect of this $\phi^4$ model is that it exhibits more favourable behaviour in this noncommutative space compared to $\phi^4$ models on classical spaces. Indeed, it is shown in \cite{GrW04} that it has no Landau ghost contrarily to the case of the usual <<commutative>> $\phi^4$ theory. Finally, note that the model is covariant under the Langmann-Szabo duality transformation \cite{LaS02}, i.e.~we have 
\begin{equation*}
S(\phi,m, \lambda,\Omega)
=\Omega^2S(\phi,\tfrac{m}{\Omega},\tfrac{\lambda}{\Omega^2},\tfrac{1}{\Omega}).
\end{equation*} 
Consequently, the frequency parameter can be restricted to $\Omega \in [0,1]$. When $\Omega = 1$, the model is self-dual. 

\paragraph{H\"ormander functional calculus}
Recall the following famous theorem of Mikhlin \cite{Mik56} which is a cornerstone of harmonic analysis. This theorem establishes sufficient conditions for a function to act as a bounded Fourier multiplier on the space $\L^p(\R^d)$. Notably, the function may possess a singularity at 0. We refer also \cite[Proposition 3 p.~204]{Zim89} for a generalization for $\UMD$ vector-valued $\L^p$-spaces and to \cite[Corollary 13.2.15 p.~273]{HvNVW23} for an operator-valued version. In the next statement describing this result, $\cal{S}(\R^d)$ denotes the space of Schwartz functions on $\R^d$ and $\hat{g}$ is the Fourier transform of a function $g$ of $\cal{S}(\R^d)$.

\begin{thm}[Mikhlin]
\label{Th-Mikhlin}
Suppose that $f \in \L^\infty(\R^d)$. Assume also that for every distributional derivative $\partial_\alpha$ where $\alpha = (\alpha_1,\ldots,\alpha_d) \in \{0,1\}^d$, we have
\begin{equation}
\label{Mikhlin}
\sup_{\xi \neq 0}\, |\xi|^{|\alpha|} \left|\partial_\alpha f(\xi)\right|
<\infty, \quad \text{where } |\alpha| \ov{\mathrm{def}}{=}\alpha_1+\cdots+\alpha_d.
\end{equation}
If $1 < p < \infty$ then the operator $M_f \co \cal{S}(\R^d) \to \cal{S}'(\R^d)$, $g\mapsto \cal{F}^{-1}\big( f \hat{g} \big)$ extends to a bounded operator on the Banach space $\L^p(\R^d)$, where $\cal{F}^{-1}$ is the inverse Fourier transform.
\end{thm}

In 1960, H\"ormander significantly contributed in \cite[Theorem 2.5 p.~120]{Hor60} to the field of Fourier multiplier theorems by establishing what is currently recognized as the Mikhlin-H\"ormander multiplier theorem. H\"ormander showed that the conclusion of Mikhlin's theorem remains valid when we replace the Mikhlin condition \eqref{Mikhlin} by an estimate on an $\L^2$-average. We refer to \cite[Corollary 8.11 p.~165]{Duo01}, \cite[Theorem 6.2.7 p.~446]{Gra14a}, \cite[Theorem 7.9.5 p.~243]{Hor03} and \cite[p.~96]{Ste70}, for different proofs of this classical result.

\begin{thm}[H\"ormander]
\label{Hormander-theo}
Let $f \co \R^d \to \mathbb{C}$ be a bounded function of class $\mathrm{C}^{\lfloor \frac{d}{2}\rfloor +1}$ on $\R^d-\{0\}$. Suppose that for any multi-index $\alpha = (\alpha_1,\ldots,\alpha_d)$ of $\N^d$
\begin{equation}
\label{hor-1}
\sup_{R > 0} R^{|\alpha|} \bigg(\frac{1}{R^d}\int_{R < |\xi| < 2R}  \big| \partial_{\alpha} f(\xi) \big|^2  \d \xi \bigg)^{\frac{1}{2}}
< \infty, \quad 0 \leq |\alpha| \leq \lfloor \tfrac{d}{2}\rfloor+1.
\end{equation} 
If $1 < p < \infty$ then the operator $M_f \co \cal{S}(\R^d) \to \cal{S}'(\R^d)$, $g\mapsto \cal{F}^{-1}\big( f \hat{g} \big)$ extends to a bounded operator on the Banach space $\L^p(\R^d)$, where $\cal{F}^{-1}$ is the inverse Fourier transform.
\end{thm} 

We refer to \cite{Hyt04} for a nice discussion on the difference between these theorems. It is worth noting that the assumption \eqref{hor-1} is more general than the condition 
\begin{equation*}
\label{}
|\xi|^{|\alpha|} |\partial_\alpha f(\xi)| \leq C
\end{equation*}
for all $\alpha \in \N^d$ with $0 \leq |\alpha| \leq \lfloor \frac{d}{2}\rfloor +1$, which is often unfortunately referred to as the Mikhlin condition in the literature, as in the short historical account \cite[Section 1]{Gra21}. A recurring theme in harmonic analysis involves providing proofs for various adaptations of these two theorems, as discussed in the following papers 
\cite{AnL86}, \cite{Ank90}, \cite{CGPT23} (Schur multipliers), \cite{CGPT23b} (group von Neumann algebras), \cite{GJP21}, \cite{Ion02} (symmetric spaces), \cite{JMP14} and \cite{JMP18}.

Theorem \ref{Hormander-theo} implies the following result for the Laplacian $\Delta$ on $\R^d$.

\begin{cor}[H\"ormander]
\label{Hormander-theorem}
Let $f \co \R_+ \to \mathbb{C}$ be a bounded function of class $\mathrm{C}^{\lfloor \frac{d}{2}\rfloor +1}$ on $\R_+^*$. Suppose that
\begin{equation*}
\label{hor-2-89}
\sup_{R> 0} R^{2k-1} \int_{R}^{2R}  \big| f^{(k)}(s) \big|^2  \d s 
< \infty, 
\quad 0 \leq k \leq \lfloor \tfrac{d}{2}\rfloor+1.
\end{equation*} 
If $1 < p < \infty$ then the operator $f(-\Delta)\co \cal{S}(\R^d) \to \cal{S}'(\R^d)$, $g\mapsto \cal{F}^{-1}\big( f(|x|^2) \hat{g} \big)$ extends to a bounded operator on the Banach space $\L^p(\R^d)$.
\end{cor} 

If $s >0$, using  the usual Sobolev norm 
\begin{equation}
\label{Sobolev-norm}
\norm{f}_{\W^{s,2}(\R)} 
\ov{\mathrm{def}}{=} \bnorm{(1+t^2)^{\frac{s}{2}}\hat{f}(t)}_{\L^2(\R)}
\end{equation}
and the inhomogeneous Sobolev space
\begin{equation}
\label{Sobolev-space}
\W^{s,2}(\R) 
\ov{\mathrm{def}}{=} \{f \in \cal{S}'(\R) : \norm{f}_{\W^{s,2}(\R) } < \infty \}
\end{equation}
it was later observed (though it is historically difficult to determine exactly who first noted this), that this theorem could be generalized to encompass all measurable functions $f$ that meet the criterion
\begin{equation}
\label{hor-2}
\norm{f}_{\Hor^s_2(\R_+^*)}
\ov{\mathrm{def}}{=} \sup_{t > 0} \norm{x \mapsto \eta(x) \cdot f(tx) }_{\W^{s,2}(\R)} 
< \infty,
\end{equation} 
where $\eta$ is any non-zero positive function of class $\mathrm{C}^\infty$ with compact support in $(0,\infty)$, and $s > \frac{d}{2}$. 
It is immediate from \eqref{hor-2} that for any $r > 0$, we have $\norm{f}_{\Hor^s_2(\R^*_+)} = \norm{f(r\,\cdot)}_{\Hor^s_2(\R^*_+)}$.
We refer to \cite[Theorem 6.24 p.~79]{Pel18} and to \cite[Theorem 8.10 p.~164]{Duo01}, \cite[Theorem 7.5.5 p.~565]{Gra14b} for a proof of slight variants and to \cite[Theorem 4.7]{CT77}, \cite[p.~63]{Kri09}, \cite[Lemma 1]{Lit65}  for related information.

Finally, note that the norm of radial Fourier multipliers which are bounded on the space $\L^p(\R^d)$ was described more recently in the paper \cite{HNS11} if $d \geq 4$ for $1 < p < \frac{2d-2}{d+1}$. When the symbol $f$ of these radial Fourier multipliers $M_{f}$ is compactly supported in the fixed dyadic annulus $\{ \xi : \frac{1}{2} < |\xi| < 2 \}$, their result takes a particularly satisfying form, namely $f$ induces a bounded operator $M_f \co \L^{p}(\R^{d}) \to \L^{p}(\R^{d})$ if and only if $\hat{f} \in \L^{p}(\R^{d})$. For $d=4$, this was extended in the paper \cite{Cla18} to the range $1 < p < \frac{36}{29}$. Note that proving such a result for $1 <p < \frac{2d}{d+1}$ would have far reaching consequences (implying the Bochner-Riesz conjecture). These results on radial Fourier multipliers are much stronger than Theorem \ref{Hormander-theo}.

H\"ormander's theorem gives results about the functional calculus for the operator $-\Delta$, i.e. for the generator of the heat semigroup on $\R^d$. A traditional broadening of this context involves replacing this semigroup by more general symmetric contraction semigroups of operators acting on classical or noncommutative $\L^p$-spaces. If $(T_t)_{t \geq 0}$ is such a semigroup with generator $-A$ acting on a noncommutative $\L^p$-space $\L^p(\cal{M})$, then the operator $A$ is positive and self-adjoint on $\L^2(\cal{M})$. The spectral theorem thus allows one to obtain a functional calculus $\scr{L}^\infty(\R_{+}) \to \B(\L^2(\cal{M}))$, $f \mapsto f(A)$ where the operator $f(A) \co \L^2(\cal{M}) \to \L^2(\cal{M})$ is well-defined and bounded on the complex Hilbert space $\L^2(\cal{M})$. If $1 < p < \infty$, the function $f$ is said to be an $\L^p$-spectral multiplier if $f(A)$ extends to a bounded operator on $\L^p(\cal{M})$. 
 Giving a version of H\"ormander's theorem in this context entails demonstrating that $f$ is an $\L^p$-spectral multiplier if it fulfils \eqref{hor-2} for some $s > 0$. This often requires additional or different assumptions. 

For instance, Christ, Mauceri and Meda proved in \cite[Theorem 1 p.~74]{Chr91} and \cite{MaM90} that the sub-Laplacian of a stratified Lie group $G$ satisfies a H\"ormander theorem for any $s > \frac{Q}{2}$, where $Q$ is the homogeneous dimension of $G$. An alternative proof of this result has been given by Sikora in \cite{Sik92}. Although the dimensional parameter $\frac{Q}{2}$ is very natural in this context, it is not always the most effective choice. Indeed, Hebisch significantly advanced this area in his work \cite[Theorem 1.1 p.~233]{Heb93}, where he demonstrated an improved condition for sub-Laplacians on products of Heisenberg-type groups. This development builds on and enhances the earlier work of M\"uller and Stein \cite{MuS94}, who studied direct products of Heisenberg groups with Euclidean groups. Specifically, Hebisch proved that the sufficient condition $s> \frac{Q}{2}$ can be more effectively replaced by $s>\frac{d}{2}$, where $d$ represents the topological dimension of the corresponding Lie group. Analogous results were proven by Cowling and Sikora in \cite[Theorem 1.1 p.~2]{CoS01} for the compact Lie group $\mathrm{SU}(2)$ and by Martini and M\"uller in \cite{MaM13} for the 6-dimensional free 2-step nilpotent Lie group $N_{3,2}$ on three generators. See also \cite{Mar15}, \cite{MaM14b} and \cite{Nie23} for similar results for other classes of groups. Actually, as proved in the paper \cite[Theorem 3]{MaM16}, it is known that one cannot do better than $\frac{d}{2}$ for any 2-step stratified Lie group of topological dimension $d$. Finally, it is worth noting that the approach of Duong in \cite[Theorem 3]{Duo96} for sub-Laplacians on stratified Lie groups relates the $\L^1$-norm of the complex time heat kernels of a sub-Laplacian on a stratified Lie group with the order of regularity $s$ needed in the H\"ormander multiplier theorem. 

On a \textit{compact} smooth Riemannian manifold $M$ of dimension $d \geq 2$ without boundary, a second order self-adjoint positive elliptic differential operator $A$ also satisfies a H\"ormander theorem for any $s > \frac{d}{2}$, by \cite[Theorem 3.1. p.~723]{SeS89} (see also \cite[Theorem 5.3.1 p.~155]{Sog17}). See also \cite[p.~469]{DOS02} and \cite{Xu07} for other proofs. It is worth noting that, in the recent paper \cite{MMN23}, it is proved that, for a general class of smooth second-order differential operators associated with a sub-Riemannian structure on a smooth $d$-dimensional manifold, regularity of order $s > \frac{d}{2}$ is necessary for a H\"ormander theorem to hold.

The literature on spectral multiplier theorems is even broader, and includes, in particular,  \cite{Ale94} (sub-Laplacians on Lie groups of polynomial growth), \cite{Blu03}, \cite{DOS02}, \cite{KuU15} (homogeneous metric measure spaces), \cite{CCMS17} (Kohn Laplacian on some spheres), \cite{GCM+01,Har19} (Ornstein-Uhlenbeck operator), \cite{Heb90} (Schr\"odinger operators), \cite{MaM14a} (Grushin operators),  \cite{Med90} (negative generators of symmetric semigroups), \cite{MPR15} (Hodge Laplacian on the Heisenberg group), \cite{CGHM94}, \cite{MPR07} (Laplacians on groups of exponential growth) \cite{Sha21} \cite{Sha22} (abstract group generators with finite speed of propagation and Sobolev embedding properties) and \cite{SWX23} (partial harmonic oscillator).

The present paper relies on two abstract approaches to spectral multiplier theorems: a characterisation using $R$-bounds in the spirit of $\H^{\infty}$ functional calculus (and the aforementioned approach of Duong) proven in \cite{KrW18} (and generalized in \cite{HaP23}), and a transference result in the style of \cite{CoW76} and \cite{NPS23}.

\paragraph{The harmonic oscillator on the Moyal-Groenewold plane $\R^d_\Theta$}

Our spectral multiplier theorem is primarily designed for the harmonic oscillator on quantum Euclidean spaces (introduced in \cite{Gro46}), which can be seen as a natural generalization of the classical harmonic oscillator when the space is quantized. This operator is formally defined by $-\sum_{k=1}^d \partial_k^2+\sum_{k=1}^d X_k^2$ where $X_k$ is the left multiplication by the <<spatial variable>> $x_k$ and where $\partial_k$ is some suitable <<quantum partial derivation>> acting on the Banach space $\L^p(\R^d_\Theta)$ and satisfying 
\begin{equation}
\label{}
\i[X_j, X_k] 
=\Theta_{jk}, \quad[\partial_j,\partial_k]
=0
\quad \text{and} \quad 
[\partial_j,X_k]
=\delta_{j=k}, \quad 1 \leq j,k\leq d.
\end{equation}
We obtain in Theorem \ref{thm-Hormander-calculus-for-quantum-euclidean-space} a bounded H\"ormander functional calculus of order $s > d + \frac12$ on the space $\L^p(\R^d_\theta)$ for a shifted version of the harmonic oscillator. Note that our approach can be used for a variant of this operator when <<magnetic fields>> are present. We refer to \cite{NeP20} for the case of the classical (euclidean) harmonic oscillator. Operators with the same algebraic structure (i.e.~sums of squares of generators of a Heisenberg algebra), acting on UMD Banach lattices, have been studied in the paper \cite{NPS23}. The case of the operator $\sum_{k=1}^d X_k^2$ acting on quantum Euclidean spaces was considered in the PhD thesis of Sharma \cite{Sha22}, and motivated our work on this topic. Going from $\sum_{k=1}^d X_k^2$ to the true harmonic oscillator 
\begin{equation*}
\label{}
-\sum_{k=1}^d \partial_k^2+\sum_{k=1}^d X_k^2
\end{equation*}
was the motivation to generalize the Weyl calculus studied in \cite{NeP20} to the $\Theta$-Weyl calculus studied here.

Note also that  it is shown in \cite{GRV06} that the propagator of the Grosse-Wulkenhaar model defined by \eqref{action-GW} is the kernel of the inverse of a shifted <<classical>> harmonic oscillator $-\Delta +\Omega^2\tilde{x} + m^2$, allowing a second proof in \cite{GMR06} of the renormalizability of this model to all orders. Recall that the <<frequency parameter>> $\Omega$ is a positive dimensionless parameter. This operator is also connected to noncommutative geometry, see \cite{GaW13}, where a related spectral triple is introduced. Our result on the harmonic oscillator $-\sum_{k=1}^d \partial_k^2+\sum_{k=1}^d X_k^2$ (which is different from the operator $-\Delta +\Omega^2\tilde{x} + m^2$) combined with the result \cite[Lemma 3.5 p.~28]{JMX06} allows us to obtain the (uniform) boundedness of the $\H^\infty$ functional calculus of the operator 
\begin{equation*}
\label{}
-\sum_{k=1}^d \partial_k^2+\sum_{k=1}^d  X_k^2 + m^2
\end{equation*}
for any <<mass>> $m > 0$.

\paragraph{Application to maximal regularity of evolution equations}
Recall that a H\"ormander type spectral multiplier theorem implies the boundedness of the $\H^\infty(\Sigma_\omega)$ functional calculus \cite{Haa06}, \cite[Chapter 10]{HvNVW18} for any angle $\omega > 0$, where 
\begin{equation*}
\label{}
\Sigma_\omega \ov{\mathrm{def}}{=}  \{ z \in \mathbb{C}^*: |\arg z| <\omega \}
\end{equation*}
is the open sector of angle $2\omega$. Our result thus implies that the harmonic oscillator on $\R^{d}_{\Theta}$ has maximal regularity. Let us recall classical concepts of maximal regularity, and refer to \cite{Den21}, \cite{HvNVW23}, \cite{KuW04} (see also \cite{Are04}) and the references therein for a much more complete account. If $A$ is a closed densely defined linear operator on a Banach space $X$, we can consider the following inhomogeneous autonomous Cauchy problem
\begin{equation}
\label{Cauchy-problem}
\begin{cases}
y'(t)+A(y(t))=f(t), \quad t \in ]0,T[ \\
y(0)=x_0
\end{cases},
\end{equation}
where $f \co ]0,T[ \to X$ is a function whose equivalence class is an element of the Bochner space $\L^1(]0,T[,X)$ with $0 < T  <\infty$ and $x_0 \in X$.

If $-A$ is the generator of a strongly continuous semigroup $(T_t)_{t \geq 0}$ on a Banach space $X$ then by \cite[Proposition 3.1.16 p.~118]{ABHN11} or \cite[Theorem G.3.2 p.~531]{HvNVW18} the Cauchy problem \eqref{Cauchy-problem} admits a unique \textit{mild} solution given by 
\begin{equation}
\label{sol-Cauchy}
y(t) 
= T_t(x_0) +\int_{0}^{t} T_{t-s}(f(s)) \d  s, \quad 0 \leq t < T. 
\end{equation}

Assume that $1 < p < \infty$. If the right-hand side $f$ of \eqref{Cauchy-problem} belongs to the Banach space $\L^p(]0,T[,X)$, <<optimal regularity>> would imply that both $y'$ and $A y$ should also belong to the space $\L^p(]0,T[,X)$. Following \cite[Definition 17.2.4 p.~577]{HvNVW23} and \cite[Proposition 17.2.11 p.~581]{HvNVW23}, we define the operator $A$ to have maximal $\L^p$-regularity on $]0,T[$ if for any vector-valued function $f \in \L^p(]0,T[,X)$ the mild solution of the Cauchy problem \eqref{Cauchy-problem} with $x_0=0$  is a <<strong $\L^p$-solution>>, meaning that the following two properties hold:
\begin{enumerate}
	\item $y$ takes values in $\dom A$ almost everywhere,
	\item the class of the function $t \mapsto A(y(t))$ belongs to the space $\L^p(]0,T[,X)$.
\end{enumerate}
 
It is a commonly known fact \cite[Theorem 17.2.15 p.~586]{HvNVW23} that this property implies that the operator $-A$ generates an analytic semigroup on $X$. Note also that by \cite[Theorem 17.2.31 p.~604]{HvNVW23}
this property is independent of $p$. A classical result \cite[Corollary 17.3.6 p.~631]{HvNVW23} (combined with \cite[Lemma 17.2.16 p.~590]{HvNVW23}) says that an operator admitting a bounded $\H^\infty(\Sigma_\theta)$ functional calculus of angle $0<\theta <\frac{\pi}{2}$ on a $\UMD$ Banach space $X$ has maximal $\L^p$-regularity on $]0,T[$. Consequently, our main result implies that the harmonic oscillator on the Moyal-Groenewold plane has $\L^p$-maximal regularity on $]0,T[$ on the Banach space $\L^q(\R^d_\Theta)$ for any $1 < p,q < \infty$. This result is in the continuity of \cite{McD24}, where a preliminary theory of partial differential equations on this structure is presented, and of \cite{GJP21}, where singular integral theory and pseudodifferential calculus are thoroughly investigated. We also refer to \cite{RST23} for more harmonic analysis results in $\L^q(\R^d_\Theta)$.

An immediate consequence of maximal $\L^p$-regularity is the following estimate \cite[p.~577]{HvNVW23}
\begin{equation}
\label{}
 \norm{Ay}_{\L^p([0,T[ ,X)}
 \leq C\norm{f}_{\L^p([0,T[ ,X)},
\end{equation}
for a positive constant $C \geq 0$ that does not depend on $f$. This bound plays a crucial role in the analysis of quasilinear problems, as it enables an effective approach via linearization techniques and the contraction mapping principle. For a comprehensive discussion on these applications, we refer to \cite[Chapter 18]{HvNVW23}.

\paragraph{Approach and structure of the paper}

In Section \ref{sec-prelim}, we collect background information in operator theory, functional calculus theory, and noncommutative analysis.

To describe the set up of our approach in Section \ref{Sec-Weyl1}, let us first recall that the Stone-von Neumann theorem \cite[Theorem 14.8 p.~286]{Hal13}, \cite[Theorem 11.24 p.~501]{Mor17} states that, if $(A_1,\ldots,A_n)$ and $(B_1,\ldots,B_n)$ are self-adjoint operators acting irreducibly on a complex Hilbert space $H$, and satisfying the exponentiated commutation relations
\begin{align}
e^{\i sA_j} e^{\i tA_k} & = e^{\i tA_k} e^{\i sA_j} \label{first1} \\
 e^{\i s B_j} e^{\i t B_k} & = e^{\i t B_k} e^{\i s B_j}  , \quad \quad \qquad s,t \in \R,\ j,k \in \{1,\ldots,n\} \label{second1} \\
e^{\i t A_j} e^{\i sB_k} &= e^{-\i ts \delta_{j=k}} e^{\i sB_k} e^{\i tA_j} \label{third},
\end{align}
then the operators are unitarily equivalent to the <<position operators>> $(X_1,\ldots,X_n)$ and <<momentum operators>> $(Q_1,\ldots,Q_n)$ defined on $\L^2(\R^d)$ by
\begin{align*}
(Q_jf)(x) = x_jf(x),
\quad (P_jf)(x) = \tfrac{1}{\i} \partial_j f(x), 
\quad x \in \R^d.
\end{align*}

In \cite[Definition 3.1 p.~262]{NeP20}, the authors investigated Weyl pairs $(A_1,\ldots,A_n)$ and $(B_1,\ldots,B_n)$ of closed densely defined operators acting on a complex \textit{Banach space }$X$ satisfying the previous relations. In Section \ref{Sec-Weyl1}, we generalize this theory by introducing (see Definition \ref{defi-Weyl-tuple}) the notion of $\Theta$-Weyl tuple $A=(A_1,\ldots,A_n)$ of operators acting on a Banach space $X$. Loosely speaking, if $\Theta \in \M_n(\R)$ denotes a fixed skew-symmetric matrix, we say that $(A_1,\ldots,A_n)$ is a $\Theta$-Weyl tuple if we have some associated groups of operators which satisfy the following commutation relations
\begin{equation}
\label{Weyl-tuple-intro}
e^{\i t A_j} e^{\i s A_k} 
= e^{\i \Theta_{jk} ts} e^{\i s A_k} e^{\i t A_j}, \quad j,k \in \{1,\ldots,n\},\: s,t \in \R. 
\end{equation}
We show in Example \ref{ex-Hall} how to associate to a $d$-dimensional Weyl pair as in \eqref{first1} a $\Theta$-Weyl tuple for some suitable matrix $\Theta$. In Lemma \ref{lem-Weyl-product}, we construct a projective representation $\R^n \to \B(X)$, $t \mapsto e^{\i t \cdot A}$ of $\R^n$ on $X$ associated with each $\Theta$-Weyl tuple $(A_1,\ldots,A_n)$. This projective representation allows us to define the associated notion of $\Theta$-Weyl functional calculus $a \mapsto a(A)$ by the formula \eqref{Weyl-calculus}.

In Section \ref{sec-twisted}, we prove a twisted transference result (Proposition \ref{Prop matricial Transference}) that is a generalization of the transference principle popularized in \cite{CoW76}. Note that this result can also be used in the more general setting described in \cite{FulG23} which encompasses the context of \cite{LeL22} and \cite{Lei11}. It estimates the norms of operators $a(A)$ defined through the $\Theta$-Weyl calculus by the norm of some twisted convolution operator. Finally, we introduce a specific concrete example of a $\Theta$-Weyl tuple $A=(A_1,\ldots,A_n)$ acting on the space $\L^p(\R^n)$, which we call the universal $\Theta$-Weyl tuple $A_\univ^\Theta$. This object plays a fundamental role in the sequel as we transfer properties of this universal $\Theta$-Weyl tuple to \textit{arbitrary} $\Theta$-Weyl tuples.

We focus on operators $a(A)$ that arise as spectral multipliers associated with a strongly continuous semigroup $(T_t)_{t \geq 0}$ generated by the abstract <<sum of squares>> $\sum_{k=1}^n A_k^2$. We give some preliminary information about this semigroup in Section \ref{subsec-semigroup-A} by connecting it to the theory of 2-step nilpotent Lie groups, using the <<Mackey group>> $\mathbb{H}_\Theta$ (as well as to some semigroups of operators previously investigated by Robinson). We focus on the case where $A_\univ^\Theta$ is the universal $\Theta$-Weyl tuple defined in \eqref{A-univ}. We are able to show in Proposition \ref{prop-HI-calculus-non-shifted}, that the operator $\sum_{k=1}^n A_k^2$ has a bounded $\HI(\Sigma_\omega)$ functional calculus on the Bochner space $\L^p(\R^n,S^p)$ for any angle $\omega > \frac{\pi}{2}$. 
An interesting technical aspect of our paper at this point is that we use contractively regular semigroups that are not necessarily positive. Prior to our work, the application of these semigroups in analysis seemed to be quite limited. In this paper, they naturally emerge in the context of twisted convolutions.

In Section \ref{subsec-shifted-HI-calculus}, and in subsequent sections of the paper, we need to examine in depth the semigroup and its integral description. We show that there exists a tensor decomposition of this semigroup which essentially reduces the difficulties to the case of the matrix $\Theta=\begin{pmatrix} 
0 & 1 \\ 
-1 & 0 
\end{pmatrix}$. It turns out that the spectrum of $\sum_{k=1}^n A_k^2$ is contained in $[\alpha,\infty)$ for some constant $\alpha \geq 0$ depending on the eigenvalues of the skew-symmetric matrix $\Theta$. We show even more, namely that the operator $\sum_{k=1}^n A_k^2 - \alpha \Id_{\L^p(\R^n,S^p)}$ has a bounded $\HI(\Sigma_\omega)$ functional calculus, for any angle $\omega \in (\frac{\pi}{2},\pi)$, see Corollary \ref{Cor-512}.

The next goal is then to enhance the boundedness of $\HI$ functional calculus of the operator $\A - \alpha \Id_{\L^p(\R^n,S^p)}$ to a bounded H\"ormander functional calculus with respect to (non-holomorphic) spectral multipliers defined on the positive half-axis. For this task, we proceed by showing $R$-boundedness of the semigroup generated by this operator to the maximal angle $\frac{\pi}{2}$. In doing so, we control the dependency on the angle as required in \cite{KrW18}. Such $R$-boundedness is usually proved with estimates of the radially decreasing kernel of the semigroup together with boundedness of the maximal Hardy-Littlewood operator. In the short Section \ref{subsec-maximal-estimates}, we explain how to use the boundedness of this maximal operator on the Bochner space $\L^p(\R^n,S^p)$ in our context, preparing the ground for the Section \ref{subsec-R-sectoriality}.

In Section \ref{sec-R-boundedness}, we prove a stability result for $R$-boundedness under the operation of tensor product which will be useful in Section \ref{subsec-R-sectoriality}. In Section \ref{sec-square-max}, we connect $\L^p$-square-max decompositions of \cite{GJP17} to $R$-boundedness by showing, in Proposition \ref{prop-square-Rad-scalar}, that the existence of an $\L^p$-square-max decomposition leads to an $R$-bound.

Such an $\L^p$-square-max decomposition is critical to our approach, and gives, in Section \ref{subsec-R-sectoriality}, the $R$-boundedness of the complex time semigroup (indeed a more general result) generated by the shifted <<sum of squares>> on $\L^p(\R^n)$ for the universal $\Theta$-Weyl tuple $A_\univ^\Theta$. A more elaborate argument relying on similar ideas gives $R$-boundedness of this complex time semigroup on the Bochner space $\L^p(\R^n,S^p)$. In Theorem \ref{thm-514}, we deduce the desired boundedness of the H\"ormander calculus on the Bochner space $\L^p(\R^n,S^p)$ by using the sufficient abstract condition of \cite{KrW18}.

In Section \ref{subsec-functional-calculus-general-Theta-Weyl-tuple}, we return to general $\Theta$-Weyl tuples of operators acting on noncommutative $\L^p$-spaces. We are able to use the transference result from Corollary \ref{cor-transference} to transfer the boundedness of the H\"ormander functional calculus of the <<sum of squares>> $\sum_{j=1}^n A_j^2$ of the universal $\Theta$-Weyl tuple $A_{\univ}^\Theta$ from Theorem \ref{thm-514} to the case of the H\"ormander functional calculus of the sum of squares of a \textit{general} $\Theta$-Weyl tuple (see Corollary \ref{cor-Hormander-calculus-for-Theta-Weyl-tuple-on-LpM}). In a final Proposition, we briefly discuss the case of $\Theta$-Weyl tuples acting on UMD Banach spaces, generalizing the case of noncommutative $\L^p$-spaces, in which many results of the whole section still hold, but for which, for now, we are only able to show the boundedness of the $\HI(\Sigma_\omega)$ functional calculus of the associated (shifted) harmonic oscillator for an angle $\omega > \frac{\pi}{2}$.
We finish the section by giving applications to Schur multipliers in Example \ref{rem-Schur-multipliers}. In Section \ref{sec-Bochner-Riesz-means}, we will obtain consequences for the Bochner-Riesz means. 

We are finally able to apply our theory to the Moyal-Groenewold plane $\R^d_\Theta$ (for any skew-symmetric matrix $\Theta \in \M_d(\R)$) in Section \ref{sec-quantum-harmonic-oscillator}. Noncommutative spatial variables $X_k$ and quantum derivations $\partial_k$ are rigorously introduced by considering the associated groups of operators. In Proposition \ref{prop-quantum-commutation-rules}, we observe that these semigroups form a $\Theta$-Weyl tuple where $
\Theta 
\ov{\mathrm{def}}{=} \begin{pmatrix} 
\theta & - \Id \\ 
\Id & 0 
\end{pmatrix}$. This allows us to define the harmonic oscillator $-\sum_{k=1}^d \partial_k^2+\sum_{k=1}^d X_k^2$ and to apply the theory developed in the previous sections to conclude that a suitable shifted version $-\sum_{k=1}^d \partial_k^2+\sum_{k=1}^d X_k^2 - \alpha \Id$ admits a H\"ormander functional calculus.



\vspace{0.2cm}

%

\section{Preliminaries}
\label{sec-prelim}

\paragraph{Gaussian integrals} 
We will use the two following classical integrals. If $a>0$, we have
\begin{equation}
\label{Gaussian-integral}
\int_{-\infty}^{\infty} \e^{-a x^2} \d x
=\frac{\sqrt{\pi}}{\sqrt{a}}
\quad\text{and} \quad
\int_{-\infty}^{\infty} x^2\e^{-a x^2} \d x
=\frac{\sqrt{\pi}}{2a^{3/2}}.
\end{equation}

\paragraph{Bochner spaces}
Over the past few decades, a key focus in Banach space theory has been exploring the extent to which results and methods applicable to scalar-valued functions can be extended to their vector-valued counterparts. In this context, for any measure space $(\Omega,\mu)$, any Banach space $X$ and any $1 \leq p < \infty$, we use the Bochner space
\begin{equation*}
\label{Bochner-space}
\L^p(\Omega,X) 
\ov{\mathrm{def}}{=} \bigg\{ \text{classes of strongly measurable } f \co \Omega \to X : \int_\Omega \norm{f(t)}_X^p \d \mu(t) <\infty \bigg\},
\end{equation*}
equipped with the norm 
\begin{equation}
\label{norm-Bochner}
\norm{f}_{\L^p(\Omega,X) } \ov{\mathrm{def}}{=} \bigg(\int_\Omega \norm{f(t)}_X^p \d\mu(t)\bigg)^{\frac{1}{p}}.
\end{equation}
We refer to the books \cite{ABHN11} and \cite{HvNVW16} for more information.

We proceed with additional background on operator theory.

\paragraph{Bounded holomorphic semigroups} We refer to \cite{Paz83} for more information on semigroups of operators. Suppose that $0 < \theta \leq \frac{\pi}{2}$. A strongly continuous semigroup $(T_t)_{t \geq 0}$ of operators acting on a Banach space $X$ is called a bounded holomorphic semigroup of angle $\theta$ if it has a bounded holomorphic extension to $\Sigma_{\theta'} \ov{\mathrm{def}}{=} \{ z \in \mathbb{C}^*: |\arg z| < \theta' \}$ for each $0 <\theta'<\theta$. The following characterization is \cite[Theorem 3.7.19 p.~157]{ABHN11} and \cite[Theorem G.5.3 p.~541]{HvNVW18}.

\begin{prop}
\label{Prop-caract-analytic}
A strongly continuous bounded semigroup $(T_t)_{t \geq 0}$ of operators acting on a Banach space $X$ with generator $-A$ is a bounded holomorphic semigroup if and only if $T_t(X) \subset \dom A$ for any $t > 0$ and 
\begin{equation*}
\label{}
\sup_{t > 0} t\norm{A T_t}_{X \to X} < \infty.
\end{equation*}
\end{prop}

\paragraph{Functional calculus of sectorial operators}
For any angle $\omega \in (0,\pi)$, we will use the open sector symmetric around the positive real half-axis with opening angle $2\omega$
\begin{equation*}
\label{def-sigma-omega}
\Sigma_{\omega} 
\ov{\mathrm{def}}{=} \big\{ z \in \C \backslash \{ 0 \} : \: | \arg z | < \omega\big\},
\end{equation*}
see Figure 1. It will be convenient to set $\Sigma_0 \ov{\mathrm{def}}{=} (0,\infty)$. 

\begin{figure}[ht]
\begin{center}
\includegraphics[scale=0.4]{sector.jpg}
\begin{picture}(0,0)

\put(-70,38){ $\omega$}

\put(-45,20){$\Sigma_\omega$}
\end{picture}

Figure 1: open sector $\Sigma_\omega$ of angle $\omega$
\end{center}
\label{figure-sector}
\end{figure}

We refer to the books \cite{Haa06}, \cite{HvNVW18}, \cite{JMX06} for background on sectorial operators and their $\HI$ functional calculus, introduced in the fundamental papers \cite{CDMY96} and \cite{McI86}. Let $A \co \dom A \subset X \to X$ be a closed densely defined linear operator acting on a Banach space $X$. We say that $A$ is a sectorial operator of type $\omega \in (0,\pi)$ if its spectrum $\sigma(A)$ is a subset of the closed sector $\ovl{\Sigma_\omega}$ and if for any angle $\nu \in (\omega,\pi)$, the set 
\begin{equation}
\label{2Sectorial}
\Big\{zR(z,A) : z \in \mathbb{C} - \overline{\Sigma_\nu}\Big\}
\end{equation}
is bounded in the algebra $\B(X)$ of bounded operators acting on $X$\label{bounded-BX}, where $R(z,A) \ov{\mathrm{def}}{=} (z\Id-A)^{-1}$\label{resolvent-operator} is the resolvent operator. The operator $A$ is said to be sectorial if it sectorial operator of type $\omega$ for some angle $\omega \in (0,\pi)$. In this case, we can introduce the angle of sectoriality
\begin{equation}
\label{equ-sectorial-operator}
\omega_{\sec}(A) 
\ov{\mathrm{def}}{=} \inf\{ \omega \in (0,\pi) : A \textrm{ is sectorial of type $\omega$} \}.
\end{equation}

Recall the following result \cite[Example 10.1.2 p.~362]{HvNVW18}.

\begin{prop}
\label{prop-sec-pi-2}
If $-A$ is the generator of a bounded strongly continuous semigroup $(T_t)_{t \geq 0}$ on a Banach space $X$ then the operator $A$ is sectorial of type $\frac{\pi}{2}$, i.e.~$\omega_{\sec}(A) \leq \frac{\pi}{2}$.
\end{prop}

Moreover, by \cite[Example 10.1.3 p.~362]{HvNVW18}, $A$ is sectorial of type $< \frac{\pi}{2}$ if and only if $-A$ generates a bounded holomorphic strongly continuous semigroup. If $A$ is a sectorial operator on a \textit{reflexive} Banach space $X$, we have by \cite[Proposition 2.1.1 (h) p.~21]{Haa06} or \cite[Proposition 10.1.9 p.~367]{HvNVW18} a topological decomposition
\begin{equation}
\label{decompo-reflexive}
X
=\ker A \oplus \ovl{\Ran A}.
\end{equation}


For any angle $\theta \in (0,\pi)$, we consider the algebra $\H^{\infty}(\Sigma_\theta)$\label{algebra-Hinfty} of all bounded analytic functions $f \co  \Sigma_\theta\to \C$, equipped with the supremum norm 
\begin{equation}
\label{norm-Hinfty}
\norm{f}_{\H^{\infty}(\Sigma_\theta)}
\ov{\mathrm{def}}{=} \sup\bigl\{\vert f(z)\vert \, :\, z\in \Sigma_\theta\bigr\}.
\end{equation}
Let $\H^{\infty}_{0}(\Sigma_\theta)$\label{algebra-Hinfty0} be the subalgebra of bounded analytic functions $f \co \Sigma_\theta\to \C$ for which there exist $s,c>0$ such that 
\begin{equation*}
\label{ine-Hinfty0}
\vert f(z)\vert
\leq C\min\{|z|^s,|z|^{-s}\}, \quad z \in \Sigma_\theta,
\end{equation*} 
as discussed in \cite[Section 2.2]{Haa06}

\begin{example} \normalfont
\label{Example-Haase}
Let $\alpha$ and $\beta$ be complex numbers such that $0 < \Re \beta < \Re \alpha$. Then for any angle $\theta \in(0,\pi)$ the function $z \mapsto \frac{z^\beta}{(1+z)^\alpha}$ belongs  to the algebra $\H^{\infty}_{0}(\Sigma_\theta)$, as observed in \cite[Example 2.2.5 p.~29]{Haa06}.
\end{example}

Let $A$ be a sectorial operator acting on a Banach space $X$. Consider some angle $\theta \in (\omega_{\sec}(A), \pi)$ and a function $f \in \H^\infty_0(\Sigma_\theta)$. Following \cite[p.~30]{Haa06}, \cite[p.~5]{LM99} (see also \cite[p.~369]{HvNVW18}), for any angle $\nu \in (\omega_{\sec}(A),\theta)$ 
 we introduce the operator
\begin{equation}
\label{2CauchySec}
f(A)
\ov{\mathrm{def}}{=} \frac{1}{2\pi \i}\int_{\partial\Sigma_\nu} f(z) R(z,A) \d z,
\end{equation}
acting on $X$, with a Cauchy integral, where the boundary $\partial\Sigma_\nu$ is oriented counterclockwise. The sectoriality condition ensures that this integral is absolutely convergent and defines a bounded operator on the Banach space $X$. Using Cauchy's theorem, it is possible to show that this definition does not depend on the choice of the chosen angle $\nu$. The resulting map $\H^\infty_0(\Sigma_\theta) \to \B(X)$, $f \mapsto f(A)$ is an algebra homomorphism.

Following \cite[Definition 2.6 p.~6]{LM99} (see also \cite[p.~114]{Haa06}), we say that the operator $A$ admits a bounded $\H^\infty(\Sigma_\theta)$ functional calculus if the latter homomorphism is bounded, i.e.~if there exists a constant $C \geq 0$ such that 
\begin{equation}
\label{Def-functional-calculus}
\norm{f(A)}_{X \to X} 
\leq C\norm{f}_{\H^\infty(\Sigma_\theta)},\quad f \in \H^\infty_0(\Sigma_\theta).
\end{equation}
If the operator $A$ has dense range and admits a bounded $\H^\infty(\Sigma_\theta)$ functional calculus, then the previous homomorphism naturally extends to a bounded homomorphism $f \mapsto f(A)$ from the algebra $\H^\infty(\Sigma_\theta)$ into the algebra $\B(X)$ of bounded linear operators on $X$. In this context, we can introduce the $\H^\infty$-angle
\begin{equation*}
\label{angle-Hinfty}
\omega_{\H^\infty}(A) 
\ov{\mathrm{def}}{=} \inf\{\theta \in (\omega_{\sec}(A),\pi) : A \text{ admits a bounded $\H^\infty(\Sigma_\theta)$ functional calculus} \}.
\end{equation*}

\begin{example} \normalfont
\label{Laplacian-funct}
If $1 < p < \infty$ and if $X$ is a $\UMD$ Banach space, then by \cite[Theorem 10.2.25 p.~391]{HvNVW18}, the Laplacian $-\Delta$ admits a bounded $\H^\infty(\Sigma_\theta)$ functional calculus on the Bochner space $\L^p(\R^n,X)$ for any angle $\theta >0$, i.e.~$\omega_{\H^\infty}(-\Delta)=0$.
\end{example}

We will make use of the following elementary and general result concerning the $\HI(\Sigma_\omega)$ functional calculus for sectorial operators shifted to the right.

\begin{lemma}
\label{lem-shifted-HI-calculus}
Let $A$ be a sectorial operator on a Banach space $X$ and $\omega \in (\omega_{\sec}(A), \pi)$. The following statements are equivalent:
\begin{enumerate}
\item $A + \epsi$ has a bounded $\mathrm{H}^\infty(\Sigma_\omega)$ functional calculus on $X$ for some $\epsi > 0$,
\item $A + \epsi$ has a bounded $\mathrm{H}^\infty(\Sigma_\omega)$ functional calculus on $X$ for all $\epsi > 0$.
\end{enumerate}
\end{lemma}

\begin{proof} 
Let $\epsi, \epsi' > 0$ be such that $\epsi' + A$ has a bounded $\mathrm{H}^\infty(\Sigma_\omega)$ functional calculus on $X$. Both $\epsi + A$ and $\epsi' + A$ are invertible and we can write
$\epsi + A 
= (\epsi ' + A) + B$ with $B 
\ov{\mathrm{def}}{=} (\epsi - \epsi') \Id_X$. By \cite[Corollary 5.5.5 p.~119]{Haa06}, the operator $\epsi + A$ admits a bounded $\mathrm{H}^\infty(\Sigma_\omega)$ functional calculus.
\end{proof}

The following is \cite[Proposition 1.4.2 p.~16]{Fac14} and \cite[Corollary 10.9 p.~214]{KuW04}. It gives rise to a large class of operators possessing such a functional calculus, although the resulting angle might be somewhat larger than desirable for applications. Recall that a strongly continuous bounded group $(U_t)_{t \geq 0}$ of operators acting on a Banach space $Y$ is a dilation of a strongly continuous semigroup $(T_t)_{t \geq 0}$ of operators acting on a Banach space $X$, if there exists bounded maps $J \co X \to Y$ and $P \co Y \to X$ such that
$T_t = PU_tJ$
for any $t \geq 0$. We also refer to \cite{AFM17} for strongly related results.

\begin{thm}
\label{thm-Hinfty-dilation}
Let $X$ be a $\UMD$ Banach space space. If the strongly continuous semigroup $(T_t)_{t \geq 0}$ on $X$ with infinitesimal generator $-A$ admits a dilation to a bounded strongly continuous group on a $\UMD$ Banach space $Y$, then $A$ admits a bounded $\H^\infty(\Sigma_\theta)$ functional-calculus for any angle $\theta \in (\frac{\pi}{2},\pi)$, i.e.~$\omega_{\H^\infty}(A) \leq \frac{\pi}{2}$.
\end{thm}

A similarity transform is an operation that preserves boundedness of the functional calculus, as shown by the following result \cite[Lemma 4.3.3 p.~135]{Fac14}.

\begin{prop}
\label{prop-isomorphism}
Let $X$ and $Y$ be Banach spaces and let $-A$ be the generator of a bounded strongly continuous semigroup $(T_t)_{t \geq 0}$ of operators acting on $X$ such that $A$ admits a bounded $\H^\infty$ functional calculus. If $S \co Y \to X$ is an isomorphism and if $-B$ generates the induced strongly continuous semigroup $(S^{-1}T_t S)_{t \geq 0}$ on $Y$, then the operator $B$ admits a bounded $\H^\infty$ functional calculus with 
\begin{equation*}
\label{}
\omega_{\H^\infty}(A)
=\omega_{\H^\infty}(B).
\end{equation*}
\end{prop}

Finally, we will use the following result \cite[Theorem 1.1 p.~353]{LM03} in Section \ref{subsec-shifted-HI-calculus}.

\begin{thm}
\label{Th-com-LM}
Let $A$ and $B$ be two commuting sectorial operators on a Banach
space $X$ and let $\theta_1,\theta_2 \in (0, \pi)$ such that $\theta_1+\theta_2 < \pi$. Assume that $A$ has a bounded $\H^\infty(\Sigma_{\theta_1})$ functional calculus, $B$ has a bounded $\H^\infty(\Sigma_{\theta_2})$ functional calculus and that $X$ has property $(\Delta)$. Then the operator $A + B$ is closed and admits a bounded $\H^\infty(\Sigma_\theta)$ functional calculus for any angle $\theta > \max\{\theta_1,\theta_2\}$.
\end{thm}

Very often in this paper, the operators we consider will have an angle of sectoriality equal to zero. In such cases, the operator $A$ may not only admit a bounded $\H^\infty(\Sigma_\theta)$ functional calculus for any angle $\theta >0$, but we can in fact define $f(A)$ meaningfully for functions $f$ belonging to the H\"ormander class introduced in \eqref{hor-2}. More precisely, we will show in Theorem~\ref{thm-514} and Corollary~\ref{cor-Hormander-calculus-for-Theta-Weyl-tuple-on-LpM} that suitable operators satisfy the following definition (see \cite[Definition 3.8]{KrW18}).

\begin{defi}
\label{defi-Hor-calculus}
Let $A$ be a sectorial operator on a Banach space $X$.
Let $s > \frac12$. We say that $A$ admits a bounded $\Hor^s_2(\R_+^*)$ H\"ormander functional calculus if for some angle $\theta > \omega_{\sec}(A)$, there exists a constant $C$ such that for any function $f \in \H^\infty_0(\Sigma_\theta)$, we have 
\begin{equation}
\label{equ-1-defi-Hor-calculus}
\norm{f(A)}_{X \to X} 
\leq C \norm{f}_{\Hor^s_2(\R_+^*)}.
\end{equation}
\end{defi}

If $s > \frac{1}{2}$, we introduce the space 
\begin{equation*}
\label{defi-mathcal-W}
\mathcal{W}^{s,2} 
\ov{\mathrm{def}}{=} \big\{ f \co (0,\infty) \to \C :  f \circ \exp \in \W^{s,2}(\R) \big\}
\end{equation*}
coming equipped with the norm 
\begin{equation}
\label{norm-Sobolev-Mellin}
\norm{f}_{\mathcal{W}^{s,2}} 
\ov{\mathrm{def}}{=} \norm{f \circ \exp}_{\W^{s,2}(\R)},
\end{equation}
where the Sobolev space $\W^{s,2}(\R)$ is introduced in \eqref{Sobolev-space}.
The space $\mathcal{W}^{s,2}$ appears for integer $s$ as the so-called Kondratiev space $V^s_{2,s-\frac12}(0,\infty)$ in \cite[p.~191]{KMR97}.
The following density result is stated in \cite[Lemma 3.5]{KrW18} which is in turn based on \cite[Lemma 4.15 p.~68]{Kri09}.

\begin{lemma}
\label{lem-density-mathcal-W}
For any angle $\theta \in (0,\pi)$, the subspace $\H^\infty(\Sigma_\theta) \cap \mathcal{W}^{s,2}$ is dense in the space $\mathcal{W}^{s,2}$.
\end{lemma}

\begin{proof}
A variant is proved in \cite[Lemma 4.15 p.~68]{Kri09} that concerns holomorphic functions on a strip $\Str_\theta \ov{\mathrm{def}}{=}  \{ z \in \C :\: |\Im z |< \theta \}$\label{strip} and the usual Sobolev space $\W^{s,2}(\R)$.
Namely the latter source says that the intersection $\H^\infty(\Str_\theta) \cap \W^{s,2}(\R)$ is dense in the space $\W^{s,2}(\R)$. Now, consider a function $f \in \mathcal{W}^{s,2}$.
Then the function $f \circ \exp$ belongs to $\W^{s,2}(\R)$ and for any $\epsi > 0$, there exists an $\epsi$-approximation $g \in \H^\infty(\Str_\theta) \cap \W^{s,2}(\R)$ of $f \circ \exp$ in the space $\W^{s,2}(\R)$. We have $g \circ \log \in \H^\infty(\Sigma_\theta) \cap \cal{W}^{s,2}$ as well as $\norm{g \circ \log - f}_{\cal{W}^{s,2}} \ov{\eqref{norm-Sobolev-Mellin}}{=} \norm{g - f \circ \exp}_{\W^{s,2}(\R)} < \epsi$.
The proof is complete.
\end{proof}

Note that the following norm estimates hold. We refer to \cite[Proposition 9.1 b) p.~1398]{HaP23} for generalizations of the (folklore) second inequality with different arguments.

\begin{lemma}
\label{lem-norm-comparisons-Hor-HI}
Let $s > \frac12$. Then
\begin{equation*}
\label{}
\norm{f}_{\Hor^s_2(\R_+^*)} 
\lesssim \norm{f}_{\mathcal{W}^{s,2}}, \quad f \in \mathcal{W}^{s,2}.
\end{equation*}
Moreover, for any angle $\theta \in (0,\pi)$, we have
\begin{equation}
\label{Hor-vs-Hinfty}
\norm{f}_{\Hor^s_2(\R_+^*)} 
\lesssim \norm{f}_{\H^\infty(\Sigma_\theta)}, \quad f \in \H^\infty(\Sigma_\theta).
\end{equation}
\end{lemma}

\begin{proof}
The first point follows from \cite[Proposition 4.9 (3) p.~62]{Kri09} and the equivalence of (2) and (3) in \cite[Proposition 4.11 p.~63]{Kri09}. Indeed, the latter proposition gives the equivalence   
\begin{equation}
\label{inter45}
\norm{f}_{\Hor^s_2(\R^*_+)} 
\cong \sup_{n \in \Z} \norm{\varphi_n (f \circ \exp)}_{\W^{s,2}(\R)},
\end{equation}
for the norm of the space $\Hor^s_2(\R^*_+)$, where $(\varphi_n)_{n \in \Z}$ is a smooth partition of unity, meaning that $\sum_{n \in \Z} \varphi_n(t) = 1$ $(t \in \R)$ and $\supp \varphi_n \subseteq [n-1,n+1]$. 
The result \cite[Proposition 4.9 (3) p.~62]{Kri09} provides the estimate $\sup_{n \in \Z} \norm{\varphi_n f}_{\W^{s,2}(\R)} \lesssim \norm{f}_{\W^{s,2}(\R)}$. By composition with the exponential function, we conclude that
\[ 
\norm{f}_{\Hor^s_2(\R^*_+)} 
\ov{\eqref{inter45}}{\cong}  \sup_{n \in \Z} \norm{\varphi_n (f \circ \exp)}_{\W^{s,2}(\R)}
\lesssim \norm{f \circ \exp}_{\W^{s,2}(\R)} 
\ov{\eqref{norm-Sobolev-Mellin}}{=} \norm{f}_{\cal{W}^{s,2}}. 
\]
This establishes the estimate stated in the first part.

For the second one, consider some function $f \in \H^\infty(\Sigma_\theta)$. For any $y > 0$, note that the ball $B(y,\frac12\sin(\theta) y)$ is a subset of the sector $\Sigma_\theta$. By the Cauchy integral formula, for any integer $k \geq 0$, we see that
\[ 
\frac{\d^kf}{\d y^k}(y) 
= \frac{k!}{2\pi i} \int_{\partial B(y,\frac12\sin(\theta) y)} \frac{f(z)}{(z - y)^{k+1}} \d z. 
\]
Thus,
\begin{equation}
\label{inter-fin-7322}
\left| y^k \frac{\d^k f}{\d y^k}(y) \right| 
\leq k! y^k \frac{1}{2\pi} \cdot 2 \pi \frac12 \sin(\theta) y \frac{1}{|\frac12 \sin(\theta) y|^{k+1}} \norm{f}_{\H^\infty(\Sigma_\theta)} \lesssim \norm{f}_{\H^\infty(\Sigma_\theta)}. 
\end{equation}
Consider a function $\eta \co ]0,\infty[ \to \R$ as in \eqref{hor-2}, supported in $[a,b]$. Consider any integer $n \in \N$ with $n \geq s$. Recall that we have a contractive inclusion $\W^{n,2}(\R) \subset \W^{s,2}(\R)$ and that the norm of a function $g$ in the space $\W^{n,2}(\R)$ is equivalent to $\sup_{0 \leq k \leq n} \norm{\frac{\d^k g}{\d x^k} }_{\L^2(\R)}$. For any $t > 0$, we have using 
\cite[Exercise after Definition 15.17]{Tah15} 
in the second inequality
\begin{align*}
\MoveEqLeft
\norm{x \mapsto \eta(x) \cdot f(t x)}_{\W^{s,2}(\R)}
\leq \norm{x \mapsto \eta(x) \cdot f(t x)}_{\W^{n,2}(\R)} \\
&\lesssim \sup \left\{ \left|\frac{\d^k(f(tx))}{\d x^k} \right| : \: x \in [a,b],\: k \leq n \right\} \\
& = \sup \left\{ \left|t^k\frac{\d^k f}{\d x^k}(t x) \right| : \: x \in [a,b],\: k \leq n \right\} \\
& \lesssim \sup \left\{ \left|(tx)^k\frac{\d^k f}{\d x^k}(t x) \right| : \: x \in [a,b],\: k \leq n \right\} \\
& \ov{\eqref{inter-fin-7322}}{\lesssim} \norm{f}_{\H^\infty(\Sigma_\theta)},
\end{align*}
with implied constant independent of $t$. With \eqref{hor-2}, we deduce \eqref{Hor-vs-Hinfty}.
\end{proof}

An immediate consequence of Lemma \ref{lem-norm-comparisons-Hor-HI} is that if $A$ is a sectorial operator admitting a bounded $\Hor^s_2(\R_+^*)$ H\"ormander functional calculus, then it also satisfies the estimate
\begin{equation*}
\label{inter-1245a}
\norm{f(A)}_{X \to X} 
\lesssim_\theta \norm{f}_{\H^\infty(\Sigma_\theta)}, \quad f \in \H^\infty_0(\Sigma_\theta)
\end{equation*}
for any angle $\theta \in (0,\pi)$. Thus, the operator $A$ admits a bounded $\H^\infty(\Sigma_\theta)$ functional calculus.

If in addition $X$ is reflexive and the operator $A$ is injective with dense range, then first by the explanations following \eqref{Def-functional-calculus}, the operator $f(A)$ is meaningfully defined for any function $f \in \H^\infty(\Sigma_\theta)$. Consider the function
\begin{equation*}
\label{function-rho}
\rho(z) 
\ov{\mathrm{def}}{=} \frac{z}{(1 + z)^{2}},
\end{equation*}
which belongs to the space $\H^\infty_0(\Sigma_\theta)$ by Example \ref{Example-Haase}. Since $\HI_0(\Sigma_\theta)$ is an ideal of the algebra $\HI(\Sigma_\theta)$, the function $f_n \ov{\mathrm{def}}{=} f \rho^{\frac{1}{n}}$ belongs to $\HI_0(\Sigma_\theta)$ for each integer $n \geq 1$. Moreover, the sequence $(f_n)$ converges pointwise to $f$ on the sector $\Sigma_\theta$ and is bounded for the norm of $\HI(\Sigma_\theta)$ since 
\begin{equation}
\label{inter-345pp}
\bnorm{\rho^{\frac{1}{n}}}_{\HI(\Sigma_\theta)}
\ov{\eqref{norm-Hinfty}}{=} \norm{\rho}_{\HI(\Sigma_\theta)}^{\frac{1}{n}}.
\end{equation} 
By using the convergence lemma \cite[Proposition 5.1.4 p.~104]{Haa06} \cite[Theorem 10.2.13 p.~381]{HvNVW18} of the $\H^\infty$ functional calculus, one can deduce that the sequence $(f_n(A))$ converges to $f(A)$ in the strong operator topology. Consequently, we have 
\begin{align}
\MoveEqLeft
\label{equ-4-defi-Hor-calculus}
 \norm{f(A)}_{X \to X} 
\lesssim \liminf_n \norm{f_n(A)}_{X \to X} 
\ov{\eqref{equ-1-defi-Hor-calculus}}{\lesssim} \liminf_n \norm{f_n}_{\Hor^s_2(\R^*_+)} \\
&\lesssim \norm{f}_{\Hor^s_2(\R^*_+)} \liminf_n \bnorm{\rho^{\frac1n}}_{\Hor^s_2(\R^*_+)} \nonumber
\ov{\eqref{Hor-vs-Hinfty}}{\lesssim} \norm{f}_{\Hor^s_2(\R^*_+)} \liminf_n \bnorm{\rho^{\frac1n}}_{\H^\infty(\Sigma_\theta)}\nonumber \\
&\ov{\eqref{inter-345pp}}{\lesssim} \norm{f}_{\Hor^s_2(\R^*_+)} \liminf_n \norm{\rho}_{\H^\infty(\Sigma_\theta)}^{\frac1n} 
 = \norm{f}_{\Hor^s_2(\R^*_+)}, \nonumber
\end{align}
where we used the strong lower semicontinuity of the operator norm \cite[Theorem 4.23 (a) p.~75]{Wei80} in the first inequality. In other words, the estimate \eqref{equ-1-defi-Hor-calculus} extends to all functions $f \in \HI(\Sigma_\theta)$.

Second, by density of the subspace $\H^\infty(\Sigma_\theta) \cap \mathcal{W}^{s,2}$ in the space $\mathcal{W}^{s,2}$ according to Lemma \ref{lem-density-mathcal-W}, $f(A)$ can be meaningfully defined by Lemma \ref{lem-norm-comparisons-Hor-HI} for $f \in \mathcal{W}^{s,2}$ and moreover, we have
\[ 
\norm{f(A)}_{X \to X} 
\lesssim \norm{f}_{\mathcal{W}^{s,2}}, \quad f \in \mathcal{W}^{s,2}. 
\]
In fact, we can say a bit more.

\begin{lemma}
Let $A$ be a sectorial operator on some reflexive Banach space $X$ with dense range. Suppose that $A$ admits a bounded $\Hor^s_2(\R^*_+)$ H\"ormander functional calculus for some $s > \frac12$. Then
\begin{equation*}
\label{Hormander-W}
\norm{f(A)}_{X \to X} 
\lesssim \norm{f}_{\Hor^s_2(\R^*_+)}, \quad f \in \mathcal{W}^{s,2}.
\end{equation*}
Moreover, $\mathcal{W}^{s,2}$ is a Banach algebra with respect to pointwise multiplication and for any functions $f,g \in \mathcal{W}^{s,2}$, we have $(fg)(A) = f(A)g(A)$.
\end{lemma}

\begin{proof}
Let $f \in \mathcal{W}^{s,2}$.
Then by density, we can find a sequence $(f_n)$ in the space $\HI(\Sigma_\theta) \cap \mathcal{W}^{s,2}$ such that $\norm{f_n - f}_{\mathcal{W}^{s,2}} \to 0$.
By Lemma \ref{lem-norm-comparisons-Hor-HI}, we also have 
\begin{equation}
\label{inter-2457654}
\norm{f}_{\Hor^s_2(\R^*_+)} 
= \lim_n \norm{f_n}_{\Hor^s_2(\R^*_+)}.
\end{equation}
Therefore,
\[
\norm{f(A)}_{X \to X} 
= \lim_n \norm{f_n(A)}_{X \to X}  
\ov{\eqref{equ-4-defi-Hor-calculus}}{\lesssim} \lim_n \norm{f_n}_{\Hor^s_2(\R^*_+)} 
\ov{\eqref{inter-2457654}}{=} \norm{f}_{\Hor^s_2(\R^*_+)}.
\]

That $\mathcal{W}^{s,2}$ is a Banach algebra follows directly from the fact that $\W^{s,2}(\R)$ is a Banach algebra for $s > \frac12$, see \cite[6.11 (a) p.~162]{Ste70} and \cite[Theorem 1 p.~222, Proposition p.~14]{RuSi96}. 
Then the multiplicativity of the $\mathcal{W}^{s,2}$ functional calculus follows from the principle of extension of algebraic identities on a dense subset.
\end{proof}



The following lemma is proved in \cite[proof of Lemma 4.3 p.~417 bottom]{KrW18}.

\begin{lemma}
\label{Lemma-Chris-Weis}
Let $s > \frac12$. For any function $f \in \Hor^s_2(\R^*_+)$ and any angle $\theta > 0$, the function $f \rho^\theta$ belongs to the space $\mathcal{W}^{s,2}$.
\end{lemma}

Let $A$ be a sectorial operator on some reflexive Banach space $X$ with dense range, admitting a bounded $\Hor^s_2(\R^*_+)$ H\"ormander functional calculus for some $s > \frac12$. We will use the notation $\Closed(X)$ for the set of closed unbounded operators acting on the Banach space $X$. Recall that $\cal{W}^{s,2}$ is a subalgebra of $\Hor^s_2(\R^*_+)$ and that each operator 
$\rho^\frac{1}{n}(A)$ is bounded and injective by \cite[Proposition 3.1.1 p.~61]{Haa06}, since the bounded operator $\rho(A) = A(1+A)^{-2}$ is injective.

Using \cite[Theorem 7.5 p.~109]{Haa18} and \cite[Lemma 1.2.1 p.~5]{Haa06} (see also \cite{Haa05}) and Lemma \ref{Lemma-Chris-Weis}, we see that the map $\Phi \co \cal{W}^{s,2} \to \B(X)$, $f \mapsto f(A)$ admits an extension $\hat{\Phi} \co \Hor^s_2(\R^*_+) \to \Closed(X)$, which is a <<calculus>> (in the sense of \cite[Section 2.4]{Haa18}), satisfying 
$$
\hat{\Phi}(f) 
= g(A)^{-1}(g f)(A).
$$
for any function $f \in \Hor^s_2(\R^*_+)$ and any function $g \in \cal{W}^{s,2}$ such that $g f \in \cal{W}^{s,2}$ with $g(A)$ injective. Note that the sequence $(\rho^\frac{1}{n})_{n \geq 1}$ is an approximate identity in $\cal{W}^{s,2}$ with respect to $\hat{\Phi}$, in the sense of \cite[Definition 7.21 p.~118]{Haa18}. This is a consequence of the convergence lemma \cite[Proposition 5.1.4 p.~104]{Haa06} \cite[Theorem 10.2.13 p.~381]{HvNVW18} of the $\H^\infty$ functional calculus.
Consequently, by \cite[Lemma 7.2.2 p.~119]{Haa18}, for any function $f \in \Hor^s_2(\R^*_+)$, the operator $\hat{\Phi}(f)$ is densely defined. Moreover, the subspace
$$
\Span \cup_{n \geq 1} \Ran \rho^\frac{1}{n}(A)
$$
is dense in the Banach space $X$ and is a core of the operator $\hat{\Phi}(f)$ for any function $f \in \Hor^s_2(\R^*_+)$. Actually, the operator $\hat{\Phi}(f)$ is bounded since for any integer $n \geq 1$, any $x \in X$ and any $\theta >0$ we have
\begin{align*}
\MoveEqLeft
 \bnorm{\hat{\Phi}(f) \rho^{\frac{1}{n}}(A)x}_X 
= \bnorm{(f\rho^\frac{1}{n})(A)x}_X
\ov{\eqref{Hormander-W}}{\lesssim} \bnorm{f \rho^\frac{1}{n}}_{\Hor^s_2(\R^*_+)} \norm{x}_X \\
& \lesssim \norm{f}_{\Hor^s_2(\R^*_+)} \bnorm{\rho^\frac{1}{n}}_{\Hor^s_2(\R^*_+)} \norm{x}_X 
 \ov{\eqref{Hor-vs-Hinfty}}{\lesssim} \norm{f}_{\Hor^s_2(\R^*_+)} \bnorm{\rho^\frac{1}{n}}_{\H^\infty(\Sigma_\theta)} \norm{x}_X  \\
&\lesssim \norm{f}_{\Hor^s_2(\R^*_+)} \norm{x}_X.
\end{align*}
We denote by $f(A) \co X \to X$ its closure. Finally, observe that by \cite[Corollary 7.23 p.~119]{Haa18}, this calculus is multiplicative, i.e.~$(fg)(A)=f(A)g(A)$.

In this paper, the sectorial operator $A$ does not have dense range in general.
However, they will be defined on a reflexive Banach space $X$, in which case, it is well-known that the space $X$ splits into a direct sum of complemented subspaces $X = \ker A \oplus \ovl{\Ran A}$, see \eqref{decompo-reflexive}. We denote by $P_{\ker A} \co X \to X$ is the projection of $X$ onto the subspace $\ker A$ along the closed subspace $\ovl{\Ran A}$. Moreover, these subspaces are $A$-invariant. Namely, we have $A = 0 \oplus A_1$, where $A_1$ is the part of $A$ on the subspace $\ovl{\Ran A}$ is again sectorial, but has in addition dense range by \cite[Proposition 10.1.8 p.~365]{HvNVW18}. If the operator $A$ admits a bounded $\H^\infty(\Sigma_\theta)$ functional calculus for some angle $\theta > 0$ in the sense of \eqref{Def-functional-calculus}, this usually enables one to set $f(A) \ov{\mathrm{def}}{=} f(0) P_{\ker A} \oplus f(A_1)$ for each function $f \in \H^\infty(\Sigma_\theta)$, whenever $f(0)$ is meaningful.

Essentially the same approach can be applied to the H\"ormander functional calculus.

\begin{lemma}
\label{lem-Hor-calculus-reflexive}
Let $X$ be a reflexive Banach space and let $A$ be a sectorial operator on $X$ having a bounded $\Hor^s_2(\R^*_+)$ functional calculus. Then the part $A_1$ of the operator $A$ on the subspace $\ovl{\Ran A}$ also admits a bounded $\Hor^s_2(\R^*_+)$ H\"ormander functional calculus. For any function $f \in \Hor^s_2(\R^*_+)$ such that $f(0) \ov{\mathrm{def}}{=} \lim_{y \to 0} f(y)$ exists,  we can introduce the bounded operator $f(A) \ov{\mathrm{def}}{=} f(0)P_{\ker A} \oplus f(A_1)$ acting on $X$. Then this functional calculus is again multiplicative and we have
\[ 
\norm{f(A)}_{X \to X} 
\lesssim |f(0)| + \norm{f}_{\Hor^s_2(\R^*_+)}.
\]
\end{lemma}

\begin{proof}
The fact that the operator $A_1$ also admits a bounded $\Hor^s_2(\R^*_+)$ H\"ormander functional calculus follows from the identity $g(A)|_{\overline{\Ran A}} = g(A_1)$ for all function $g \in \H^\infty_0(\Sigma_\theta)$.
That $f(A)$ is a bounded operator on $X$ follows from the fact that $f(A_1)$ is bounded on the subspace $\ovl{\Ran A}$ and from boundedness of the projection $P_{\ker A}$. Finally, the multiplicativity of the functional calculus is easy to check.
\end{proof}

\paragraph{Commutation relations} The following is extracted from \cite{Kat62}.

\begin{prop}
\label{prop-Kato}
Let $(e^{tA})_{t \geq 0}$ and $(e^{tB})_{t \geq 0}$ be two strongly continuous semigroups on a Banach space $X$ such that
\begin{equation*}
\label{}
e^{s A}e^{t B}
=e^{cst}e^{t B}e^{s A}, \quad 0 \leq s, t < \infty
\end{equation*}
for some constant $c \in \C$. Then the subspace $\dom(BA) \cap \dom(A B)$ is dense in $X$ and 
\begin{equation*}
\label{}
(AB-BA)f
=cf
\end{equation*}
if $f \in \dom(BA) \cap \dom(A B)$. 
\end{prop}


\paragraph{2-cocycles}
Let $G$ be a locally compact group equipped with a left Haar measure $\mu_G$. We first recall that a Borel measurable 2-cocycle on $G$ with values in the one-dimensional torus $\T \ov{\mathrm{def}}{=} \{z \in \mathbb{C} : |z| = 1 \}$ \label{torus} is a Borel measurable map $\sigma \co G \times G \to \T$ such that 
\begin{equation}
\label{equation-2-cocycle}
\sigma(s,t)\sigma(st,r) 
= \sigma(s,tr)\sigma(t,r), \quad	s, t, r \in G. 
\end{equation}
The terminologies <<multiplier>> or <<factor system>> are also used in the literature. This implies that
\begin{equation*}
\sigma(s,e)
= \sigma(e,e) 
= \sigma(e,s) , \quad s \in G.
\end{equation*}
Indeed, the cocycle condition \eqref{equation-2-cocycle} evaluated with $(s^{-1},s,e)$ instead of $(s,t,r)$ gives
\begin{equation*}
\label{}
\sigma(s^{-1},s)\sigma(e,e)
=\sigma(s^{-1},s)\sigma(s,e).
\end{equation*}
Hence $
\sigma(e,e) 
= \sigma(s,e)$. 
Similarly the 2-cocycle condition applied with $(e,s,s^{-1})$ instead of $(s,t,r)$ gives
\begin{equation*}
\label{}
\sigma(e,s)\sigma(s,s^{-1})
=\sigma(e,e)\sigma(s,s^{-1}).
\end{equation*}
Thus $\sigma(e,s) = \sigma(e,e)$. The set of Borel measurable 2-cocycles is a group under pointwise multiplication. The 2-cocycle $\sigma$ is said to be normalized if 
\begin{equation}
\label{normalized}
\sigma(s,s^{-1}) = 1, \quad s \in G
\end{equation}
In particular, $\sigma(e,e)=1$. 

\begin{example} \normalfont
\label{cocycle-Rd}
Observe that for any $d \times d$ real skew-symmetric matrix $\Theta$,  the 2-cocycle $\sigma_\Theta \co \R^d \times \R^d \to \T$, $(s,t) \mapsto \e^{\frac{1}{2}\i\langle s, \Theta t\rangle}$ defined in \eqref{def-cocycle-intro} is normalized. Actually, according to \cite[Theorem 7.38 p.~276]{Var85}, these 2-cocycles are the only ones that the locally compact group $\R^d$ admits.
\end{example}

\begin{lemma}
\label{Lemma-2-cocycle}
If the 2-cocycle $\sigma\co G \times G \to \T$ is normalized, we have
\begin{equation}
\label{very-useful-1}
\sigma(s,s^{-1}t)
=\sigma(t^{-1},s), \quad s,t \in G.
\end{equation}
\end{lemma}

\begin{proof}
Using the cocycle condition with $s$ replaced by $t^{-1}$, $t$ by $s$ and $r=s^{-1}t$, we obtain
\begin{equation*}
\label{}
\sigma(t^{-1},s)
\ov{\eqref{normalized}}{=} \sigma(t^{-1},s)\sigma(t^{-1}s,s^{-1}t)
\ov{\eqref{equation-2-cocycle}}{=} \sigma(t^{-1}, t)\sigma(s,s^{-1}t)
\ov{\eqref{normalized}}{=} \sigma(s,s^{-1}t).
\end{equation*}
\end{proof}

%
%

\paragraph{Twisted group von Neumann algebras}
Let $G$ be a locally compact group. Recall that if $X$ is a Banach space then a map $\pi \co G \to \B(X)$ for some Banach space $X$ is called a (strongly) continuous $\sigma$-projective unitary representation of $G$ on $X$ for some Borel measurable 2-cocycle $\sigma \co G \times G \to \T$ if the map $G \to X$, $s \mapsto \pi_sx$ is continuous for any $x \in X$ and if we have
\begin{equation}
\label{eq-proj-rep}
\pi_s\pi_t 
= \sigma(s,t)\pi_{st}, \quad s,t \in G.
\end{equation} 
The representation is said to be a (strongly) continuous $\sigma$-projective unitary representation if $X$ is a Hilbert space $H$ and if each $\pi_s$ is a unitary operator.  
We refer to the nice survey \cite{Pac08} for a description of some applications of projective representations of locally compact groups, and to \cite[Chapter VIII, Section 10]{FeD88}, \cite{Kle74} and \cite{Mor17} for more information. 

If $\sigma \co G \times G \to \T$ is a Borel measurable 2-cocycle, an example of continuous $\sigma$-projective unitary representation is given by the left regular $\sigma$-projective representation $\lambda_\sigma \co G \to \B(\L^2(G))$, which is defined by
\begin{equation}
\label{def-lambda-sigma}
(\lambda_{\sigma,s}f)(t) 
\ov{\mathrm{def}}{=} \sigma(s,s^{-1}t)f(s^{-1}t),\quad s,t \in G, f \in \L^2(G).
\end{equation}
For any $s,t \in G$, we have 
\begin{equation}
\label{product-adjoint-twisted}
\lambda_{\sigma,s} \lambda_{\sigma,t} 
= \sigma(s,t) \lambda_{\sigma,st}, 
\qquad 
\big(\lambda_{\sigma,s}\big)^* 
= \ovl{\sigma(s,s^{-1})} \lambda_{\sigma,s^{-1}}.	
\end{equation}
For any function $f \in \L^1(G)$, we let 
\begin{equation*}
\label{def-lambda-sigma}
\lambda_\sigma(f) 
\ov{\mathrm{def}}{=} \int_G f(s)\lambda_{\sigma,s} \d\mu_G(s),
\end{equation*}
where the integral is understood in the weak sense. We define the twisted group von Neumann algebra $\VN(G,\sigma)$ of the group $G$ to be the von Neumann algebra generated by $\{\lambda_\sigma(f) : f \in \L^1(G)\}$ (or equivalently by $\{ \lambda_{\sigma,s} : s\in G \}$) in the algebra $\B(\L^2(G))$.


\paragraph{A central extension} 

Suppose that $G$ and $K$ are two locally compact groups. Each neutral element is denoted by $e$. An extension of $G$ by $K$ is an exact sequence $\{e\} \to K \xra{i} H \xra{j} G \to \{e\}$ where $H$ is a locally compact group, $i$ is a topological isomorphism onto its image and where $j$ is continuous and open. It is said to be a central extension if $i(K)$ is contained in the center of the group $H$. 

Assume that $G$ is second countable. Let $\sigma \co G \times G \to \T$ be a continuous normalized 2-cocycle on $G$. Consider the set
\begin{equation}
\label{def-central-extension}
G_\sigma 
\ov{\mathrm{def}}{=} G \times \T
\end{equation}
endowed with the group law 
\begin{equation*}
\label{}
(s,z) \cdot (t,w) 
\ov{\mathrm{def}}{=} (st,zw\sigma(s,t)),
\end{equation*}
where $s,t \in G$ and $z,w \in \T$. As in \cite[Section 2]{Mac58}, \cite[Section 3]{Par69} and \cite[Theorem 7.8]{Var85}, we equip $G_\sigma$ with the product topology.   
Then $G_\sigma$ becomes a second countable locally compact group with left Haar measure $\mu_G \ot \mu_\T$ by \cite[p.~13]{Par69} and \cite[Remark p.~255]{Var85}.  We warn the reader that in the theory of Mackey, we can use cocycles that are only Borel measurable instead of continuous. In this case, the used topology on the <<Mackey group>> $G_\sigma$ is the <<Weil topology>> and this is not identical in general to the product topology on $G \times \T$.

Introducing the continuous maps $i \co \T \to G_\sigma$, $z \mapsto (e,z)$ and $j \co G_\sigma \to G$, $(s,z) \mapsto s$ then by \cite[Theorem 7.8]{Var85} the triple $(G_\sigma,i,j)$ defines a central extension of the group $G$ by $\T$. This means that we have a short exact sequence 
\begin{equation}
\label{central-extension}
\{e\} \to \T \xra{i} G_\sigma \xra{j} G \to \{e\}
\end{equation}
such that the image of $\T$ is a subgroup of the center of the group $G_\sigma$.

We will use the following result \cite[Corollary 5.2 p.~220]{Rag94} or \cite[Theorem 7.21 p.~262]{Var85}.

\begin{thm}
\label{Th-Lie-extension}
Let $G$ be a connected Lie group. Let $\sigma \co G \times G \to \T$ be a continuous 2-cocycle. Then there exists a unique analytic structure on $G_\sigma$ which converts $G_\sigma$ into a Lie group.
\end{thm}

Let $\mathfrak{a}$ and $\mathfrak{g}$ be Lie algebras. An extension of $\mathfrak{g}$ by $\mathfrak{a}$ is a short exact sequence 
\begin{equation*}
\label{}
0 \to \mathfrak{a} \xra{\lambda} \mathfrak{h} \xra{\mu} \mathfrak{g} \to 0,
\end{equation*}
where $\mathfrak{h}$ is a Lie algebra, $\mu$ is a surjective homomorphism from $\mathfrak{h}$ onto $\mathfrak{g}$ and where $\lambda$ is an injective homomorphism from $\mathfrak{a}$ onto the kernel of $\mu$. Two extensions 
\begin{equation*}
\label{}
0 \to \mathfrak{a} \xra{\lambda} \mathfrak{h} \xra{\mu} \g \to 0 
\quad \text{and} \quad
0 \to \mathfrak{a} \xra{\lambda'} \mathfrak{h}' \xra{\mu'} \g \to 0
\end{equation*}
are said to be equivalent if there exists a homomorphism $\nu \co \mathfrak{g} \to \mathfrak{g}'$ such that $\lambda'=\nu \circ \lambda$ and $\mu=\mu' \circ \nu$. In this case, $\nu$ is an isomorphism. 
An extension is said to be central if $\mathfrak{a}$ is contained in the center of the Lie algebra $\mathfrak{h}$, that is if $[\mathfrak{a},\mathfrak{h}] = 0$. 

A 2-cocycle on a Lie algebra $\mathfrak{g}$ with values in an \textit{abelian} Lie algebra $\mathfrak{a}$, is a bilinear map $c \co \mathfrak{g} \times \mathfrak{g} \to \mathfrak{a}$, which is antisymmetric (i.e.~satisfies $c(X,Y) = -c(Y,X)$ for any $X,Y \in \mathfrak{g}$) and satisfying the cocycle identity
\begin{equation*}
\label{Jacobi-cocycle}
c([X, Y], Z) + c([Z, X], Y) + c([Y,Z],X) 
= 0, \quad X,Y,Z \in \mathfrak{g}.
\end{equation*}
For any linear form $f \co \g \to \mathbb{C}$, one can define a 2-cocycle $c_f$ by the formula
\begin{equation*}
\label{def-cf}
c_f(X,Y) \ov{\mathrm{def}}{=} f([X,Y]),
\end{equation*}  
where $X,Y \in \g$. Such a 2-cocycle is called a 2-coboundary or a trivial 2-cocycle on $\g$. A 2-cocycle $c_1$ is said to be equivalent to a 2-cocycle $c_2$ if the difference $c_1-c_2$ is trivial.

Recall that there is a one-to-one correspondence between the equivalence classes of central extensions of $\g$ by an abelian Lie algebra $\mathfrak{a}$ and the equivalence classes of  cocycles. In particular, for any central extension $\mathfrak{h}$ of $\mathfrak{g}$ by $\mathfrak{a}$ there exists by \cite[p.~68]{Sch08} a 2-cocycle $c \co \mathfrak{g} \times \mathfrak{g} \to \mathfrak{a}$ such that the Lie algebra $\mathfrak{h}$ is isomorphic to the Lie algebra defined by the vector space $\mathfrak{g} \oplus \mathfrak{a}$ and by the bracket 
\begin{equation}
\label{bracket-central}
[(X,Z),(Y,Z')]
=[X,Y]_{\mathfrak{g}}+c(X,Y), \quad X,Y \in \mathfrak{g},Z,Z' \in \mathfrak{a}.
\end{equation}

In the situation of Theorem \ref{Th-Lie-extension}, we obtain by \cite[p.~64]{Sch08} that the Lie algebra $\g_\sigma$ \label{def-gsigma} of the Lie group $G_\sigma$ is a central extension of the Lie algebra $\mathfrak{t}$ of the Lie group $\mathbb{T}$ by the Lie algebra $\g$ of the Lie group $G$. This means that we have a short exact sequence 
\begin{equation}
\label{central-extension-algebras}
0 \to \mathfrak{t} \to \g_\sigma \to \g \to 0.
\end{equation}
with $[\mathfrak{t},\g_\sigma] = 0$. Assume that the 2-cocycle $\sigma \co G \times G \to \T$ is smooth. If $\gamma_t \co \R \to G$ and $\beta_s \co \R \to G$ are smooth curves in $G$ such that $\frac{\d}{\d t}|_{t=0} \gamma_t = X$ and $\frac{\d}{\d s}|_{s=0} \beta_s = Y$ then the formula
\begin{equation}
\label{recov-cocycle}
c(X,Y) 
=\frac{\d^2}{\d t\d s}|_{t=0,s=0} \sigma(\gamma_t,\beta_s) -\frac{\d^2}{\d t\d s}|_{t=0,s=0} \sigma(\beta_s, \gamma_t),
\end{equation}
defines by a particular case of \cite[Proposition 3.14 p.~26]{KhW09} a 2-cocycle $c \co \g \times \g \to \mathfrak{t}$ defining the corresponding central extension \eqref{central-extension-algebras}. See also \cite[Corollary 1.8]{Sun91} for a different but similar formula.

It is well-established that the representations of the group $G$ and the $\sigma$-projective representations of the Mackey group $G_\sigma$ are closely related, as detailed in \cite[Chapter VIII, Section 10]{FeD88}, \cite[Section 2]{Mac58} and \cite[Corollary p.~223]{Kle74}.
Roughly speaking, the theory of $\sigma$-projective representations of $G$ can be reduced to the theory of unitary representations of the slightly larger group $G_\sigma$. We need a  part of this correspondence for $\sigma$-projective representations of locally compact groups on \textit{Banach spaces}, which we now define. If $X$ is a Banach space, a map $\pi \co G \to \B(X)$ is called a (strongly) continuous $\sigma$-projective representation on $G$ if the map $G \to X$, $s \mapsto \pi_s(x)$ is continuous for any $x \in X$ and if we have the equality \eqref{eq-proj-rep}.


\begin{prop}
\label{prop-corre}
Let $X$ be a Banach space. Consider a second countable locally compact group $G$ equipped with a continuous 2-cocycle $\sigma \co G \times G \to \T$. For each continuous $\sigma$-projective representation $\pi \co G \to \B(X)$, the map $\tilde{\pi} \co G_\sigma \to \B(X)$, $(s,z) \mapsto z\pi_s$ defines a continuous representation of the group $G_\sigma$.
\end{prop}

\begin{proof}
For any elements $(s,z)$ and $(t,w)$ of the group $G_\sigma$, we have the equalities 
\begin{equation*}
\label{}
\tilde{\pi}(s,z)\tilde{\pi}(t,w)
=zw\pi_s \pi_t
\end{equation*}
and 
\begin{equation*}
\label{}
\tilde{\pi}((s,z)\cdot (t,w))=\tilde{\pi}(st, zw\sigma(s,t))=zw\sigma(s,t)\pi_{st}\ov{\eqref{eq-proj-rep}}{=} zw\pi_{s}\pi_{t}.
\end{equation*} 
Consequently, $\tilde{\pi}$ is a representation of the group $G_\sigma$. The continuity is obvious.
\end{proof}

\chapter{$\Theta$-Weyl tuples and associated operators}

\section{Weyl calculus for tuples of $\mathrm{C}_0$-groups with $\Theta$-commutation relations}
\label{Sec-Weyl1}

Let $n \in \N$. In all the paper, $\Theta \in \M_n(\R)$ denotes a fixed skew-symmetric matrix, i.e.~satisfying $\Theta^T = - \Theta$, where $\Theta^T$ is the transpose \label{transpose} of the matrix $\Theta$. In the sequel, it is convenient to introduce the upper and lower triangle parts \label{upper} of the matrix $\Theta$ by $(\Theta^\uparrow)_ {jk} \ov{\mathrm{def}}{=} \delta_{j \leq k} \Theta_{jk}$ and $(\Theta^\downarrow)_{jk} \ov{\mathrm{def}}{=} \delta_{j \geq k} \Theta_{jk}$ where $1 \leq j,k \leq n$. We clearly have the equalities 
\begin{equation*}
\label{}
\Theta = \Theta^\uparrow + \Theta^\downarrow
\end{equation*}
(recall that $\Theta$ vanishes on the diagonal) and 
\begin{equation*}
\label{}
(\Theta^\uparrow)^T 
= - \Theta^\downarrow.
\end{equation*}

We start by introducing the following fundamental definition. Independently, Gerhold and Shalit introduced in \cite[Definition 1.1]{GeS23} a similar notion in the Hilbert space setting. 

\begin{defi}
\label{defi-Weyl-tuple}
Let $X$ be a Banach space. For each $k \in \{1,\ldots,n\}$, consider some strongly continuous bounded group $(e^{\i t A_k} )_{t \in \R}$ of operators acting on $X$ with generator $\i A_k$. We say that $(A_1,\ldots,A_n)$ is a $\Theta$-Weyl tuple if the groups satisfy the following commutation relations:
\begin{equation}
\label{Weyl-tuple}
e^{\i t A_j} e^{\i s A_k} 
= e^{\i \Theta_{jk} ts} e^{\i s A_k} e^{\i t A_j}, \quad j,k \in \{1,\ldots,n\},\: s,t \in \R. 
\end{equation}
\end{defi}

We will use the following useful notation that we will see as a generalization of \cite[notation 14.3 p.~284]{Hal13} and \cite[(4.1) p.~268]{NeP20}.

\begin{defi}
Let $(A_1,\ldots,A_n)$ be a $\Theta$-Weyl tuple. Let $(t_1,\ldots,t_n) \in \R^n$.
We define
\begin{equation}
\label{def-eitA}
e^{\i t \cdot A} 
\ov{\mathrm{def}}{=} e^{-\frac12 \i \langle t , \Theta^\uparrow t \rangle} \prod_{k=1}^n e^{\i t_k A_k} 
= e^{-\frac12 \i \langle t , \Theta^\uparrow t \rangle} e^{\i t_1 A_1} e^{\i t_2 A_2} \cdots e^{\i t_n A_n}. 
\end{equation}
\end{defi}


\begin{example} \normalfont
\label{ex-Hall}
Suppose that $n=2d$ for some integer $d \geq 1$. Consider a $d$-dimensional Weyl pair $(A,B)=(A_1,\ldots,A_d,B_1,\ldots,B_d)$ as in \cite[Definition 3.1 p.~262]{NeP20} i.e.~satisfying the relations \eqref{first1}, \eqref{second1} and \eqref{third}. We will use the matrix $\Theta=\begin{bmatrix}
  0   & - \I_d \\
  \I_d   &  0 \\
\end{bmatrix}$. We have $\Theta^\uparrow=\begin{bmatrix}
  0   &  -\I_d \\
  0   &  0 \\
\end{bmatrix}$. By introducing the notation $(A_{d+1},\ldots,A_{2d}) \ov{\mathrm{def}}{=} (B_{1},\ldots,B_{d})$, we can write $(A_1,\ldots,A_{2d})=(A_1,\ldots,A_{d},B_{1},\ldots,B_{d})$ and this notation defines a $\Theta$-Weyl tuple. Furthermore, if $t=(u,v)=(u_1,\ldots,u_d,v_1,\ldots,v_d)$ is an element of $\R^{2d}$, we obtain $\langle t , \Theta^\uparrow t \rangle=- u \cdot v$. Consequently, we recover the notation
\begin{equation}
\label{Weyl-Portal-2}
e^{\i (u A+v B)} 
\ov{\mathrm{def}}{=} e^{\frac12 \i u \cdot v} e^{\i u_1 A_1} \cdots e^{\i u_d A_d}
e^{\i v_1 B_1}  \cdots e^{\i v_d B_d}
\end{equation}
of \cite[Notation 14.3 p.~284]{Hal13} with $\hbar=1$ and $d$ instead of $n$ and the one of \cite[(4.1) p.~268]{NeP20}. In particular, the case $d=1$ gives the equality
\begin{equation}
\label{Weyl-Portal}
e^{\i (u A_1+v B_2)} 
\ov{\mathrm{def}}{=} e^{\frac12 \i u v} e^{\i u A_1} e^{\i v B_2}, \qquad u,v \in \R.
\end{equation}
\end{example}

Now, we prove a commutation rule for the operators defined by the equation \eqref{def-eitA}. This result says that we have a projective representation of $\R^n$ with respect to the 2-cocycle defined in \eqref{def-cocycle-intro}.
\begin{lemma}
\label{lem-Weyl-product}
Let $(A_1,\ldots,A_n)$ be a $\Theta$-Weyl tuple. Let $(t_1,\ldots,t_n),\: (s_1,\ldots,s_n) \in \R^n$. Then
\begin{equation}
\label{commute-eit}
e^{\i t \cdot A} \circ e^{\i s \cdot A} 
= e^{\frac{1}{2}\i  \langle t, \Theta s \rangle} e^{\i(t+s)\cdot A} . 
\end{equation}
Moreover, the map $\R^n \to \B(X)$, $t \mapsto e^{\i t \cdot A}$ is continuous when the space $\B(X)$ is equipped with the strong operator topology.
\end{lemma}

\begin{proof}
According to the commutation relations from Definition \ref{defi-Weyl-tuple}, we have for any $j \in \{1,\ldots,n-1\}$,
\begin{align}
\left( \prod_{k=j}^n e^{\i t_k A_k} \right) e^{\i s_j A_j} 
&=e^{\i t_j A_j} e^{\i t_{j+1} A_{j+1}} \cdots e^{\i t_n A_n} e^{\i s_j A_j} \nonumber \\
& \ov{\eqref{Weyl-tuple}}{=} e^{\i \sum_{k = j+1}^n \Theta_{kj} t_k s_j} e^{\i t_jA_j} e^{\i s_j A_j} \prod_{k=j+1}^n e^{\i t_k A_k} \nonumber \\
& = e^{-\i s_j (\Theta^\uparrow t)_j} e^{\i(t_j+s_j)A_j} \prod_{k=j+1}^n e^{\i t_k A_k} \label{error}
\end{align}
since $\sum_{k=j+1}^n \Theta_{kj} t_k s_j = - \sum_{k=j+1}^n \Theta_{jk} t_k s_j = - s_j(\Theta^\uparrow t)_j $. Using the equality \eqref{error} $n$ times with $j=1,\ldots,j=n$ to gradually eliminate the factor $ \prod_{j=1}^n e^{\i s_j A_j}$, we deduce that
\begin{align*}
e^{\i t \cdot A} \circ e^{\i s \cdot A} 
& \ov{\eqref{def-eitA}}{=}  e^{-\frac12 \i \langle t , \Theta^\uparrow t \rangle} e^{-\frac12 \i \langle s , \Theta^\uparrow s \rangle}  \left( \prod_{k=1}^n e^{\i t_k A_k} \right)\left( \prod_{j=1}^n e^{\i s_j A_j} \right) \\
&\ov{\eqref{error}}{=}  e^{-\frac12 \i ( \langle t, \Theta^\uparrow t \rangle + \langle s, \Theta^\uparrow s \rangle)} e^{-\i s_1 (\Theta^\uparrow t)_1} e^{\i(t_1+s_1)A_1} \left( \prod_{k=2}^n e^{\i t_k A_k} \right) \left( \prod_{j=2}^n e^{\i s_j A_j} \right) \\
&\ov{\eqref{error}}{=} \cdots 
= e^{-\frac12 \i ( \langle t, \Theta^\uparrow t \rangle + \langle s, \Theta^\uparrow s \rangle)} e^{-\i \sum_{k=1}^n s_k (\Theta^\uparrow t)_k} \prod_{k=1}^n e^{\i(t_k+s_k)A_k}.
\end{align*}
We conclude since the exponents in the exponential function simplify to
\begin{align*}
\MoveEqLeft -\frac12 \i \big[ \langle t, \Theta^\uparrow t \rangle + \langle s, \Theta^\uparrow s \rangle\big] - \i \langle s, \Theta^\uparrow t \rangle 
= - \frac12 \i \big[ \langle t, \Theta^\uparrow t \rangle + \langle s , \Theta^\uparrow s \rangle  + 2 \langle s , \Theta^\uparrow t \rangle\big] \\
& = - \frac12 \i \big[ \langle t, \Theta^\uparrow t \rangle + \langle s, \Theta^\uparrow s \rangle + \langle s, \Theta^\uparrow t \rangle + \langle t, \Theta^\uparrow s \rangle \big] + \frac12 \i \langle t, \Theta^\uparrow s \rangle - \frac12 \i \langle s, \Theta^\uparrow t \rangle \\
& = - \frac12 \i \big\langle (t+s), \Theta^\uparrow (t+s) \big\rangle + \frac12 \i \langle t, \Theta^\uparrow s \rangle - \frac12 \i \langle s, \Theta^\uparrow t \rangle \\
& = - \frac12 \i \big\langle (t+s), \Theta^\uparrow (t+s) \big\rangle + \frac12 \i \langle t, \Theta^\uparrow s \rangle + \frac12 \i \langle \Theta^\downarrow s , t \rangle \\
& = - \frac12 \i \big\langle (t+s), \Theta^\uparrow (t+s) \big\rangle + \frac12 \i \langle t, \Theta s \rangle.
\end{align*}
The strong continuity is a consequence of the joint continuity of the composition of operators on bounded sets for the strong operator topology, see \cite[Proposition C.19 p.~517]{EFHN15}.
\end{proof}


\begin{example} \normalfont
\label{ex-Hall-2}
Using the notations of Example \ref{ex-Hall}, , we observe that for any elements $t=(u_1,\ldots,u_d,v_1,\ldots,v_d)$ and $s=(u_1',\ldots,u_d',v_1',\ldots,v_d')$ in $\R^{2d}$ we have $\langle t , \Theta s \rangle= -u \cdot v' +v\cdot u'$. Hence, we recover the commutation rule
\begin{equation}
\label{equ-Hall-2}
e^{\i (u A+v B)} e^{\i (u' A+v' B)} 
=e^{-\frac12\i (u \cdot v' -v\cdot u')} e^{\i ((u+u') A+(v+v')B)} 
\end{equation}
of \cite[Notation 14.10]{Hal13} with $\hbar=1$ and the one of \cite[(4.2) p.~268]{NeP20}.
\end{example}


We proceed by introducing the Weyl functional calculus which is suitable for our $\Theta$-Weyl tuples.

\begin{defi}
\label{defi-Weyl-calculus}
Let $(A_1,\ldots,A_n)$ be a $\Theta$-Weyl tuple on a Banach space $X$. Consider some Schwartz class function $a \in \cal{S}(\R^n)$. 
We define the $\Theta$-Weyl functional calculus by
\begin{equation}
\label{Weyl-calculus}
a(A)x 
\ov{\mathrm{def}}{=} \frac{1}{(2 \pi)^{\frac{n}{2}}} \int_{\R^n} \hat{a}(u) e^{\i u \cdot A}x \d u, \quad x \in X,
\end{equation}
where we use the definition of the Fourier transform
\begin{equation}
\label{Fourier-transform}
\hat{a}(u) 
\ov{\mathrm{def}}{=} \frac{1}{(2\pi)^{\frac{n}{2}}} \int_{\R^n} a(x) e^{-\i \langle u , x \rangle } \d x. 
\end{equation}
\end{defi}

Since the groups the groups $(e^{\i tA_k})_{t \in \R}$ are bounded for any integer $k \in \{1, \ldots,n\}$, it follows that $a(A)$ defines a bounded operator on $X$ for any function $a \in \L^1(\R^n)$ with $\hat{a} \in \L^1(\R^n)$.

\begin{remark} \normalfont
\label{rem-tempered}
We can define $a(A)$ if $a$ is a tempered distribution on $\R^n$ with $\hat{a} \in \L^1(\R^n)$.
\end{remark}


\begin{example} \normalfont
\label{Ex-Portal-3}
With the notations of Example \ref{ex-Hall} where $n=2d$, we recover the Weyl calculus 
\begin{equation}
\label{}
a(A,B)x
\ov{\mathrm{def}}{=} \frac{1}{(2\pi)^{d}} \int_{\R^{2d}} \hat{a}(u,v) e^{\i(u A+v B)} x \d u \d v, \quad x \in X.
\end{equation}
of \cite[Definition 4.2]{NeP20}. 
\end{example}

\section{Twisted transference principle and universal $\Theta$-Weyl tuples}
\label{sec-twisted}

\paragraph{Twisted convolution}
Let $G$ be a locally compact group equipped with a left Haar measure $\mu_G$ and a Borel measurable 2-cocycle $\sigma \co G \times G \to \T$ on $G$. Consider a Banach space $X$. Following \cite[p.~120]{EdL69a}, we define the twisted convolution $f*_{\sigma} g \co G \to X$ of two functions $f \co G \to \mathbb{C}$ and $g \co G \to X$, when it exists, by
\begin{equation}
\label{twisted-convolution}
(f *_\sigma g)(t) 
\ov{\mathrm{def}}{=} \int_G \sigma(s,s^{-1}t)f(s) g(s^{-1}t) \d\mu_G(s), \quad t \in G.
\end{equation}
If $X=\mathbb{C}$, it is immediate to see that the twisted convolution is associative on $\L^1(G)$ and generalizes the classical convolution defined by
\begin{equation}
\label{Convolution-formulas}
(f*g)(t)
\ov{\mathrm{def}}{=} \int_G f(s)g(s^{-1}t) \d\mu_G(s), \quad t \in G.
\end{equation}
If the cocycle $\sigma$ is normalized, note that 
\begin{align}
\MoveEqLeft
\label{twisted-second}
(f *_\sigma g)(t)          
\ov{\eqref{very-useful-1}}{=} \int_G \sigma(t^{-1},s) f(s) g(s^{-1}t)\d\mu_G(s). 
\end{align}

\begin{example} \normalfont 
In the case of the group $G=\R^d$ equipped with the 2-cocycle $\sigma_\Theta \co \R^d \times \R^d \to \T$, $(s,t) \mapsto \e^{\frac{1}{2}\i\langle s, \Theta t\rangle}$, as introduced in \eqref{def-cocycle-intro}, we obtain with \eqref{twisted-second} and we will use the notation
\begin{equation}
\label{twisted-Rn}
(f *_{\Theta} g)(x) 
\ov{\mathrm{def}}{=} \int_{\R^d} f(y)g(x-y) \e^{-\frac{\i}{2} \langle x, \Theta y\rangle} \d y.
\end{equation}
\end{example}
We need the notion of regular operator. Indeed, in Section \ref{subsec-semigroup-A}, we will have to deal with contractive operators that are non-positive yet contractively regular.

\paragraph{Regular operators} Recall that a linear operator $T \co \L^p(\Omega) \to \L^p(\Omega')$ is regular if and only if for any Banach space $X$ the map $T \ot \Id_X$ induces a bounded operator between the Bochner spaces $\L^p(\Omega,X)$ and $\L^p(\Omega',X)$. In this case, the regular norm is given by
\begin{equation*}
\label{Norm-reg-c}
\norm{T}_{\reg,\L^p(\Omega) \to \L^p(\Omega')}
\ov{\mathrm{def}}{=} \sup_X \norm{T \ot \Id_X}_{\L^p(\Omega,X)\to \L^p(\Omega',X)},
\end{equation*}
where the supremum runs over all Banach spaces $X$. The operator $T$ is said to be contractively regular if $\norm{T}_{\reg,\L^p(\Omega) \to \L^p(\Omega')} \leq 1$. We refer to \cite{AbA02}, \cite{ArK23} and \cite{Pis10} for more information on regular operators. 

\begin{example} \normalfont
\label{Example-regular}
If $T \co \L^p(\Omega) \to \L^p(\Omega')$ is a positive operator, 
then $T$ is regular and we have the equality  
$
\norm{T}_{\reg,\L^p(\Omega) \to \L^p(\Omega')}
=\norm{T}_{\L^p(\Omega) \to \L^p(\Omega')}$. For any $p \in [1,\infty)$ with $p \not= 2$, it is known that any isometry $T \co \L^p(\Omega) \to \L^p(\Omega')$ is contractively regular. An absolute contraction, or a Dunford–Schwartz operator, over a measure
space $\Omega$ is an operator $T \co \L^1(\Omega) + \L^\infty(\Omega) \to \L^1(\Omega') + \L^\infty(\Omega')$ satisfying 
\begin{equation*}
\label{}
\norm{T(f)}_{\L^1(\Omega')} \leq \norm{f}_{\L^1(\Omega)}\text{ and } \norm{T(f)}_{\L^\infty(\Omega')} \leq \norm{f}_{\L^\infty(\Omega)}, \quad f \in \L^1(\Omega) \cap \L^\infty(\Omega).
\end{equation*}
According to \cite[Lemma 3.20 p.~36]{ArK23} and \cite[p.~13]{Pis10}, $T$ extends uniquely to a contractively regular operator $T \co \L^p(\Omega) \to \L^p(\Omega')$ for any $1 \leq p \leq \infty$. 
\end{example}

The proof of the following Young's inequality is elementary. 

\begin{lemma}
\label{lemma-Young-twisted}
Let $G$ be a locally compact group equipped with a left Haar measure $\mu_G$ and a Borel measurable 2-cocycle $\sigma \co G \times G \to \T$ on $G$. For any $1 \leq p \leq \infty$ and any Banach space $X$, the operator $C_{f,\sigma} \co \L^p(G,X) \to \L^p(G,X)$, $g \mapsto f *_{\sigma} g$\label{def-Cfsigma} is a well-defined linear operator on the Bochner space $\L^p(G,X)$ for any function $f \in \L^1(G)$ and we have
\begin{equation}
\label{Young-twisted}
\norm{f *_{\sigma} g}_{\L^p(G,X)}
\leq \norm{f}_{\L^1(G)} \norm{g}_{\L^p(G,X)}, \quad g \in \L^p(G,X).
\end{equation}
\end{lemma}

\begin{proof}
It suffices to show that $C_{f,\sigma}$ induces a bounded operator on the Banach spaces $\L^1(G)$ and $\L^\infty(G)$, because then, it is a regular operator on any $\L^p(G)$ by Example \ref{Example-regular}, and as such extends to a bounded operator on the Bochner space $\L^p(\R^n,X)$ for any Banach space $X$ by definition \cite[p.~13]{Pis10}. For any $t \in G$, one has
\begin{align*}
\MoveEqLeft
|(C_{f,\sigma} g)(t)| 
\ov{\eqref{twisted-convolution}}{=} \left|\int_G \sigma(s,s^{-1}t)f(s) g(s^{-1}t) \d\mu_G(s)\right| \\
&\leq  \int_{G} |f(s)| \, |g(s^{-1}t)| \d \mu_G(s) 
\ov{\eqref{Convolution-formulas}}{=} (|f| \ast |g|)(t).
\end{align*}
Therefore, by the classical Young inequality \cite[Proposition 2.40 p.~57]{Fol16}, we obtain the inequalities 
\begin{equation*}
\label{}
\norm{C_{f,\sigma} g}_{\L^1(G)} \leq \norm{f}_{\L^1(G)} \norm{g}_{\L^1(G)}
\end{equation*}
and 
\begin{equation*}
\label{}
\norm{C_{f,\sigma} g}_\infty \leq \norm{f}_{\L^1(G)} \norm{g}_{\L^\infty(G)}.
\end{equation*}
\end{proof}

\begin{remark} \normalfont
\label{twisted-conv-is-regular}
In particular, for any function $f \in \L^1(G)$, the twisted convolution operator $C_{f,\sigma} \co \L^p(G) \to \L^p(G)$, $g \mapsto f *_{\sigma} g$ is regular with 
\begin{equation*}
\label{}
\norm{C_{f,\sigma}}_{\reg, \L^p(G) \to \L^p(G)} 
\leq \norm{f}_{\L^1(G)}.
\end{equation*}
\end{remark}

\paragraph{Integrated form of a projective representation} Generalizing the notion of integrated form of a continuous representation of a locally compact group to continuous projective representations, we introduce the following definition. A continuous projective representation $\pi \co G \to \B(X)$ is said to be uniformly bounded if there exists a constant $M$ such that $\norm{\pi(s)}_{X \to X} \leq M$ for any $s \in G$.

\begin{defi}
\label{defi-integrated}
Let $G$ be a locally compact group endowed with a left Haar measure $\mu_G$. Consider a continuous uniformly bounded projective representation $\pi \co G \to \B(X)$ of $G$ on a Banach space $X$. For any function $f \in \L^1(G)$ and any $x \in X$, we let
\begin{equation}
\label{integrated-rep}
\pi(f)x 
\ov{\mathrm{def}}{=} \int_{G} f(s) \pi_{s}x \d \mu_G(s).
\end{equation}
\end{defi}
By \cite[Theorem A.22 p.~285]{Fol16}, the integral is a \textit{weak integral} and we have for any $x \in X$ and any function $f \in \L^1(G)$
\begin{equation*}
\label{}
\norm{\int_{G} f(s) \pi_{s}x \d \mu_G(s)}_X
\leq M\norm{x}_{X}\norm{f}_{\L^1(G)}.
\end{equation*} 
 So we have a linear operator $\pi(f) \co X \to X$ and this operator is clearly bounded with 
\begin{equation*}
\label{}
\norm{\pi(f)}_{X \to X} 
\leq M\norm{f}_{\L^1(G)}.
\end{equation*}
The operator $\pi(f)$ is also denoted by $\int_{G} f(s)\pi_{s} \d\mu_G(s)$.

\paragraph{Twisted transference} 
The following result is a generalization of the transference result \cite[Theorem 2.8 p.~6]{BGM89}. See also \cite{BPW94} and \cite[Theorem 3.15 p.~19]{CoW76}. 


\begin{prop}
\label{Prop matricial Transference} 
Let $G$ be an amenable locally compact group equipped with a normalized Borel measurable 2-cocycle $\sigma$. Consider a continuous $\sigma$-projective representation $\pi \co G \to \B(X)$ of $G$ on a separable Banach space $X$ which is uniformly bounded with
\begin{equation}
\label{def-de-M}
M
\ov{\mathrm{def}}{=} \sup\big\{\norm{\pi_s}_{X \to X} \co s \in G \big\}
<\infty.
\end{equation}
Suppose that $1 \leq p < \infty$. Then for any function $f \in \L^1(G)$, we have
\begin{equation}
\label{}
\Bgnorm{\int_{G} f(s)\pi_{s} \d\mu_G(s)}_{X \to X} 
\leq M^2 \norm{C_{f,\sigma}}_{\L^p(G,X) \to \L^p(G,X)},
\end{equation}
where we use the twisted convolution operator $C_{f,\sigma} \co g \to  f *_{\sigma} g$.
\end{prop}

\begin{proof}
Note that both sides of the inequality define continuous functions on $\L^1(G)$. So by density, we can suppose that the function $f$ is continuous with compact support $\cal{K}$. Let $\epsi >0$. Since the group $G$ is amenable, we can use Leptin's characterization of amenability \cite[Theorem 7.9 p.~70]{Pie84}. There exists a measurable subset $V$ of $G$ with $0 < \mu(V) < \infty$ such that 
\begin{equation*}
\label{}
\frac{\mu(\cal{K}^{-1}V)}{\mu(V)}
\leq 1+\epsi.
\end{equation*}
For any $x \in X$, we define the vector-valued function $h \co G \to X$ by
\begin{equation}
\label{def-func}
h(s)
\ov{\mathrm{def}}{=} 1_{V^{-1}\cal{K}}(s^{-1})\pi_{s^{-1}} x.
\end{equation}
Note that $h$ is strongly measurable and that
\begin{align}
\label{div-543}
\norm{h}_{\L^p(G,X)}^p
&\ov{\eqref{def-func}}{=} \int_G \bnorm{1_{V^{-1}\cal{K}}(s^{-1})\pi_{s^{-1}}x}_{X}^p \d s
=\int_G \norm{\pi_{s^{-1}}x}_{X}^p 1_{V^{-1}\cal{K}}(s^{-1}) \d s \\
&\ov{\eqref{def-de-M}}{\leq} M^p \int_G \norm{x}_{X}^p 1_{V^{-1}\cal{K}}(s^{-1}) \d s
=\mu(\cal{K}^{-1}V) M^p\norm{x}_{X}^p. \nonumber
\end{align}
Introducing the relation $\pi_{t}\pi_{t^{-1}} \ov{\eqref{eq-proj-rep}}{=} \sigma(t,t^{-1})\pi_e \ov{\eqref{normalized}}{=} \Id_X$ in the first equality, we observe that for any element $t \in V$
\begin{align*}
\MoveEqLeft
\Bgnorm{\int_{G} f(s) \pi_{s}x \d s}_{X}
= \Bgnorm{\pi_{t}\pi_{t^{-1}}\bigg(\int_{G} f(s)\pi_{s}x \d s\bigg)}_{X} 
\ov{\eqref{def-de-M}}{\leq} M \Bgnorm{\bigg(\int_{G} f(s) \pi_{t^{-1}}\pi_{s}x \d s\bigg) }_{X} \\
&\ov{\eqref{eq-proj-rep}}{=} M \Bgnorm{\int_{G} \sigma(t^{-1},s) f(s) \pi_{t^{-1}s}x \d s}_{X}.
\end{align*}
Raising to the $p^{\textrm{th}}$ power and averaging over $V$, we obtain 
\begin{align*}
\MoveEqLeft
\Bgnorm{\int_{G} f(s)\pi_{s}x \d s}_{X}^p
\leq M^p\frac{1}{\mu(V)}\int_{V}\Bgnorm{\int_{K} \sigma(t^{-1},s) f(s) \pi_{t^{-1}s}x \d s}_{X}^p \d t \\
&= M^p \frac{1}{\mu(V)} \int_{V}\Bgnorm{\int_{G} \sigma(t^{-1},s) f(s) 1_{V^{-1}\cal{K}}(t^{-1}s)\pi_{t^{-1}s}(x) \d s}_{X}^p \d t \\
&\ov{\eqref{def-func}}{\leq} M^p \frac{1}{\mu(V)} \int_{G} \Bgnorm{\int_{G} \sigma(t^{-1},s) f(s) h(s^{-1}t) \d s}_{X}^p \d t\\
&\ov{\eqref{twisted-second}}{=} M^p \frac{1}{\mu(V)} \int_{G} \bnorm{(f *_{\sigma}h)(t)}_{X}^p \d t \\
&\leq M^p \frac{1}{\mu(V)} \norm{C_{f,\sigma}}_{\L^p(G,X) \to \L^p(G,X)}^p \norm{h}_{\L^p(G,X)}^p  \\
&\ov{\eqref{div-543}}{\leq} M^{2p}(1+\epsi) \norm{C_{f,\sigma}}_{\L^p(G,X) \to \L^p(G,X)}^p \norm{x}_{X}^p.
\end{align*}
The result follows by taking the limit $\epsi \to 0$.
\end{proof}

\begin{remark} \normalfont
\label{rem-not-clear}
It seems to us that we cannot deduce this result from the standard transference result by using the Mackey group and the lifting of the projective representation on this group.
\end{remark}

Using the projective representation $t \mapsto e^{\i t\cdot A}$ provided by Lemma \ref{lem-Weyl-product}, we deduce the following result.

\begin{cor}
\label{cor-transference}
Let $A = (A_1,\ldots,A_n)$ be a $\Theta$-Weyl tuple acting on some separable Banach space $X$. Set 
\begin{equation*}
\label{}
M_A \ov{\mathrm{def}}{=} \sup_{t \in \R^n} \norm{e^{\i t\cdot A}}_{X \to X}.
\end{equation*}
Suppose that $1 \leq p < \infty$. Then for any tempered distribution $a$ with $\hat{a} \in \L^1(\R^n)$, we have
\begin{equation}
\label{transference-estimate}
\norm{a(A)}_{X \to X} 
\leq M_A^2 \norm{C_{\hat{a}}^\Theta}_{\L^p(\R^n,X) \to \L^p(\R^n,X)}
\end{equation}
where
\begin{equation}
\label{convolution-bis}
(C_a^\Theta g)(x) 
\ov{\mathrm{def}}{=} \frac{1}{(2\pi)^{\frac{n}{2}}} \int_{\R^n} e^{\frac12 \i \langle y , \Theta x \rangle} a(y) g(x-y) \d y. 
\end{equation}
\end{cor}

\paragraph{Universal $\Theta$-Weyl tuple} Suppose that $1 \leq p < \infty$. For each $k \in \{ 1, \ldots, n\}$, consider the strongly continuous groups $(e^{\i sQ_k})_{t \in \R}$ and $(e^{\i sP_k})_{t \in \R}$ of operators acting on the Banach space $\L^p(\R^n)$ defined by
\begin{equation}
\label{def-groupes}
(e^{\i sQ_k} f)(x) 
= e^{\i sx_k} f(x) 
\text{ and }
(e^{\i sP_k}f)(x) 
= f(x+s e_k), \quad f \in \L^p(\R^n).
\end{equation}
Note that the infinitesimal generators act on their domains by 
\begin{equation}
\label{pos-mom}
(Q_k f)(x) 
= x_k f(x)
\quad \text{and} \quad
(P_kf)(x) 
= -\i \partial_k f(x).
\end{equation}
We refer to \cite[p.~65]{EnN06} for more information. Now, we introduce the universal $\Theta$-Weyl tuple, which plays a special role in the theory of $\Theta$-Weyl tuples.

\begin{prop}
\label{prop-universal-Theta-Weyl-tuple}
Consider the families $Q = (Q_1,\ldots,Q_n)$ and $P = (P_1,\ldots,P_n)$. The tuple
\begin{equation}
\label{A-univ}
A_{\univ}^\Theta 
\ov{\mathrm{def}}{=} \bigg(\frac12 \sum_{l=1}^{n} \Theta_{1l}Q_l - P_1,\ldots,\frac12 \sum_{l=1}^{n} \Theta_{kl}Q_l - P_k,\ldots,\frac12 \sum_{l=1}^{n} \Theta_{nl}Q_l - P_n\bigg)
\end{equation}
is a $\Theta$-Weyl tuple. Moreover, its $\Theta$-Weyl calculus satisfies for any function $a \in \cal{S}(\R^n)$ the equality
\begin{equation}
\label{aA-as-twisted-conv}
a(A_{\univ}^\Theta) 
= C_{\hat{a}}^\Theta. 
\end{equation}
More precisely, we have
\begin{equation}
\label{equ-universal-Theta-Weyl-explicit-formula}
(e^{\i t \cdot A_{\univ}^\Theta}f)(x) 
= e^{\frac12 \i \langle t, \Theta x \rangle} f(x-t), \quad x,t \in \R^n.
\end{equation}
The tuple $A_{\univ}^\Theta$ is said to be the universal $\Theta$-Weyl tuple, also denoted by 
\begin{equation*}
\label{}
A_{\univ}^\Theta
=\frac12 \Theta Q - P.
\end{equation*}
\end{prop}

\begin{proof}
For any $k \in \{1, \ldots, n\}$, we have 
\begin{equation*}
\label{}
A_k = \frac12 \sum_{l=1}^{n} \Theta_{kl}Q_l - P_k.
\end{equation*}
Since $Q_l$ and $P_k$ commute for any distinct indices $l$ and $k$ and since $\Theta_{kk} = 0$ by skew-adjointness of $\Theta$, all the operators constituting $A_k$ commute. Thus, we see that
\begin{equation}
\label{equa-inter-678}
(e^{\i t A_k}f)(x) 
= \prod_{l=1}^n e^{\frac12 \i t \Theta_{kl} Q_l} e^{-\i t P_k}f(x) 
\ov{\eqref{def-groupes}}{=} \prod_{l=1}^n e^{\frac12 \i t \Theta_{kl} x_l} f(x - te_k). 
\end{equation}
We infer that
\begin{align*}
\MoveEqLeft
(e^{\i t A_j} e^{\i s A_k}f)(x) 
\ov{\eqref{equa-inter-678}}{=} e^{\i t A_j} \prod_{l=1}^n e^{\frac12 \i s \Theta_{kl} x_l} f(x - s e_k) \\
& \ov{\eqref{equa-inter-678}}{=} \prod_{m=1}^n e^{\frac12 \i t \Theta_{jm} x_m} \left(\prod_{l\neq j} e^{\frac12 \i s \Theta_{kl} x_l}\right) e^{\frac12 \i s \Theta_{kj} (x_j - t)} f(x- se_k - te_j) \\
& = \prod_{m=1}^n e^{\frac12 \i t \Theta_{jm} x_m} \prod_{l=1}^n e^{\frac12 \i s \Theta_{kl} x_l} e^{-\frac12 \i s t \Theta_{kj}} f(x - se_k - te_j).
\end{align*}
In the same way, we obtain
\begin{align*}
\MoveEqLeft
(e^{\i s A_k } e^{\i t A_j} f)(x) 
= \prod_{m=1}^n e^{\frac12 \i t \Theta_{jm} x_m} \prod_{l=1}^n e^{\frac12 \i s \Theta_{kl} x_l} e^{+\frac12 \i s t \Theta_{kj}} f(x - se_k - te_j) \\
& = e^{\i st \Theta_{kj}} e^{\i t A_j} e^{\i s A_k}f(x).
\end{align*}
By \eqref{Weyl-tuple}, we infer that $A_{\univ}^\Theta$ is a $\Theta$-Weyl tuple.

Next, we note that for any $t = (t_1,\ldots,t_n) \in \R^n$ and any $k = 1, \ldots, n$, we have 
\[ 
(e^{\i t_k A_k} f)(x) 
\ov{\eqref{equa-inter-678}}{=} \prod_{l=1}^n e^{\frac12 \i t_k \Theta_{kl} x_l} f(x - t e_k) 
= e^{\frac12 \i t_k (\Theta x)_k} f(x - te_k) .
\]
Thus we obtain
\begin{align*}
\MoveEqLeft
\Big(\prod_{k=1}^n e^{\i t_k A_k}f\Big)(x) 
= e^{\frac12 \i t_1 (\Theta x)_1} e^{\frac12 \i t_2(\Theta(x-t_1e_1))_2} \cdots e^{\frac12 \i t_n(\Theta(x-t_1e_1-\ldots - t_{n-1}e_{n-1}))_n}f(x-t) \\
& = e^{\frac12 \i \sum_{k=1}^n t_k (\Theta x)_k} e^{-\frac12 \i \sum_{k=2}^n \sum_{l = 1}^{k-1} t_k t_l \Theta_{kl}}f(x-t) \\
& = e^{\frac12 \i \langle t, \Theta x \rangle} e^{-\frac12 \i \langle t, \Theta^\downarrow t \rangle} f(x-t).
\end{align*}
We deduce that
\[
(e^{\i t \cdot A_{\univ}^\Theta}f)(x) 
\overset{\eqref{def-eitA}}{=} e^{-\frac12 \i \langle t , \Theta^\uparrow t \rangle } e^{-\frac12 \i \langle t , \Theta^\downarrow t \rangle} e^{\frac12 \i \langle t, \Theta x \rangle} f(x-t) 
= e^{\frac12 \i \langle t, \Theta x \rangle} f(x-t).
\]
We have proven \eqref{equ-universal-Theta-Weyl-explicit-formula}. Therefore, the Weyl calculus is given by 
\begin{align*}
\MoveEqLeft
(a(A_{\univ}^\Theta)f)(x) 
\ov{\eqref{Weyl-calculus}}{=} \frac{1}{(2\pi)^{\frac{n}{2}}} \int_{\R^n} \hat{a}(u) e^{\i u \cdot A_{\univ}^\Theta}f(x) \d u \\
&=  \frac{1}{(2\pi)^{\frac{n}{2}}} \int_{\R^n} \hat{a}(u) e^{\frac12 \i \langle u, \Theta x \rangle} f(x-u) \d u    \ov{\eqref{convolution-bis}}{=} C_{\hat{a}}^\Theta f(x).
\end{align*}
\end{proof}

\begin{example} \normalfont
\label{Example-J-n=2}
If $n=2$ and if $\Theta=\begin{bmatrix}
  0   & - 1 \\
  1   &  0 \\
\end{bmatrix}$, we recover the Weyl pair 
\begin{equation*}
\label{}
(A,B)=(-\tfrac{1}{2}Q_2-P_1,\tfrac{1}{2}Q_1-P_2)
\end{equation*}
considered in \cite[Lemma 6.4 p.~284]{NeP20}. Note that by Proposition \ref{prop-Kato}, we have 
\begin{equation*}
\label{}
[A,B]=-\i.
\end{equation*} 
\end{example}

\section{Composition property of the Weyl functional calculus}
\label{sec-composition-property}

The next result states that $\pi \co \L^1(G) \to \B(X)$ is a homomorphism of algebras, where the Banach space $\L^1(G)$ is equipped with the twisted convolution product $\ast_\sigma$, introduced in \eqref{twisted-convolution}.

\begin{prop}
\label{prop-integrated}
Let $\pi \co G \to \B(X)$ be a continuous $\sigma$-projective uniformly bounded representation of a locally compact group $G$ on a Banach space $X$. For any functions $f, g \in \L^1(G)$, we have 
\begin{equation}
\label{compose-proj-rep}
\pi(f)\pi(f)
=\pi(f \ast_\sigma g).
\end{equation}
\end{prop}

\begin{proof}
For any $x \in X$, we have in a \textit{weak sense}
\begin{align*}
\pi(f) \pi(g) x
&\ov{\eqref{integrated-rep}}{=} \bigg(\int_{G} f(s) \pi_{s} \d s\bigg) \bigg(\int_{G} g(t) \pi_{t}x \d t\bigg) \\
&= \int_{G} \int_{G} f(s) g(t) \pi_{s} \pi_{t}x \d s \d t \\
&\ov{\eqref{eq-proj-rep}}{=} \int_{G} \int_{G} \sigma(s,t) f(s) g(t) \pi_{s t}x \d s \d t \\
&=\int_{G} \int_G \sigma(s,s^{-1}t)f(s) g(s^{-1}t) \pi_{t}x\d\mu_G(s)  \d t\\
&\ov{\eqref{twisted-convolution}}{=} \int_{G} (f \ast_\sigma g)(t) \pi_{t}x \d t
\ov{\eqref{integrated-rep}}{=} \pi(f \ast_\sigma g).
\end{align*}
\end{proof}
Let $G$ be a locally compact abelian group. The Moyal product $a \star_{\sigma} b \co G \to \mathbb{C}$ of two (enough regular) functions $a,b \co \hat{G} \to \mathbb{C}$ is defined by
\begin{equation}
\label{def-Moyal-product-33}
\widehat{a \star_{\sigma} b}
\ov{\mathrm{def}}{=} \hat{a} \mathbin{\ast_{\sigma}} \hat{b},
\end{equation}
where $\hat{a}$ and $\hat{b}$ are the Fourier transform of $a$ and $b$.

\begin{example} \normalfont
\label{exa-Moyal-product-star-Theta}
Let $\Theta \in \M_n(\R)$ be a skew-symmetric matrix. Moreover, consider the associated 2-cocycle $\sigma_\Theta \co \R^n \times \R^n \to \T$, $(x,y) \mapsto e^{\frac12 \i \langle x, \Theta y \rangle}$ of Example \ref{cocycle-Rd}. Then for any pair of Schwartz functions $f,g \in \cal{S}(\R^n)$, we have
\begin{equation}
\label{Moyal-def}
(2\pi)^{\frac{n}{2}} \widehat{f \star_\Theta g}
= \hat{f} \mathbin{\ast_{\sigma_\Theta}} \hat{g}.
\end{equation}
%
Indeed, first observe that the twisted convolution of Fourier transforms equals
\begin{align*}
\MoveEqLeft
(\hat{f} \mathbin{\ast_{\sigma_\Theta}} \hat{g})(u) 
\ov{\eqref{twisted-convolution}}{=} \int_{\R^n} \sigma_\Theta(\xi,u-\xi)\hat{f}(\xi) \hat{g}(u-\xi)  \d \xi \\
&\ov{\eqref{def-cocycle-intro}}{=} 
\int_{\R^n} \e^{\frac{1}{2}\i\langle \xi, \Theta (u-\xi)\rangle} \hat{f}(\xi)\hat{g}(u-\xi)  \d \xi,
\end{align*}
where $u \in \R^n$. If $A(u) \ov{\mathrm{def}}{=} e^{\i\langle u, x \rangle} \hat{g}(u)$, the inverse Fourier transform of $\hat{f} \mathbin{\ast_{\sigma_\Theta}} \hat{g}$ is given by (note that since $\Theta$ is skew-symmetric, we have$\langle \xi , \Theta \xi \rangle = 0$)
\begin{align*}
\MoveEqLeft
 \frac{1}{(2\pi)^{\frac{n}{2}}} \int_{\R^n}e^{\i \langle u , x \rangle} \bigg(\int_{\R^n} e^{\frac12 \i \langle \xi,\Theta(u - \xi) \rangle}\hat{f}(\xi)\hat{g}(u-\xi)   \d\xi \bigg) \d u\label{equ-1-lemma-homo}\\
& = \frac{1}{(2\pi)^{\frac{n}{2}}} \int_{\R^n} \bigg(\int_{\R^n} \underset{\ov{\mathrm{def}}{=}B_u(\xi)}{\underbrace{\left[ e^{\i \langle \xi , x + \frac12 \Theta u \rangle } \hat{f}(\xi)\right]}}\underset{= A(u-\xi)}{\underbrace{\left[ e^{\i\langle u - \xi, x \rangle} \hat{g}(u-\xi) \right]}}  \d\xi \bigg) \d u\nonumber \\
&\ov{\eqref{Convolution-formulas}}{=} \frac{1}{(2\pi)^{\frac{n}{2}}} \int_{\R^n}(B_u \ast A)(u) \d u 
=  \int_{\R^n} \mathcal{F}\left[\check{B}_u \cdot \check{A} \right](u) \d u \nonumber \\
& \ov{\eqref{Fourier-transform}}{=} \frac{1}{(2\pi)^{\frac{n}{2}}} \int_{\R^n} \int_{\R^n} e^{- \i \langle u, s \rangle} \check{B}_u(s) \check{A}(s)  \d s \d u \nonumber \\
&=\frac{1}{(2\pi)^{\frac{n}{2}}}\int_{\R^n} \int_{\R^n} e^{-\i \langle u , s \rangle }  f(s+x + \textstyle{\frac{1}{2}} \Theta u) g(x+s)\d s \d u. \nonumber
\end{align*}
Now, assume that the matrix $\Theta$ is invertible. Noting that its inverse, $\Theta^{-1}$ is also skew-symmetric (in particular $\langle 2\Theta^{-1}s,s \rangle= 0$ for any $s \in \R^n$), we proceed with a change of variables and continue the computation as follows.
Note that if $\Theta$ is invertible, $u \mapsto f(s+x+\frac12 \Theta u)$ is rapidly decaying and we can apply Fubini's theorem. Hence
\begin{align*}
\MoveEqLeft 
\mathcal{F}^{-1}(\hat{f} \ast_{\sigma_\Theta} \hat{g})(x) = \frac{1}{(2\pi)^{\frac{n}{2}}} \int_{\R^n} \bigg(\int_{\R^n} e^{- \i \langle u, s \rangle} f\big(s+x + {\textstyle \frac{1}{2}} \Theta u\big) g(x+s)  \d s \bigg)\d u \\
&= \frac{1}{(2\pi)^{\frac{n}{2}}} \int_{\R^n} \bigg(\int_{\R^n} e^{- \i \langle u, s \rangle} f\big(s+x + \textstyle{\frac{1}{2}} \Theta u\big) g(x+s)  \d u\bigg) \d s\\
& \ov{u = v - 2 \Theta^{-1}s}{=} \frac{1}{(2\pi)^{\frac{n}{2}}} \int_{\R^n} \int_{\R^n} e^{- \i \langle (v -2 \Theta^{-1}s), s \rangle}  f\big(s+x + \textstyle{\frac{1}{2}} \Theta (v-2 \Theta^{-1}s)\big) g(x+s) \d v \d s\\
& = \frac{1}{(2\pi)^{\frac{n}{2}}} \int_{\R^n} \int_{\R^n} e^{- \i \langle v, s \rangle}  f\big(x +\textstyle{\frac{1}{2}} \Theta v\big) g(x+s) \d v \d s \\
& = \frac{1}{(2\pi)^{\frac{n}{2}}} \int_{\R^n} \int_{\R^n} e^{- \i \langle v, s \rangle}  f\big(x +\textstyle{\frac{1}{2}} \Theta v\big) g(x+s) \d s \d v\\ 
&\ov{\eqref{star-product}}{=} (2\pi)^{\frac{n}{2}}(f \star_\Theta g)(x).
\end{align*}
We now turn to the general case. The skew-symmetric matrix $\Theta \in \M_n(\R)$ can always be described by a block diagonal form through an orthogonal change of basis. More precisely, there exists an orthogonal matrix $O \in \M_n(\R)$ such that  
\begin{equation}
\label{diago-Theta-bis}
\Theta 
= O^T J O,
\end{equation}  
where $J$ is a block diagonal matrix consisting of $2 \times 2$ skew-symmetric blocks, given by  
\begin{equation}
\label{def-de-J-bis}
J \ov{\mathrm{def}}{=}  
\begin{pmatrix} 
J_1 & 0 & 0 & \ldots & 0 \\ 
0 & J_2 & 0 & \ldots & 0\\ 
\vdots & \ddots & \ddots & \vdots & \vdots \\ 
0 & \ldots & \ddots & J_k & 0 \\ 
0 & \ldots & \ldots & \ldots & 0 
\end{pmatrix},  
\quad \text{where }  
J_j \ov{\mathrm{def}}{=} 
\begin{pmatrix} 
0 & \alpha_j \\ 
- \alpha_j & 0 
\end{pmatrix},  
\quad j = 1,\ldots,k.
\end{equation}  
Here, $k$ is an integer satisfying $0 \leq k \leq \frac{n}{2}$, and the scalars $\alpha_1, \dots, \alpha_k$ are strictly positive real numbers. 

In the following calculation, we will use that if $h$ is a sufficiently regular and decaying function, then
\begin{equation}
\label{useful-3456}
\frac{1}{(2\pi)^m}\int_{\R^m} \int_{\R^m} e^{-\i \langle u, s\rangle} h(s+x)\d s \d u 
= h(x).
\end{equation}
Indeed, we have
\begin{align*}
\MoveEqLeft
\frac{1}{(2\pi)^m} \int_{\R^m} \int_{\R^m} e^{-\i \langle u , s \rangle} h(s+x)\d s \d u 
= \frac{1}{(2\pi)^m} \int_{\R^m} \int_{\R^m} e^{-\i \langle u , s-x \rangle} h(s) \d s \d u \\
&=\frac{1}{(2\pi)^m} \int_{\R^m} \bigg(\int_{\R^m} e^{-\i \langle u , s \rangle} h(s)\d s\bigg)e^{\i \langle u,x \rangle}  \d u \\
& \ov{\eqref{Fourier-transform}}{=} \frac{1}{(2\pi)^\frac{m}{2}} \int_{\R^m} \hat{h}(u) e^{\i \langle u,x \rangle} \d u 
= h(x).
\end{align*}
Assume now that $\Theta$ has block diagonal form $\Theta = \begin{pmatrix} \Theta' & 0 \\ 0 & 0 \end{pmatrix}$ over $\R^n = \R^{2k+(n-2k)}$ with $\Theta' \in \M_{2k}(\R)$ skew-symmetric and invertible.
Then according to the invertible case (over $\R^{2k}$), with variables $x = (x_1,x_2) \in \R^{2k} \times \R^{n-2k} = \R^n$ and similar notation for $s,u \in \R^n$, we obtain the following.
Hereby, note firstly that according to the invertible case applied with the functions $f'(s_1) \ov{\mathrm{def}}{=} f(s_1,s_2+x_2)$, $g'(s_1) \ov{\mathrm{def}}{=} g(s_1,s_2+x_2)$, the invertible matrix $\Theta'$ and $x_1$ instead of $x$, the calculation preceding \eqref{diago-Theta-bis} yields the identity
\begin{align}
\label{equ-1-exa-Moyal-product-star-Theta}
\MoveEqLeft
\int_{\R^{2k}}\int_{\R^{2k}} e^{-\i \langle u_1,s_1 \rangle}f\big(s+x+\tfrac12 (\Theta'u_1,0)\big) g(s+x) \d u_1 \d s_1\\
& = \int_{\R^{2k}}\int_{\R^{2k}} e^{-\i \langle u_1,s_1 \rangle} f\big((0,s_2)+x+\tfrac12(\Theta'u_1,0)\big) g(s+x) \d u_1 \d s_1. \nonumber
\end{align}
Then note secondly that we can use Fubini's theorem to interchange the integrals $\int_{\R^{2k}}$ and $\int_{\R^{n}}$, since 
\[ 
\int_{\R^{2k}} \int_{\R^{n}} \left| e^{-\i \langle u_1, s_1 \rangle} f\big(s+x+\tfrac12 (\Theta'u_1,0)\big)g(s+x)\right| \d s \d u_1 
< \infty.
\]
We get
\begin{align*}
\MoveEqLeft
(2\pi)^{\frac{n}{2}} \mathcal{F}^{-1}(\hat{f} \ast_{\sigma_\Theta} \hat{g})(x) 
= \int_{\R^n} \bigg(\int_{\R^n} e^{- \i \langle u, s \rangle} f\big(s+x + \tfrac{1}{2} \Theta u\big) g(x+s)  \d s \bigg)\d u \\
& =  \int_{\R^{n-2k}} \int_{\R^{n-2k}} e^{-\i \langle u_2,s_2\rangle} \bigg(\int_{\R^{2k}} \int_{\R^{2k}} e^{- \i \langle u_1,s_1\rangle}  f\big(s+x+\tfrac{1}{2} (\Theta' u_1,0)\big) g(s+x) \\
&\d u_1 \d s_1 \bigg)\d s_2 \d u_2 \\
& \overset{\eqref{equ-1-exa-Moyal-product-star-Theta}}{=} \int_{\R^{n-2k}} \int_{\R^{n-2k}} e^{-\i \langle u_2,s_2\rangle} \bigg(\int_{\R^{2k}} \int_{\R^{2k}} e^{- \i \langle u_1,s_1\rangle}   f\big((0,s_2)+x + \tfrac{1}{2} (\Theta' u_1,0)\big)  \\
&g(s+x)\d u_1  \d s_1 \bigg)\d s_2 \d u_2
=\int_{\R^{n-2k}} \int_{\R^{n-2k}} e^{-\i \langle u_2,s_2\rangle} h(s_2+x_2) \d s_2 \d u_2,
\end{align*}
where 
\begin{equation*}
\label{}
h(s_2) 
\ov{\mathrm{def}}{=} \int_{\R^{2k}} \int_{\R^{2k}} e^{- \i \langle u_1 , s_1 \rangle} f(x_1+\tfrac12\Theta'u_1,s_2)g(s_1+x_1,s_2) \d u_1 \d s_1.
\end{equation*}
Thus, according to \eqref{useful-3456} with $m$ replaced by $n-2k$, the previous double integral equals
\begin{equation}
\label{inter-789}
(2\pi)^{n-2k} h(x_2)
= (2\pi)^{n-2k}\int_{\R^{2k}} \int_{\R^{2k}} e^{- \i \langle u_1, s_1 \rangle} f\big(x_1+\tfrac12\Theta'u_1,x_2\big) g(s_1+x_1,x_2) \d u_1 \d s_1. 
\end{equation}
In a similar manner, we have
\begin{align*}
\MoveEqLeft
(2\pi)^n (f\star_\Theta g)(x) 
\ov{\eqref{star-product}}{=} \int_{\R^n}\int_{\R^n} e^{-\i \langle u,s \rangle}f(x+\tfrac{1}{2}\Theta u)g(x+s)  \d s \d u\\
& =  \int_{\R^{n-2k}} \int_{\R^{n-2k}}  e^{-\i \langle u_2,s_2\rangle}\bigg( \int_{\R^{2k}} \int_{\R^{2k}} e^{- \i \langle u_1,s_1\rangle} f\big(x+\tfrac{1}{2} (\Theta' u_1,0)\big) g(s+x)\\
&\d u_1  \d s_1\bigg) \d s_2 \d u_2 
= \int_{\R^{n-2k}} \int_{\R^{n-2k}} e^{-\i \langle u_2,s_2\rangle} l(s_2+x_2)\d s_2\d u_2  ,
\end{align*}
where $l(s_2) \ov{\mathrm{def}}{=} \int_{\R^{2k}} \int_{\R^{2k}} e^{- \i \langle u_1, s_1 \rangle} f\big(x +\tfrac12 (\Theta' u_1,0)\big) g(s_1+x_1,s_2)\d u_1 \d s_1$.
Thus, according to \eqref{useful-3456} with $h$ replaced by $l$, the previous double integral equals
\begin{align*}
\MoveEqLeft
(2\pi)^{n-2k}l(x_2) 
= (2\pi)^{n-2k}\int_{\R^{2k}} \int_{\R^{2k}} e^{- \i \langle u_1, s_1 \rangle}f\big(x_1+\tfrac12\Theta'u_1,x_2\big) g(s_1+x_1,x_2)  \d u_1 \d s_1 \\
&\ov{\eqref{inter-789}}{=} (2\pi)^{n-2k}h(x_2).         
\end{align*}
We have shown \eqref{Moyal-def} in the block diagonal case.

In the case of a general skew-adjoint matrix $\Theta$, write $\Theta = O^T J O$ with $J$ in block diagonal form as previously and $O$ some orthogonal matrix. Then some simple changes of variables reveal that
\begin{align*}
\MoveEqLeft
(2\pi)^{\frac{n}{2}} \mathcal{F}^{-1} (\hat{f} \ast_{\sigma_\Theta} \hat{g})(x) = \int_{\R^n} \bigg(\int_{\R^n} e^{- \i \langle u, s \rangle} f\big(s+x + \tfrac{1}{2} \Theta u\big) g(x+s)  \d s \bigg)\d u \\
&=\int_{\R^n} \bigg(\int_{\R^n} e^{- \i \langle u, s \rangle} f\big(s+x + \tfrac{1}{2} O^T J O u\big) g(x+s)  \d s \bigg)\d u\\
&\ov{v=O u}{=}\int_{\R^n} \bigg(\int_{\R^n} e^{- \i \langle O^Tv, s \rangle} f\big(s+x + \tfrac{1}{2} O^T J v\big) g(x+s)  \d s \bigg)\d v\\
&\ov{r=O s}{=} \int_{\R^n} \int_{\R^n} e^{ - \i \langle v , r \rangle}  f\big(O^T(r+ Ox + \tfrac12 Jv)\big) g\big(O^T(Ox+r)\big) \d r \d v \\
& = (2\pi)^{\frac{n}{2}} \cal{F}^{-1}\Big(\widehat{f(O^T \cdot)} \ast_{\sigma_J} \widehat{g(O^T \cdot)}\Big)(Ox).
\end{align*}
and similarly
\begin{align*}
\MoveEqLeft
(2\pi)^n (f \star_\Theta g)(x) 
\ov{\eqref{star-product}}{=} \int_{\R^n}\int_{\R^n} e^{-\i \langle u,s \rangle} f\big(x+\tfrac{1}{2}\Theta u\big)g(x+s)  \d s \d u \\
&=\int_{\R^n} \int_{\R^n} e^{-\i \langle u,s \rangle} f\big(x+\tfrac{1}{2}O^T J O u\big)g(x+s)  \d s \d u\\
&\ov{v=O u}{=} \int_{\R^n} \int_{\R^n} e^{-\i \langle O^T v,s \rangle} f\big(x+\tfrac{1}{2}O^T J v\big)g(x+s)  \d s \d v\\
&\ov{r=O s}{=} \int_{\R^n} \int_{\R^n} e^{-\i \langle v,r \rangle} f\big(O^T(Ox + \tfrac12 Jv)\big) g\big(O^T(Ox+r)  \big) \d r \d v \\
& \ov{\eqref{star-product}}{=} (2\pi)^n \big(f(O^T \cdot) \star_J g(O^T \cdot)\big)(Ox).
\end{align*}
Thus the general case follows from the block diagonal case.
\end{example}


Using the $\sigma_\Theta$-projective representation $\pi \co \R^n \to \B(X)$, $u \mapsto e^{\i u \cdot A}$ introduced in \eqref{lem-Weyl-product}, we obtain the following composition formula for the Weyl calculus defined in \eqref{Weyl-calculus}. 

\begin{prop}
\label{prop-compo-Weyl-34}
Let $(A_1,\ldots,A_n)$ be a $\Theta$-Weyl tuple on a Banach space $X$. For any Schwartz class functions $a,b \in \cal{S}(\R^n)$, we have
\begin{equation*}
\label{}
a(A) \circ b(A) 
= (a \star_\Theta b)(A).
\end{equation*}
\end{prop}

\begin{proof}
We have
\begin{align*}
\MoveEqLeft
a(A) \circ b(A) 
\ov{\eqref{Weyl-calculus}}{=} \frac{1}{(2\pi)^{n}} \bigg(\int_{\R^n} \hat{a}(u)\pi_u  \d u\bigg)\bigg(\int_{\R^n} \hat{b}(u) \pi_u \d u\bigg)       
\ov{\eqref{integrated-rep}}{=} \frac{1}{(2\pi)^{n}}\pi(\hat{a}) \pi(\hat{b}) \\
&\ov{\eqref{compose-proj-rep}}{=} \frac{1}{(2\pi)^{n}}\pi(\hat{a} \ast_{\sigma_\Theta} \hat{b})
\ov{\eqref{Moyal-def}}{=} \frac{1}{(2\pi)^{\frac{n}{2}}}\pi(\widehat{a \star_{\Theta} b}) \\
&\ov{\eqref{integrated-rep}}{=} \frac{1}{(2\pi)^{\frac{n}{2}}} \int_{\R^n} (\widehat{a \star_{\Theta} b})(u) \pi_u \d u
\ov{\eqref{Weyl-calculus}}{=} (a \star_\Theta b)(A).
\end{align*}
\end{proof}

\section{Associated semigroup and functional calculus of the operator $\sum_{k=1}^n A_k^2$}
\label{subsec-semigroup-A}

In this section, we introduce a semigroup generated by the <<sum of squares>> $\sum_{k=1}^n A_k^2$ associated with a $\Theta$-Weyl tuple $A = (A_1,\ldots,A_n)$. In the particular case where $A$ is the universal $\Theta$-Weyl tuple $A_\univ ^{\Theta}$, we can explicitly describe the integral kernel of this semigroup and show the complete boundedness of the $\HI(\Sigma_\theta)$ functional calculus of the generator for  any angle $\frac{\pi}{2}<\theta<\pi$.

\paragraph{Nilpotent Lie algebras and nilpotent Lie groups} We need to briefly review some elementary theory of nilpotent Lie groups.  Let $\g$ be a finite-dimensional Lie algebra over $\mathbb{R}$. The descending central series of $\g$ is defined inductively as in \cite[Definition 5.2.1 p.~91]{HiN12} by
\[
\g^{1}
\ov{\mathrm{def}}{=} \g, \quad \g^{2}=[\g,\g], \quad \g^{3}=[\g,[\g,\g]]
\quad \text{and} \quad
\g^{k+1}
\ov{\mathrm{def}}{=}[\g,\g^{k}].
\]
We have $\g^{k+1} \subset \g^{k}$ for any integer $k \geq 1$. The Lie algebra $\mathfrak{g}$ is said to be nilpotent if $\g^{s+1}=\{0\}$ for some $s \in \mathbb{N}$. If $s$ is the smallest integer such that $\g^{s+1}=0$, then $\g$ is said to be nilpotent of step $s$. Every abelian Lie algebra is nilpotent of step 1. By \cite[Theorem 11.2.5 p.~444]{HiN12}, a connected Lie group $G$ is nilpotent if and only if its Lie algebra is nilpotent. A connected nilpotent Lie group $G$ is said to be of step $s$ if its Lie algebra has step $s$.

A stratification \cite[Definition 2.2.3 p.~122]{BL07} (see also \cite[Definition 3.1.5 p.~93]{FiR16}) 
of a finite-dimensional real Lie algebra $\mathfrak{g}$ is a vector space decomposition $
\mathfrak{g}
= \mathfrak{g}_1 \oplus \dots \oplus \mathfrak{g}_s$ such that 
\begin{equation*}
\g_{k+1}
=[\g_1 , \g_k ], \quad 1 \leq k \leq s-1, \quad [\g_1 , \g_s ]=0.
\end{equation*}
In other words, we have
\begin{equation*}
\label{}
\g
=\g_1 \oplus [\g_1,\g_1] \oplus [\g_1,[\g_1,\g_1]] \oplus [\g_1,[\g_1,[\g_1,\g_1]]] \oplus \cdots.
\end{equation*}
The subspace $\g_{1}$ is often referred to as the first layer or the horizontal layer and generates $\g$ as a Lie algebra by \cite[Proposition A.17 p.~346]{BiB19}. 
By \cite[Lemma A.20 p.~347]{BiB19}, the existence of such a stratification with $\g_s\not= \{0\}$ implies that the Lie algebra $\g$ is nilpotent of step $s$. We say that a simply connected nilpotent Lie group $G$ is stratified when its Lie algebra is equipped with a stratification and we say that it has $m$ generators if $m \ov{\mathrm{def}}{=} \dim \g_1$.



%

\paragraph{The Mackey group} Let $n \geq 1$. We start by describing some information on the Mackey group $\mathbb{H}_\Theta \ov{\mathrm{def}}{=} \R^n \times \T$ \label{def-Htheta}introduced in \eqref{def-central-extension} associated to the 2-cocycle $\sigma_\Theta$ defined in \eqref{def-cocycle-intro} on the group $\R^n$ by the skew-symmetric matrix $\Theta$.

\begin{prop}
\label{prop-Lie-group-exists}
There exists a unique structure of Lie group on the locally compact group $\mathbb{H}_\Theta$. Moreover, this group is connected  and nilpotent. Furthermore,  its Lie algebra $\mathfrak{h}_\Theta$ identifies to $\R^n \oplus \R$ equipped with the Lie bracket
\begin{equation}
\label{Lie-bracket}
[(x_1,z_1),(x_2,z_2)]
=(0,\la x_1,\Theta x_2 \ra), \quad x_1,x_2 \in\R^n, z_1,z_2 \in \R
\end{equation}
and if $\Theta \not=0$, this decomposition defines a 2-step stratification $\mathfrak{h}_\Theta=\mathfrak{g}_1 \oplus \mathfrak{g}_2$ of $\mathfrak{h}_\Theta$.
\end{prop}

\begin{proof}
Recall that by \eqref{central-extension} the group $\mathbb{H}_\Theta$ is a central extension of the group $\R^n$ by $\T$. Consequently, it admits by Theorem \ref{Th-Lie-extension} a canonical structure of Lie group. Moreover, by \eqref{central-extension-algebras} we deduce that its Lie algebra $\mathfrak{h}_\Theta$ is a central extension of the Lie algebra $\mathfrak{t}$ of $\mathbb{T}$ by the Lie algebra $\g$ of $\R^n$. This means that we have an exact sequence 
\begin{equation*}
\label{}
0 \to \mathfrak{t} \to \mathfrak{h}_\Theta \to \g \to 0
\end{equation*}
with $[\mathfrak{t},\mathfrak{h}_\Theta] = 0$. By \eqref{bracket-central}, the Lie algebra $\mathfrak{h}_\Theta$ identifies as a vector space to $\g \oplus \mathfrak{t}= \R^n \oplus \i\R$ and that the Lie bracket has the form
\begin{equation*}
\label{}
[(x_1,\i z_1),(x_2,\i z_2)]
\ov{\eqref{bracket-central}}{=} (0,c(x_1,x_2)), \quad x_1,x_2 \in\R^n, z_1,z_2 \in \R,
\end{equation*}
where $c \co \R^n \times \R^n \to \i\R$ is the associated Lie algebra 2-cocycle. Here, we use the well-known identification of $\mathfrak{t}$ with the real vector space $\i\R$, e.g.~\cite[p.~71]{HiN12}. We can obtain a concrete description of this cocycle by using a suitable differentiation. Indeed, for any $x_1,x_2 \in \R^n$, we have 
\begin{align*}
\MoveEqLeft
c(x_1,x_2) \\
&\ov{\eqref{recov-cocycle}}{=} \frac{\d^2}{\d s\d t}|_{t=0,s=0} \big[\sigma_\Theta(\exp_{\R^n}(t x_1),\exp_{\R^n}(s x_2)) - \sigma_\Theta(\exp_{\R^n}(s x_2),\exp_{\R^n}( tx_1)\big] \\
&=\frac{\d^2}{\d s\d t}|_{t=0,s=0} \big[\sigma_\Theta(tx_1,s x_2) - \sigma_\Theta(s x_2, t x_1)\big] \\
& \ov{\eqref{def-cocycle-intro}}{=} \frac{\d^2}{\d s\d t}|_{t=0,s=0} \big[e^{\frac{1}{2}\i\langle t x_1, \Theta sx_2\rangle}  -e^{\frac{1}{2}\i\langle  s x_2, \Theta t x_1\rangle} \big] \\
&=\frac{\d^2}{\d s\d t}|_{t=0,s=0} \big[e^{\frac{1}{2}\i t s\langle  x_1, \Theta x_2\rangle}  -e^{\frac{1}{2}\i s t\langle   x_2, \Theta  x_1\rangle} \big] \\
&=\frac{\i}{2}\la x_1,\Theta x_2 \ra-\frac{\i}{2}\la x_2,\Theta x_1 \ra
=\i\la x_1,\Theta x_2 \ra.         
\end{align*}
Now, identifying $\i\R$ with $\R$ we obtain \eqref{Lie-bracket}. Furthermore, it is clear that the decomposition $\mathfrak{h}_\Theta = \R^n \oplus \R$ defines a 2-step stratification of the Lie algebra $\mathfrak{h}_\Theta$ if $\Theta \not=0$.
\end{proof}

\begin{remark} \normalfont
We can easily describe some additional properties of this group. Indeed, it is unimodular by \cite[p.~53]{Fol16} or \cite[Proposition 10.4.11 p.~410]{HiN12} 
and amenable since any solvable locally compact group is amenable by \cite[Corollary 13.5 p.~120]{Pie84}.
\end{remark}


Recall that by \cite[Proposition 9.5.1 p.~335]{HiN12} the Lie group $\mathbb{H}_\Theta$ and its universal covering $\tilde{\mathbb{H}}_\Theta$ admit isomorphic Lie algebras. Consequently, the universal covering $\tilde{\mathbb{H}}_\Theta$ is a (simply connected) 2-step stratified nilpotent Lie group. So the previous result leads us to introduce some useful concepts for the study of 2-step stratified Lie groups. It is widely recognized that these groups are non-abelian Lie groups which closely resemble abelian groups, although they display distinct phenomena not present in abelian groups.

\paragraph{Kaplan $J$-operators on 2-step stratified Lie groups} 
Let $G$ be a 2-step stratified connected Lie group and $\mathfrak{g} = \mathfrak{g}_1 \oplus \mathfrak{g}_2$ be the stratification of its Lie algebra. Let $\mathfrak{g}_2^*$ be the dual of $\mathfrak{g}_2$.
Any element $\mu \in \g_2^*$ gives rise to a skew-symmetric bilinear form $\omega_\mu \co \mathfrak{g}_1 \times \mathfrak{g}_1 \to \R$ defined by
\begin{equation*}
\label{def-omega-mu}
\omega_\mu(x,y) 
= \mu([x,y]),\quad x,y \in \g_1.
\end{equation*}
Following \cite[Proposition 3.7.3 p.~174]{BL07}, we say that the group $G$ is called a M\'etivier group if $\omega_\mu$ is non-degenerate for all \textit{non-zero} linear forms $\mu \in \g_2^*$. 

Consider an inner product $\la \cdot,\cdot \ra$ on the Lie algebra $\g$  for which the decomposition $\g = \g_1 \oplus \g_2$ is orthogonal. The inner product $\la \cdot,\cdot \ra$ induces a norm on the dual $\g_2^*$ that we denote by $|\cdot|$. Let $J_\mu \co \g_1 \to \g_1$ be the unique skew-symmetric endomorphism such that
\begin{equation}
\label{endomorphism}
\langle J_\mu(x),y \rangle 
= \mu([x,y]), \qquad x,y \in \g_1.
\end{equation} 
Then it is immediate that $G$ is a M\'etivier group if and only if $J_\mu$ is invertible for all non-zero linear forms $\mu \in \g_2^*$. We call $G$ a Heisenberg type group (or $H$-type group) if the endomorphisms $J_\mu$ are orthogonal for all $\mu \in \g_2^*$ of norm 1. This is equivalent to say that
\[
J_\mu^2 
= - |\mu|^2 \Id_{\g_1}\quad\text{for all } \mu\in \g_2^*.
\]
This class of groups was introduced by Kaplan in \cite{Kap80}. The class of Heisenberg type groups is a strict subclass of that of M\'etivier groups.



If $G$ is a simply-connected 2-step nilpotent Lie group and equipped with the left invariant metric determined by the inner product on $\g=\mathrm{T}_eG$, all the geometry of $G$ can be described by the maps $J_\mu$, as showed in \cite{Ebe94}, at least in the case where $\g_2$ is the center of $\g$. These operators were introduced by Kaplan in \cite{Kap80}, \cite{Kap81} and \cite{Kap83}, initially for studying the Heisenberg type groups.



\paragraph{Exponential maps on connected nilpotent Lie groups} 
If $\frak{g}$ is a nilpotent Lie algebra, then by \cite[Corollary 11.2.6 p.~445]{HiN12} the Dynkin series defines a polynomial map 
\begin{equation*}
\label{}
* \co \mathfrak{g} \times \mathfrak{g} \to \mathfrak{g}, (x, y) \mapsto x + y +\frac{1}{2}[x, y] + \cdots.
\end{equation*}
Moreover, $(\frak{g},*)$ defines a Lie group whose Lie algebra is $\frak{g}$ with $\exp_\frak{g} = \Id_\frak{g}$.  For nilpotent algebras of step $2$, observe that the product is given by
\begin{equation}
\label{prod-*}
x*y
=x+y+\frac{1}{2}[x,y],\quad x,y \in \mathfrak{g}.
\end{equation}
Recall that by \cite[Corollary 11.2.7 p.~446]{HiN12}, the exponential map 
\begin{equation*}
\label{exponential-map}
\exp_G \co (\frak{g},*) \to G
\end{equation*} 
of a connected nilpotent Lie group is the universal covering morphism of $G$. In particular, the exponential map of $G$ is surjective. If in addition $G$ is simply connected then the exponential map $\exp_G \co \frak{g} \to G$ is a diffeomorphism and can identify $G$ with $(\mathfrak{g},*)$. 
In this case, the Haar measure on $G$ coincides with the Lebesgue measure on $\mathfrak{g}$ by \cite[Theorem 1.2.10 p.~19]{CoG90}.

\paragraph{The universal covering of $\mathbb{H}_\Theta$}
Now, we will describe the universal covering of the connected nilpotent Lie group $\mathbb{H}_\Theta$. We can consider the canonical basis $(a_1,\ldots,a_{n+1})$ of $\R^n \oplus \R$. In particular, we obtain
\begin{equation}
\label{basis-aj}
c(a_i,a_j)
=\Theta_{ij}, \quad i,j \in \{1,\ldots,n\}.
\end{equation}
Note that we have an isomorphism $\mathfrak{g}_2^* \to \R$, $\mu \mapsto \mu(a_{n + 1})$.
  Consider a basis $(a_1,\dots,a_{n})$ of $\mathfrak{g}_1$ where $n =\dim \mathfrak{g}_1$. Let $\langle \cdot,\cdot \rangle$ be the inner product on $\mathfrak{g}_1$ that turns $(a_1,\dots,a_{n+1})$ into an orthonormal basis. 

\begin{prop}
\label{universal-covering}
The universal covering $\tilde{\mathbb{H}}_\Theta$ of the Lie group $\mathbb{H}_\Theta$ is given by $\tilde{\mathbb{H}}_\Theta=\mathbb{R}^n \times \R$ equipped with the product 
\begin{equation}
\label{law-1}
(s,u) \cdot (s', u') 
=  \big(s + s', u + u' + \textstyle{\frac{1}{2}} \langle s, \Theta s'\rangle\big), \quad s,s' \in \R^n, u,u' \in \R
\end{equation}
and the morphism $\tilde{\mathbb{H}}_\Theta \to \mathbb{H}_\Theta$ is defined by $(s,u) \mapsto (s,e^{\i u})$. Moreover, for any $\mu \in \R$, the skew-symmetric endomorphism $J_\mu \co \R^n \to \R^n$ is determined by $J_\mu = -\mu  \Theta$. Finally, the Lie group $\tilde{\mathbb{H}}_\Theta$ is a M\'etivier group if and only if the matrix $\Theta$ is invertible and a Heisenberg type group if and only if the matrix $\Theta$ is orthogonal.
\end{prop}

\begin{proof}
The first sentence is a consequence of \eqref{Lie-bracket} and \eqref{prod-*}. For any $\mu \in \R$, we have
\begin{equation*}
\label{}
\langle J_\mu x, y \rangle_{\R^n} 
\ov{\eqref{endomorphism}}{=} \mu [x, y] 
\ov{\eqref{Lie-bracket}}{=} \mu \la x,\Theta y\ra
=-\mu  \langle \Theta x, y \rangle_{\R^n}.
\end{equation*} 
We deduce that $J_\mu = -\mu  \Theta$. The last sentence is obvious.
\end{proof}

Note that by \cite[Theorem 3.2.2 p.~160]{BL07} the groups defined in \eqref{law-1} are \textit{exactly} the $n+1$-dimensional stratified groups of step two with $n$ generators. 

\paragraph{Partial Fourier transform} Let $G$ be a 2-step stratified simply connected Lie group and $\mathfrak{g} = \mathfrak{g}_1 \oplus \mathfrak{g}_2$ be the stratification of its Lie algebra. If $f \in \L^1(G)$ and $\mu \in \mathfrak{g}_2^*$, then we denote by $f^\mu$ the $\mu$-section of the partial Fourier transform of $f$ along $\mathfrak{g}_2$, given by
\begin{equation}
\label{partial-Fourier}
f_\mu(s) 
\ov{\mathrm{def}}{=} \int_{\mathfrak{g}_2} f(s,u) e^{-\i \langle \mu,u\rangle} \d u, \quad s \in \mathfrak{g}_1.
\end{equation}
We also define the even meromorphic functions $S$ and $R$ by 
\begin{equation}
\label{functions-S-and-R}
S(z) 
\ov{\mathrm{def}}{=}  \frac{z}{\sin(z)} 
\quad \text{and} \quad 
R(z) 
\ov{\mathrm{def}}{=}  \frac{z}{\tan(z)}, \quad 
z \in \mathbb{C} - \{k\pi : 0 \not= k \in \Z\}
\end{equation}
and the square root determined by the principal branch of the logarithm 
\begin{equation*}
\label{}
\log \co \C \backslash (-\infty,0] \to \C.
\end{equation*} 
Let $a_1,\ldots,a_{n}$ be a basis of $\frak{g}_1$ and let 
\begin{equation*}
\label{}
\Delta = -\sum_{i=1}^{n} a_i^2
\end{equation*}
be the corresponding subelliptic Laplacian, where $n=\dim \frak{g}_1$, see \cite[Definition 2.2.25 p.~144]{BL07} or \cite[p.~296]{Rob91}. For any $t > 0$, we denote by $K_t$ the kernel of the operator $e^{-t\Delta}$. Observe that by \cite[Lemma 4.15 p.~327]{Rob91}, the semigroup $(e^{-t\Delta})_{t \geq 0}$ on $\L^2(G)$ is a bounded holomorphic semigroup of angle $\frac{\pi}{2}$. We denote by $K_z$ its kernel.

Note that $J_\mu \co \frak{g}_1 \to \frak{g}_1$ is naturally identified with a skew-symmetric endomorphism of the complexification $(\frak{g}_1)_{\mathbb{C}}$ of $\frak{g}_1$, endowed with the corresponding hermitian inner product, and, for all $z \in \mathbb{C}$, the map $zJ_\mu$ is a normal endomorphism of $(\frak{g}_1)_{\mathbb{C}}$. It is stated in \cite[Proposition 4]{MaM16} (see also \cite{Cyg79}) that
\begin{equation}
\label{prop-heatkernel-MMformula}
K_{z,\mu}(s) 
= \frac{1}{(4\pi z)^{n / 2}}\sqrt{\det( S(z J_\mu))} \, \exp\Bigl(-\frac{1}{4z} \langle R(zJ_\mu)s, s \rangle \Bigr), \quad \Re z > 0
\end{equation}
for any $\mu \in \mathfrak{g}_2^*$ and any $s \in \frak{g}_1$. We also refer to \cite[Theorem 4, p.~318]{BGG6} and \cite[Theorem 10.2.7, p.~249]{CCFI11}, which present formulas that are likely equivalent for real times.

\paragraph{The sum of squares $\sum_{k=1}^n A_k^2$ and its associated semigroup}
Now, we connect the <<sum of squares>> associated to a Weyl pair to a class of <<subelliptic>> operators studied in \cite[Chapter IV]{Rob91}, while a less general setting is investigated in the thesis of Langlands, see \cite{Lan60} and \cite{Lan59}. These operators are constructed from a connected real finite-dimensional unimodular Lie group $G$, equipped with a Lie algebraic basis $(a_1,\ldots,a_n)$ of the Lie algebra $\frak{g}$, that is a finite sequence of linearly independent elements of $\frak{g}$ generating $\frak{g}$,  and a continuous representation $\pi \co G \to \B(X)$ on a Banach space $X$. Using the classical notation $\d\pi(a_i)$\label{dpi} of \cite[Definition 3.12 p.~51]{Mag92} for the infinitesimal generator of the strongly continuous group $(\pi(\exp_G(t a_i)))_{t \in \R}$, we can consider the operator
\begin{equation}
\label{sub-Rob}
H
\ov{\mathrm{def}}{=} -(\d\pi(a_1))^2-\cdots -(\d\pi(a_n))^2.
\end{equation}
We warn the reader that the notation $A_i$ is often used for the operator $\d\pi(a_i)$ but in our paper, the notation $A_i$ denotes an element of a $\Theta$-Weyl tuple.

The following result is a combination of \cite[Theorem 2.1]{ElR92}, \cite[Proposition 3.2 p.~127]{Jor75} and \cite[Corollary 4.17 p.~332]{Rob91}.

\begin{thm}
\label{th-Rob}
\begin{enumerate}
	\item The operator $H$ defined by \eqref{sub-Rob} is closable and its closure $\ovl{H}$ generates a continuous bounded holomorphic strongly continuous semigroup $(T_t)_{t \geq 0}$ on the Banach space $X$ of angle $0 < \varphi \leq \frac{\pi}{2}$ with the property that $\Ran  T_t$ is a subspace of the domain of any monomial in the operators $\d\pi(a_1),\ldots, \d\pi(a_n)$ for all $t > 0$.

\item If $K_t$ is the kernel of the operator $e^{-t\Delta}$ where $\Delta=-\sum_{i=1}^{n} a_i^2$ then 
\begin{equation}
\label{semi-int-rep}
T_t 
= \int_G K_t(s) \pi(s) \d s,\quad t > 0.
\end{equation}


\end{enumerate}
\end{thm}


Now, we are able to prove the following result.

\begin{prop}
\label{prop-semigroup}
Let $A = (A_1,\ldots,A_n)$ be a $\Theta$-Weyl tuple on some Banach space $X$.
\begin{enumerate}
\item The opposite of the closure $\A$ of the operator $\sum_{k=1}^n A_k^2$ generates a strongly continuous semigroup $(T_t)_{t \geq 0}$ of operators acting on $X$ satisfying
\begin{equation}
\label{def-Tt}
T_t
= \int_{\tilde{\QH}_\Theta} K_t(s,u)\pi(s,u) \d s\d u, \quad t \geq 0
\end{equation}
where $(K_t)_{t \geq 0}$ is the kernel semigroup associated to the subelliptic Laplacian $\Delta$ on the universal covering $\tilde{\mathbb{H}}_\Theta$ of the nilpotent Lie group $\mathbb{H}_\Theta$ and where $\pi \co \tilde{\mathbb{H}}_\Theta \to \B(X)$ is a continuous representation.

\item The semigroup generated by $-\A$ is bounded holomorphic of some angle $\varphi \in (0,\frac{\pi}{2}]$ and can be written as
\begin{equation}
\label{Tz-descr}
T_z x 
=\int_{\R^n} p_z^\Theta(s)  e^{\i s \cdot A} x \d s, \quad x \in X, |\arg z| <\varphi
\end{equation}
where the kernel $(p_z^\Theta)_{\Re z>0}$ is given by
\begin{equation}
\label{def-pzTheta}
p_z^\Theta(s) 
\ov{\mathrm{def}}{=} \frac{1}{(4 \pi z)^{\frac{n}{2}}} \sqrt{\det S(z\Theta)} \exp\bigg(-\frac{1}{4z} \langle R(z \Theta)s, s \rangle \bigg)
=\int_{\R} e^{\i u} K_z(s,u)  \d u.
\end{equation}
where $s \in \R^n$.

\item In particular, if $A_{\univ}^\Theta$ is the universal $\Theta$-Weyl tuple on $\L^p(\R^n)$ from \eqref{A-univ}, then 
\begin{equation}
\label{Autre-rep-Tz}
(T_zg)(x) 
=  \int_{\R^n} k_z^\Theta(x,y)  g(x-y) \d y, \quad \text{a.e. }x \in \R^d,\: |\arg z| <\varphi,\: g \in \L^p(\R^n)
\end{equation}
where the kernel $(k_z^\Theta)_{\Re z>0}$ is defined by
\begin{equation}
\label{kzxy}
k_z^\Theta(x,y) 
\ov{\mathrm{def}}{=} e^{\frac12 \i \langle y , \Theta x \rangle} p_z^\Theta(y), \quad x,y \in \R^d.
\end{equation}
\end{enumerate}
\end{prop}

\begin{proof}
Let $(A_1,\ldots,A_n)$ be a $\Theta$-Weyl tuple on a Banach space $X$. Recall that by Lemma \ref{lem-Weyl-product} we have a projective representation $\R^n \to \B(X)$, $s \mapsto e^{\i s \cdot A}$. By Proposition \ref{prop-corre}, we can consider its lifting $\pi_\Theta \co \mathbb{H}_\Theta \to \B(X)$ defined by 
\begin{equation*}
\label{}
\pi_\Theta(s,e^{\i u})=e^{\i u} e^{\i s\cdot A},
\end{equation*}
which is a continuous representation. By composition, we obtain a representation $\pi \co \tilde{\mathbb{H}}_\Theta \to \B(X)$ of the universal covering $\tilde{\mathbb{H}}_\Theta$ of the Lie group $\mathbb{H}_\Theta$. 

Recall that by \cite[Proposition 9.5.1 p.~335]{HiN12} the Lie groups $\mathbb{H}_\Theta$ and $\tilde{\mathbb{H}}_\Theta$ admit isomorphic Lie algebras. Then the family $(a_1,\ldots,a_{n})$ of \eqref{basis-aj} is a H\"ormander system. We define $(K_z)_{\Re z >0}$ to be the heat kernel of the subelliptic Laplacian $\Delta = -\sum_{j = 1}^n a_j^2$.


For any $1 \leq j \leq n$ and any $t \in \R$, using \cite[Proposition 6.12 p.~99]{Upm85} note that 
\begin{equation*}
\label{}
\pi(\exp(ta_j))
=\exp(\pi_*(ta_j))
=e^{\i tA_j},
\end{equation*}
where the computation of $\pi_*$ is left to the reader. We deduce that $\d\pi(a_j) =\i A_j$. Consequently, by Theorem \ref{th-Rob},  
the unbounded operator 
\begin{equation*}
\label{}
\sum_{j = 1}^n A_j^2 
= - \sum_{j = 1}^n (\d\pi(a_j))^2
\end{equation*}
defined on the subspace $\bigcap_{j = 1}^n \dom A_j^2$ of the Banach space $X$ is closable, and its closure $\mathcal{A}$ generates a strongly continuous semigroup $(T_t)_{t \geq 0}$ of operators acting on $X$ defined by \eqref{def-Tt}.
 Furthermore, the same result says that this semigroup is bounded analytic of angle $0<\varphi \leq \frac{\pi}{2}$. Its extension by $(T_z)_{|\arg z| <\varphi}$ is given by 
\begin{equation}
\label{Tz-def}
T_z
= \int_{\tilde{\mathbb{H}}_\Theta} K_z(s,u)\pi(s,u) \d s \d u, \quad |\arg z| <\varphi.
\end{equation}


2. 
For any $x \in X$ and any complex number $z \in \mathbb{C}^*$ with $|\arg z| <\varphi$, we can describe $T_z$ by the formula
\begin{align*}
\MoveEqLeft
T_z(x) 
\ov{\eqref{Tz-def}}{=} \int_{\tilde{\mathbb{H}}_\Theta} K_z(s,u) \pi(s,u) x  \d s \d u
\ov{\eqref{prop-Lie-group-exists}}{=} \int_{\tilde{\mathbb{H}}_\Theta} K_z(s,u) e^{\i u} e^{\i s\cdot A}x \d s \d u \\
&=\int_{\R^{n}} \Bigl( \int_{\R} e^{\i u} K_z(s,u)  \d u\Bigr) e^{\i s \cdot A}x  \d s
\ov{\eqref{prop-heatkernel-Weylsymbol}}{=} \int_{\R^n} \hat{a}_z(s) e^{\i s \cdot A} x \d s
\end{align*}
where
\begin{equation}
\label{prop-heatkernel-Weylsymbol}
\hat{a}_z(s)
\ov{\mathrm{def}}{=}  \int_{\R} e^{\i u} K_z(s,u) \d u.
\end{equation}
We obtain
\begin{align*}
\MoveEqLeft
\hat{a}_z(s)          
\ov{\eqref{prop-heatkernel-Weylsymbol}}{=} \int_{\R} e^{\i u} K_z(s,u)  \d u
\ov{\eqref{partial-Fourier}}{=} K_{z,{-1}}(s) \\
&\ov{\eqref{prop-heatkernel-MMformula}}{=}  \frac{1}{(4\pi z)^{n / 2}} \sqrt{\det( S(z \Theta))} \, \exp\Bigl(-\frac{1}{4z} \langle R(z\Theta)s, s \rangle \Bigr).
\end{align*}

3. This follows directly from the explicit description of $e^{\i u \cdot A}$ from \eqref{equ-universal-Theta-Weyl-explicit-formula} in the case where $A$ is the universal $\Theta$-Weyl tuple.
\end{proof}

\begin{remark} \normalfont
Using \cite[II.22 p.~25]{DtER03}, we see that the local dimension of the couple $(\tilde{\mathbb{H}}_\Theta,(X_1,\ldots,X_n))$ is equal to $n+2$.
\end{remark}

\begin{remark} \normalfont
\label{Rem-angle}
It might be feasible to replace the angle $\varphi$ by $\frac{\pi}{2}$ in the previous result by re-examining the method of \cite{Kis76} that establishes holomorphicity at angle $\frac{\pi}{2}$.
\end{remark}

We now turn to the functional calculus properties of the operator $\A$, when $A_{\univ}^\Theta$ is the universal $\Theta$-Weyl tuple from Proposition \ref{prop-universal-Theta-Weyl-tuple}.

\begin{prop}
\label{prop-HI-calculus-non-shifted}
Let $X$ be a $\UMD$ Banach space. Consider the operator $\A$ associated to the universal $\Theta$-Weyl tuple $A_{\univ}^\Theta$. Then $\A \ot \Id_X$ admits a bounded $\HI(\Sigma_\omega)$ functional calculus on the Bochner space $\L^p(\R^n,X)$ for any $1 < p <\infty$ and any angle $\omega \in (\frac{\pi}{2},\pi)$.
\end{prop}

\begin{proof}
The heat kernel $p_t$ on $\tilde{\QH}_\Theta$ from Proposition \ref{prop-semigroup} satisfies $\norm{p_t}_{\L^1(\tilde{\QH}_\Theta)} = 1$ for all $t > 0$. Note that by the formula \eqref{twisted-Rn} defining the twisted convolution, we see the formula \eqref{Autre-rep-Tz} says that each operator $T_t$ is a twisted convolution operator. Then Remark \ref{twisted-conv-is-regular} implies that each $T_t$ is contractively regular (and even an absolute contraction).

Consequently, we have a semigroup of contractively regular operators. So, we can use \cite[p.~738]{Fen97} with Example \ref{Example-regular} to obtain a dilation of the semigroup $(T_t \ot \Id_X)_{t \geq 0}$. If $p \neq 2$, using the contractive regularity of the operators of the dilation, we conclude with Theorem \ref{thm-Hinfty-dilation} that $\A \ot \Id_X$ admits a bounded $\HI(\Sigma_\omega)$ functional calculus on the Bochner space $\L^p(\R^n,X)$ for any angle $\omega \in (\frac{\pi}{2},\pi)$.
Then if $p = 2$, the operator $\A \ot \Id_X$ still admits a bounded $\HI(\Sigma_\omega)$ functional calculus on $\L^2(\R^n,X)$ for any angle $\omega \in (\frac{\pi}{2},\pi)$, since we can interpolate its spectral multipliers as bounded operators on the Bochner space $\L^{p_1}(\R^n,X)$ and $\L^{p_2}(\R^n,X)$ to obtain bounded operators on the Banach space $\L^2(\R^n,X)$, where $\frac12 =  \frac{\theta}{p_2} + \frac{1-\theta}{p_1}$ for some suitable choice of the parameters $p_1,p_2$ and $\theta$.
\end{proof}

\begin{remark} \normalfont
Since the strongly continuous groups $(e^{\i t A_k})_{t \in \R}$ consist in particular of $\L^2$ isometries, the operators $A_k$ are self-adjoint on $\L^2$ by Stone's theorem ($k=1,\ldots,n$). Thus $(T_t)_{t \geq 0}$ is a diffusion semigroup in the sense of \cite[Section 5]{JMX06} on a \textit{commutative} von Neumann algebra.
\end{remark}

\begin{example} \normalfont
\label{Example-twisted-Laplacian}
It is worth noting that the <<sum of squares>> associated to the Weyl pair
\begin{equation*}
\label{}
(A,B)=(-\tfrac{1}{2}Q_2-P_1,\tfrac{1}{2}Q_1-P_2)
\end{equation*}
of Example \ref{Example-J-n=2} identifies to the <<twisted Laplacian>> $L$ on $\R^2$. A formal computation gives
\begin{align*}
\MoveEqLeft
(-\tfrac{1}{2}Q_2-P_1)^2+(\tfrac{1}{2}Q_1-P_2)^2        
=P_1^2+P_2^2+\frac{1}{4}Q_1^2+\frac{1}{4}Q_2^2+Q_2P_1-Q_1P_2 \\
&\ov{\eqref{pos-mom}}{=} -(\partial_1^2+\partial_2^2)+\frac{1}{4}(x_1^2+x_2^2)-\i(x_2\partial_1-x_1\partial_2).
\end{align*}
Note that this linear operator $L$ has been studied by many people, see e.g.~\cite{Car20}, \cite[Section 7.9 p.~477]{DOS02}, \cite{KoR07}, \cite{JLR23}, \cite{JLR22}, \cite{Nar03}, \cite{StZ98}, \cite[(1.4.23) p.~19]{Tha98}, \cite{Tha12}, \cite{Tie06} and \cite{Won05}. It is a perturbation of the operator 
\begin{equation*}
\label{}
-(\partial_1^2+\partial_2^2)+\frac{1}{4}(x_1^2+x_2^2)
\end{equation*}
by the partial differential operator $\i(x_1\partial_2-x_2\partial_1)$. The twisted Laplacian $L$ emerges in physics as the Hamiltonian representing the movement of a free electron in an infinite two-dimensional plane, subjected to a constant magnetic field oriented perpendicularly to the plane, following the seminal work \cite{Lan30} of Landau. Furthermore, it is well-known that this operator is connected to the subelliptic Laplacian 
\begin{equation*}
\label{}
\Delta \ov{\mathrm{def}}{=} -X^2-Y^2
\end{equation*}
of \cite[(18.1) p.~139]{Won23} on the Heisenberg group $\mathbb{C} \times \R$. Here, the group law is defined following e.g.~\cite[p.~6]{Tha12} by 
\begin{equation}
\label{law-2}
(s,u) \cdot (s',u') 
=\left (s+s',u+u'+\tfrac{1}{2}\Im (s \ovl{s'})\right), \quad (s,u),(s',u') \in \mathbb{C} \times \R,
\end{equation}
and $X$ and $Y$ are the left-invariant vector fields on the Heisenberg group $\mathbb{C} \times \R$ defined by $X \ov{\mathrm{def}}{=}\partial_x +\tfrac{1}{2}y\partial_u$ and $Y \ov{\mathrm{def}}{=} \partial_y-\tfrac{1}{2}x\partial_u$ where the complex variable $s$ identifies to $x+\i y$. These vector fields satisfy $[X,Y]=-T$ where $T \ov{\mathrm{def}}{=} \partial_u$ and $(X,Y,T)$ is a basis of the Lie algebra. Indeed, by \cite[p.~9]{Tha12} for any suitable function $f \co \R^2 \to \R$ we have
\begin{equation*}
\label{}
\Delta(e^{\i t} f)
=e^{\i t}L(f), \quad t \in \R.
\end{equation*}
Finally, replacing the matrix $\Theta$ by the matrix $\cal{J} \ov{\mathrm{def}}{=}\begin{bmatrix}
  0   & - 1 \\
  1   &  0 \\
\end{bmatrix}$ in the group law \eqref{law-1}, we recover the group law \eqref{law-2} of the Heisenberg group observing that $\langle s, \cal{J} s'\rangle=\Im (s\ovl{s'})$ for any elements $s,s'$ of $\R^2=\mathbb{C}$. Denoting $(K_t)_{t \geq 0}$ the kernel semigroup associated to the subelliptic Laplacian $\Delta$, we have by \cite[Theorem 2.8.1 p.~84]{Tha04} the formula
\begin{equation}
\label{kernel-Heis}
\int_{\R} e^{\i \lambda u} K_t(s,u) \d u 
=\frac{\lambda}{4\pi \sinh \lambda t} \exp\bigg(-\frac{\lambda\coth (t\lambda)}{4}(s_1^2+s_2^2)\bigg), \quad (s,u) \in \R^2 \times \R,
\end{equation}
where $\lambda \in \R$. So, for $s \in \R^2$ and $t > 0$ we recover the formula
\begin{align*}
\MoveEqLeft         
p_t^\cal{J}(s)  \ov{\eqref{def-pzTheta}}{=} \int_{\R} e^{\i u} K_t(s,u)  \d u 
\ov{\eqref{kernel-Heis}}{=} \frac{1}{4\pi \sinh t} \exp\bigg(-\frac{\coth (t)}{4}(s_1^2+s_2^2)\bigg)
\end{align*}
proved in \cite[Theorem 4.3 p.~277]{Won05} using an alternative approach. 
We will generalize this formula in \eqref{kernel-cal-J} to complex times by a different method. 
%
\end{example}

\begin{remark} \normalfont
\label{Rem-5-10}
If the skew-symmetric matrix $\Theta$ is orthogonal, we have seen in Proposition \ref{universal-covering} that the Lie group $\tilde{\mathbb{H}}_\Theta$ is an $H$-type Lie group. Then using the description of the heat kernel of \cite{Ran96} (rediscovered in \cite[Theorem 3.2]{YaZ08}), we can determine similarly an expression for $p_t^\Theta$ for any $t > 0$.
\end{remark}

\begin{example} \normalfont
\label{twisted-bis}
More generally, using the Weyl pair described in \cite[Lemma 6.4 p.~284]{NeP20}, the <<sum of squares>> identifies to the twisted Laplacian on $\R^d$.
\end{example}

%

%
%
%
%
%
%
%
%
%
%
%

\begin{example} \normalfont
\label{osc-harmo}
Suppose that $n=2d$ for some integer $d \geq 1$. Consider the standard Weyl pair $(Q_1,\ldots,Q_d,P_1,\ldots,P_d)$ considered in \cite[Example 3.3 p.~263]{NeP20}, where the operators $Q_j$ and $P_j$ are defined in \eqref{pos-mom}. By Example \ref{ex-Hall}, this sequence defines a $\Theta$-Weyl tuple on the Banach space $\L^p(\R^d)$ for the matrix $\Theta=\begin{bmatrix}
  0   & - \I_d \\
  \I_d   &  0 \\
\end{bmatrix}$. The <<sum of squares>> identifies to the Hermite operator
\begin{equation*}
\label{}
-\Delta+|x|^2
\end{equation*}
acting on a dense subspace of the Banach space $\L^p(\R^d)$. By \cite[p.~62]{CCFI11} or \cite[p.~10]{Tha12}, we have Mehler's formula
\begin{align*}
\MoveEqLeft
(T_tf)(x) \\
&=\frac{1}{(2\pi\sinh 2t)^{\frac{d}{2}}}\int_{\R^d} \exp\bigg(-\frac{1}{2\sinh(2t)}\big[(|x|^2+|y|^2)\coth(2t)-2 \la x, y\ra\big] \bigg) \\
&f(y) \d y,        
\end{align*}
where $t > 0$ and $x \in \R^d$, which describes the associated semigroup $(T_t)_{t \geq 0}$ on the space $\L^p(\R^d)$. Now, consider the simplest case $d=1$ with the notations $P \ov{\mathrm{def}}{=} P_1$, $Q \ov{\mathrm{def}}{=} Q_1$ and $\Theta=\begin{bmatrix}
  0   & - 1 \\
  1   &  0 \\
\end{bmatrix}$. For any $s=(s_1,s_2)$ of $\R^2$, note the equality $\langle s , \Theta^\uparrow s \rangle=-s_1s_2$. Using the second part of Lemma \ref{lem-technical-kernel}, the action of the operators of the associated semigroup on any function $f \in \L^p(\R)$ can also be described by
\begin{align*}
\MoveEqLeft
T_z f 
\ov{\eqref{Tz-descr}}{=} \int_{\R^2} p_z^\Theta(x)  e^{\i s \cdot A} f \d s      
\ov{\eqref{def-eitA}}{=} \int_{\R^2} p_z^\Theta(x)   e^{\i (x P+ y Q)} f \d x \d y \\
& \ov{\eqref{kernel-cal-J}}{=} \frac{1}{4\pi\sinh(z)}\int_{\R^2}  \exp\bigg(- \frac{\coth(z)}{4} (x^2 + y^2) \bigg) e^{\i (x P+ y Q)} f \d x\d y,
\end{align*}
where $|\arg z| < \varphi$, in the spirit of the formula \cite[(3) p.~50]{Lus06} for the operator $T_{\frac{t^2}{2}}$. This formula seems to be new.
\end{example}

\begin{example} \normalfont
\label{Example-Heat}
For any integer $n \geq 1$, the choice $\Theta=0_n$ in Proposition \ref{prop-universal-Theta-Weyl-tuple} gives the universal $0_n$-Weyl tuple
\begin{equation*}
\label{}
A_{\univ}^{0_n}=(-P_1,\ldots,-P_n),
\end{equation*}
where the operator $P_i$ is defined in \eqref{pos-mom}. The <<sum of squares>> is the Laplacian $-\Delta$ on the Banach space $\L^p(\R^n)$. 
\end{example}

\section{$\HI$ calculus for the shifted harmonic oscillator of the univ. $\Theta$-Weyl tuple}
\label{subsec-shifted-HI-calculus}

Consider a skew-symmetric matrix $\Theta \in \M_n(\R)$. In the following, it will turn out that the <<sum of squares>> $\A$, defined as the closure of the operator $\sum_{k=1}^n A_k^2$ of the universal $\Theta$-Weyl tuple, has a spectral gap at $0$. The value of this gap is explicitly given in \eqref{def-alpha}.

Note that the skew-symmetric matrix $\Theta \in \M_n(\R)$ is orthogonally similar to a suitable block diagonal matrix, as observed in \eqref{diago-Theta-bis}. More precisely, there exists a decomposition 
\begin{equation}
\label{diago-Theta}
\Theta 
= O^T J O
\end{equation}
where $O$ is an orthogonal matrix of $\M_n(\R)$ and $J$ is the $2 \times 2$ block diagonal matrix
\begin{equation}
\label{def-de-J}
J 
\ov{\mathrm{def}}{=} \begin{pmatrix} 
J_1 & 0 & 0 & \ldots & 0 \\ 
0 & J_2 & 0 & \ldots & 0\\ 
\vdots & \ddots & \ddots & \vdots & \vdots \\ 
0 & \ldots & \ddots & J_k & 0 \\ 
0 & \ldots & \ldots & \ldots & 0 
\end{pmatrix}, 
\quad \text{where }
J_j 
\ov{\mathrm{def}}{=} \begin{pmatrix} 
0 & \alpha_j \\ 
- \alpha_j & 0 
\end{pmatrix},  
\quad j = 1,\ldots,k
\end{equation}
for some integer $k \in [0,\frac{n}{2}]$ with $\alpha_1,\ldots,\alpha_k >0$. 
We put
\begin{equation}
\label{def-alpha}
\alpha \ov{\mathrm{def}}{=} \sum_{j=1}^k \alpha_j > 0
\end{equation}
with the convention that $\alpha = 0$ if $\Theta = 0$.
We also set $\beta \ov{\mathrm{def}}{=} \min\{\alpha_1,\ldots,\alpha_k\} > 0$ in the case where $\Theta \neq 0$.
Note that $\alpha$ and $\beta$ do not depend on the particular choice of $O$ and $J$ in the decomposition $\Theta = O^T J O$ hereabove.
Finally, we will use the matrix
\begin{equation}
\label{def-cal-J}
\cal{J}
\ov{\mathrm{def}}{=} \begin{pmatrix} 
0 & 1 \\ 
-1 & 0 
\end{pmatrix}.
\end{equation}

Let us now calculate more explicitly the integral kernel of the semigroup from Proposition \ref{prop-semigroup} in some particular cases. The kernel of the third point is the heat kernel on $\R^n$.

\begin{lemma}
\label{lem-technical-kernel}
Let $p_z^\Theta$ denote the kernel defined in \eqref{def-pzTheta}.
\begin{enumerate}
\item
Decomposing $\Theta=  O^T J O$ as in \eqref{diago-Theta}, we have
\begin{equation}
\label{similar-pz}
p_z^\Theta(x) 
= p_z^J(Ox), \quad \Re z>0, \: x \in \R^n.
\end{equation}
\item The kernel $p_z^\cal{J}$ can be written as
\begin{align}
\MoveEqLeft
\label{kernel-cal-J}
p_z^\cal{J}(x) 
= \frac{1}{2\pi} \frac{1}{e^{z} - e^{- z}} \exp\bigg(- \frac{1}{4} \frac{e^{z} + e^{-z}}{e^{z} - e^{- z}} (x_{1}^2 + x_{2}^2) \bigg) \nonumber \\
&=\frac{1}{4\pi\sinh(z)} \exp\bigg(- \frac{\coth(z)}{4} (x_{1}^2 + x_{2}^2) \bigg).         
\end{align}

\item The kernel $p_z^{0_n}$ is equal to 
\begin{equation}
\label{heat-kernel}
p_z^{0_n}(x)
= \frac{1}{(4\pi z)^{\frac{n}{2}}} \exp\bigg(-\frac{1}{4z}  \sum_{j=1}^n x_j^2 \bigg).
\end{equation}
\item Let $J = \alpha \cal{J}$ for some $\alpha > 0$.
Then we have
\begin{equation}
\label{rescaled-3}
p_z^{\alpha \cal{J}} (x) 
= \alpha p_{\alpha z}^{\cal{J}}(\sqrt{\alpha} x) .
\end{equation}
\end{enumerate}
\end{lemma}

\begin{proof}
Recall that the meromorphic functions $S(z) = \frac{z}{\sin(z)}$ and $R(z) = \frac{z}{\tan(z)}$ are defined in \eqref{functions-S-and-R}.

1. We have 
\begin{equation*}
\label{}
S(z\Theta) \ov{\eqref{diago-Theta}}{=} S(z O^T J O) =O^T S(zJ) O.
\end{equation*}
Thus, $\det S(z\Theta) = \det S(zJ)$. On the other hand, $R(z\Theta) \ov{\eqref{diago-Theta}}{=} R(z O^T J O)= O^T R(zJ) O$, so that 
\begin{equation*}
\label{}
\langle R(z\Theta)x,x \rangle  = \langle O^T R(zJ)Ox, x \rangle = \langle R(zJ)Ox,Ox \rangle.
\end{equation*}
Recalling the formula \eqref{def-pzTheta} for $p_z^\Theta$, we deduce the first part of the lemma.

2. The matrix $\cal{J}$ is skew-adjoint hence normal.
Its eigenvalues being $\i$ and $-\i$, we can write $\cal{J} = U^{-1} \begin{bmatrix} \i & 0 \\ 0 & - \i \end{bmatrix} U$ for some invertible $U$.
As $S$ is even, functional calculus of matrices yields 
\[
S(z \cal{J}) 
= U^{-1} S\left( \begin{bmatrix} 
\i z & 0 \\ 
0 & -\i z 
\end{bmatrix} \right) U 
= U^{-1} 
\begin{bmatrix} 
S(\i z) & 0 \\ 
0 & S(-\i z) 
\end{bmatrix} 
U = S(\i z) U^{-1} \Id U 
= S(\i z) \Id. 
\]
Therefore, 
\[ 
\sqrt{\det(S(z\cal{J}))} 
= \sqrt{S(\i z)^2} 
\ov{\eqref{functions-S-and-R}}{=} \frac{\i z }{\sin(\i z)} 
= \frac{2z}{e^{z} - e^{-z}}. 
\]
Recall that the function $R$ is also even. Thus, we have similarly
\begin{equation}
\label{div-R-e}
R(z\cal{J})
=\begin{bmatrix}
  R(\i z)   & 0  \\
   0  &  R(\i z) \\
\end{bmatrix}
\ov{\eqref{functions-S-and-R}}{=} \begin{bmatrix}
  z \frac{e^{z} + e^{- z}}{e^{ z} - e^{- z}}  & 0  \\
   0  &  z \frac{e^{z} + e^{- z}}{e^{ z} - e^{- z}} \\
\end{bmatrix}.
\end{equation}
Consequently, 
we deduce that
\begin{align*}
\MoveEqLeft
\exp\bigg(-\frac{1}{4z} \big\la R(z\cal{J})x,x \big\ra \bigg) 
\ov{\eqref{div-R-e}}{=} \exp\bigg(-\frac{1}{4z} \bigg\langle\begin{bmatrix}
  z \frac{e^{z} + e^{- z}}{e^{ z} - e^{- z}}  & 0  \\
   0  &  z \frac{e^{z} + e^{- z}}{e^{ z} - e^{- z}} \\
\end{bmatrix}
x,x\bigg\rangle\bigg) \\
&= \exp\bigg(-\frac{1}{4} \frac{e^{z} + e^{-z}}{e^{z} - e^{- z}}(x_{1}^2 + x_{2}^2) \bigg).
\end{align*}

3. We have $S(z0_{n})=\I_n$ and $R(z0_{n})=\I_n$. Consequently
\begin{equation*}
\label{}
p_z^{0_{n}}(x) 
\ov{\eqref{def-pzTheta}}{=} \frac{1}{(4 \pi z)^{\frac{n}{2}}} \sqrt{\det \I_n} \exp\bigg(-\frac{1}{4z} \langle \I_nx, x \rangle \bigg)
=\frac{1}{(4\pi z)^{\frac{n}{2}}} \exp\bigg(-\frac{1}{4z}  \sum_{j=1}^n x_j^2 \bigg).
\end{equation*}

4. Using the explicit description of $p_z^{\Theta}$, we have
\begin{align*}
\MoveEqLeft
p_z^{\alpha \cal{J}}(x) 
\ov{\eqref{def-pzTheta}}{=} \frac{1}{4\pi z} \sqrt{S(z \alpha \cal{J})} \exp\left( - \frac{1}{4z} R(z \alpha \cal{J}) x, x \rangle \right) \\
& = \alpha \cdot \frac{1}{4 \pi \alpha z} \sqrt{S(\alpha z \cal{J})} \exp\left( - \frac{1}{4 \alpha z} R(\alpha z \cal{J}) \sqrt{\alpha}x, \sqrt{\alpha}x \rangle \right) \\
&\ov{\eqref{def-pzTheta}}{=} \alpha p_{\alpha z}^{\cal{J}}(\sqrt{\alpha} x).
\end{align*}
\end{proof}

In the sequel, we will show that the shifted operator $(\A - \alpha \Id_{\L^p(\R^n)}) \ot \Id_{X}$ admits a bounded $\HI(\Sigma_\omega)$ functional calculus on the Bochner space $\L^p(\R^n,X)$, provided $X$ has the UMD property. For that, we will use the following shorthand notation
\begin{equation}
\label{equ-A-alpha}
\A_\alpha 
\ov{\mathrm{def}}{=} (\A - \alpha \Id_{\L^p(\R^n)}) \ot \Id_X,
\end{equation}
where $-\A$ is the generator of the semigroup associated with the universal $\Theta$-Weyl tuple for some fixed skew-symmetric matrix $\Theta \in \M_n(\R)$. Recall that the kernel $k_z^\Theta$ of the semigroup $(e^{-z\A})_{\Re z >0}$ is defined in \eqref{kzxy}.

The semigroup generated by $\A$ naturally splits into tensor factors acting on $\L^p(\R^{2})$ and on $\L^p(\R^{n-2k})$, corresponding to the block diagonal form of the matrix $J$ with non-zero $2 \times 2$ blocks and the big vanishing block.
This is the content of the next lemma.

\begin{lemma}
\label{lem-kernel-splitting}
Let $\Theta = J$ be a $2 \times 2$ block diagonal skew-symmetric matrix of $\in \M_n(\R)$ as in \eqref{def-de-J}. Let $A_\univ^J$ be the universal $J$-Weyl tuple. 
Then the kernel $k_t^J(x,y) $ of Proposition \ref{prop-semigroup} splits accordingly to the space decomposition $\R^n = \R^{2} \oplus \cdots \oplus \R^{2} \oplus \R^{n-2k}, x = x_1+\cdots +x_k + x''$ to
\begin{equation}
\label{kt-as-a-product}
k_t^J(x,y) 
= k_t^{J_1}(x_1,y_1) \cdots k_t^{J_k}(x_k,y_k) k_t^{0_{n-2k}}(x'',y'').
\end{equation}
Consequently, the semigroup $(e^{-t\A})_{t \geq 0}$ splits into 
\begin{equation}
\label{}
e^{-t\A} 
= e^{-t\A^{J_1}} \ot \cdots \ot e^{-t\A^{J_k}}\ot e^{-t\A^{0_{n-2k}}}
\end{equation}
on the space $\L^p(\R^n) = \L^p(\R^{2})\ot_p \cdots \ot_p \L^p(\R^{2}) \ot_p \L^p(\R^{n-2k})$ where $e^{-t \A^{J_j}}$ is associated with the kernel $k_t^{J_j}$. Moreover, the generator $\A$ is the closure of 
\begin{equation*}
\label{}
\A^{J_1} \ot \Id_{\L^p(\R^{n-2})} +\cdots +\Id_{\L^p(\R^{2k-2})}\ot \A^{J_k} \ot \Id_{\L^p(\R^{n-2k})}+\Id_{\L^p(\R^{2k})} \ot \A^{0_{n-2k}}
\end{equation*}
defined on $\dom \A^{J_1} \ot \cdots \ot \dom \A^{0_{n-2k}}$. 
\end{lemma}

\begin{proof}
For any $t \geq 0$, note that
\begin{equation}
\label{}
\det S(tJ)
=\prod_{j=1}^{k+1} \det S(t J_j)
\end{equation}
where $J_{k+1}=0_{n-2k}$. Thus, according to Lemma \ref{lem-technical-kernel}, the kernel splits as
\begin{align*}
\MoveEqLeft
k_t^J(x,y) 
\ov{\eqref{kzxy}}{=} e^{\frac12 \i \langle y , J x \rangle} p_t^J(y)\\
&\ov{\eqref{def-pzTheta}}{=} \frac{1}{(4 \pi t)^{\frac{n}{2}}}  e^{\frac12 \i \langle y , J x \rangle}\sqrt{\det S(tJ)} \exp\bigg(-\frac{1}{4t} \langle R(t J)y, y \rangle \bigg) \\
&=\frac{1}{(4 \pi t)^{\frac{n}{2}}}  e^{\frac12 \i \langle y , J x \rangle}\prod_{j=1}^{k+1} \sqrt{\det S(t J_j)} \\
&\exp\left(-\frac{1}{4t} \left\langle \begin{bmatrix}
   R(tJ_1)  & 0 & \cdots & 0\\
    0 & \ddots & &	\vdots\\
    \vdots &  & R(tJ_k)  & 0\\
		0&  & 0  & R(t0_{n-2k})\\
\end{bmatrix}y, y \right\rangle \right) \\
&=\bigg[\frac{1}{4 \pi t} e^{\frac12 \i \langle y_1 , J x_1 \rangle}\sqrt{\det S(tJ_1)} \exp\bigg(-\frac{1}{4t} \langle R(t J_1)y_1, y_1 \rangle \bigg)\bigg]\cdots\\
& \bigg[\frac{1}{4 \pi t}  e^{\frac12 \i \langle y_k , J_k x_k \rangle}\sqrt{\det S(tJ_k)} \exp\bigg(-\frac{1}{4t} \langle R(t J_k)y_k, y_k \rangle \bigg)\bigg] \\
&\bigg[\frac{1}{(4 \pi t)^{\frac{n-2k}{2}}}  \sqrt{\det S(t0_{n-2k})} \exp\bigg(-\frac{1}{4t} \langle R(t 0_{n-2k})y'', y'' \rangle \bigg)\bigg] \\
&\ov{\eqref{def-pzTheta}}{=} \big[e^{\frac12 \i \langle y_1 , J_1 x_1 \rangle} p_t^{J_1}(y_1)\big]\cdots \big[e^{\frac12 \i \langle y_k , J_k x_k \rangle} p_t^{J_k}(y_k)\big]p_t^{0_{n-2k}}(y'') \\
&\ov{\eqref{kzxy}}{=} k_t^{J_1}(x_1,y_1) \cdots k_t^{J_k}(x_k,y_k)k_t^{0_{n-2k}}(x'',y'').
\end{align*}
By definition, $k_t^J(x,y)$ is the kernel of the semigroup $T_t = e^{-t\A}$ on $\L^p(\R^n)$.
Moreover, it is well-known that $k_t^{0_{n-2k}}(x'',y'')$ is the kernel of the heat semigroup $T_t''$ on $\L^p(\R^{n-2k})$.
Denote the operator $T_t^j$ given by the integral kernel $k_t^{J_j}(x_j,y_j)$, $j = 1,\ldots,k$.
Then $(T_t^j)_{t \geq 0}$ is the semigroup on $\L^p(\R^{2})$ associated with the universal $J_j$-Weyl tuple. 
By identification, we find 
\begin{equation*}
\label{}
T_t(f_1 \ot f_2 \ot \cdots \ot f_k \otimes f'') = T_t^1(f_1) \ot T_t^2(f_2) \ot \cdots \ot T_t^k(f_k) \ot T_t''(f'').
\end{equation*}
The last assertion on the generator is a direct consequence of \cite[p.~23]{AGG86}.
\end{proof}

\begin{prop}
\label{prop-kernel-estimate}
Suppose that $1 < p < \infty$ and $A_\univ^\cal{J}$ be the universal $\cal{J}$-Weyl tuple acting on the Banach space $\L^p(\R^2)$. 
Then for any complex number $z$ with $\Re z>0$, there exists a smooth even function $\h_z \co \R \to \R_+$ which is decreasing on $\R_+$ such that
\begin{equation}
\label{majo-by-gz}
|k_z^\cal{J}(x,y)|
=|p_z^\cal{J}(y)|
\leq g_z(y) 
\ov{\mathrm{def}}{=} \h_z(y_1)\h_z(y_2), \quad x,y \in \R^2
\end{equation}
with
\begin{equation}
\label{estim-sup}
\sup_{\Re z >0} \cos(\arg z)  e^{\Re z} \norm{g_z}_{\L^1(\R^2)} 
< \infty . 
\end{equation}
\end{prop}

\begin{proof}
We have
\begin{equation*}
\label{}
|k_z^\cal{J}(x,y)| 
\ov{\eqref{kzxy}}{=} |e^{\frac12 \i \langle y , \cal{J} x \rangle} p_z^\cal{J}(y)|
=|p_z^\cal{J}(y)|.
\end{equation*}
Taking into account the explicit formula for $p_z^\cal{J}$ of the second part of Lemma \ref{lem-technical-kernel}, 
we see that it suffices to put 
\begin{equation}
\label{def-gz1}
\h_z(y) 
\ov{\mathrm{def}}{=} \frac{1}{\sqrt{4\pi|\sinh(z)|}} \exp\bigg(- \frac{\Re\coth(z)}{4} y^2 \bigg), \quad y \in \R.
\end{equation}

Clearly, this function takes positive values and it is smooth, even and decreasing on $\R_+$. Then the estimate of the $\L^1$-norm can be easily achieved. Observe that $\Re \coth(z) > 0$ if $\Re z > 0$ since $\Re\coth(x+\i y)=\frac{\coth x(1+\cot^2 y)}{\coth^2 x+\cot^2 y}$ and that 
\begin{equation*}
\label{}
\Re \big[(1+e^{-2z})(1-e^{-2\overline{z}})\big]=\Re \big[1-e^{-2\Re z}+e^{-2z}-e^{-2\overline{z}}\big]=1-e^{-2\Re z}.
\end{equation*}
Hence, we have
\begin{align*}
\MoveEqLeft
\norm{g_z}_{\L^1(\R^2)}
=\bigg(\int_\R \h_z(y) \d y \bigg)^2
\ov{\eqref{def-gz1}}{=} \frac{1}{4\pi|\sinh(z)|}\bigg(\int_\R \exp\bigg(- \frac{\Re\coth(z)}{4} y^2 \bigg)^2 \\
&\ov{\eqref{Gaussian-integral}}{\cong} \frac{|e^{- z}|}{|1 - e^{-2 z}|} \frac{1}{\Re\coth(z)} 
=\frac{e^{- \Re z}}{|1 - e^{-2 z}|} \left[\Re \left(\frac{1+e^{-2 z}}{1 - e^{-2 z}}\right)\right]^{-1} \\
&=\frac{e^{- \Re z}}{|1 - e^{-2 z}|} \left[\Re \left( \frac{(1+e^{-2z})(1-e^{-2\overline{z}})}{|1 - e^{-2 z}|^2} \right) \right]^{-1} 
= e^{- \Re z} \frac{|1 - e^{-2 z}|}{1 -e^{-2 \Re z}}.         
\end{align*}
We let $x \ov{\mathrm{def}}{=} \Re z > 0$ and $y \ov{\mathrm{def}}{=} \Im z$.
Observe that
\begin{align*}
\MoveEqLeft
1-e^{-2z}
=1- e^{-2x} e^{-2y\i} \\
&=1- (1 - 2x + O(x^2))(1 - 2y\i + O(y^2)) 
\underset{0}{\sim}  2x + 2y\i. 
\end{align*} 
Consequently, we have $|1 - e^{-2z}| \underset{0}{\sim} |2x + 2y\i| = 2\sqrt{x^2 + y^2}$. Since $1 -e^{-2 x} \underset{0}{\sim} 2x$, we conclude that
\begin{equation*}
\label{}
\norm{g_z}_{\L^1(\R^2)}
=e^{- x}O\bigg(\frac{\sqrt{x^2 + y^2}}{x}\bigg).
\end{equation*}
%
%
\end{proof}

In the following result and in the sequel of paper, $\lesssim$\label{lesssim} stands for an inequality up to a constant. 

\begin{lemma}
\label{lem-kernel-decay}
Let $(p_t^\cal{J})_{t \geq 0}$ be the semigroup kernel of the universal $\cal{J}$-Weyl tuple from Proposition \ref{prop-semigroup}.
Then 
\begin{equation}
\label{estimate-876}
\norm{\frac{\d}{\d t} p_t^\cal{J} + p_t^\cal{J}}_{\L^1(\R^2)} 
\lesssim e^{-3 t}, \quad t \in \R^+. 
\end{equation}
\end{lemma}

\begin{proof}
Recall that 
\begin{equation*}
\label{}
\frac{\d}{\d t}\frac{1}{\sinh t}
=-\frac{\coth t}{\sinh t}
\quad \text{and} \quad 
\frac{\d}{\d t}\coth t 
= -\frac{1}{\sinh^2 t}
\end{equation*}
for any $t > 0$. With the formula of the second part of Lemma \ref{lem-technical-kernel}, 
we calculate
\begin{align*}
\MoveEqLeft
\frac{\d}{\d t}p_t^\cal{J}(x) + p_t^\cal{J}(x)
\ov{\eqref{kernel-cal-J}}{=} -\coth(t) p_t^\cal{J}(x) 
+  \frac{x_{1}^2 + x_{2}^2}{4\sinh^2 t} p_t^\cal{J}(x)
+  p_t^\cal{J}(x) \\
& = (1-\coth t) p_t^\cal{J}(x) 
+  \frac{x_{1}^2 + x_{2}^2}{4\sinh^2 t}  p_t^\cal{J}(x)
=-\frac{e^{-t}}{\sinh t} p_t^\cal{J}(x) 
+  \frac{x_{1}^2 + x_{2}^2}{4\sinh^2 t}  p_t^\cal{J}(x).
\end{align*}
We estimate the $\L^1(\R^2)$-norms of the two terms separately. On the one hand, we have
\begin{align*}
\MoveEqLeft
\norm{ \frac{e^{-t}}{\sinh t} p_t^\cal{J}}_{\L^1(\R^2)}
\ov{\eqref{kernel-cal-J}}{\cong} \frac{e^{-t}}{\sinh^2 t} \int_{\R^2} \exp\bigg(-\frac{\coth(t)}{4} (x_{1}^2 + x_{2}^2)\bigg) \d x \\
& \ov{\eqref{Gaussian-integral}}{\cong}  \frac{e^{-t}}{\sinh^2 t}\frac{1}{\coth^2 t}
=\frac{e^{- t}}{\cosh^2 t} 
\lesssim e^{- 3t}.
\end{align*}
On the other hand, we have
\begin{align*}
\MoveEqLeft
\norm{\frac{x_{1}^2 + x_{2}^2}{4\sinh^2 t} p_t^\cal{J} }_{\L^1(\R^2)} 
\ov{\eqref{kernel-cal-J}}{\cong} \frac{1}{\sinh^3 t}\int_{\R^{2}} (x_{1}^2 + x_{2}^2)
 \exp\bigg(- \frac{\coth(t)}{4}(x_{1}^2 + x_{2}^2)\bigg) \d x \\
& \ov{\eqref{Gaussian-integral}}{\cong} \frac{1}{\sinh^3 t} \frac{1}{\coth^3 t}  
= \frac{1}{\cosh^3 t} 
 \lesssim e^{- 3t}.
\end{align*}
\end{proof}

\begin{cor}
\label{lem-spectral-gap-2}
Suppose that $1 < p < \infty$ and let $A_\univ^{\cal{J}}$ be the universal $\cal{J}$-Weyl tuple acting on the Banach space $\L^p(\R^2)$. Consider a Banach space $X$.
\begin{enumerate}
\item
Then the operator $\A_1 \ov{\eqref{equ-A-alpha}}{=} (\A - \Id_{\L^p(\R^n)}) \ot \Id_X$ generates a bounded holomorphic semigroup on the Bochner space $\L^p(\R^2,X)$.

\item More precisely, we have the following norm estimate
\[
\norm{t \A_1 \exp(-t\A_1) }_{\L^p(\R^2,X) \to \L^p(\R^2,X)} 
\lesssim te^{-2 t}. 
\]

\item In the same way, if $B_\univ^{\alpha\cal{J}} = (B_1,B_2)$ is the universal $\alpha \cal{J}$-Weyl tuple on $\L^p(\R^2)$ for some $\alpha > 0$, and $X$ is any Banach space, then the closure $\mathcal{B}_\alpha$ of $(B_1^2 + B_2^2 - \alpha \Id_{\L^p(\R^2)}) \ot \Id_X$ generates a bounded holomorphic semigroup on $\L^p(\R^2,X)$ with estimate
\begin{equation}
\label{estimate-biz}
\norm{t \mathcal{B}_\alpha \exp(-t\mathcal{B}_\alpha) }_{\L^p(\R^2,X) \to \L^p(\R^2,X)} \lesssim te^{-2 \alpha t} .
\end{equation}
\end{enumerate}
\end{cor}

\begin{proof}
1. and 2. 
Recall that the semigroup $(T_t)_{t \geq 0}$ is bounded holomorphic by the second part of Proposition \ref{prop-semigroup} and that $\frac{\d}{\d t} T_t=\A T_t$ by \cite[Lemma 4.2.8 p.~290]{Jac01} for any $t>0$. Note that by \eqref{kzxy} the operator $T_t$ is a twisted convolution operator $p^{\cal{J}}_t *_{\cal{J}} \cdot $ by the kernel $p^{\cal{J}}_t$. A standard argument shows that the operator $\frac{\d}{\d t} T_t$ is the twisted convolution operator $(\frac{\d}{\d t}p^{\cal{J}}_t )*_{\cal{J}} \cdot $.

We will use Proposition \ref{Prop-caract-analytic}. We start with the domain condition. It is clear that if $h \ot x$ belongs to $\L^p(\R^2) \ot X$ then the element $(T_t \ot \Id_X)(h \ot x)=T_t(h) \ot x$ belongs to $\dom(\A) \ot X$ by the first part of Theorem \ref{th-Rob}, hence to the domain of the closure $\A \otvn \Id_X$ of $\A \ot \Id_X$. By linearity, $(T_t \ot \Id_X)(\L^p(\R^2) \ot X)$ is a subset of $\dom(\A \otvn \Id_X)$.

Let $g \in \L^p(\R^2,X)$. There exists a sequence $(g_n)$ of elements of the space $\L^p(\R^2) \ot X$ such that $g_n \to g$ in the Banach space $\L^p(\R^2,X)$. By continuity, we see that $(T_t \ot \Id_X)(g_n) \to (T_t \ot \Id_X)(g)$. It is easy to check with that the estimate \eqref{estimate-876} and the vector-valued twisted Young inequality from Lemma \ref{lemma-Young-twisted} that the operator 
\begin{equation*}
\label{}
\A T_t \ot \Id_X
=\frac{\d}{\d t} T_t \ot \Id_X
\end{equation*}
induces a bounded operator. Now, we deduce that
\begin{equation*}
\label{}
(\A \otvn \Id_X)(T_t \ot \Id)(g_n)
=(\A T_t \ot \Id_X)(g_n) 
\to (\A T_t \ot \Id_X)(g).
\end{equation*}
For any $t>0$, we conclude that $(T_t \ot \Id_X)(g)$ belongs to the space $\dom \A \otvn \Id_X$. Indeed, we proved that $(T_t \ot \Id_X)_{t \geq 0}$ is bounded holomorphic.

According to \eqref{kzxy} and \cite[2.2 p.~60]{EnN00}, the operator 
\begin{equation*}
\label{}
\exp(-t (\A - \Id))
=e^t\exp(-t\A)
\end{equation*}
is the twisted convolution operator by the kernel $e^t p_t^{\cal{J}}(y)$ for any $t>0$. For any $g \in \L^p(\R^2,X)$ and any $t > 0$, it follows from the previous paragraph that the vector $\exp(-t\A_1)g$ belongs to $\dom \A_1$ and that $t \mapsto \exp(-t\A_1)$ is differentiable. Using \cite[Lemma 4.2.8 p.~290]{Jac01} in the first equality, we obtain
\begin{equation*}
\label{}
\A_1 \exp(-t\A_1)g 
= - \frac{\d}{\d t} \big[\exp(-t\A_1)g\big]
= - \bigg(\frac{\d}{\d t} \big[e^{t} p^{\cal{J}}_t \big]\bigg) *_{\cal{J}} g.
\end{equation*}

Using again the vector-valued twisted Young inequality from Lemma \ref{lemma-Young-twisted} in the first inequality, we obtain
\begin{align*}
\MoveEqLeft
\norm{t\A_1 \exp(-t\A_1)}_{\L^p(\R^2,X) \to \L^p(\R^2,X)} 
=\norm{t \frac{\d}{\d t} \exp(-t \A_1)}_{\L^p(\R^2,X) \to \L^p(\R^2,X)} \\
&\ov{\eqref{Young-twisted}}{\leq} \norm{t \frac{\d}{\d t} \left(e^{t} p^{\cal{J}}_t \right)}_{\L^1(\R^{2})} 
= t e^{t}\norm{\frac{\d}{\d t} p^{\cal{J}}_t + p^{\cal{J}}_t}_{\L^1(\R^2)} 
\ov{\eqref{estimate-876}}{\lesssim}  t e^{-2 t}.
\end{align*}
We conclude with Proposition \ref{Prop-caract-analytic}.  

3. In the same way as for the first two assertions, we have with the fourth part of Lemma \ref{lem-technical-kernel} and the result from the second point,
\begin{align*}
\MoveEqLeft
\bnorm{t \mathcal{B}_\alpha \exp(-t \mathcal{B}_\alpha) }_{\L^p(\R^2,X) \to \L^p(\R^2,X)} 
= \norm{t \frac{\d}{\d t} \exp(-t \mathcal{B}_\alpha) }_{\L^p(\R^2,X) \to \L^p(\R^2,X)} \\
&\leq \norm{ t \frac{\d}{\d t} \left( e^{\alpha t} p_t^{\alpha \cal{J}} \right) }_{\L^1(\R^2)} 
\ov{\eqref{rescaled-3}}{=} \norm{ t \frac{\d}{\d t} \left( e^{\alpha t} \alpha p_{\alpha t}^{\cal{J}}(\sqrt{\alpha} \,\cdot) \right) }_{\L^1(\R^2)} \\
&= \norm{ t \frac{\d}{\d t} \left( e^{\alpha t} p_{\alpha t}^{\cal{J}} \right) }_{\L^1(\R^2)} \\
&= \alpha t\norm{ \frac{\d} {\d s} \left(e^{s} p_{s}^{\cal{J}} \right)|_{s = \alpha t} }_{\L^1(\R^2)} 
= \alpha t e^{t\alpha}\norm{ \left(   p_{s}^{\cal{J}}+ \frac{\d}{\d s} p_{s}^{\cal{J}} \right)|_{s = \alpha t} }_{\L^1(\R^2)} \\
&  \ov{\eqref{estimate-876}}{\lesssim} \alpha t e^{-2 \alpha t} \cong t e^{-2 \alpha t}.
\end{align*}
\end{proof}

Note that by \cite[2.2 p.~60]{EnN00} the (negative) generator of the semigroup 
\begin{equation*}
\label{}
(e^{t\alpha} \exp(-t \A) \ot \Id_X)_{t \geq 0}
\end{equation*} 
is the shifted generator $\A_\alpha = (\A - \alpha \Id_{\L^p(\R^n)}) \ot \Id_X$. To show the boundedness of the $\HI$ functional calculus for the shifted operator, we rely on the specific case where $\Theta = J$ takes the block diagonal form as given in \eqref{def-de-J}.

\begin{prop}
\label{prop-Hinfty-shifted}
Let $\Theta = J$ be a $2 \times 2$ block diagonal skew-symmetric matrix of $\M_n(\R)$ as in \eqref{def-de-J}.
Suppose that $1 < p < \infty$ and let $A_\univ^J$ be the universal $J$-Weyl tuple acting on the Banach space $\L^p(\R^n)$. Let $X$ be a Banach space. In the first three statements, we suppose that the matrix $J$ is invertible.
\begin{enumerate}
\item Then $\A_\alpha$ generates a bounded holomorphic semigroup on the Bochner space $\L^p(\R^n,X)$.
\item More precisely, we have the following norm estimate
\begin{equation}
\label{div-estim-778}
\norm{t \A_\alpha \exp(-t\A_\alpha) }_{\L^p(\R^n,X) \to \L^p(\R^n,X)} 
\lesssim t e^{-2\beta t},
\end{equation}
where $\beta=\min\{\alpha_1,\ldots,\alpha_k\}$ is defined at the beginning of this section.
\item The operator $\A_\alpha - 2 \beta \Id_{\L^p(\R^n,X)}$ is sectorial of type $\frac{\pi}{2}$ on the subspace $\overline{\Ran\A_\alpha}$ of $\L^p(\R^n,X)$.
\item In this part, $J$ is not necessarily invertible. Assume that the Banach space $X$ is $\UMD$. The operator $\A_\alpha$ admits a bounded $\HI(\Sigma_\omega)$ functional calculus on the Bochner space $\L^p(\R^n,X)$ for any angle $\omega \in (\frac{\pi}{2},\pi)$, i.e.~$\omega_{\H^\infty}(\A_\alpha) \leq \frac{\pi}{2}$.
\end{enumerate}
\end{prop}

\begin{proof}
1. and 2. In the case where $J=\begin{pmatrix} 
0 & \alpha \\ 
- \alpha & 0 
\end{pmatrix}$ for some $\alpha >0$, \eqref{estimate-biz} shows the estimate \eqref{div-estim-778}, with $\beta = \alpha$ in this case, and hence also 1.
If $J$ is a bigger $2 \times 2$ block diagonal matrix, according to Lemma \ref{lem-kernel-splitting}, 
\begin{align*}
\MoveEqLeft
\exp(-t \A_\alpha) 
= e^{t \alpha} \exp(-t \A) \ot \Id_X 
= \prod_{j=1}^k \left( e^{t \alpha_j} \exp(-t \A^j ) \ot \Id_{\L^p(\R^{n-2},X)} \right) \\ 
&= \prod_{j=1}^k \exp(-t \A^j_{\alpha_j}) \ot \Id_{\L^p(\R^{n-2})}.         
\end{align*}
Thus, $\exp(-t \A_\alpha)$ is the product of operators $\exp(-t \A^j_{\alpha_j}) \ot \Id_{\L^p(\R^{n-2})}$ which are as in the first case with $\alpha$ replaced by $\alpha_j$ and $X$ replaced by the space $\L^p(\R^{n-2},X)$.
Noting that 
\begin{equation*}
\label{}
\A_\alpha \exp(-t\A_\alpha) 
= -\frac{\d}{\d t} \exp(-t \A_\alpha),
\end{equation*}
applying Leibniz rule, and observing that the semigroups $\exp(-t \A^j_{\alpha_j}) \ot \Id_{\L^p(\R^{n-2})}$ are uniformly bounded, we obtain the second part since
\begin{align*}
\MoveEqLeft
\norm{t \A_\alpha \exp(-t \A_\alpha)}_{\L^p(\R^n,X) \to \L^p(\R^n,X)} \\
&  \leq \sum_{j=1}^k \norm{t \A^j_{\alpha_j} \exp(-t \A^j_{\alpha_j}) \otimes \Id_{\L^p(\R^{n-2})}} \prod_{l \neq j} \norm{\exp(-t \A^l_{\alpha_l}) \otimes \Id_{\L^p(\R^{n-2})} } \\
&   \ov{\eqref{estimate-biz}}{\lesssim} t \sum_{j=1}^k e^{-2 \alpha_j t} 
\cong t e^{-2 \beta t}.
\end{align*}
3. By Proposition \ref{prop-sec-pi-2}, it suffices to check that the norm of $\exp(-t(\A_\alpha-2\beta\Id))=e^{2\beta t}\exp(-t\A_\alpha)$ on the subspace $\ovl{\Ran \A_\alpha}$ is bounded as $t \to \infty$. If $f = \A_\alpha g$ for some $g \in \dom\A_\alpha$, one has $\lim_{s \to \infty} \exp(-s \A_\alpha)f = \lim_{s \to \infty} \A_\alpha \exp(-s \A_\alpha)g = 0$ according to the second part. Therefore, again by the second part, we have for any $t \geq 1$ the estimate
\begin{align*}
\MoveEqLeft
\norm{ \exp(-t \A_\alpha)f}_{\L^p(\R^n,X)} 
= \norm{ \int_t^\infty - \frac{\d}{\d s} \exp(-s \A_\alpha) f\d s} \\
&\leq \int_t^\infty \norm{\A_\alpha \exp(-s \A_\alpha)f} \d s\\
&  \ov{\eqref{div-estim-778}}{\lesssim} \int_t^\infty e^{-2\beta s} \d s \norm{f}_{\L^p(\R^n,X)} 
\lesssim e^{-2 \beta t} \norm{f}_{\L^p(\R^n,X)}.
\end{align*}
We obtain the third point.

4. Now, $J$ is no longer necessarily invertible. 
According to Lemma \ref{lem-kernel-splitting}, the operator $\A_\alpha$ equals (the closure of) $\A^1_\alpha \ot \Id_{\L^p(\R^{n-2k})} + \Id_{\L^p(\R^{2k},X)} \ot \A^{0_{n-2k}}$, where $\A^1$ is the semigroup generator associated with the universal Weyl tuple of an invertible skew-symmetric matrix, and $\A^{0_{n-2k}}$ is the ordinary Laplacian on $\L^p(\R^{n-2k})$.
Recalling that the $\UMD$ Banach space $Y =\L^p(\R^n,X)$ has property $(\Delta)$, according to Theorem \ref{Th-com-LM}, if suffices to establish the boundedness of the $\HI(\Sigma_\omega)$ functional calculus for the operators $\A^1_\alpha \ot \Id_{\L^p(\R^{n-2k})}$ and $ \Id_{\L^p(\R^{2k})} \ot \A^{0_{n-2k}}$ on $Y$ separately. Note that when applying Le Merdy's theorem, the angles $\omega_1$ and $\omega_2$ for $\A^1_\alpha$ and $\A^{0_{n-2k}}$ add together. Therefore, it is important that we are able to choose $\omega_2 > 0$ to be as small as desired. Since the Banach spaces $\L^p(\R^{n-2k},X)$ and $\L^p(\R^{2k},X)$ are UMD spaces, we are reduced to show the boundedness of the $\HI(\Sigma_\omega)$ functional calculus for the operators $\A^1_\alpha$ and $\A^{0_{n-2k}} \ot \Id_X$ on theses spaces for a generic UMD Banach space $X$.
According to Proposition \ref{prop-HI-calculus-non-shifted}, the operator 
\begin{equation*}
\label{}
\A^1 \ot \Id_X
=(\A^1 - \alpha \Id_{\L^p(\R^n)}) \ot \Id_X+ \alpha \Id_{\L^p(\R^n,X)}
= (\A^1_\alpha - 2 \beta \Id) + (\alpha + 2 \beta)\Id
\end{equation*}
has a bounded $\HI(\Sigma_\omega)$ functional calculus on the Bochner space $\L^p(\R^{2k},X)$, and thus, on the invariant subspace $\overline{\Ran\A^1_\alpha}$.
Since the operator $\A^1_\alpha - 2 \beta \Id$ is sectorial on the space $\ovl{\Ran\A^1_\alpha}$ according to the third part, Lemma \ref{lem-shifted-HI-calculus} implies that the operator
\begin{equation*}
\label{}
\A ^1_\alpha 
= (\A^1_\alpha - 2 \beta \Id) + 2 \beta \Id
\end{equation*}
also has a bounded $\HI(\Sigma_\omega)$ functional calculus on the space $\overline{\Ran\A^1_\alpha}$. One has trivially the boundedness of the $\HI(\Sigma_\omega)$ functional calculus of the operator $\A^1_\alpha$ on $\Ker \A^1_\alpha$. Since the Banach space $X$ is UMD, $\L^p(\R^{2k},X)$ is reflexive and therefore, the sectorial operator $\A^1_\alpha$ (see the first part) induces a space decomposition 
\begin{equation*}
\label{}
\L^p(\R^{2k},X) 
\ov{\eqref{decompo-reflexive}}{=} \overline{\Ran\A^1_\alpha} \oplus \Ker\A^1_\alpha.
\end{equation*}
So the operator $\A^1_\alpha$ has a bounded $\HI(\Sigma_\omega)$ functional calculus on the space $\L^p(\R^{2k},X)$.
Finally, $\A^{0_{n-2k}} \ot \Id_X$ has also a bounded $\HI(\Sigma_\omega)$ functional calculus on the space $\L^p(\R^{n-2k},X)$ by Example \ref{Laplacian-funct} since it is ordinary Laplacian on $\L^p(\R^{n-2k})$ and since the Banach space $X$ is UMD.
\end{proof}

%



We are now able to state the main result of this section for a general skew-symmetric matrix. Recall that  
$\alpha$ is defined in \eqref{def-alpha}.

\begin{cor}
\label{Cor-512}
Consider a skew-symmetric matrix $\Theta \in \M_n(\R)$ and the universal $\Theta$-Weyl tuple $A_\univ^\Theta$ acting on $\L^p(\R^n)$. Let $X$ be a $\UMD$ Banach space. Then the operator $\A_\alpha = (\A - \alpha \Id_{\L^p(\R^n)}) \ot \Id_{X}$ admits a bounded $\HI(\Sigma_\omega)$ functional calculus on the Bochner space $\L^p(\R^n,X)$ for any angle $\omega \in (\frac{\pi}{2},\pi)$, i.e.~$\omega_{\H^\infty}(\A_\alpha) \leq \frac{\pi}{2}$.
\end{cor}

\begin{proof}
Recall again the formula $\Theta = O^T J O$ from \eqref{diago-Theta}. Consider the isometry 
\begin{equation*}
\label{}
\tilde{O} \co \L^p(\R^n) \to \L^p(\R^n), \: f \mapsto f(O^T \cdot).
\end{equation*} 
Denote the semigroups generated by $\A -\alpha \Id_{\L^p(\R^n)}$ associated with the universal $J$-Weyl tuple and universal $\Theta$-Weyl tuple by $(T_t^J)_{t \geq 0}$ and $(T_t^\Theta)_{t \geq 0}$.

In the same manner, we recall the notation $p_t^J$ resp. $p_t^\Theta$ from Proposition \ref{prop-semigroup}. Then for any $g \in \L^p(\R^n)$, we have
\[ 
T_t^Jg(x) 
\ov{\eqref{Autre-rep-Tz}}{=} \int_{\R^n} e^{\frac12 \langle y, Jx \rangle} p_t^J(y) g(x-y) \d y. 
\]
Moreover, since $p_t^\Theta(x) \ov{\eqref{similar-pz}}{=} p_t^J(Ox)$, we have
\begin{align*}
\MoveEqLeft
\tilde{O}^{-1} T_t^J \tilde{O} g(x) 
= \int_{\R^n} e^{\frac12 \i \langle y, JOx \rangle} p_t^J(y) g\big(x-O^Ty\big) \d y \\
& = \int_{\R^n} e^{\frac12 \i \langle Oy, JO x \rangle} p_t^J(Oy) g(x-y) \d y 
 = \int_{\R^n} e^{\frac12 \i \langle y, \Theta x \rangle} p_t^\Theta(y) g(x-y) \d y \\
&\ov{\eqref{Autre-rep-Tz}}{=} T_t^\Theta g(x).
\end{align*}
Since $\tilde{O}^{\pm 1}$ is an isometry on all $\L^p(\R^n)$, $1 \leq p \leq \infty$, it extends to an isometry on $\L^p(\R^n,X)$. As the generator of $(T_t^J \ot \Id_X)_{t \geq 0}$ has a bounded $\HI(\Sigma_\omega)$ functional calculus according to Proposition \ref{prop-Hinfty-shifted}, also that of $(T_t^\Theta \ot \Id_X)_{t \geq 0}$ admits a bounded $\HI(\Sigma_\omega)$ functional calculus by Proposition \ref{prop-isomorphism}.
\end{proof}



%
%

\begin{prop}
\label{prop-not}
Let $\Theta \in \M_n(\R)$ be any skew-symmetric matrix. Consider any $\Theta$-Weyl tuple $A$ acting on some Banach space $X$. Let $(T_t)_{t \geq 0}$ be the semigroup from Proposition \ref{prop-semigroup} and denote $S_t \ov{\mathrm{def}}{=} e^{\alpha t} T_t$ if $t \geq 0$. Then the rescaled semigroup $(S_t)_{t \geq 0}$, with (negative) generator $\A - \alpha \Id_{X}$, is bounded holomorphic of angle $\frac{\pi}{2}$. Its extension is given by
\begin{equation}
\label{def-Sz}
S_z x 
= \int_{\R^n} b_z(u) e^{\i u \cdot A}x \d u, \quad \Re z >0, x \in X
\end{equation}
for some function $b_z \in \L^1(\R^n)$ and in addition we have 
\begin{equation}
\label{estimate}
\sup_{z \in \Sigma_\omega} \norm{b_z}_{\L^1(\R^n)} < \infty, \quad \omega \in (0,\textstyle{\frac{\pi}{2}}).
\end{equation}
Finally, $b_z=b_z^{\Theta}$ does not depend on $A$.
\end{prop}

\begin{proof}
By Proposition \ref{prop-semigroup}, we have that there exists an angle $\varphi \in (0,\frac{\pi}{2}]$ such that
$$
S_z x 
=\int_{\R^n} e^{\alpha z}p_z^\Theta(s)  e^{\i s \cdot A} x \d s, \quad x \in X, z \in \Sigma_{\varphi}.
$$
Let $z$ be a complex number with $\Re z >0$. We are thus led to introduce
\begin{equation}
\label{inter-9876}
b_z(u) 
\ov{\mathrm{def}}{=} e^{\alpha z}p^\Theta_z(u) \ov{\eqref{similar-pz}}{=} e^{\alpha z} p^J_z(Ou) , \quad u \in \R^{n}.
\end{equation}
Using the tensor product decomposition of Lemma \ref{lem-kernel-splitting}, we obtain for any $u \in \R^n$ that
\begin{align*}
\MoveEqLeft
|b_z(O^{T}u) |
\ov{\eqref{inter-9876}}{=} |e^{\alpha z} p^J_z(u)| 
\ov{\eqref{kzxy}}{=} |e^{\alpha z} p^J_z(u)|
= e^{\alpha \Re z} |k_z^J(u)| \\ 
& \ov{\eqref{kt-as-a-product}}{=} e^{\alpha \Re z} \big|k_z^{J_1}(u_1,u_2) \cdots k_z^{J_k}(u_{2k-1},u_{2k}) k_z^{0_{n-2k}}(x'',y'')\big| \\
&\ov{\eqref{majo-by-gz}}{=} e^{\alpha \Re z} \big|p_{z}^{\alpha_{1}\mathcal{J}}(u_1,u_2) \cdots p_{z}^{\alpha_{k}\mathcal{J}}(u_{2k-1},u_{2k}) p_{z}^{0_{n-2k}}(x'',y'')\big|\\
&\ov{\eqref{rescaled-3}}{=} e^{\alpha \Re z} \big|\alpha_1 p_{\alpha_1 z}^{\cal{J}}(\sqrt{\alpha_1}u_1,\sqrt{\alpha_1}u_2) \cdots \\
&\alpha_k p_{\alpha_k z}^{\mathcal{J}}(\sqrt{\alpha_k} u_{2k-1},\sqrt{\alpha_k} u_{2k}) p_{z}^{0_{n-2k}}(x'',y'')\big|.
\end{align*}
Using $\alpha=\sum_{j=1}^{k} \alpha_j$, we obtain that
\begin{align*}
|b_z(O^{T}u) | 
= |p_{z}^{0_{n-2k}}(x'',y'')|\prod \limits_{j=1}^{k} e^{\alpha_{j} \Re z}\alpha_{j} \big|p_{\alpha_{j}z}^{\mathcal{J}}(\sqrt{\alpha_{j}}u_{2j-1},\sqrt{\alpha_{j}}u_{2j})\big| ,
\quad u \in \R^{n}.
\end{align*}
Therefore, using Proposition \ref{prop-kernel-estimate}, we have that, for all angle $\omega \in (0,\frac{\pi}{2})$, there exists a positive constant $C$ such that for any complex number $z \in \Sigma_{\omega}$ 
\begin{align*}
\norm{b_{z}}_{\L^1(\R^n)} 
&\lesssim \bigg(\frac{1}{\cos(\arg z)}\bigg)^{\frac{n-2k}{2}} \bigg(\prod \limits _{j=1} ^{k}\frac{1}{\cos(\arg z)}\bigg) 
\lesssim \frac{1}{(\cos \omega)^{\frac{n}{2}}},      
\end{align*}
where we used the estimate on the classical heat kernel provided by \cite[Remark 3.7.10 p.~153]{ABHN11} (or \cite[p.~543]{HvNVW18}). This proves \eqref{estimate}. Note that for any $x \in X$, we have
\begin{align*}
\MoveEqLeft
\norm{\int_{\R^n} b_z(u) e^{\i u \cdot A}x \d u}_X        
\leq \int_{\R^n} |b_z(u)| \norm{e^{\i u \cdot A}x}_X \d u 
\ov{\eqref{def-eitA}}{\lesssim} \int_{\R^n} |b_z(u)| \d u \norm{x}_X,
\end{align*}
This implies that
\begin{equation*}
\label{}
z\mapsto \bigg[x \mapsto \int_{\R^n} b_z(u) e^{\i u \cdot A}x \d u\bigg]
\end{equation*}
defines a bounded holomorphic semigroup on the sector $\Sigma_{\omega}$, equal to $(S_{z})_{z \in \Sigma_{\varphi}}$ on $\Sigma_{\varphi}$. The result follows from the identity theorem for analytic functions.
\end{proof}

\chapter{Maximal estimates and $R$-boundedness}
\section{Maximal estimates in noncommutative $\L^p$-spaces}
\label{subsec-maximal-estimates}

In this section, we delve into some aspects of harmonic analysis within noncommutative $\L^p$-spaces. Specifically, we give estimates of some convolution type maximal function defined by well-behaving kernels.

\paragraph{$\ell^\infty$-valued noncommutative $\L^p$-spaces} Let $\cal{M}$ be a von Neumann algebra equipped with a normal semifinite faithful trace and let $\Omega$ be a measure space. Suppose that $1 \leq p < \infty$. 
We define $\L^p(\cal{M},\ell^\infty_I)$ \label{LpMellI} as the space consisting of all families $x=(x_i)_{i \in I}$ of elements of $\L^p(\cal{M})$ that admit a factorization 
\begin{equation*}
\label{}
x_i
=a y_i b, \quad i \in I,
\end{equation*}
where $a, b \in \L^{2p}(\cal{M})$ and where $(y_i)_{i \in I}$ is a bounded family of elements of $\L^\infty(\cal{M})$. The norm of each element $x$ of $\L^p(\cal{M},\ell^\infty_I)$ is defined by
\begin{equation*}
\label{}
\norm{x}_{\L^p(\cal{M},\ell^\infty_I)}
\ov{\mathrm{def}}{=}\inf\Big\{\norm{a}_{\L^{2p}(\cal{M})} \sup_{i \in I} \norm{y_i}_{\L^\infty(\cal{M})} \norm{b}_{\L^{2p}(\cal{M})}\Big\} ,
\end{equation*}
where the infimum runs over all factorizations of $x$ as previously. We will follow the classical convention which is that the norm $\norm{x}_{\L^p(\cal{M},\ell^\infty_I)}$ is denoted by $\norm{\sup_{i \in I}^+ x_i}_{\L^p(\cal{M})}$ or $\norm{\sup_{i \in I}^+ x_i}_p$. We refer to \cite{JuX07} for more information. In the sequel, we will use the fact that if $x_i \geq 0$ for any $i \in I$, we have
\begin{equation}
\label{Norm-Lp-Linfty-Omega}
\norm{{\sup_{i \in I}}^+ x_i}_p
=\inf \big\{\norm{y}_{\L^p(\cal{M})} : x_i \leq y \text{ for any } i \in I \big\}.
\end{equation}
Let $\Omega$ be a measurable space and $Z$ be an index set. We will use the following elementary fact that if $(\mu_{z})_{z \in Z}$ is a family of finite positive measures on $\Omega$ and $(R_{\omega})_{\omega \in \Omega}$ is a family of positive operators $R_{\omega} \co \L^p(\cal{M}) \to \L^p(\cal{M})$ such that $\omega \mapsto R_{\omega}(u)$ is strongly $\mu_{z}$-measurable for any $u \in \L^p(\cal{M})$ and any $z \in Z$, then the following bound holds for any positive element $u \in \L^p(\cal{M})$
\begin{equation}
\label{averageMax}
\norm{{\mathop{\mathrm{sup}^+}_{z \in Z}} \Big\{ \int_{\Omega} R_{\omega}(u) \d \mu_{z}( \omega) \Big\} }_p
\leq \bigg( \sup_{z \in Z} \mu_{z}(\Omega) \bigg) \norm{{\mathop{\mathrm{sup}^+}_{\omega \in \Omega}} R_{\omega}(u)}_p.
\end{equation}

\paragraph{Hardy-Littlewood maximal operators}
Let $B$ be a symmetric convex body in $\R^d$ and let $\cal{N}$ be a semifinite von Neumann algebra equipped with a normal semifinite faithful trace. For any $t > 0$, we may consider the operator-valued Hardy-Littlewood operator $\Phi_t \co \L^p(\R^d,\L^p(\cal{N})) \to \L^p(\R^d,\L^p(\cal{N}))$ defined by 
\begin{equation}
\label{def-Hardy}
(\Phi_t f)(x)
\ov{\mathrm{def}}{=} \frac{1}{\mu(B)}\int_{B} f(x-ty) \d y, \quad x \in \R^d.
\end{equation}
By \cite[Theorem 5.13 (3) p.~79]{HWW23}, if $B$ is the closed Euclidean unit ball of $\R^d$, then there exists a positive constant $C_p$ which is independent of the dimension $d$ such that for any $1 < p < \infty$, we have
\begin{equation}
\label{estimte-Hardy}
\norm{{\sup_{t > 0}}^+ \Phi _{t}f}_p
\leq C_p\norm{f}_{\L^p(\R^d,\L^p(\cal{N}))}, \qquad  f \in \L^p(\R^d,\L^p(\cal{N})).
\end{equation}
Using an argument based on direct sums of semifinite von Neumann algebras, one can show that $C_p$ can be chosen independently of the von Neumann algebra $\mathcal{N}$. We also refer to the papers \cite{HLW21} and\cite{Hon20} for related facts. For any $t>0$ and any integer $1 \leq j \leq n$, we introduce the operator $R_t^j \co \L^p(\R^n,S^p) \to \L^p(\R^n,S^p)$ defined by
\begin{equation}
\label{def-Rt}
(R_t^j f)(x)
\ov{\mathrm{def}}{=} \frac{1}{2t} \int_{x_j-t}^{x_j+t} f(x_1,\ldots,x_{j-1},y,x_{j+1},\ldots,x_n) \d y, \quad \text{a.e. }x \in \R^n. 
\end{equation}
Using \cite[Proposition 1.2.24 p.~25]{HvNVW16} together with \cite[Proposition 1.2.2 p.~16]{ABHN11}, it is easy to check that the map $(0,\infty) \to \L^p(\R^n,S^p)$, $t \mapsto R_t^j f$ is continuous, hence strongly measurable for any $1 \leq j \leq n$ and any vector-valued function $f \in\L^p(\R^n,S^p)$. 

From the estimate \eqref{estimte-Hardy} applied with $d = 1$, $B= [-1,1]$ and $\L^p(\cal{N}) = \L^p(\R^{n-1},S^p)$, we immediately deduce with \eqref{Norm-Lp-Linfty-Omega} the following result, using the equality 
\begin{equation*}
\label{}
\L^p(\R,\L^p(\cal{N})) = \L^p(\R^n,S^p).
\end{equation*}

\begin{prop}
\label{prop-nc-maximal-inequality}
Let $1 < p < \infty$ and $j \in \{ 1,\ldots,n\}$. Let $u$ be a positive element of the space $\L^p(\R^n,S^p)$. Then there exists some positive element $v \in \L^p(\R^n,S^p)$ such that for any $t>0$
\[ 
R_t^j(u) \leq v 
\quad \text{with} \quad
\norm{v}_{\L^p(\R^n,S^p)} 
\leq C_{p} \norm{u}_{\L^p(\R^n,S^p)}. 
\]
\end{prop}

The next result is fundamental for the next section.

\begin{lemma}
\label{lem-nc-maximal-inequality-convolution}
Let $Z$ be an index set. For any integer $j \in \{1, \ldots, n\}$ and any $z \in Z$, consider a smooth even integrable function $\h_z \co \R \to \R_+$ which is decreasing on $\R_+$. Write $g_z(y_1,y_2) \ov{\mathrm{def}}{=} \h_z(y_1)\h_z(y_2)$ for any $y_1,y_2 \in \R$. Assume that 
\begin{equation*}
\label{}
\sup_{z \in Z} \norm{g_z}_{\L^1(\R^2)} 
< \infty.
\end{equation*}
Suppose that $1 < p < \infty$. Let $u \in \L^p(\R^2,S^p)$ be a positive element. Then
\begin{equation}
\label{div-4556}
\norm{{\mathop{\mathrm{sup}^+}_{z \in Z}} g_z \ast u }_{\L^p(\R^2,S^p)}
\leq C_p^2 \sup_{z \in Z} \norm{g_z}_{\L^1(\R^2)}\norm{u}_{\L^p(\R^2,S^p)}.
\end{equation}
\end{lemma}

\begin{proof}
Consider a positive element $u$ of the space $\L^p(\R^{2},S^p)$. Let $\epsi > 0$. By Proposition \ref{prop-nc-maximal-inequality}, there exists some positive element $u_2 \in \L^p(\R^2,S^p)$ such that $R_{t}^2(u) \leq u_2$ for all $t > 0$ with 
\begin{equation*}
\label{}
\norm{u_2}_{\L^p(\R^2,S^p)} 
\leq C_p \norm{u}_{\L^p(\R^2,S^p)}+\epsi.
\end{equation*}
For any $t \in (0,\infty)$, note that $R_{t}^1(u_2)$ is a positive element. Using again Proposition \ref{prop-nc-maximal-inequality}, we see that there exists some positive element $u_{1} \in \L^p(\R^2,S^p)$ such that $R_{t}^{1}(u_2) \leq u_{1}$ for all $t > 0$ with 
\begin{equation*}
\label{}
\norm{u_{1}}_{\L^p(\R^2,S^p)} 
\leq C_p \norm{u_2}_{\L^p(\R^2,S^p)}+\epsi
\leq C_p^2 \norm{u}_{\L^p(\R^2,S^p)}+2\epsi.
\end{equation*}
In particular, the positivity of each operator $R_{t}^{1}$ entails that 
\begin{equation*}
\label{}
R_{t_{1}}^{1} R_{t_2}^2(u) \leq R_{t_{1}}^{1}(u_2) \leq u_{1}
\end{equation*}
for any $t_1,t_2 > 0$. By \eqref{Norm-Lp-Linfty-Omega}, we deduce that
\begin{equation*}
\label{}
\norm{{\sup}^+_{t_1,t_2>0} R_{t_{1}}^{1} R_{t_2}^2(u)}_p 
\leq C_p^2 \norm{u}_{\L^p(\R^2,S^p)}.
\end{equation*} 
Now, consider a family $(\mu_z)_{z \in Z}$ of positive finite measures on the product $(0,\infty) \times (0,\infty)$ such that $
\sup_{z \in Z} \norm{\mu_z}_{\M((0,\infty)^2)} 
< \infty$. According to \eqref{averageMax}, we obtain that
\begin{equation*}
\norm{{\mathop{\mathrm{sup}^+}_{z \in Z}} \int_0^\infty \int_0^\infty R_{t_{1}}^{1} R_{t_2}^2(u) \d \mu_{z}(t_1,t_2)  }_p
\leq C_p^2 \bigg( \sup_{z \in Z} \norm{\mu_{z}}_{\M((0,\infty)^2)} \bigg) \norm{u}_{\L^p(\R^2,S^p)}.
\end{equation*}
It suffices now to identify $g_z \ast u$ with an integral $\int_0^\infty \int_0^\infty R_{t_1}^1 R_{t_2}^2(u) \d\mu_z(t_1,t_2)$ for some suitable measure $\mu_z$. Note that the assumptions imply that $\lim_{y \to \infty} y^\gamma \h_z(y) = 0$ for any $z \in Z$ and any $\gamma \in ]0,1[$ (this is a classical exercise\footnote{
Since the function $h_z$ is decreasing on $\R_+$, if $0 \leq t \leq y$ we have $h_z(y) \leq h_z(t)$. 
Integrating over the interval $[0, y]$, this yields
\[
h_z(y) \leq \frac{1}{y} \int_0^y h_z(t) \d t.
\]
Multiplying both sides by $y^\gamma$, we get:
\[
y^\gamma h_z(y) \leq y^{\gamma - 1} \int_0^y h_z(t)\d t.
\]
Since $h_z$ is integrable, there exists a constant $C$ such that $\int_0^y h_z(t) \d t \leq C$ for any $y > 0$. Thus
\[
y^\gamma h_z(y) \leq C y^{\gamma - 1}.
\]
Since $\gamma - 1 < 0$, we see that $y^{\gamma - 1} \to 0$ as $y \to \infty$. We conclude that
\[
\lim_{y \to \infty} y^\gamma h_z(y) = 0.
\]
}). Using H\"older's inequality, this implies that, for all $x \in \R^2$,
\begin{align*}
\MoveEqLeft
\underset{|(y_1,y_2)| \to \infty}{\lim} h_{z}(y_1)h_{z}(y_2) \int_{-y_1}^{y_1} 
\int_{-y_2}^{y_2} \norm{u(x+w)}_{S^p} \d w_1 \d w_2 \\
&\lesssim \norm{u}_{\L^{p}(\R^2,S^{p})} \underset{y \to \infty}{\lim} \big(y^{\frac{1}{p^*}} h_{z}(y)\big)^{2} 
= 0.         
\end{align*}
Therefore, integrating by parts in the third equality (the previous computation shows that the bracket of the integration by parts
converges to zero), we obtain that 
\begin{align*}
\MoveEqLeft
(g_z \ast u)(x) 
\ov{\eqref{Convolution-formulas}}{=} \int_{\R^2} g_z(y) u(x-y) \d y 
= \int_{\R^2} \h_z(y_1)\h_z(y_2) u(x-y) \d y \\
&= \int_0^\infty \int_0^\infty \h_z(y_1)\h_z(y_2) \sum_{v \in \{\pm 1\}^2} u(x+v \cdot y) \d y_1 \d y_2 \\
&= \int_0^\infty  \int_0^\infty (\h_z)'(y_1)(\h_z)'(y_2) \bigg(\int_0^{y_1} \int_0^{y_2} \sum_{v \in \{\pm 1\}^2} u(x+v \cdot w) \d w_1 \d w_2\bigg) \d y_1 \d y_2 \\
&\ov{\eqref{def-Rt}}{=} \int_0^\infty \int_0^\infty R_{y_1}^1 u(x) R_{y_2}^2 u(x) \times 4y_1 y_2 (\h_z)'(y_1)(\h_z)'(y_2) \d y_1 \d y_2.
\end{align*}
For any $z \in Z$, we are led to set $ 
\mu_z(y_1,y_2) 
\ov{\mathrm{def}}{=} 4y_1y_2 (\h_z)'(y_1)(\h_z)'(y_2) \d y_1 \d y_2$. 
Note that the density of $\mu_z$ maintains a constant sign since the function $\h_z$ is decreasing on $\R_+$. Thus we can easily calculate, again by integration by parts without boundary terms, 
\begin{align*}
\MoveEqLeft
\mu_z((0,\infty)^2)
= 4\int_0^\infty \int_0^\infty y_1 y_2 (\h_z)'(y_1) (\h_z)'(y_2) \d y_1 \d y_2 \\
&= 4 \int_0^\infty \int_0^\infty \h_z(y_1) \h_z(y_2) \d y_1 \d y_2
=4 \int_0^\infty \int_0^\infty g_z(y_1,y_2) \d y_1 \d y_2 \\
&= \norm{g_z}_{\L^1(\R^2)}.
\end{align*}
Therefore, we have 
\begin{equation*}
\label{}
g_z \ast u 
= \int_0^\infty  \int_0^\infty R_{t_1}^1 R_{t_2}^2(u) \d \mu_z(t_1,t_2).
\end{equation*}
The last sentence is a consequence of \eqref{Norm-Lp-Linfty-Omega}.
\end{proof}

\begin{remark} \normalfont
\label{rem-nc-maximal-inequality-convolution}
Note that Lemma \ref{lem-nc-maximal-inequality-convolution} has a direct counterpart for convolutions on $\R^n$ for other dimensions $n$ than $2$.
In the proof of the fourth point of Proposition \ref{prop-LpSp-square-max-cal-J}, we will use the one-dimensional version.
\end{remark}

\section{Background on $R$-boundedness}
\label{sec-R-boundedness-def}

Here, we provide information and background on $R$-boundedness. This notion is a probabilistic generalization of the classical boundedness for sets of operators acting on Hilbert spaces. It has attracted considerable interest in recent years due to its wide range of applications, such as maximal regularity theory for evolution equations, stochastic evolution equations, and vector-valued harmonic analysis.

Suppose that $1 < p < \infty$. Following \cite[Definition 8.1.1, Remark 8.1.2 p.~165]{HvNVW18}, we say that a set $\cal{F}$ of bounded operators on a Banach space $X$ is $R$-bounded provided that there exists a constant $C\geq 0$ such that for any operators $T_1,\ldots, T_n$ in $\cal{F}$ and any vectors $x_1,\ldots,x_n$ in $X$, we have
\begin{equation}
\label{R-boundedness}
\Bgnorm{\sum_{k=1}^{n} \epsi_k \ot T_k (x_k)}_{\L^p(\Omega_0,X)}
\leq C \Bgnorm{\sum_{k=1}^{n} \epsi_k \ot x_k}_{\L^p(\Omega_0,X)},
\end{equation}
where $(\epsi_{k})_{k \geq 1}$ is a sequence of independent Rademacher variables\label{def-Rademacher} on some probability space $\Omega_0$. It is known that this property is independent of $p$. The best constant in \eqref{R-boundedness} is denoted by $\mathscr{R}_p(\mathcal{F})$\label{R-bound} and is referred to as the $R$-bound of the family $\cal{F}$. 
By \cite[Theorem 8.1.3 p.~166]{HvNVW18}, any $R$-bounded set is necessarily bounded. It is worth noting that by \cite[Corollary 8.6.2 p.~235]{HvNVW18} $X$ is isomorphic to a Hilbert space if and only if each bounded subset of $\B(X)$ is $R$-bounded. Finally, a singleton $\{T\}$ is automatically $R$-bounded by \cite[Example 8.1.7 p.~170]{HvNVW18}.

If $X$ is a complex Banach lattice with finite cotype, we will use a classical theorem of Maurey which asserts that we have a uniform equivalence
\begin{equation}
\label{equiv-lattice}
\Bgnorm{\sum_{k=1}^n \epsi_k \ot x_k}_{\L^p(\Omega_0,X)}
\approx \Bgnorm{\bigg(\sum_{k=1}^n |x_k|^{2}\bigg)^{\frac{1}{2}}}_X
\end{equation}
for any vectors $x_1,\ldots,x_n$ of $X$, see e.g.~\cite[Theorem 16.18 p.~338]{DJT95} (or \cite[Definition 8.1.1 p.~165]{HvNVW18} for the real case).

We also will use in Section \ref{subsec-R-sectoriality} the following well-known result \cite[Proposition 8.4.1 p.~211]{HvNVW18}. Recall that by \cite[Theorem 7.4.23 p.~124]{HvNVW18} a Banach space is $K$-convex if and only if it has non-trivial type. So a noncommutative $\L^p$-space is $K$-convex by \cite[Corollary 5.5 p.~1481]{PiX03} for any $1 < p < \infty$.  

\begin{prop}
\label{prop-R-bounded-dual}
Let $X$ be a $K$-convex Banach space. Then a family $\{T_z :  z \in Z \}$ of operators acting on $X$ is $R$-bounded if and only if the family $\{T_z^*:  z \in Z \}$ of adjoints is $R$-bounded over the dual space $X^*$.
\end{prop}

\section{$R$-boundedness and tensor product of operators}
\label{sec-R-boundedness}

In Section \ref{subsec-R-sectoriality}, we need Proposition \ref{prop-2-tensor-R-bdd}, which describe a stability of $R$-boundedness under tensor products. First, we establish the following lemma on $R$-boundedness, which is an immediate consequence of Kahane-Khintchine's inequalities.

\begin{lemma}
\label{lem-tensor-R-bdd}
Let $\{S_z : z \in Z \}$ be an $R$-bounded family of operators acting on some Banach space $X$.
Consider  a $\sigma$-finite measure space $(\Omega,\mu)$ and $1 \leq p < \infty$. Then for any $z \in Z$, the map $\Id_{\L^p(\Omega)} \ot S_z$ extends to a bounded operator on the Bochner space $\L^p(\Omega,X)$. Moreover, the set $\{ \Id_{\L^p(\Omega)} \ot S_z : z \in Z \}$ is $R$-bounded over the space $\L^p(\Omega,X)$.
\end{lemma}

\begin{proof}
The first assertion is \cite[Proposition 2.1.2 p.~69]{HvNVW16}. Consider some elements $z_1,\ldots z_n$ of $Z$ and some functions $f_1,\ldots,f_n \in \L^p(\Omega,X)$. We have
\begin{align*}
\MoveEqLeft
\E \norm{ \sum_{j=1}^n \epsi_j \ot (\Id_{\L^p(\Omega)} \ot S_{z_j})(f_j) }_{\L^p(\Omega,X)}^p \\
&= \E \int_\Omega \norm{ \sum_{j=1}^n \epsi_j \ot (\Id_{\L^p(\Omega)} \ot S_{z_j})(f_j)(\omega)}_X^p \d \mu(\omega) \\
& =  \int_\Omega \E \norm{ \sum_{j=1}^n \epsi_j \ot S_{z_j}(f_j(\omega))}_X^p \d \mu(\omega) \\
&\ov{\eqref{R-boundedness}}{\lesssim} \mathscr{R}_p\left(\left\{ S_z :  z \in Z \right\} \right)^p \int_\Omega \E \norm{ \sum_{j=1}^n \epsi_j \ot f_j(\omega)}_X^p \d \mu(\omega) \\
&\ov{\eqref{norm-Bochner}}{=} \mathscr{R}_p\left(\left\{ S_z :  z \in Z \right\} \right)^p \E \norm{ \sum_{j=1}^n \epsi_j \ot f_j}_{\L^p(\Omega,X)}^p.
\end{align*}
\end{proof}

We deduce the following result on $R$-boundedness of tensor products of operators.

\begin{prop}
\label{prop-2-tensor-R-bdd}
Let $(\Omega_1,\mu_1),\ldots,(\Omega_n,\mu_n)$ be $\sigma$-finite measure spaces.
Consider the product measure space
\begin{equation*}
\label{}
(\Omega,\mu) \ov{\mathrm{def}}{=} (\Omega_1 \times \Omega_2 \times \cdots \times \Omega_n, \mu_1 \ot \mu_2 \ot \ldots \ot \mu_n).
\end{equation*} 
Let $1 \leq p < \infty$ and let $X$ be a Banach space. For $k = 1,\ldots,n$ and $z \in Z$, assume that $S^k_z \co \L^p(\Omega_k) \to \L^p(\Omega_k)$ is a bounded operator such that the family $\{S^k_z \ot \Id_X:  z \in Z \}$ is $R$-bounded over the Bochner space $\L^p(\Omega_k,X)$.
Then for any $z \in Z$ the tensor product 
\begin{equation*}
\label{}
S_z \ov{\mathrm{def}}{=} S^1_z \ot S^2_z \ot \cdots \ot S^n_z \ot \Id_X
\end{equation*} 
extends to a bounded operator on $\L^p(\Omega,X)$ and the family $\{S_z :  z \in Z \}$ is $R$-bounded over the Bochner space $\L^p(\Omega,X)$.
\end{prop}

\begin{proof}
We proceed by induction over $n$. If $n = 1$, the statement follows directly from the assumptions. Now, assume that the lemma holds for $n-1$, that is, $S_z' \ov{\mathrm{def}}{=} S_z^2 \ot \cdots \ot S_z^n \ot \Id_X$ define an $R$-bounded family over the space $\L^p(\Omega',X)$, where $\Omega' \ov{\mathrm{def}}{=} \Omega_2 \times \cdots \times \Omega_n$ is equipped with the measure $\mu' \ov{\mathrm{def}}{=} \mu_2 \ot \cdots \ot \mu_n$.
Then for any $z \in Z$ we have
\begin{equation*}
\label{}
S_z 
= (S_z^1 \ot \Id_{\L^p(\Omega')} \ot \Id_X ) \cdot (\Id_{\L^p(\Omega_1)} \ot S_z').
\end{equation*}
The family $\{ S_z^1 \ot \Id_X :  z \in Z \}$ is $R$-bounded over the space $\L^p(\Omega_1,X)$ by assumption.
Then Lemma \ref{lem-tensor-R-bdd} and Fubini's theorem yield that the set $\{S_z^1 \ot \Id_{\L^p(\Omega')} \ot \Id_X  :  z \in Z \}$ is $R$-bounded over the space $\L^p(\Omega,X)$. Moreover, according to the induction hypothesis, the family $\{ S_z' :  z \in Z \}$ is $R$-bounded over the space $\L^p(\Omega',X)$. Then again Lemma \ref{lem-tensor-R-bdd} yields that the set 
\begin{equation*}
\label{}
\{ \Id_{\L^p(\Omega_1)} \ot S_z' :  z \in Z \}
\end{equation*}
is $R$-bounded over the Bochner space $\L^p(\Omega,X)$. We conclude since the product of $R$-bounded families is again $R$-bounded by \cite[Proposition 8.1.19 p.~178]{HvNVW18}.
\end{proof}

\section{$R$-boundedness and square-max decompositions}
\label{sec-square-max}

As a result of independent interest, we establish in this section a connection between $R$-boundedness and some appropriate variants of square-max decompositions introduced in the paper \cite{GJP17}.

If $1 < p < \infty$, note that a sufficient condition for the $R$-boundedness of a family $(C_k)_{k \in \cal{K} }$ of integral operators $C_k \co \L^p(\R^d, X) \to \L^p(\R^d,X)$, $f \mapsto k*f$ on the Bochner space $\L^p(\R^d, X)$ is well-known if $X$ is a $\UMD$ \textit{Banach function space}. Indeed, if $M$ is the Hardy-Littlewood maximal operator and if
\begin{equation*}
\cal{K} 
\ov{\mathrm{def}}{=} \big\{ k \in \L^1(\R^d) : |k| * |f| \leq Mf \text{ a.e. for all simple} \, f  \co \R^d \to \C \big\}
\end{equation*}
then the family $(C_k)_{k \in \cal{K} }$ of operators is $R$-bounded on the Bochner space $\L^p(\R^d,X)$, see e.g.~\cite[Proposition 3.1 p.~371]{Lor19}
. As pointed in \cite[p.~371]{Lor19}, this class of kernels contains any radially decreasing $k \in \L^1(\R^d)$ with $\norm{k}_{\L^1(\R^d)} \leq 1$. 

Unfortunately for us, the $\UMD$ Banach space $S^p$ is not a Banach function space. Indeed, the proof of the previous result relies on the proof of an estimate of <<$\ell^\infty$-boundedness>> 
\begin{equation*}
\label{}
\norm{\sup_{1 \leq j \leq n} |C_{k_j}f_j|}_{\L^p(\R^d,X)} 
\lesssim \norm{\sup_{1 \leq j \leq n} |f_j|}_{\L^p(\R^d,X)}
\end{equation*}
for any $k_1,\ldots,k_n \in \cal{K}$, which has no meaning in the noncommutative setting. So we need to introduce and to use some kind of substitute of the supremum. This substitute will be an estimate of the kind
\begin{equation*}
\label{}
\norm{ {\sup_{z \in Z}}^+ C_{k_z} f }_{\L^p(\R^d,S^p)} 
\lesssim \norm{f}_{\L^p(\R^d,S^p)}, \quad f \in \L^p(\R^d,S^p),
\end{equation*}
where $(k_z)_{z \in Z}$ is a suitable family of kernels.


\paragraph{Warm-up: the scalar case}

Suppose that $1 < p < \infty$. Let $\Omega$ be a measure space. For a measurable function $k \co \Omega \times \Omega \to \C$ such that for any $f \in \L^p(\Omega)$ the function $y \mapsto  k(x,y)f(y)$ belongs to $\L^1(\Omega)$ for almost all $x \in \Omega$ we define the integral operator
\begin{equation}
\label{integral-operator}
(\TC_kf)(x) 
\ov{\mathrm{def}}{=} \int_\Omega k(x,y)f(y) \d y, \quad \text{for almost all }x \in \Omega.
\end{equation}
We refer to \cite[Chapter 5]{AbA02} for a nice discussion of these operators. Note that the adjoint of a well-defined bounded integral operator $\TC_k \co \L^p(\Omega) \to \L^p(\Omega)$ is again an integral operator, namely
\begin{equation}
\label{equ-adjoint-twisted-convolution-group}
(\TC_k)^* = \TC_{\check{k}} 
\quad\text{with} \quad \check{k}(x,y) 
\ov{\mathrm{def}}{=} k(y,x), \quad x,y \in \Omega.
\end{equation}

Inspired by the $\L^p$-square-max decompositions suitable for family of operators affiliated to group von Neumann algebras, as introduced in \cite[Definition 1 p.~893]{GJP17}, we propose the following variant.

\begin{defi}
\label{defi-square-max-scalar}
Let $\Omega$ be a measure space. Consider an index set $Z$. Suppose that $1 \leq p < \infty$. We say that a family $(k_z)_{z \in Z}$, where $k_z \co \Omega \times \Omega \to \mathbb{C}$, admits an $\L^p$-square-max decomposition if there exists a decomposition $k_z = s_z m_z$ with measurable functions $s_z,m_z \co \Omega \times \Omega \to \C$  
such that
\begin{equation}
\label{estimate-square}
\sup_{z \in Z} \sup_{x \in \Omega} \norm{s_z(x,\cdot)}_{\L^2(\Omega)} 
< \infty, 
\quad \text{and} \quad
\norm{ {\sup_{z \in Z}}^+ \TC_{|\check{m}_z|^2}(f) }_{\L^p(\Omega)} 
\lesssim \norm{f}_{\L^p(\Omega)},
\end{equation}
for any function $f \in \L^p(\Omega)$.
\end{defi}

\begin{example} \normalfont
\label{Ex-Gonzales}
Let $G$ be a unimodular locally compact group equipped with a Haar measure $\mu_G$. Consider an index set $Z$. We consider a family $(g_z)_{z \in Z}$ of functions $g_z \co G \to \mathbb{C}$ with factorization
$$
g_z
=S_zM_z
$$
for some measurable functions $S_z,M_z \co G \to \mathbb{C}$ with 
$$
\sup_{z \in Z} \norm{S_z}_{\L^2(G)} < \infty
$$
and 
$$
\norm{ {\sup_{z \in Z}}^+  |\check{M}_z|^2 * f  }_{\L^p(G)} 
\lesssim \norm{f}_{\L^p(G)}.
$$
These are essentially the assumptions of \cite[Definition 1. p.~893]{GJP17}. For any $s,t \in G$ and any function $g \co G \to \mathbb{C}$, we let $g^\HS(s,t) \ov{\mathrm{def}}{=} g(st^{-1})$. For any $t \in G$, we have
\begin{align*}
\MoveEqLeft
\label{}
(\TC_{g^\HS}f)(s)
\ov{\eqref{integral-operator}}{=} \int_G g^\HS(s,t)f(t) \d\mu_G(t) \\
&=\int_G g(s t^{-1}) f(t) \d\mu_G(t)
\ov{\eqref{Convolution-formulas}}{=} (g*f)(s).         
\end{align*}
For any $z \in Z$, we let $k_z(s,t) \ov{\mathrm{def}}{=} g_z^\HS$, $s_z \ov{\mathrm{def}}{=} S_z^\HS$, $m_z \ov{\mathrm{def}}{=} M_z^\HS$. Clearly, we have $k_z = s_z m_z$. 
Furthermore, we have
\begin{align*}
\MoveEqLeft
\sup_{z \in Z} \sup_{s \in G} \norm{s_z(s,\cdot)}_{\L^2(G)}
=\sup_{z \in Z} \sup_{s \in G} \norm{S_z^\HS(s,\cdot)}_{\L^2(G)} \\
&=\sup_{z \in Z} \sup_{s \in G} \norm{S_z(s\cdot^{-1})}_{\L^2(G)} 
=\sup_{z \in Z}  \norm{S_z}_{\L^2(G)} .        
\end{align*}
Hence the family $(k_z)_{z \in Z}$ admits an $\L^p$-square-max decomposition.
\end{example}

Equipped with this definition, we are now in a position to prove the following result, which establishes a new sufficient condition for $R$-boundedness. If $2 \leq p < \infty$, notice that the conjugate exponent of $\frac{p}{p-2}$ is equal to $(\frac{p}{p-2})^*=\frac{p}{2}$.

\begin{prop}
\label{prop-square-Rad-scalar}
Let $\Omega$ be a measure space. Consider an index set $Z$. Suppose that $2 \leq p < \infty$. For any $z \in Z$, consider a bounded integral operator $\TC_{k_z} \co \L^p(\Omega) \to \L^p(\Omega)$. Suppose that the family $(k_z)_{z \in Z}$ has an $\L^{\frac{p}{p-2}}$-square-max decomposition. Then the family $(C_{k_z})_{z \in Z}$ is $R$-bounded.
\end{prop}

\begin{proof}
Consider an integer $n \geq 1$. Let $f_1,\ldots,f_n$ be some functions belonging to the space $\L^p(\Omega)$ and let $z_1,\ldots,z_n$ be some elements of $Z$. We have
\begin{equation}
\label{First-equation-Fourier}
\norm{\bigg(\sum_{j=1}^{n} |\TC_{k_{z_j}}(f_j)|^2\bigg)^{\frac{1}{2}}}_{\L^p(\Omega)}
=\norm{\sum_{j=1}^{n} |\TC_{k_{z_j}}(f_j)|^2}_{\L^{\frac{p}{2}}(\Omega)}^{\frac{1}{2}}.
\end{equation} 
By duality, there exists a unique positive element $u \in {\L^{\frac{p}{p-2}}(\Omega)}$ of norm 1 such that
\begin{equation}
\label{Eq-divers-456-Fourier}
\Bgnorm{\sum_{j=1}^{n} |\TC_{k_{z_j}}(f_j)|^2}_{\L^{\frac{p}{2}}(\Omega)}
=\int_\Omega u\sum_{j=1}^{n} |\TC_{k_{z_j}}(f_j)|^2.
\end{equation}
Now, we have to estimate the term inside the sum. 
Using Cauchy-Schwarz inequality
\begin{align}
\MoveEqLeft
\label{div-23467}
|\TC_{k_{z_j}}(f_j)|^2          
\ov{\eqref{integral-operator}}{=} \left|\int_\Omega k_{z_j}(x,y) f_j(y) \d y\right|^2 \\
&=\left|\int_\Omega s_{z_j}(x,y)m_{z_j}(x,y) f_j(y) \d y\right|^2 \nonumber\\
&\leq \int_\Omega \left|s_{z_j}(x,y)\right|^2 \d \xi \int_\Omega \left|m_{z_j}(x,y) f_j(y)\right|^2 \d \xi \nonumber\\
&\ov{\eqref{integral-operator}}{\leq} \sup_{z \in Z} \sup_{x \in \Omega} \norm{s_z(x,\cdot)}_{\L^2(\Omega)}^2 \TC_{|m_{z_j}|^2} (|f_j|^2) \nonumber\\
&\ov{\eqref{estimate-square}}{\lesssim} \TC_{|m_{z_j}|^2} (|f_j|^2).\nonumber
\end{align}
Since $u \in \L^{\frac{p}{p-2}}(\Omega)$ and since $|f_j|^2 \in \L^{\frac{p}{2}}(\Omega)$, we have 
\begin{align*}
\MoveEqLeft
\Bgnorm{\bigg(\sum_{j=1}^n |\TC_{k_{z_j}}(f_j)|^2 \bigg)^{\frac{1}{2}}}_{\L^p(\Omega)}^2
\ov{\eqref{First-equation-Fourier}\eqref{Eq-divers-456-Fourier}}{=} \sum_{j=1}^{n} \int_\Omega u\, |\TC_{k_{z_j}}(f_j)|^2\\
&\ov{\eqref{div-23467}}{\leq} \sum_{j=1}^{n} \int_\Omega u\big( \TC_{|m_{z_j}|^2} (|f_j|^2) \big) \\ 
&=\sum_{j=1}^{n} \int_\Omega \TC_{|\check{m_{z_j}}|^2}(u) |f_j|^2 
\leq  \inf_{\TC_{|\check{m}_z|^2}(u) \leq h} \int_\Omega h\sum_{j=1}^{n} |f_j|^2\\
&\leq  \bigg(\inf_{\TC_{|\check{m}_{z}|^2}(u) \leq h}  \norm{h}_{\L^{\frac{p}{p-2}}(\Omega)} \bigg) \Bgnorm{\sum_{j=1}^{n} |f_j|^2 }_{\L^{\frac{p}{2}}(\Omega)} \\
&\ov{\eqref{Norm-Lp-Linfty-Omega}}{=} \norm{{\sup_{z \in Z}}^+ \TC_{|\check{m}_{z}|^2}(u)}_{\L^{\frac{p}{p-2}}(\Omega)} \Bgnorm{\sum_{j=1}^{n} |f_j|^2}_{\L^{\frac{p}{2}}(\Omega)}
\\
&\ov{\eqref{estimate-square}}{\lesssim} \norm{u}_{\L^{\frac{p}{p-2}}(\Omega)} \Bgnorm{\sum_{j=1}^{n} |f_j|^2}_{\L^{\frac{p}{2}}(\Omega)}
=\Bgnorm{\sum_{j=1}^{n} |f_j|^2}_{\L^{\frac{p}{2}}(\Omega)}.
\end{align*}
%
%
Taking the square roots, we conclude that
\begin{equation*}
\label{}
\norm{\left(\sum_{j=1}^n |\TC_{k_{z_j}}(f_j)|^2 \right)^{\frac12}}_{\L^p(\Omega)} 
\lesssim \norm{\left(\sum_{j=1}^n |f_j|^2\right)^{\frac12}}_{\L^p(\Omega)}.
\end{equation*}
So with \eqref{equiv-lattice}, we conclude that the family $(\TC_{k_{z}})_{z \in Z}$ of operators is $R$-bounded. 
\end{proof}

\paragraph{The vector-valued case}


The following is a variant of Definition \ref{defi-square-max-scalar} for Bochner spaces with values in the Schatten space $S^p$.

\begin{defi}
\label{defi-square-max-Sp}
Suppose that $1 < p < \infty$.
Consider an index set $Z$ and a measure space $\Omega$.
We say that a family $(k_z)_{z \in Z}$, where $k_z \co \Omega \times \Omega \to \C$ admits an $\L^p(S^p)$-square-max decomposition if there exists a decomposition $k_z = s_z m_z$ such that
\begin{equation}
\label{estimate-square-Sp}
\sup_{z \in Z}\sup_{x \in \Omega} \norm{s_z(x,\cdot)}_{\L^2(\Omega)} 
< \infty
\end{equation}
and
\begin{equation}
\label{estimate-max-Sp}
\norm{{\sup_{z \in Z}}^+ (\TC_{|m_z|^2} )^*(u)}_{\L^p(\Omega,S^p)} \ov{\eqref{equ-adjoint-twisted-convolution-group}}{=} \norm{{\sup_{z \in Z}}^+ \TC_{|\check{m}_z|^2}(u)}_{\L^p(\Omega,S^p)} \lesssim \norm{u}_{\L^p(\Omega,S^p)}.
\end{equation}
\end{defi}

By replacing the Cauchy-Schwarz inequality in the proof of Proposition \ref{prop-square-Rad-scalar} by the operator-valued inner product \cite[Proposition 1.1 p.~3]{Lan95}, we can prove the next result Theorem \ref{thm-square-Rad-vectorial}. For the proof, it will be necessary to consider two naturals variants of this notion, introduced in \cite[Chapter 4]{JMX06}. 

Suppose that $1 < p<\infty$. Consider a von Neumann algebra $\cal{M}$ endowed with a normal semifinite faithful trace. A subset $\mathcal{F}$ of the space $\B(\L^p(\cal{M}))$ is Col-bounded (resp. Row-bounded) if there exists a constant $C\geq 0$ such that for any finite families $T_1,\ldots,T_n$ in $\mathcal{F}$, and any vectors $x_1,\ldots,x_n$ in $\L^p(\cal{M})$,
we have
\begin{equation}
\label{colbound}
\Bgnorm{\bigg(\sum_{k=1}^{n}\big|T_k(x_k)\big|^2\bigg)^{\frac{1}{2}}}_{\L^p(\cal{M})}
\leq C\Bgnorm{\bigg(\sum_{k=1}^{n}|x_k|^2\bigg)^{\frac{1}{2}}}_{\L^p(\cal{M})}
\end{equation}

\begin{equation}
\label{rowbound} \Bigg( \text{resp.
}\Bgnorm{\bigg(\sum_{k=1}^{n}\big|T_k(x_k)^*\big|^2\bigg)^{\frac{1}{2}}}_{\L^p(\cal{M})}
\leq C\Bgnorm{\bigg(\sum_{k=1}^{n}|x_k^*|^2\bigg)^{\frac{1}{2}}}_{\L^p(\cal{M})}\Bigg).
\end{equation}
Obviously any Col-bounded set or Row-bounded set is bounded. It is a consequence of noncommutative Khintchine inequalities that if a subset $\mathcal{F}$ of $\B(\L^p(\cal{M}))$ is both Col-bounded and Row-bounded, then it is necessarily $R$-bounded.

Observe that, unlike the case of $R$-boundedness, a singleton set ${T}$ need not necessarily be Col-bounded or Row-bounded, as noted in \cite[Section 4.A]{JMX06}. Additionally, there exist subsets $\mathcal{F}$ that are both $R$-bounded and Col-bounded, yet fail to be Row-bounded. Finally, there also exist subsets that are $R$-bounded and Row-bounded but not Col-bounded, as well as sets that are $R$-bounded but neither Row-bounded nor Col-bounded.


Recall that $(\frac{p}{p-2})^*=\frac{p}{2}$ if $2 \leq p < \infty$.

\begin{thm}
\label{thm-square-Rad-vectorial}
Let $Z$ be an index set. Let $\Omega$ be a measure space. Suppose that $2 \leq p <\infty$. For any $z \in Z$, consider a bounded integral operator $\TC_{k_z} \co \L^p(\Omega) \to \L^p(\Omega)$. Suppose that the family $(k_z)_{z \in Z}$ has an $\L^{\frac{p}{p-2}}(S^{\frac{p}{p-2}})$-square-max decomposition. Then the family $(\TC_{k_z} \ot \Id_{S^p})_{z \in Z}$ is Col-bounded, Row-bounded and $R$-bounded on the Bochner space $\L^p(\Omega,S^p)$.
\end{thm}

\begin{proof}
Let $N \in \N$ and $f_1,\ldots,f_N \in \L^p(\Omega,S^p)$.
Consider some $z_1,\ldots,z_N \in Z$ and write in short $\TC_l \ov{\mathrm{def}}{=} \TC_{z_l}$, $k_{l} \ov{\mathrm{def}}{=} k_{z_l}$, $s_{l} \ov{\mathrm{def}}{=} s_{z_l}$, $m_{l} \ov{\mathrm{def}}{=} m_{z_l}$, where we denote the $\L^{(\frac{p}{2})^*}(S^{(\frac{p}{2})^*})$-square-max decomposition $k_z = s_z m_z$.
So in order to prove Col-boundedness, we are left to show that for some uniform constant $C$, we have
\begin{align}
\norm{\bigg(\sum_l |(\TC_l \ot \Id_{S^p})(f_l)|^2 \bigg)^{\frac12}}_{\L^p(\Omega,S^p)}
& \leq C \norm{\bigg(\sum_l |f_l|^2\bigg)^{\frac12}}_{\L^p(\Omega,S^p)}  \label{equ-3-proof-prop-cbdd-Harris-group}
\end{align}

We have
\begin{equation}
\label{Divers-678-group}
\norm{\bigg(\sum_l |(\TC_{l} \ot \Id_{S^p})(f_l)|^2\bigg)^{\frac12}}_{\L^p(\Omega,S^p)}^2 
=\norm{\sum_l |(\TC_{l} \ot \Id_{S^p})(f_l)|^2}_{\L^{\frac{p}{2}}(\Omega,S^{\frac{p}{2}})}. 
\end{equation}
By duality, there exists a unique positive element $u \in \L^{(\frac{p}{2})^*}(\Omega,S^{(\frac{p}{2})^*})$ of norm one such that
\begin{equation}
\label{Avec-u-group}
\norm{\sum_l |(\TC_{l} \ot \Id_{S^p})(f_l)|^2}_{\L^{\frac{p}{2}}(\Omega,S^{\frac{p}{2}})} 
=\bigg(\int_{\Omega} \ot \tr\bigg) \bigg( \sum_l |(\TC_l \ot \Id_{S^p})(f_l)|^2 u \bigg). 
\end{equation}
We shall give an upper bound of $|(\TC_{l} \ot \Id_{S^p})(f_l)|^2$ in the ordered space $\L^{\frac{p}{2}}(\Omega,S^{\frac{p}{2}})$. Note that by a standard approximation argument it suffices to consider the case $f_l \in \L^p(\Omega,S^p_m)$ where $S^p_m$ is a finite-dimensional Schatten space if we obtain a bound in \eqref{equ-3-proof-prop-cbdd-Harris-group} that is independent of the size $m \in \N$.
We obtain
\begin{align*}
\MoveEqLeft
(\TC_{l} \ot \Id_{S^p_m})(f_l) 
= \sum_{i,j=1}^m \TC_l (f_{l,ij}) \ot e_{ij} 
\ov{\eqref{integral-operator}}{=} \sum_{i,j=1}^m \int_{\Omega} k_l(\cdot,y) f_{l,ij}(y) \d y \ot e_{ij} \\
& \ov{\mathrm{Def. }\ref{defi-square-max-Sp}}{=} \sum_{i,j=1}^m \int_{\Omega} s_{l}(\cdot,y)m_{l}(\cdot,y) f_{l,ij}(y) \d y \ot e_{ij}\\
&= \bigg(\int_{\Omega} \ot \Id_{\L^\infty(\Omega)} \ot \Id_{S^\infty_m} \bigg)\\
&\left((s_l(\cdot,y) \ot 1_{S^\infty_m})(m_{l}(\cdot,y) \ot 1_{S^\infty_m})\bigg(\sum_{i,j} f_{l,ij} \ot 1_{\L^\infty(\Omega)} \ot e_{ij} \bigg)\right).
\end{align*}
Now, we can consider the operator-valued inner product 
\begin{equation*}
\label{}
\langle X , Y \rangle \ov{\mathrm{def}}{=}  (\Id_{\L^\infty(\Omega)} \ot \int_{\Omega} \ot \Id_{S^\infty_m})(X^*Y)
\end{equation*}
on $\L^2(\Omega,\L^\infty(\Omega,S^\infty_m))$. It satisfies according to \cite[Proposition 1.1 p.~3]{Lan95} 
\begin{equation}
\label{weak-Cauchy-schwarz-bis-group}
|\langle X , Y \rangle|^2 
\leq \norm{ \langle X , X \rangle}_{\L^\infty(\Omega,S^\infty_m)} \langle Y , Y \rangle.
\end{equation}
So with $X \co y \mapsto \ovl{s_{l}}(\cdot,y) \ot 1_{S^\infty_m}$ and $Y \co y \mapsto (m_{l}(\cdot,y) \ot 1_{S^\infty_m})\big( \sum_{i,j} f_{l,ij} \ot 1_{\L^\infty(\Omega)} \ot e_{ij} \big)$, we obtain
\begin{align*}
\MoveEqLeft
|(\TC_l \ot \Id_{S^p_m})(f_l)|^2 
= | \langle X,Y \rangle|^2 
\ov{\eqref{weak-Cauchy-schwarz-bis-group}}{\leq} \norm{\langle X,X \rangle}_{\L^\infty(\Omega,S^\infty_m)} \langle Y, Y \rangle \\
&= \bigg(\sup_{x \in \Omega} \int_{\Omega} |s_{l}(x,y)|^2 \d y \bigg) \times\bigg(\Id_{\L^\infty(\Omega)} \ot \int_{\Omega} \ot \Id_{S^\infty_m}\bigg)\\
& \qquad \left( \left| (m_{l}(\cdot,y) \ot 1_{S^\infty_m})\bigg( \sum_{i,j}  f_{l,ij} \ot 1_{\L^\infty(\Omega)} \ot e_{ij}\bigg) \right|^2 \right).
\end{align*}
We estimate the first factor by 
\begin{equation}
\label{first-group}
\sup_{x \in \Omega} \int_{\Omega} |s_{l}(x,y)|^2 \d y \ov{\eqref{estimate-square-Sp}}{\leq} C. 
\end{equation}
For the second factor, we see that
\begin{align}
\MoveEqLeft
\label{second-group}
\bigg(\Id_{\L^\infty(\Omega)} \ot \int_{\Omega} \ot \Id_{S^\infty_m}\bigg)\left( \left| (m_{l}(\cdot,y) \ot 1_{S^\infty_m}) \bigg(\sum_{i,j} f_{l,ij} \ot 1_{\L^\infty} \ot e_{ij}\bigg) \right|^2 \right) \\
&= \int_{\Omega} |m_{l}(\cdot,y)|^2 \left|\sum_{i,j} f_{l,ij}(y) \ot e_{ij}\right|^2 \d y \nonumber
= \int_{\Omega} |m_l(\cdot,y)|^2 \, |f_l(y)|^2 \d y \\
&\ov{\eqref{integral-operator}}{=} \TC_{|m_l|^2} \ot \Id_{S^{p/2}_m}( |f_l|^2). \nonumber
\end{align}
Thus we can estimate the expression at the beginning of the proof by
\begin{align*}
\MoveEqLeft
\norm{ \bigg(\sum_{l} |(\TC_l \ot \Id_{S^p_m}) (f_l)|^2 \bigg)^{\frac12} }_{\L^p(\Omega,S^p)}^2 
\ov{\eqref{Divers-678-group}\eqref{Avec-u-group}}{=} \bigg(\int_{\Omega} \ot \tr\bigg) \left( \sum_l |( \TC_l \ot \Id_{S^p_m})(f_l)|^2 u \right) \\
&\ov{\eqref{first-group}\eqref{second-group}}{\leq} C \bigg(\int_{\Omega} \ot \tr\bigg) \bigg( \sum_l \TC_{|m_l|^2 } \ot \Id_{S^{p/2}_m}(|f_l|^2) u \bigg)  \\
& \ov{\eqref{equ-adjoint-twisted-convolution-group}}{=} C \bigg(\int_{\Omega} \ot \tr\bigg) \bigg( \sum_l |f_l|^2 (\TC_{|\check{m_l}|^2} \ot \Id_{S^{(p/2)^*}_m}(u)) \bigg) \\
& \leq C \inf \left\{ \bigg(\int_{\Omega} \ot \tr\bigg) \bigg( \sum_l |f_l|^2 B \bigg) : \: B \geq 0,\right.\\
&\left. \TC_{|\check{m_l}|^2} \ot \Id_{S^{(p/2)^*}_m}(u) \leq B,\: l = 1,\ldots,N \right\}.
\end{align*}
According to \eqref{Norm-Lp-Linfty-Omega} and \eqref{estimate-max-Sp}, we can find such $B \geq 0$ as in the previous infimum with the property $\norm{B}_{(\frac{p}{2})^*} \leq C \norm{u}_{(\frac{p}{2})^*} = C$ (up to $\epsi > 0$). We infer that
\begin{align*}
\MoveEqLeft
\norm{ \left( \sum_l |(\TC_l \ot \Id_{S^p})(f_l)|^2 \right)^{\frac12} }_{\L^p(\Omega,S^p)}^2 
 \lesssim \bigg(\int_{\Omega} \ot \tr\bigg) \bigg( \sum_l |f_l|^2 B \bigg) \\
&\lesssim \norm{ \sum_l |f_l|^2 }_{\frac{p}{2}} = \norm{\bigg( \sum_l |f_l|^2 \bigg)^{\frac12} }_{\L^p(\Omega,S^p)}^2 .
\end{align*}
Thus \eqref{equ-3-proof-prop-cbdd-Harris-group} is proved with constant 
\[
\sup_{z \in Z} \sup_{x \in \Omega} \norm{s_z(x,\cdot)}_{\L^2(\Omega)} \cdot \norm{{\sup_{z \in Z}}^+ \TC_{|\check{m}_z|^2}(\cdot)}_{\L^{(p/2)^*}(\Omega,S^{(p/2)^*}) \to \L^{(p/2)^*}(\Omega,S^{(p/2)^*})}^{\frac12}. 
\]

Now, we deduce Row-boundedness of $(\TC_{k_z} \ot \Id_{S^p})_{z \in Z}$ from Col-boundedness of that family. In order to obtain Row-boundedness, we have to show that 
\begin{equation}
\norm{\bigg(\sum_l |[(\TC_l \ot \Id_{S^p})(f_l)]^*|^2 \bigg)^{\frac12}}_{\L^p(\Omega,S^p)}
\leq C \norm{\bigg(\sum_l |f_l^*|^2\bigg)^{\frac12}}_{\L^p(\Omega,S^p)}  \label{equ-4-proof-prop-cbdd-Harris}.
\end{equation}
Note that $[\TC_l \ot \Id_{S^p}(f_l)]^* = \TC_{\overline{k_l}} \ot \Id_{S^p}(f_l^*)$.
Thus Row-boundedness follows from Col-boundedness once we know that the family $(\overline{k_z})_{z \in Z}$ has also an $\L^{\left(\frac{p}{2}\right)^*}(S^{\left(\frac{p}{2}\right)^*})$-square-max decomposition.
But this is immediate from $\overline{k_z} = \overline{s_z} \overline{m_z}$ and the identites 
\begin{equation*}
\label{}
\norm{\overline{s_z(x,\cdot)}}_{\L^2(\Omega)} = \norm{s_z(x,\cdot)}_{\L^2(\Omega)}
\end{equation*}
and $|\overline{m_z}|^2  = |m_z|^2$, together with \eqref{estimate-square-Sp} and \eqref{estimate-max-Sp}.

Finally, note that Row-boundedness in conjunction with Col-boundedness implies $R$-boundedness according to \cite[(2.21) and p.~31]{JMX06}.
\end{proof}

\chapter{Functional calculus}
\section{Functional calculus of the harmonic oscillator of the universal $\Theta$-Weyl tuple}
\label{subsec-R-sectoriality}

In this section, we apply the maximal estimate from Lemma \ref{lem-nc-maximal-inequality-convolution} of Section \ref{subsec-maximal-estimates} to show the $R$-boundedness of the complex time semigroup $(T_{z})_{z \in \Sigma_{\varphi}}$ (in a controlled way as the angle $\varphi$ increases to $\frac{\pi}{2}$) of the operator $\A - \alpha$ on the Bochner space $\L^p(\R^n,S^p)$. The $R$-boundedness in such a noncommutative $\L^p$-space can be expressed, thanks to (noncommutative) Khintchine inequalities, by square function estimates.


\begin{prop}
\label{prop-existence-square-max}
For any complex number $z$ such that $\Re z > 0$, consider the function $G_z \co \R^2 \times \R^2 \to \R_+$, $(x,y) \mapsto g_z(x-y)$ where $g_z$ is defined in Proposition \ref{prop-kernel-estimate}. Suppose that $1 < p < \infty$.  Then the family 
\begin{equation*}
\label{}
\big(\cos(\arg z) e^{\Re z}G_z\big)_{\Re z >0}
\end{equation*}
admits an $\L^{p}$-square-max decomposition.
\end{prop}

\begin{proof}
If $\Re z>0$ then with the functions $s_z \ov{\mathrm{def}}{=} m_z \ov{\mathrm{def}}{=} \sqrt{\cos(\arg z) e^{\Re z}g_z}$ we can write
 $
\cos(\arg z) e^{\Re z}g_z
=s_z m_z$.
Now, decompose 
\begin{equation*}
\label{}
\cos(\arg z) e^{\Re z}G_z(x,y) 
= S_z(x,y) M_z(x,y),
\end{equation*}
with $S_z(x,y) \ov{\mathrm{def}}{=} s_z(x-y)$ and $M_z(x,y) \ov{\mathrm{def}}{=} m_z(x-y)$.
Now, we check that this decomposition is an $\L^p$-square-max decomposition. We have
\begin{align*}
\MoveEqLeft
\sup_{\Re z >0} \sup_{x \in \R^2} \norm{S_z(x,\cdot)}_{\L^2(\R^2)}
=\sup_{\Re z >0} \sup_{x \in \R^2} \norm{s_z(x-\cdot)}_{\L^2(\R^2)} \\
&=\sup_{\Re z >0}  \norm{s_z}_{\L^2(\R^2)}
= \sup_{\Re z > 0} \norm{s_z^2}_{\L^1(\R^2)}^{\frac12}
= \sup_{\Re z > 0} \norm{\cos(\arg z) e^{\Re z} g_z}_{\L^1(\R^2)}^{\frac12}
\ov{\eqref{estim-sup}}{<} \infty .
\end{align*}
Recall that the function $g_z \co \R^2 \to \R_+$ is even. For any function $u \in \L^p(\R^2)$, Lemma \ref{lem-nc-maximal-inequality-convolution} combined with \eqref{estim-sup} give the estimate 
\begin{align*}
\MoveEqLeft
\norm{  {\sup_{\Re z >0}}^+ (\TC_{|M_z|^2} )^*(u)}_{\L^p(\R^2)}
=\norm{ {\sup}^+_{\Re z >0}( \cos(\arg z) e^{\Re z} g_{z}) \ast u }_{\L^p(\R^2)} \\
&\lesssim \sup_{\Re z > 0} \cos(\arg z) e^{\Re z} \norm{g_z}_{\L^1(\R^2)}\norm{u}_{\L^p(\R^2)} \ov{\eqref{estim-sup}}{\lesssim} \norm{u}_{\L^p(\R^2)}.
\end{align*}
\end{proof}

Now, we can obtain an $R$-bound for the family $(T_z)_{\Re z >0}$ of operators from Proposition \ref{prop-semigroup} using a standard domination trick in the spirit of \cite[Example 2.10 p.~89]{KuW04} and \cite[Proposition 8.1.10 p.~171]{HvNVW18}. Recall that $\cal{J} \ov{\eqref{def-cal-J}}{=} \begin{pmatrix} 0 & 1 \\ -1 & 0 \end{pmatrix}$.

\begin{prop}
\label{prop-R-sectorial-bis}
Suppose that $1 < p<  \infty$. Let $A_\univ^\cal{J}$ be the universal $\cal{J}$-Weyl tuple on the Banach space $\L^p(\R^2)$. The family
\[ 
\left\{ \cos(\arg z) e^{\Re z} T_z \, : \, \Re z >0 \right\} 
\]
is $R$-bounded over the Banach space $\L^p(\R^2)$.
\end{prop}


\begin{proof}
First note that since $R$-boundedness is preserved under taking adjoints by Proposition \ref{prop-R-bounded-dual} and due to self-adjointness of the operator $\A$, it suffices to prove the statement for $p \geq 2$. Moreover, since for $p = 2$, $R$-boundedness is just boundedness, essentially proved in Proposition \ref{prop-kernel-estimate}, we can restrict to the case $p > 2$. Consider an integer $n \geq 1$ and $f_1,\ldots,f_n \in \L^p(\R^2)$. Suppose that the complex numbers $z_1,\ldots,z_n$ satisfy $\Re z_1>0, \ldots,\Re z_n>0$. Using Proposition \ref{prop-square-Rad-scalar} combined with Proposition \ref{prop-existence-square-max} in the last inequality, we obtain that
\begin{align*}
\MoveEqLeft
\norm{\left(\sum_{j=1}^n \big|\cos(\arg z_j) e^{\Re z_j}T_{z_j}(f_j) \big|^2 \right)^{\frac12}}_{\L^p(\R^2)} \\      
& \ov{\eqref{Autre-rep-Tz}}{=} \norm{\left(\sum_{j=1}^n \left|\int_{\R^{2}} \cos(\arg z_j) e^{\Re z_j}k_{z_j}^\cal{J}(\cdot,y)f_j(\cdot-y)\d y\right|^2 \right)^{\frac12}}_{\L^p(\R^2)}    \\
&\leq \norm{\left(\sum_{j=1}^n \bigg(\int_{\R^{2}} \cos(\arg z_j) e^{\Re z_j}\big|k_{z_j}^\cal{J}(\cdot,y)\big| \, |f_j(\cdot-y)|\d y\bigg)^2 \right)^{\frac12}}_{\L^p(\R^2)}\\
&\ov{\eqref{majo-by-gz}}{\leq} \norm{\left(\sum_{j=1}^n \bigg(\int_{\R^{2}} \cos(\arg z_j) e^{\Re z_j} g_{z_j}(y) |f_j(\cdot-y)|\d y\bigg)^2 \right)^{\frac12}}_{\L^p(\R^2)} \\
& \ov{\eqref{integral-operator}}{=} \norm{\left(\sum_{j=1}^n \Big[\TC_{\cos(\arg z) e^{\Re z}G_{z_j}}(|f_j|)\Big]^2 \right)^{\frac12}}_{\L^p(\R^2)} 
\lesssim \norm{\left(\sum_{j=1}^n |f_j|^2 \right)^{\frac12}}_{\L^p(\R^2)}.
\end{align*}
\end{proof}

For an arbitrary matrix $\Theta$, we could obtain from this result an $R$-bound. However, we do not prove it here since we will obtain in Theorem \ref{thm-R-sectorial} a more general vector-valued result.

%

\begin{prop}
\label{prop-LpSp-square-max-cal-J}
Suppose that $1 < p < \infty$.
\begin{enumerate}
\item For any complex number $z$ such that $\Re z > 0$, let $G_z \co \R^2 \times \R^2 \to \R_+$, $(x,y) \mapsto g_z(x-y)$ where $g_z$ is defined in Proposition \ref{prop-kernel-estimate}. 
Then the family 
\begin{equation*}
\label{}
\big(\cos(\arg z) e^{\Re z}G_z\big)_{\Re z >0}
\end{equation*} 
admits an $\L^{p}(S^p)$-square-max decomposition.

\item Let $K_z^{\cal{J}} : \R^2 \times \R^2 \to \R_+$ be the integral kernel such that the semigroup of the universal $\cal{J}$-Weyl tuple $A_{\univ}^{\cal{J}}$ on the Banach space $\L^p(\R^2)$ is given by $\TC_{K_z^{\cal{J}}}$. Then the family 
\begin{equation*}
\label{}
\big(\cos(\arg z) e^{\Re z}K_z^{\cal{J}}\big)_{\Re z >0}
\end{equation*}
admits an $\L^{p}(S^p)$-square-max decomposition.

\item In the same manner, if $\alpha > 0$ and if $K_z^{\alpha \cal{J}}$ is the integral kernel associated with the semigroup of the universal $\alpha \cal{J}$-Weyl tuple $A_{\univ}^{\alpha\cal{J}}$ on the space $\L^p(\R^2)$, then the family 
\begin{equation*}
\label{}
\big(\cos(\arg z) e^{\alpha \Re z}K_z^{\alpha \cal{J}}\big)_{\Re z >0}
\end{equation*} 
admits an $\L^{p}(S^p)$-square-max decomposition.

\item Let $p_z$ be the standard heat kernel on $\R$. Then the family
\begin{equation*}
\label{}
\big((x,y) \mapsto (\cos \arg z)^{\frac12} p_z(x-y)\big)_{\Re z > 0}
\end{equation*} 
admits an $\L^p(S^p)$-square-max decomposition.
\end{enumerate}
\end{prop}

\begin{proof}
1. For any $x,y \in \R^2$, we write 
\begin{equation}
\label{Decompose}
\cos(\arg z) e^{\Re z}G_z(x,y) 
= S_z(x,y) M_z(x,y)
\end{equation}
as in the proof of Proposition \ref{prop-existence-square-max}. In particular 
\begin{equation*}
\label{}
M_z(x,y)
=\sqrt{\cos(\arg z) e^{\Re z}g_z(x-y)}.
\end{equation*}
Then this proof actually does not only show an $\L^p$-square-max decomposition, but even an $\L^p(S^p)$-square-max decomposition using Lemma \ref{lem-nc-maximal-inequality-convolution}.

2. According to \eqref{kzxy} and \eqref{majo-by-gz}, we have for all $x,y \in \R^2$ and $\Re z > 0$, the inequality 
\begin{equation}
\label{mystere}
|K_z^{\cal{J}}(x,y)|
\ov{\eqref{Autre-rep-Tz}}{=} |p_z^{\cal{J}}(x-y)|
\ov{\eqref{majo-by-gz}}{\leq} 
g_z(x-y)
=G_z(x,y).
\end{equation}
Now, we decompose
\begin{equation}
\label{decompo-ultime}
\cos(\arg z)e^{\Re z}K_z^{\cal{J}} 
=  S_z' M_z
\end{equation}
with $M_z$ as in the first part of the proof and $S_z' \ov{\mathrm{def}}{=} \frac{1}{M_z}(\cos \arg z ) e^{\Re z}|K_z^{\mathcal{J}}|$ (note that $g_z > 0$ and hence also $M_z > 0$ everywhere). Then we obtain the estimate
\begin{equation*}
\label{}
|S_z'(x,y)| 
\leq \frac{(\cos \arg z) e^{\Re z} G_z(x,y)}{M_z(x,y)} 
\ov{\eqref{Decompose}}{=} S_z(x,y).
\end{equation*} 
Hence the first part shows that $K_z^{\cal{J}}$ admits an $\L^p(S^p)$-square-max decomposition.

3. 
Since we have the decomposition \eqref{decompo-ultime}, we obtain here (note that $\arg(z) = \arg(\alpha z)$)
\begin{align}
\MoveEqLeft
\label{sans-fin-4}
\cos(\arg z) e^{\alpha \Re z} K_z^{\alpha \cal{J}}(x,y) 
=\cos(\arg z) e^{\alpha \Re z} e^{-\frac12 \i \langle y , \alpha \cal{J} x \rangle} p_z^{\alpha \cal{J}}(x-y)\\
&\ov{\eqref{rescaled-3}}{=} 
 e^{-\frac12 \i \langle y , \alpha \cal{J} x \rangle} \alpha\cos(\arg z) e^{\Re (\alpha z)} p_{\alpha z}^{\cal{J}}(\sqrt{\alpha} (x-y))\nonumber \\
&\ov{\eqref{kzxy}}{=} e^{\i \phi(x,y)} \alpha\cos(\arg z) e^{\Re (\alpha z)} K_{\alpha z}^{\cal{J}}(\sqrt{\alpha} x,\sqrt{\alpha} y) \nonumber \\
&\ov{\eqref{decompo-ultime}}{=} e^{\i \phi(x,y)} \alpha S_{\alpha z}'(\sqrt{\alpha}x,\sqrt{\alpha}y) M_{\alpha z}(\sqrt{\alpha}x,\sqrt{\alpha}y) 
\ov{\mathrm{def}}{=} S_z^\alpha(x,y) M_z^\alpha(x,y) , \nonumber
\end{align}
for some measurable phase function $\phi$ 
with
\[ 
S_z^\alpha(x,y) 
\ov{\mathrm{def}}{=} e^{ \i \phi(x,y)}\sqrt{\alpha} S_{\alpha z}'(\sqrt{\alpha}x,\sqrt{\alpha}y)
\quad \text{and} \quad
M_z^\alpha(x,y) 
\ov{\mathrm{def}}{=} \sqrt{\alpha} M_{\alpha z}(\sqrt{\alpha}x,\sqrt{\alpha}y) . 
\]

Now, we show that this is an $\L^p(S^p)$-square-max decomposition. Indeed, we have 
\[ 
\norm{S_z^\alpha(x,\cdot)}_{\L^2(\R^2)} 
= \norm{\sqrt{\alpha} S_{\alpha z}'(\sqrt{\alpha}x,\sqrt{\alpha}\,\cdot)}_{\L^2(\R^2)} 
\approx \norm{S_{\alpha z}'(\sqrt{\alpha}x,\cdot)}_{\L^2(\R^2)}
 \leq C.
\]
Moreover, it is not difficult to see that we have the factorization 
\begin{equation*}
\label{}
\TC_{|M_z^\alpha|^2} 
= D_{\sqrt{\alpha}} \TC_{|M_{\alpha z}|^2} D_{\sqrt{\alpha}^{-1}},
\end{equation*}
where $D_\beta \co \L^p(\R^2,S^p) \to \L^p(\R^2,S^p)$, $u \mapsto u(\beta\, \cdot)$ is a positive map such that 
\begin{equation*}
\label{}
\|D_\beta u\|_{\L^p(\R^2,S^p)}
= \beta^{-\frac{2}{p}} \norm{u}_{\L^p(\R^2,S^p)}.
\end{equation*}
Suppose that $u$ is a positive element of $\L^p(\R^2,S^p)$.
Then $D_{\sqrt{\alpha}^{-1}}u$ is again a positive element of $\L^p(\R^2,S^p)$.
By the maximal estimate from the first part, there exists some positive element $B$ of $\L^p(\R^2,S^p)$ such that $\TC_{|M_{\alpha z}|^2}D_{\sqrt{\alpha}^{-1}}u \leq B$ for all $\Re z > 0$ and in addition $\norm{B}_p \leq C \bnorm{D_{\sqrt{\alpha}^{-1}}u}_p$.
Thus, we have 
\begin{equation*}
\label{}
\TC_{|M_z^\alpha|^2}u =  D_{\sqrt{\alpha}} \TC_{|M_{\alpha z}|^2} D_{\sqrt{\alpha}^{-1}}u \leq D_{\sqrt{\alpha}}B
\end{equation*}
and
\begin{equation*}
\label{}
\norm{D_{\sqrt{\alpha}}B}_p 
= \alpha^{-\frac1p} \norm{B}_p 
\leq C \alpha^{-\frac1p}\bnorm{D_{\sqrt{\alpha}^{-1}}u}_p 
= C \norm{u}_p.
\end{equation*}
Consequently, we obtain
\[ 
\norm{{\sup_{\Re z > 0}}^+ \TC_{|M_z^\alpha|^2} u}_{\L^p(\R^2,S^p)}
\lesssim \norm{u}_{\L^p(\R^2,S^p)} .
\]

4. Using the formula $\Re(\frac{1}{z})=\frac{\cos \arg z}{|z|}$, we can write for $x,y \in \R$
\begin{align*}
\MoveEqLeft
\cos(\arg z)^{\frac12} p_z(x-y) 
\ov{\eqref{heat-kernel}}{=} e^{\i \phi(x,y,z)} \cos(\arg z)^{\frac12} \frac{1}{\sqrt{4 \pi |z|}} e^{-\cos(\arg z)\frac{(x-y)^2}{4 |z|}} \\
&= s_z(x,y)m_z(x,y),
\end{align*}
with some measurable real-valued phase function $\phi \co \R^2 \times\C_+ \to \R$, the maximal part 
\begin{equation*}
\label{}
m_z(x,y) \ov{\mathrm{def}}{=}  \cos(\arg z)^{\frac14}\frac{1}{\sqrt[4]{4 \pi |z|}} e^{-\cos(\arg z)\frac{(x-y)^2}{8 |z|}},
\end{equation*}
and the square part 
\begin{equation*}
\label{}
s_z(x,y) 
\ov{\mathrm{def}}{=}  e^{\i\phi(x,y,z)} m_z ^0(x,y).
\end{equation*}
Thus
\begin{equation*}
\label{}
|s_z(x,y)|^2 
= |m_z(x,y)|^2 
= \cos(\arg z)^{\frac12} |p_z(x-y)|.
\end{equation*}
Then with the classical estimate of \cite[Remark 3.7.10 p.~153]{ABHN11} (see also \cite[p.~543]{HvNVW18}), we obtain for any $x \in \R$
\begin{equation*}
\label{}
\norm{s_z(x,\cdot)}_{\L^2(\R)}^2 
= \bnorm{\cos(\arg z)^{\frac12} p_z(x-\cdot)}_{\L^1(\R)}
\leq C,
\end{equation*}
which implies the square estimate \eqref{estimate-square-Sp}. Moreover, the 1-dimensional version of Lemma \ref{lem-nc-maximal-inequality-convolution} described in Remark \ref{rem-nc-maximal-inequality-convolution} together with the $\L^1$-norm estimate
\begin{equation*}
\label{}
\norm{|m_z(x,\cdot)|^2}_{\L^1(\R)} 
= \bnorm{\cos(\arg z)^{\frac12} p_z(x-\cdot)}_{\L^1(\R)} 
\leq C, \quad x \in \R
\end{equation*}
implies the maximal estimate \eqref{estimate-max-Sp}.
\end{proof}

\begin{cor}
\label{cor-Jay-univ-R-bounded}
Let $\alpha > 0$. Suppose that $1 < p < \infty$ and let $A_{\univ}^{\alpha \cal{J}}$ be the universal $\alpha \cal{J}$-Weyl tuple on the Banach space $\L^p(\R^2)$. If $(T_z)_{\Re z > 0}$ denotes the complex time semigroup of \eqref{Autre-rep-Tz} generated by $-\A$ associated with $A_{\univ}^{\alpha \cal{J}}$ then the set
\[ 
\left\{ \cos(\arg z) e^{\alpha \Re z} T_z \ot \Id_{S^p}:  \Re z > 0 \right\} 
\]
is $R$-bounded over the Bochner space $\L^p(\R^2,S^p)$.
\end{cor}

\begin{proof}
Let us first consider the case where $p \geq 2$. According to Proposition \ref{prop-LpSp-square-max-cal-J}, if the integral kernel of $T_z$ is denoted by $K_z^{\alpha \cal{J}}$ (that is, $T_z = \TC_{K_z^{\alpha \cal{J}}}$), then the family 
\begin{equation*}
\label{}
\big(\cos(\arg z) e^{\alpha \Re z}K_z^{\alpha \cal{J}}\big)_{\Re z > 0}
\end{equation*}
admits an $\L^{q}(S^q)$-square-max decomposition for $q = \big(\frac{p}{2}\big)^*$. Then Theorem \ref{thm-square-Rad-vectorial} implies that the family 
\begin{equation*}
\label{}
\big\{ \cos(\arg z) e^{\alpha \Re z} T_z \ot \Id_{S^p} :  \Re z > 0 \big\} 
= \big\{ \cos(\arg z) e^{\alpha \Re z} \TC_{K_z^{\alpha \cal{J}}} \ot \Id_{S^p} :  \Re z > 0 \big\}
\end{equation*} 
is $R$-bounded over the Bochner space $\L^p(\R^2,S^p)$. 

If $1 < p < 2$, then the $R$-boundedness is obtained from a duality argument in the following way. By Proposition \ref{prop-R-bounded-dual}, it suffices to show that the family 
\begin{equation*}
\label{}
\left\{ \cos(\arg z) e^{\alpha \Re z} T_z^* \ot \Id_{S^{p^*}}:  \Re z > 0 \right\}
\end{equation*}
is $R$-bounded over the Banach space $\L^{p^*}(\R^2,S^{p^*})$, where $p^*$ is the conjugate exponent of $p$. According to \eqref{equ-adjoint-twisted-convolution-group}, we have $T_z^* = \TC_{\check{K}_z}$. Now, we will observe that the family $\{ \cos(\arg z) e^{\alpha \Re z} \check{K}_z :\: \Re z > 0\}$ also has an $\L^q(S^q)$-square-max decomposition for any $1 < q < \infty$. Indeed, using the notation introduced in the  proof of Proposition \ref{prop-LpSp-square-max-cal-J}, we see that
\begin{align*}
\MoveEqLeft
\cos(\arg z)e^{\alpha \Re z} \check{K}_z^{\alpha \cal{J}}(x,y) 
\ov{\eqref{equ-adjoint-twisted-convolution-group}}{=} 
\cos(\arg z) e^{\alpha \Re z} K_z^{\alpha \cal{J}}(y,x) \\
&\ov{\eqref{sans-fin-4}}{=} S^\alpha_z(y,x) M^\alpha_z(y,x) 
= S_z''(x,y)M_z''(x,y)
\end{align*}
with $S_z''(x,y) \ov{\mathrm{def}}{=} S^\alpha_z(y,x)$ and $M_z''(x,y) \ov{\mathrm{def}}{=} M_z^\alpha(y,x)$. With  \eqref{mystere} it is easy to see that \eqref{estimate-square-Sp} holds for $S_z''$.
Since $M_z''(x,y) = M_z^\alpha(x,y)$, we can confirm that \eqref{estimate-max-Sp} holds for $M_z''$.
\end{proof}

Arguing in the same manner, but using the fourth point of Proposition \ref{prop-LpSp-square-max-cal-J} instead of the third, we obtain the following result. In \cite[Example 8.3.6 p.~196]{HvNVW18}, the $R$-boundedness of the set $\{ T_z \ot \Id_{X} : |\arg z| \leq \theta \}$ is proved with $R$-bound less than $1+\frac{2}{\cos \theta}$ up to a constant, for any $\UMD$ Banach space $X$.

\begin{cor}
\label{cor-heat-R-bounded}
Suppose that $1 < p < \infty$. Let $(T_z)_{\Re z > 0}$ denote the standard heat semigroup on $\R$. Then the family
\[ 
\left\{ \cos(\arg z)^{\frac12} T_z \ot \Id_{S^p} :  \Re z > 0 \right\} 
\]
is $R$-bounded over the Bochner space $\L^p(\R,S^p)$.
\end{cor}

In the following statement, recall that the value $\alpha  \ov{\mathrm{def}}{=} \sum_{j=1}^k \alpha_j > 0$ is defined in \eqref{def-alpha} and that $A_\univ^\Theta$ is defined in \eqref{A-univ}.

\begin{thm}
\label{thm-R-sectorial}
Let $\Theta \in \M_n(\R)$ be a skew-adjoint matrix and $1 < p < \infty$. Consider the universal $\Theta$-Weyl tuple $A_\univ^\Theta$ on the space $\L^p(\R^n)$ and denote $(T_z)_{\Re z > 0}$ the complex time semigroup of \eqref{Autre-rep-Tz} associated with $A_\univ^\Theta$. Then the set
\[ 
\left\{ \cos(\arg z)^{\frac{n}{2}} e^{\alpha \Re z} T_z \ot \Id_{S^p} :  \Re z > 0 \right\} 
\]
is $R$-bounded over the Bochner space $\L^p(\R^n,S^p)$.
\end{thm}

\begin{proof}
Assume first that $\Theta = J$ has the block diagonal form with matrices $\alpha_j \cal{J}$ on the diagonal, see \eqref{def-de-J}. Then the operators of the semigroup $(T_z)_{\Re z > 0}$ split into tensor products
\[ 
T_z 
= T_z^1 \ot \cdots \ot T_z^k \ot T_z' \ot \cdots \ot T_z', \quad \Re z > 0,
\]
acting on the space 
\begin{equation*}
\label{}
\L^p(\R^2) \ot_p \cdots \ot_p \L^p(\R^2) \ot_p \L^p(\R) \ot_p \cdots \ot_p \L^p(\R),
\end{equation*}
where $T_z^j$ is the semigroup associated with the universal $\alpha_j \cal{J}$-Weyl tuple on $\L^p(\R^2)$, and $T_z'$ is the standard heat semigroup on the space $\L^p(\R)$. This follows from Lemma \ref{lem-kernel-splitting} for real positive $z$, and then by uniqueness of holomorphic extensions for an arbitrary complex number $z$ with $\Re z >0$. Put now $S_z \ov{\mathrm{def}}{=} \cos(\arg z)^{\frac{n}{2}} e^{\alpha \Re z} T_z \ot \Id_{S^p}$, $S_z^j \ov{\mathrm{def}}{=} \cos(\arg z) e^{\alpha_j \Re z} T_z^j$ for any $j = 1,\ldots,k$ and $S_z' \ov{\mathrm{def}}{=} \cos(\arg z)^{\frac{1}{2}} T_z'$. Then
\[ 
S_z 
= S_z^1 \ot S_z^2 \ot \cdots S_z^k \ot S_z' \ot \cdots \ot S_z' \ot \Id_{S^p}, \quad \Re z > 0.
\]
Moreover, the families $\{S^k_z \ot \Id_{S^p} :  \Re  z > 0 \}$ are $R$-bounded over the space $\L^p(\R^2,S^p)$ according to Corollary \ref{cor-Jay-univ-R-bounded}. The family $\{S_z' \ot \Id_{S^p} :  \Re z > 0 \}$ is also $R$-bounded over the space $\L^p(\R,S^p)$ according to Corollary \ref{cor-heat-R-bounded}. Then Lemma \ref{prop-2-tensor-R-bdd} implies that the family $\{ S_z :  \Re z > 0 \}$ is $R$-bounded over the Bochner space $\L^p(\R^n,S^p)$.

Consider the case of general skew-adjoint $\Theta \in \M_n(\R)$. Let $O \in \M_n(\R)$ be the orthogonal matrix such that $\Theta = O^T J O$ (see \eqref{diago-Theta}), where $J$ is a block diagonal matrix as in the first case.
Denote the semigroups associated with $\Theta$ and $J$ by $(T_z^\Theta)_{\Re z>0}$ and $(T_z^J)_{\Re z>0}$. Then according to the proof of Corollary \ref{Cor-512} (first for real $z$, then by uniqueness of holomorphic extensions for complex $z$), $T_z^\Theta = \tilde{O}^{-1} T_z^J \tilde{O}$ for the isometry $\tilde{O} : \L^p(\R^n) \to \L^p(\R^n), \: f \mapsto f(O^T \cdot)$. Since $\tilde{O}$ is an isometry for all $1 \leq p \leq \infty$, it extends to an isometry $\tilde{O} \ot \Id_{S^p}$ on $\L^p(\R^n,S^p)$. Then the set
\begin{align*}
\MoveEqLeft
\left\{ \cos(\arg z)^{\frac{n}{2}} e^{\alpha \Re z} T_z^\Theta \ot \Id_{S^p} :  \Re z > 0 \right\} \\
&= (\tilde{O}^{-1} \ot \Id_{S^p}) \left\{ \cos(\arg z)^{\frac{n}{2}} e^{\alpha z} T_z^J \ot \Id_{S^p} :  \Re z > 0 \right\} (\tilde{O} \ot \Id_{S^p})
\end{align*}
is $R$-bounded over the Bochner space $\L^p(\R^n,S^p)$ by the first case and stability of $R$-boundedness under products \cite[Propsition 8.1.19 p.~178]{HvNVW18}.
\end{proof}

Finally, we establish the principal result of this section, Theorem \ref{thm-514}, by employing the  following result as a foundational basis. Recall that a Banach space $X$ is of type $p$ for some $p \in [1,2]$ if there exists a constant $C \geq 1$ such that
\begin{equation*}
\label{type-inequality}
\Bgnorm{\sum_{k=1}^{n} \epsi_k \ot x_k}_{\L^p(\Omega_0,X)}
\leq C\bigg(\sum\limits_{k=1}^n \norm{x_k}_X^p\bigg)^{\frac{1}{p}}
\end{equation*}
for all integer $n \geq 1$ and any elements $x_1,\ldots, x_n \in X$. We say that $X$ is of cotype $q$ for some $q \in [2,\infty]$ if there exists a constant $C \geq 1$ such that
\begin{equation*}
\label{cotype-inequality}
\left(\sum\limits_{k=1}^n \norm{x_k}_X^q\right)^{\frac{1}{q}}
\leq C\Bgnorm{\sum_{k=1}^{n} \epsi_k \ot x_k}_{\L^q(\Omega_0,X)}
\end{equation*}
for all integer $n \geq 1$ and any elements $x_1,\ldots, x_n \in X$. In the case $q=\infty$, the left-hand side must be replaced by $\max_{1 \leq k \leq n} \norm{x_k}_X$. For any Banach space $X$, we let
\begin{equation*}
\label{cotype-inequalities}
\type X
\ov{\mathrm{def}}{=} \sup\{ p\geq 1 : X \text{ has type }p\}  \text{ and } 
\cotype X \ov{\mathrm{def}}{=} \inf\{q \geq 2: X \text{ has cotype }q \}.
\end{equation*}
See the book \cite{DJT95} for a complete presentation.

The sufficient criterion that we will rely on throughout this work to establish a bounded H\"ormander calculus is given in \cite[Theorem 7.1 (2), p.~424]{KrW18}. For convenience, we recall it here in a form adapted to reflexive Banach spaces and to sectorial operators that are not necessarily injective.

\begin{thm}
\label{prop-KrW-R-bounded}
Let $X$ be a reflexive Banach space and $A$ be a sectorial operator acting on $X$ admitting a bounded $\H^\infty(\Sigma_\sigma)$ functional calculus for some angle $\sigma \in (0,\pi)$.
Let $r \in (1,2]$ such that $\frac1r > \frac{1}{\type X} - \frac{1}{\cotype X}$.
Consider some $\beta \geq 0$.
If for all $t > 0$ and any angle $\theta \in (-\frac{\pi}{2},\frac{\pi}{2})$, the set
\begin{equation*}
\label{}
\left\{ \exp(- e^{\i\theta} t 2^n A) :\: n \in \Z \right\}
\end{equation*}
is $R$-bounded with a control of its $R$-bound by $C \left( \frac{\pi}{2} - |\theta| \right)^{-\beta}$, where the constant $C$ is independent of $t$ and $\theta$, then the operator $A$ admits a bounded $\Hor^s_2(\R^*_+)$ H\"ormander functional calculus for any $s > \beta + \frac1r$.
\end{thm}

\begin{proof}
Note that the assumptions imply that $\omega_{\sec}(A) = 0$.
The result is shown in \cite[Theorem 7.1 (2)]{KrW18} in the case that $A$ also has dense range.
Then if the range of $A$ is not dense, its part $A_1$ on $\ovl{\Ran A}$ is sectorial of type 0, and the $R$-bound of its semigroup has the same control as the semigroup generated by $A$, since $\exp(-zA_1) = \exp(-zA)|_{\ovl{\Ran A}}$ according to \cite[Proposition p.~60]{EnN00}. Moreover, the subspace $\ovl{\Ran A}$ inherits type and cotype from $X$, as easily observed in \cite[p.~219]{DJT95}. Thus, applying \cite[Theorem 7.1 (2)]{KrW18} to $A_1$, we deduce that $f(A_1)$ is a bounded operator on $\ovl{\Ran A}$ for any function $f \in \Hor^s_2(\R^*_+)$. We conclude with the identity $f(A)|_{\ovl{\Ran A}} = f(A_1)$ valid for every function $f \in \H^\infty_0(\Sigma_\theta)$.
\end{proof}

This leads directly to the following result.

\begin{prop}
\label{prop-KrW18}
Suppose that $1 < p < \infty$ and $\beta \geq 0$. Let $-A$ be the generator of a bounded holomorphic semigroup $(T_z)_{\Re z > 0}$ acting on a noncommutative $\L^p$-space $\L^p(\cal{M})$. Suppose that $A$ has a bounded $\HI(\Sigma_\omega)$ functional calculus for some angle $\omega \in (0,\pi)$. Assume in addition that the family
\begin{equation}
\label{family-23}
\left\{ (\cos \arg z)^\beta T_z :  \Re z > 0 \right\} 
\end{equation}
of operators is $R$-bounded over the space $\L^p(\cal{M})$. Then the operator $A$ admits a bounded $\Hor^s_2(\R_+^*)$ H\"ormander calculus of any order $s > \beta + \frac12$.
\end{prop}

\begin{proof}
This is an immediate consequence of Theorem \ref{prop-KrW-R-bounded} (see also \cite[Theorem 7.1 (2) p.~424]{KrW18}). Indeed, note first that 
\begin{equation*}
\label{}
\frac{1}{\type X} - \frac{1}{\cotype X}
=\left|\frac{1}{2}-\frac{1}{p}\right| 
< \frac12,
\end{equation*}
where $X=\L^p(\cal{M})$ by \cite[Corollary 5.5 p.~1481]{PiX03}, so we can take $r=2$ in \cite[Theorem 7.1]{KrW18}. If the family \eqref{family-23} is $R$-bounded, then the subset 
\begin{equation*}
\label{}
\left\{ (\cos\theta)^\beta T_{e^{\i\theta}2^n t} :\: n \in \Z \right\}
\end{equation*} 
is also $R$-bounded, with uniformly controlled $R$-bound in $\theta \in (-\frac{\pi}{2},\frac{\pi}{2})$ and $t > 0$.
Then the $R$-boundedness assumption of Theorem \ref{prop-KrW-R-bounded} (see also \cite[(7.2)]{KrW18}) follows by dividing the aforementioned expression by $(\cos \theta)^\beta$ and observing that this expression is equivalent to $(\frac{\pi}{2} - |\theta|)^\beta$.
\end{proof}

We conclude with the following theorem.

\begin{thm}
\label{thm-514}
Suppose that $1 < p < \infty$. Consider a skew-symmetric matrix $\Theta \in \M_n(\R)$. Let $-\A$ be the infinitesimal generator of the semigroup of operators acting on $\L^p(\R^n)$ associated with the universal $\Theta$-Weyl tuple $A_\univ^\Theta = (A_1,\ldots,A_n)$ from Proposition \ref{prop-universal-Theta-Weyl-tuple}, that is, $\A$ is the closure of the operator $\sum_{j=1}^n A_j^2$. Then the operator $(\A - \alpha \Id) \ot \Id_{S^p}$ admits a bounded $\Hor^s_2(\R_+^*)$ H\"ormander functional calculus of order $s > \frac{n}{2} + \frac12$ on the Bochner space $\L^p(\R^n,S^p)$.
\end{thm}

\begin{proof}
According to Corollary \ref{Cor-512}, the operator $(\A - \alpha \Id) \ot \Id_{S^p}$ has a bounded $\HI(\Sigma_\omega)$ functional calculus on the Bochner space $\L^p(\R^n, S^p)$ for any angle $\omega \in (\frac{\pi}{2},\pi)$. Moreover, according to Theorem \ref{thm-R-sectorial}, the semigroup $(S_z \ot \Id_{S^p})_{\Re z > 0}$ of operators acting on the space $\L^p(\R^n, S^p)$ of \eqref{def-Sz}, generated by the operator $(\A - \alpha \Id) \ot \Id_{S^p}$, is $R$-bounded in the following sense:
\[ 
R \left( \left\{ \left(\cos \arg z \right)^{\frac{n}{2}} S_z \ot \Id_{S^p}:  \Re z >0 \right\} \right) 
< \infty .
\]
Recall that 
\begin{equation*}
\label{}
\frac{1}{\type X} - \frac{1}{\cotype X}
=\left|\frac{1}{2}-\frac{1}{p}\right| 
< \frac12
\end{equation*}
 for the Banach space $X = \L^p(\R^n,S^p)$. Then according to Proposition \ref{prop-KrW18}, the operator $(\A - \alpha \Id) \ot \Id_{S^p}$ admits a bounded $\Hor^s_2(\R_+^*)$ H\"ormander calculus of order $s > \frac{n}{2} + \frac12$ on the Bochner space $\L^p(\R^n,S^p)$.
\end{proof}

Recall that for any angle $\omega > 0$ the algebra $\HI(\Sigma_\omega)$ injects continuously in the space $\Hor^s_2(\R_+^*)$ by Lemma \ref{lem-norm-comparisons-Hor-HI}. This implies the following result which says, with the terminology of \cite[Definition 3.7 p.~31]{JMX06}, that the operator $\A - \alpha \Id$ admits a \textit{completely} bounded $\H^\infty(\Sigma_\omega)$ functional calculus for any angle $\omega > 0$. 

\begin{cor}
\label{cor-bis}
Suppose that $1 < p < \infty$. Consider a skew-symmetric matrix $\Theta \in \M_n(\R)$. Let $-\A$ be the infinitesimal generator of the semigroup of operators acting on $\L^p(\R^n)$ associated with the universal $\Theta$-Weyl tuple $A_\univ^\Theta = (A_1,\ldots,A_n)$ from Proposition \ref{prop-universal-Theta-Weyl-tuple}, that is, $\A$ is the closure of the operator $\sum_{j=1}^n A_j^2$. Then the operator $(\A - \alpha \Id) \ot \Id_{S^p}$ admits a bounded $\H^\infty(\Sigma_\omega)$ functional calculus for any angle $\omega >0$  on the Bochner space $\L^p(\R^n,S^p)$, i.e.~$\omega_{\H^\infty}(\A - \alpha \Id) \ot \Id_{S^p})=0$.
\end{cor}

\begin{example} \normalfont
\label{example-twisted-Lap-III}
In light of Example \ref{Example-twisted-Laplacian} and Example \ref{twisted-bis}, both results apply to the twisted Laplacian.
\end{example}

\begin{example} \normalfont
\label{example-Lap-new}
By Example \ref{Example-Heat}, Theorem \ref{thm-514} applies to the classical Laplacian $-\Delta$ on the Euclidean space $\R^n$. Note that the order $s > \frac{n}{2}+\frac12$ in our H\"ormander functional calculus on the Bochner space $\L^p(\R^n,S^p)$ appears to be new, even for this fundamental operator.
Previously, it was shown in~\cite[Theorem 1.1 (3)]{Hyt04} that the operator $-\Delta$ admits a Mikhlin calculus of order $N = \lfloor \tfrac{n}{\min(p, p^*)} \rfloor + 1$. This means that spectral multipliers $f(-\Delta)$ are bounded on the Bochner space $\L^p(\R^n,S^p)$ provided that
\begin{equation}
\label{equ-example-Lap-new}
\big| \xi^k f^{(k)}(\xi) \big| 
\leq C, \quad \xi > 0,\: k= 0, 1, \ldots, N.
\end{equation}

Note that a H\"ormander $\Hor^s_2(\R^*_+)$ function $f$ with order $s > N + \frac12$ satisfies \eqref{equ-example-Lap-new}. Indeed, according to the Sobolev embedding theorem \cite[Proposition 2.5, p.~172]{LaM89}, the Sobolev space $\W^{s,2}(\R)$ embeds into the Banach space $\mathrm{C}^N_b(\R)$\label{CbRn} of functions of class $\mathrm{C}^k$ with bounded derivatives up to order $N$, for any $s > N + \frac{1}{2}$. Recall that the norm of a function $u$ in the space $\mathrm{C}^N_b(\R)$ is given by
\begin{equation*}
\label{}
\norm{u}_{\mathrm{C}^N_b(\R)}
\ov{\mathrm{def}}{=} \sup_{x \in \R^n,0 \leq k \leq N} |u^k(x)|.
\end{equation*}

Consider some function $f \in \Hor^s_2(\R^*_+)$.  
The function $g_t \ov{\mathrm{def}}{=} \eta f(t\, \cdot)$ belongs to the space $\W^{s,2}(\R)$ uniformly in $t > 0$. In particular, we have the Sobolev embedding $\W^{s,2}(\R) \hookrightarrow \mathrm{C}_b^k(\R)$ if $s > k + \tfrac{1}{2}$. Thus for any integer $k \in \llbracket 0, N \rrbracket$ the function $g_t^{(k)}$ is continuous and uniformly bounded by a constant $C$. Fix any point $x_0$ with $\eta(x_0) \not=0$ and write $\xi \ov{\mathrm{def}}{=} t x_0$. We aim to estimate $|\xi^k f^{(k)}(\xi)|$ by induction. If $k = 0$ then  
\[
|f(\xi)|
=|f(tx_0)| 
= \left|\frac{g_t(x_0)}{\eta(x_0)} \right|
\leq \frac{C}{\eta(x_0)}.
\]
Assume that for all integer $j \in \llbracket 0,k \rrbracket$, there exists a positive constant $C_j > 0$ such that
\begin{equation}
\label{inter-final}
|f^{(j)}(\xi)| 
\leq C_j |\xi|^{-j}.
\end{equation}
Using Leibniz's rule, we see that
\[
g_t^{(k)}(x) 
= \sum_{j=0}^k \binom{k}{j} \eta^{(k-j)}(x)  t^j f^{(j)}(t x), \quad x>0.
\]
Replacing $x$ by $x_0$ and isolating the leading term, we obtain
\begin{align*}
\MoveEqLeft
|t^k f^{(k)}(\xi)| 
= \left|\frac{g_t^{(k)}(x_0)}{\eta(x_0)} - \sum_{j=0}^{k-1} \binom{k}{j} \frac{\eta^{(k-j)}(x_0)}{\eta(x_0)} t^j f^{(j)}(\xi)\right| \\         
&\leq \frac{\big|g_t^{(k)}(x_0)\big|}{\eta(x_0)} + \sum_{j=0}^{k-1} \binom{k}{j} \frac{\eta^{(k-j)}(x_0)}{\eta(x_0)} t^j \big|f^{(j)}(\xi)\big| \\
&\ov{\eqref{inter-final}}{\leq} \frac{C}{\eta(x_0)}+\sum_{j=0}^{k-1} \binom{k}{j} \frac{\eta^{(k-j)}(x_0)}{\eta(x_0)} \frac{t^j}{|\xi|^{j}} C_j 
=\frac{C}{\eta(x_0)}+\sum_{j=0}^{k-1} \binom{k}{j} \frac{\eta^{(k-j)}(x_0)}{\eta(x_0)} \frac{1}{x_0^{j}} C_j.
\end{align*}
Hence $|\xi|^k |f^{(k)}(\xi)| \leq C_k$ for some constant $C_k$. By induction up to order $N$, we obtain the claim.


The H\"ormander functional calculus of order $s > \frac{n}{2} + \frac{1}{2}$ for the classical Laplacian on the Bochner space $\L^p(\R^n, Y)$, where $Y$ is a UMD lattice, was established in~\cite[Corollary 4.23 (2)]{DKK21}.
\end{example}

\section{Functional calculus for the harmonic oscillator associated to $\Theta$-Weyl tuples}
\label{subsec-functional-calculus-general-Theta-Weyl-tuple}

Consider a skew-symmetric matrix $\Theta \in \M_n(\R)$. Thanks to the transference result from Corollary \ref{cor-transference}, we can generalize the H\"ormander functional calculus of the semigroup associated with the universal $\Theta$-Weyl tuple $A_\univ^\Theta$ to semigroups associated with any $\Theta$-Weyl tuple acting on a noncommutative $\L^p$-space associated with an injective (=approximate finite-dimensional) semifinite von Neumann algebra. We will use the notion of completely bounded maps between operator spaces and we refer to the books \cite{BLM04}, \cite{EfR00} and \cite{Pis98} for background. If $1 \leq p < \infty$, it is known \cite[Lemma 1.7 p.~23]{Pis98} that a linear map $T \co E \to F$ between operator spaces is completely bounded\label{def-completely-bounded} if and only if it induces a bounded operator $\Id_{S^p} \ot T \co S^p(E) \to S^p(F)$ on the vector-valued Schatten space $S^p(E)$ of \cite[p.~18]{Pis98}. In this case, we have
\begin{equation*}
\label{norm-cb-Sp}
\norm{T}_{\cb, E \to F}
=\norm{\Id_{S^p} \ot T}_{S^p(E) \to S^p(F)}.
\end{equation*}
Moreover, for any \textit{injective} von Neumann algebra $\cal{M}$ equipped with a normal semifinite faithful trace, a completely bounded map $T \co E \to F$ between operator spaces induces by \cite[(3.1) p.~39]{Pis98} a completely bounded map $\Id_{\L^p(\cal{M})} \ot T \co \L^p(\cal{M},E) \to \L^p(\cal{M},F)$ with 
\begin{equation}
\label{majo-cb}
\norm{\Id_{\L^p(\cal{M})} \ot T}_{\cb, \L^p(\cal{M},E) \to \L^p(\cal{M},F)} 
\leq \norm{T}_{\cb,E \to F}.
\end{equation}

\begin{prop}
\label{prop-nc-transference}
Let $\cal{M}$ be an injective von Neumann algebra with separable predual equipped with a normal semifinite faithful trace. Suppose that $1 < p < \infty$. Let $A$ be a $\Theta$-Weyl tuple on the noncommutative $\L^p$-space $\L^p(\cal{M})$. Let us introduce the constant 
\begin{equation*}
\label{}
M_A \ov{\mathrm{def}}{=} \sup_{t \in \R^n} \norm{e^{\i t\cdot A}}_{\L^p(\cal{M}) \to \L^p(\cal{M})}.
\end{equation*} 
Let $A_{\univ}^\Theta$ be the universal $\Theta$-Weyl tuple on $\L^p(\R^{n})$ from Proposition \ref{prop-universal-Theta-Weyl-tuple}. For any tempered distribution $a$ with $\hat{a} \in \L^1(\R^n)$, we have
\begin{equation}
\label{trans-134}
\norm{a(A)}_{\L^p(\cal{M}) \to \L^p(\cal{M})} 
\leq M_A^2 \norm{a(A_{\univ}^\Theta)}_{\cb, \L^p(\R^{n}) \to \L^p(\R^{n})}.
\end{equation}
\end{prop}

\begin{proof}
Using the transference result of Corollary \ref{cor-transference} in the first estimate and Fubini's theorem \cite[(3.6')]{Pis98} in the second equality, we obtain
\begin{align*}
\MoveEqLeft
\norm{a(A)}_{\L^p(\cal{M}) \to \L^p(\cal{M})}          
\ov{\eqref{transference-estimate}}{\leq} M_A^2 \norm{C_{\hat{a}}^\Theta}_{\L^p(\R^n,\L^p(\cal{M})) \to \L^p(\R^n,\L^p(\cal{M}))} \\
&\ov{\eqref{aA-as-twisted-conv}}{=} M_A^2 \norm{a(A_{\univ}^\Theta) \ot \Id_{\L^p(\cal{M})}}_{\L^p(\R^n,\L^p(\cal{M})) \to \L^p(\R^n,\L^p(\cal{M}))} \\
&=M_A^2 \norm{\Id_{\L^p(\cal{M})} \ot a(A_{\univ}^\Theta)}_{\L^p(\cal{M},\L^p(\R^n)) \to \L^p(\cal{M},\L^p(\R^n))} \\
&\ov{\eqref{majo-cb}}{\leq} M_A^2 \norm{a(A_{\univ}^\Theta)}_{\cb, \L^p(\R^{n}) \to \L^p(\R^{n})}.
\end{align*}
\end{proof}

\begin{remark} \normalfont
In this paper, we only use the case where the von Neumann algebra $\cal{M}$ is injective. However, we could replace in Proposition \ref{prop-nc-transference} the word <<injective>> by the much more general assumption $\QWEP$\label{QWEP}, which is the Kirhcberg's quotient weak expectation property. Indeed, we can use \cite[(iii) p.~984]{Jun04}, whose proof is unfortunately unpublished, instead of \cite[(3.1)]{Pis98}. The same remark  also applies to Corollary \ref{cor-Hormander-calculus-for-Theta-Weyl-tuple-on-LpM}. Recall that by \cite[Theorem 1.1]{AHW16} a von Neumann algebra $\cal{M}$ acting on a complex Hilbert space $H$ has $\QWEP$ if and only if it belongs to the closure of the subset $\cal{F}_{\mathrm{inj}}$ of injective factors on $H$ in the space of von Neumann algebras acting on $H$, endowed with the Effros-Mar\'echal topology. We refer to \cite{AHW16}, \cite{Gol22} and \cite{Oza04} for more information on $\QWEP$ von Neumann algebras.
\end{remark}

Let $\M(\R_+)$ be the space of all bounded Borel measures on $\R_+$. Then $\M(\R_+)$ is a Banach algebra under convolution of measures. 
The Laplace transform $\scr{L}\mu$ of a bounded Borel measure $\mu$ on $[0,\infty)$ is defined by
\begin{equation*}
\label{Laplace-transform}
(\scr{L}\mu)(z) 
\ov{\mathrm{def}}{=} \int_0^\infty e^{-sz} \d\mu(s),
\quad \Re z > 0.
\end{equation*}
This is a bounded holomorphic function on $\C_+$.
Let 
\begin{equation*}
\label{def-LM}
\scr{L}\M 
\ov{\mathrm{def}}{=}  \{\scr{L}\mu : \mu \in \M(\R_+)\}.
\end{equation*}
The map $\M(\R_+) \to \H^\infty(\C_+)$, $\mu \mapsto \scr{L} \mu$ is an injective algebra homomorphism. Thus $\scr{L}\M$ becomes a Banach algebra for the norm $\norm{\scr{L}\mu}_{\HP} \ov{\mathrm{def}}{=} \norm{\mu}_{\M(\R_+)}$ and the pointwise product.

Consider a strongly continuous bounded semigroup $(T_t)_{t \geq 0}$ of operators acting on a Banach space $X$ with infinitesimal generator $-A$. The Hille-Phillips functional calculus \cite[Chapter XV]{HiP74} \cite[Section 3.3]{Haa06} for the operator $A$ is defined as follows. If $f=\scr{L}\mu$ is the Laplace transform of a bounded Borel measure $\mu$ on $\R_+$, we can define the bounded linear operator $f(A)_{\HP} \co X \to X$ by
\begin{align}
\label{Hille-Phillips}
f(A)_{\HP}x
\ov{\mathrm{def}}{=} \int_{0}^{\infty} T_t x \d \mu(t), \quad x \in X
\end{align} 
where the integral converges in the strong operator topology. In this context, we use the subscript $\HP$ to avoid confusion with the notation \eqref{Weyl-calculus} introduced for the $\Theta$-Weyl functional calculus. The map $\scr{L}\M \to \B(X)$, $f \mapsto f(A)_{\HP}$ is a continuous algebra homomorphism with norm less than $\sup_{t \geq 0} \norm{T_t}$, known as the Hille-Phillips calculus. It is well-known that the function $e_t \co \mathbb{C}_+ \to \mathbb{C}$, $z \mapsto e^{-tz}$ belongs to $\scr{L}\M$ for any $t>0$ and that $e_t(A) = \exp(-tA)$ for any $t>0$.

The Hille-Phillips calculus is compatible with the sectorial functional calculus, i.e.~any $f$ of $\HI_0(\Sigma_\omega)$, where $\omega \in (\frac{\pi}{2},\pi)$ belongs to $\scr{L}\M$ and we have $f(A)_{\HP}=f(A)_{\mathrm{sect}}$ where $f(A)_{\mathrm{sect}}$ is defined by the sectorial functional calculus. We refer to \cite[Remark 3.3.3 p.~74]{Haa06}.

The following lemma addresses the technical difficulty to exploit the transference from Proposition \ref{prop-nc-transference} when we are faced with an operator $f(A)$ provided by the functional calculus.


\begin{lemma}
\label{lem-Hille-Phillips} 
Consider a strongly continuous projective representation $\pi$ of $\R^n$ on a separable Banach space $X$ and a strongly continuous semigroup $(T_t)_{t \geq 0}$ of operators on $X$ such that there exists a family $(p_t)_{t > 0}$ of integrable functions over $\R^n$ such that
\begin{equation}
\label{semi-int-rep-bis}
T_t 
= \int_{\R^n} p_t(s) \pi_s \d s,\quad t > 0 
\quad \text{with} \quad \sup_{t > 0} \norm{p_t}_{\L^1(\R^n)} 
< \infty.
\end{equation}
Let $f \in \scr{L}\M$ be the Laplace transform of some bounded Borel measure $\mu$ on $\R_+$. Then there exists a function $k_f \in \mathrm{C}_0(\R^n)$ such that $\hat{k}_f \in \L^1(\R^n)$ and $f(A)_{\HP}x = \int_{\R^n} \hat{k}_f(s) \pi_s x \d s$ for any $x \in X$. Here $\hat{k}_f$ is the Fourier transform of the tempered distribution $k_f$.
\end{lemma}

\begin{proof}
The Hille-Phillips calculus yields 
\begin{equation}
\label{equ-1-proof-lem-Hille-Phillips}
f(A)_{\HP}x 
\ov{\eqref{Hille-Phillips}}{=} \int_0^\infty T_t x \d\mu(t) 
\ov{\eqref{def-Sz}}{=} \int_0^\infty \bigg(\int_{\R^n} p_t(s) \pi_s x \d s \bigg) \d\mu(t).
\end{equation}
According to \eqref{semi-int-rep-bis}, we have
\[
\int_0^\infty \bigg(\int_{\R^n} |p_t(s)| \d s \bigg) \d\mu(t) 
\lesssim \sup_{t > 0} \bnorm{p_t}_{\L^1(\R^n)} \norm{\mu}_{\M(\R^+)} 
< \infty.
\]
This shows that the function $g \co \R^n \to \mathbb{C}$, $s \mapsto \int_0^\infty p_t(s) \d\mu(t)$ belongs to the space $\L^1(\R^n)$. So we can introduce the inverse Fourier transform $k_f \ov{\mathrm{def}}{=} \scr{F}^{-1}g$ which belongs to $\mathrm{C}_0(\R^n)$. Now
\[
\int_0^\infty \bigg(\int_{\R^n} |p_t(s)| \norm{\pi_s x}_X \d s \bigg) \d\mu(t) 
\lesssim \sup_{t > 0} \norm{p_t}_{\L^1(\R^n)} \norm{\mu}_{\M(\R^+)} 
< \infty.
\]
Thus we can change the order of integration in \eqref{equ-1-proof-lem-Hille-Phillips} and we obtain
\[ 
f(A)_{\HP}x 
= \int_{\R^n} \bigg(\int_0^\infty p_t(s) \d\mu(t) \bigg) \pi_s x \d s 
= \int_{\R^n} \hat{k}_f(s) \pi_s x \d s.
\]
\end{proof}

Recall that $\alpha \geq 0$ is defined in \eqref{def-alpha}.

\begin{cor}
\label{cor-Hormander-calculus-for-Theta-Weyl-tuple-on-LpM}
Let $A=(A_1,\ldots,A_n)$ be a $\Theta$-Weyl tuple on a noncommutative $\L^p$-space $\L^p(\cal{M})$  associated with an injective von Neumann algebra $\cal{M}$ with separable predual, where $1 < p < \infty$. Then the closure $\A -\alpha \Id$ of the operator $\sum_{k=1}^n A_k^2 - \alpha \Id$ admits a bounded H\"ormander $\Hor^s_2(\R_+^*)$ functional calculus of order $s > \frac{n}{2} + \frac12$ on the Banach space $\L^p(\cal{M})$.
\end{cor}

\begin{proof}
We claim that for any angle $\omega \in (\frac{\pi}{2},\pi)$ and any function $f$ of the space $\HI_0(\Sigma_\omega)$ we have the estimate
\begin{equation}
\label{equ-1-proof-cor-Hormander-calculus-for-Theta-Weyl-tuple-on-LpM}
\norm{f(\A - \alpha \Id) }_{\L^p(\cal{M}) \to \L^p(\cal{M})} 
\leq C \norm{f}_{\Hor^s_2(\R_+^*)}.
\end{equation}
Indeed, according to \cite[Lemma 3.3.1 p.~72]{Haa06}, \cite[Lemma 2.12]{LM99}, the function $f$ belongs to the space $\scr{L} M$. Using Lemma \ref{lem-Hille-Phillips} and Proposition \ref{prop-not}, we can write 
\begin{equation*}
\label{}
f(\A - \alpha \Id)_{\HP} 
= k_{f}(A)
\end{equation*}
for some function $k_{f} \in \mathrm{C}_0(\R^n)$ with $\hat{k}_{f} \in \L^1(\R^n)$. Thus, by Proposition \ref{prop-nc-transference} and Theorem \ref{thm-514} applied to the universal $\Theta$-Weyl tuple $A_{\univ}^\Theta = (A_{\univ,1},\ldots,A_{\univ,n})$ acting on the space $\L^p(\R^n)$ and the closure $\cal{B}$ of the unbounded operator $\sum_{k=1}^n A_{\univ,k}^2$, we obtain
\begin{align*}
\MoveEqLeft
\norm{f(\A - \alpha \Id)_{\HP}}_{\L^p(\cal{M}) \to \L^p(\cal{M})}  
= \norm{k_{f}(A)}_{\L^p(\cal{M}) \to \L^p(\cal{M})}  \\ 
&\ov{\eqref{trans-134}}{\leq} M_A^2 \norm{k_{f}(A_{\univ})}_{\cb,\L^p(\R^n) \to \L^p(\R^n)} \\
&= M_A^2 \norm{f(\cal{B} - \alpha \Id)_{\HP}}_{\cb,\L^p(\R^n) \to \L^p(\R^n)} 
\lesssim \norm{f}_{\Hor^s_2(\R_+^*)}.         
\end{align*}
Consequently \eqref{equ-1-proof-cor-Hormander-calculus-for-Theta-Weyl-tuple-on-LpM} is proved. 
\end{proof}

Using the fact that $\Hor^s_2(\R_+^*)$ is a Banach algebra into which $\HI(\Sigma_\omega)$ injects continuously according to Lemma \ref{lem-norm-comparisons-Hor-HI}, 
we obtain the following consequence.

\begin{cor}
\label{cor-Hormander-calculus-for-Theta-Weyl-tuple-on-LpM-bis}
Let $A=(A_1,\ldots,A_n)$ be a $\Theta$-Weyl tuple on a noncommutative $\L^p$-space $\L^p(\cal{M})$  associated with an injective von Neumann algebra $\cal{M}$ with separable predual, where $1 < p < \infty$. Then the closure $\A -\alpha \Id$ of the operator $\sum_{k=1}^n A_k^2 - \alpha \Id$ admits a bounded $\H^\infty(\Sigma_\omega)$ functional calculus for any angle $\omega >0$ on the Banach space $\L^p(\cal{M})$, i.e.~$\omega_{\H^\infty}(\A -\alpha \Id)=0$.
\end{cor}

Let us mention that parts of the previous program also work if the $\Theta$-Weyl tuple $A$ acts on a $\UMD$ Banach space $X$ instead of a noncommutative $\L^p$-space.

\begin{prop}
\label{prop-Theta-Weyl-tuple-HI-calculus-UMD-space}
Let $A = (A_1,\ldots,A_n)$ be a $\Theta$-Weyl tuple acting on a separable $\UMD$ Banach space $X$. Then the closure $\A - \alpha \Id$ of the operator $\sum_{k=1}^n A_k^2 - \alpha \Id$ acting on $X$ admits a bounded $\HI(\Sigma_\omega)$ functional calculus for any angle $\omega \in (\frac{\pi}{2},\pi)$, i.e.~$\omega_{\H^\infty}(\A - \alpha \Id) \leq \frac{\pi}{2}$.
\end{prop}

\begin{proof}
Assume $1 < p < \infty$. Consider the universal $\Theta$-Weyl tuple 
\begin{equation*}
\label{}
A_{\univ}^\Theta 
= (A_{\univ,1},\ldots,A_{\univ,n}),
\end{equation*}
acting on the Banach space $\L^p(\R^n)$ and the closure $\mathcal{B}$ of the operator $\sum_{k=1}^n A_{\univ,k}^2$. By Corollary \ref{Cor-512}, the tensor extension of $\cal{B}_\alpha = (\cal{B} - \alpha \Id_{\L^p(\R^n)}) \ot \Id_{X}$ admits a bounded $\HI(\Sigma_\omega)$ functional calculus on the Bochner space $\L^p(\R^n,X)$ for any angle $\omega \in (\frac{\pi}{2},\pi)$.


If $\omega \in (\frac{\pi}{2},\pi)$ and $f \in \HI_0(\Sigma_\omega)$ then using Lemma \ref{lem-Hille-Phillips} together with Theorem \ref{thm-514}, we can write $f(\A - \alpha \Id)_{\HP} = k_{f}(A)$ for some function $k_{f}$ with $\hat{k}_{f} \in \L^1(\R^n)$. We obtain
\begin{align*}
\MoveEqLeft
\norm{f(\A - \alpha \Id)_{\HP}}_{X \to X}  
= \norm{k_{f}(A)}_{X \to X}   
\ov{\eqref{aA-as-twisted-conv}, \eqref{transference-estimate}}{\leq} M_A^2 \norm{k_{f}(A_{\univ}^\Theta)}_{\L^p(\R^n,X) \to \L^p(\R^n,X)} \\
&= M_A^2 \norm{f(\cal{B} - \alpha \Id)_{\HP}}_{\L^p(\R^n,X) \to \L^p(\R^n,X)} 
\lesssim \norm{f}_{\HI(\Sigma_\omega)}.         
\end{align*}
In other words, the operator $\A - \alpha \Id$ admits a bounded $\HI(\Sigma_\omega)$ functional calculus.
\end{proof}

\begin{remark}\normalfont
\label{rem-interpolation}
In case that $A = (A_1,\ldots,A_n)$ form a $\Theta$-Weyl tuple not only on a single Banach space, but consistently on the whole scale of noncommutative $\L^p$-spaces associated with an injective von Neumann algebra $\cal{M}$ with $1 < p < \infty$, such that on $\L^2(\mathcal{M})$, the closure $\A$ of $\sum_{k=1}^n A_k^2$ is self-adjoint, we can slightly improve on the differentiation order of the H\"ormander functional calculus by complex interpolation in Corollary \ref{cor-Hormander-calculus-for-Theta-Weyl-tuple-on-LpM}.

Let $p \in (1,\infty)$ be fixed.
Choosing $p_0$ close to $1$ (if $p \leq 2$) or close to $\infty$ (if $p \geq 2$), we get by Corollary \ref{cor-Hormander-calculus-for-Theta-Weyl-tuple-on-LpM} that the operator $\A - \alpha \Id$ has a bounded $\Hor^s_2(\R^*_+)$ calculus on $\L^{p_0}$ with $s > \frac{n}{2} + \frac12$.
On $\L^2$, the operator $\A - \alpha \Id$ admits a bounded $\Hor^s_2(\R^*_+)$ calculus with $s  > \frac12$ since then, $\Hor^s_2(\R^*_+)$ is contained in the space of bounded Borel functions (classical self-adjoint functional calculus).
Then by complex interpolation of the functional calculus bilinear mapping 
\begin{equation*}
\label{}
\Phi \co \L^p \times \Hor^s_2(\R^*_+) \to \L^p ,\: (x,m) \mapsto m(\A - \alpha \Id)x
\end{equation*}
between levels $\L^{p_0}$ (with $p_0$ close to $1$ resp. $\infty$) and $\L^2$, we deduce that on $\L^p$ (with $p< 2$ resp. $p > 2$), $\A - \alpha \Id$ has a bounded $\Hor^s_2(\R^*_+)$ functional calculus on $\L^p$ with 
\[ 
s > \begin{cases} 
(\frac{2}{p} - 1)(\frac{n}{2}+\frac12) + \frac1{p'} & : \: p \leq 2 \\ 
(\frac{2}{p'} - 1)(\frac{n}{2} + \frac12) + \frac1p & : \: p \geq 2 . 
\end{cases} 
\]
See e.g.~\cite[Proposition 4.83 (2)]{Kri09}.

Note that under the previous assumptions, it is even possible to lower the differentiation order still a bit by interpolating $\Hor^s_2(\R^*_+)$ calculus on $\L^{1+\epsi}$ or $\L^{p_0}$ with large $p_0$ and $\Hor^\epsilon_\infty(\R^*_+)$ on $\L^2$. Here one defines the space $\Hor^s_q(\R^*_+)$ in the same way as the space $\Hor^s_2(\R^*_+)$, but over the Sobolev space $\W^{s,q}(\R)$. One gets that the operator $\A  - \alpha \Id$ has a $\Hor^s_q(\R^*_+)$ functional calculus on $\L^p$ for $\frac1q > \left| \frac1p - \frac12 \right|$ and $s = \frac{2}{q} \left( \frac{n}{2} + \frac12 \right)$.
\end{remark}

\begin{example}\normalfont
\label{multiplications}
Let $\cal{M}$ be a von Neumann algebra acting on some Hilbert space $H$, equipped with a normal semifinite faithful trace. Suppose that $1 \leq p < \infty$. For any $a \in \cal{M}$, we can define a bounded linear operator 
\begin{equation*}
\label{Left-multiplication-operator}
\cal{L}_a \co \L^p(\cal{M}) \to \L^p(\cal{M}), x\mapsto ax,
\end{equation*}
called left multiplication by $a$ on $\L^p(\cal{M})$. In \cite[Section 8.A]{JMX06}, it is explained how to define a left multiplication operator $\cal{L}_a$ in the more general case where $a \co \dom a \subset H \to H$ is a closed densely defined operator with $\rho(a) \not=\emptyset$ and affiliated with $\cal{M}$. If $z \in \rho(a)$, this closed densely defined operator is defined by
\begin{equation}
\label{7mult0}
\cal{L}_a 
\ov{\mathrm{def}}{=} z - \cal{L}_{R(z,a)}^{-1} \co D \to \L^p(\cal{M})
\end{equation}
where $D$ is the range of the injective operator $\cal{L}_{R(z,a)}$ with dense range. This definition does not depend on $z$. Moreover $\rho(a) \subset \rho(\cal{L}_a)$ and we have
\begin{equation*}
\label{7mult}
R(z,\cal{L}_a) 
= \cal{L}_{R(z,a)},\quad z\in \rho(a).
\end{equation*}
Similarly, we can define the right multiplication operator $\cal{R}_a$\label{def-right-mult} by $a$.

If $a$ and $b$ are two self-adjoint operators affiliated with $\cal{M}$, it is proved in \cite[p.~86]{JMX06} that the intersection $\dom \cal{L}_a \cap \dom \cal{R}_b$ is a dense subspace of the Banach space $\L^p(\cal{M})$ and that the difference operator
\begin{equation*}
\label{}
\cal{L}_a - \cal{R}_b \co \dom \cal{L}_a \cap \dom \cal{R}_b \to \L^p(\cal{M}), x \mapsto \cal{L}_a(x)-\cal{R}_b(x)
\end{equation*}
is closable. Following \cite[p.~87]{JMX06}, we define the closed operator $\Ad_{(a,b)}$ as the closure of $\cal{L}_a-\cal{R}_b$, i.e.~
\begin{equation}
\label{7ad}
\Ad_{(a,b)} 
\ov{\mathrm{def}}{=} \ovl{\cal{L}_a -\cal{R}_b}.
\end{equation}
Note that by \cite[Lemma 8.10 p.~87]{JMX06}, the operator $\i\Ad_{(a,b)}$ generates a strongly continuous group of isometries $(U_t)_{t \in \R}$ acting on the Banach space $\L^p(\cal{M})$, where 
\begin{equation*}
\label{}
U_t \co \L^p(\cal{M}) \to \L^p(\cal{M}), x \mapsto e^{\i t a}x e^{-\i t b}
\end{equation*}
for any $t \in \R$.

Now, consider two finite commuting families $(a_1,\ldots,a_n)$ and $(b_1,\ldots,b_n)$ of self-adjoint operators affiliated with $\cal{M}$, i.e.~$a_ia_j = a_ja_i$ and $b_ib_j=b_jb_i$ for any $1 \leq i,j \leq n$. Suppose that $1 < p < \infty$. For any $1 \leq j \leq n$, we consider the closed operator 
\begin{equation*}
\label{}
A_j \ov{\mathrm{def}}{=} \Ad_{(a_j,b_j)}
\end{equation*}
acting on the Banach space $\L^p(\cal{M})$. As observed in the proof of \cite[Theorem 8.12 p.~87]{JMX06}, the operators $(A_1,\ldots,A_n)$ are pairwise commuting in the resolvent sense, i.e.
\begin{equation*}
\label{}
R(z_1,A_i)R(z_2,A_j)
=R(z_2,A_j)R(z_1,A_i), \quad z_1 \in \rho(A_i) ,z_2 \in \rho(A_j), i,j \in \{1,\ldots n\}.
\end{equation*}
With the results stated in \cite[p.~176]{Dav07}, it is easy to deduce that the strongly continuous groups $(e^{\i t A_1})_{t \in \R},\ldots, (e^{\i t A_n})_{t \in \R}$ commute. So, we have a $0_n$-Weyl tuple in the sense of Definition \ref{defi-Weyl-tuple}. If $\cal{M}$ is \textit{injective}, we deduce by Corollary \ref{cor-Hormander-calculus-for-Theta-Weyl-tuple-on-LpM} that the <<sum of squares>> 
\begin{equation*}
\label{}
\cal{A} 
\ov{\mathrm{def}}{=} A_1^{2}+\cdots +A_n^{2}
\end{equation*}
admits a bounded H\"ormander $\Hor^s_2(\R_+^*)$ functional calculus of order $s > \frac{n}{2} + \frac12$ on the Banach space $\L^p(\cal{M})$, recovering and improving the boundedness of the $\H^\infty(\Sigma_\theta)$ functional calculus for any angle $\theta > 0$ provided by \cite[Theorem 8.12 p.~87]{JMX06}.

We obtain another generalization if we assume that the strongly continuous groups of unitaries $(e^{\i t a_1})_{t \in \R},\ldots,(e^{\i t a_n})_{t \in \R}$ admit a commutation relation
\[ e^{\i t a_j} e^{\i s a_k} = e^{\i t s \Theta^{(1)}} e^{\i s a_k} e^{\i t a_j}\]
and similarly,
\[ e^{\i t b_j} e^{\i s b_k} = e^{\i t s \Theta^{(2)}} e^{\i s b_k} e^{\i t b_j}.\]
In this case, the strongly continuous groups $(e^{\i t A_1})_{t \in \R},\ldots,(e^{\i t A_n})_{t \in \R}$ form a $(\Theta^{(1)} - \Theta^{(2)})$-Weyl tuple over the Banach space $\L^p(\cal{M})$.
\end{example}


\begin{example}\normalfont
\label{rem-Schur-multipliers}
The preceding example applies in particular to Schur multipliers. Namely, we consider the case of the von Neumann algebra $\cal{M} = \B(\L^2(\R))$ of bounded operators on the (complex) Hilbert space $\L^2(\R)$. Recall that the affiliated operators with this von Neumann algebra are precisely the closed densely defined operators acting on the Hilbert space $\L^2(\R)$. If $a,b \co \R \to \R$ are some real-valued measurable functions, then by  the multiplication operators $\Mult_{a} \co \L^2(\R) \to \L^2(\R)$, $f \mapsto af$ and $\Mult_{b} \co \L^2(\R) \to \L^2(\R)$, $f \mapsto bf$ are self-adjoint operators affiliated with the von Neumann algebra $\B(\L^2(\R))$. We obtain that the operator $A_1 \ov{\mathrm{def}}{=} \Ad_{(a,b)}$ is a $0_1$-Weyl tuple over the Schatten space $S^p_{\R} \ov{\mathrm{def}}{=} S^p(\B(\L^2(\R)))$ in the sense of Definition \ref{defi-Weyl-tuple} for any $1 < p < \infty$. Moreover, by an <<unbounded version>> of \cite[Lemma 5.1]{Arh24}, the closed unbounded operator $A_1$ is a Schur multiplier with symbol $a-b \co \R \to \R$. So Corollary \ref{cor-Hormander-calculus-for-Theta-Weyl-tuple-on-LpM} yields a H\"ormander $\Hor^s_2(\R^*_+)$ functional calculus of order $s > 1$ for the Schur multiplier with symbol $(a-b)^2$.

In particular, with the choice $a(x)=x$ and $b(y)=y$, Corollary \ref{cor-Hormander-calculus-for-Theta-Weyl-tuple-on-LpM} implies that if $f \in \Hor^s_2(\R^*_+)$ then the symbol 
\begin{equation*}
\label{}
\psi \co \R \times \R \mapsto \mathbb{C}, (x,y) \mapsto f((x-y)^2)
\end{equation*}
induces a bounded Schur multiplier $M_\psi \co S^p_{\R}\to S^p_{\R}$ on the Schatten class $S^p_{\R}$. Actually, as stated in \cite[Proposition 4.8 (4)]{Kri09}, the mapping $f \mapsto f((\cdot)^2)$ is an isomorphism of $\Hor^s_2(\R^*_+)$ for any $s > \frac12$. So the preceding assertion implies that if $s > 1$ each $f \in \Hor^s_2(\R^*_+)$ induces a bounded Schur multiplier $M_\psi \co S^p_{\R}\to S^p_{\R}$ with $\psi \co \R \times \R \mapsto \mathbb{C}$, $(x,y) \mapsto f(|x-y|)$.
Moreover, if $f : \R \to \C$ is a function such that both $f 1_{\R_+}$ and $f(-\cdot) 1_{\R_+}$ belong to $\Hor^s_2(\R^*_+)$, then 
\begin{equation*}
\label{}
\psi(x,y) = f(x-y) = 1_{x-y \geq 0} f(|x-y|) + 1_{x-y < 0} f(-|x-y|)
\end{equation*}
induces a bounded Schur multiplier $M_\psi \co S^p_{\R} \to S^p_{\R}$.
Indeed, the symbols $1_{x-y \geq 0}$ and $1_{x-y < 0}$ induce $S^p$ (completely) bounded operators by the transference \cite{CdlS15} to the Fourier multiplier of the Hilbert transform which is indeed completely bounded on $\L^p(\R)$ since $S^p$ has the UMD property \cite[Proposition 5.4.2]{HvNVW16}.
Note that this is a particular case of the recent result \cite[Theorem A']{CGPT23} on bounded Schur multipliers on $S^p_{\R^n} \ov{\mathrm{def}}{=} S^p(\B(\L^2(\R^n)))$ associated to non-necessarily Toeplitz symbols as we consider hereabove.




The same method applies to discrete Schur multipliers. If $\cal{M} = \B(\ell^2)$, consider some diagonal matrices $a_1,\ldots,a_n,b_1,\ldots,b_n$ with real entries $a_{j,k}$ and $b_{j,k}$ on the diagonal. Then these matrices are self-adjoint operators affiliated with the von Neumann algebra $\B(\ell^2)$. We obtain that the operators $A_j = \Ad_{(a_j,b_j)}$ form a $0_n$-Weyl tuple over the Schatten space $S^p \ov{\mathrm{def}}{=}  S^p(\ell^2)$ in the sense of Definition \ref{defi-Weyl-tuple} for $1 < p < \infty$.
Moreover, the operator $A_j$ is a Schur multiplier with symbol $(a_{j,k} - b_{j,l})_{kl}$. So Corollary \ref{cor-Hormander-calculus-for-Theta-Weyl-tuple-on-LpM} yields a H\"ormander $\Hor^s_2(\R^*_+)$ functional calculus of order $s > \frac{n}{2} + \frac12$ for the Schur multiplier with symbol
\[ 
\left(\sum_{j = 1}^n (a_{j,k} - b_{j,l})^2 \right)_{kl} .
\]
In particular, with $n=1$ and $a_{1,k} = b_{1,k} = k$, we obtain that a Toeplitz Schur multiplier with symbol $f((k-l)^2)$ induces a bounded operator on $S^p$ provided that $f \in \Hor^s_2(\R^*_+)$ with $s>1$. As in the continuous case, we obtain that a (Toeplitz) Schur multiplier with symbol $f(k-l)$ is bounded on $S^p$ for any function $f \in \Hor^s_2(\R^*_+)$ if $s > 1$.

\end{example}

\section{Bochner-Riesz means}
\label{sec-Bochner-Riesz-means}

A prominent example of H\"ormander spectral multipliers are the Bochner-Riesz means, which we introduce here. We will use the notation $t_+ \ov{\mathrm{def}}{=} \max\{0,t\}$ for any $t \in \R$. For any exponent $\nu \geq 0$ and any parameter $R > 0$, we can define the function $\sigma_R^\nu \co \R^+ \to \R$ by
\begin{equation*}
\label{Bochner-means}
\sigma_R^\nu(\lambda)
\ov{\mathrm{def}}{=} \bigg(1 - \frac{\lambda}{R}\bigg)_+^\nu 
= \begin{cases} (1 - \tfrac{\lambda}{R})^\nu &  \text{ if }0 \leq \lambda \leq R \\ 
0 & \text{ if } \lambda > R 
\end{cases}.
\end{equation*}
See Figure 2.
\begin{figure}[h!]
\centering
\begin{tikzpicture}
  \begin{axis}[
    width=12cm,
    height=7cm,
    domain=0:2,
    samples=200,
    xlabel={$\lambda$},
    ylabel={$\sigma_1^\nu(\lambda)$},
		ylabel style={rotate=-90},
    legend style={at={(0.5,-0.20)}, anchor=north, legend columns=4},
    grid=major,
  ]
    \addplot[blue, thick]    {(x <= 1) * (1 - x)^0.5};
    \addlegendentry{$\nu = 0.5$}
    
    \addplot[red, thick, dashed] {(x <= 1) * (1 - x)^1};
    \addlegendentry{$\nu = 1$}
    
    \addplot[green!60!black, thick, dotted] {(x <= 1) * (1 - x)^2};
    \addlegendentry{$\nu = 2$}
    
    \addplot[purple, thick, dashdotted] {(x <= 1) * (1 - x)^5};
    \addlegendentry{$\nu = 5$}
  \end{axis}
\end{tikzpicture}

Figure 2: $\sigma_R^\nu(\lambda)$ for different values of $\nu$ with $R=1$.
\end{figure}

Then for any positive self-adjoint operator $B$ we can define the operator $\sigma_R^\nu(B)$ using the functional calculus of Hilbert and von Neumann. This operator is called the Riesz or the Bochner-Riesz means of order $\nu$ with respect to $R$ corresponding to the operator $B$. The case $\nu = 0$ corresponds to the spectral projector $1_{[0, R]}(B)$. For any $\nu > 0$, the operator $\sigma_R^\nu(B)$ can be interpreted as a smoothed version of the spectral projector, with larger values of $\nu$ leading to higher degrees of smoothness. If the operator $B$ acts on a (noncommutative) $\L^2$-space $\L^2(\cal{M})$ and if $1 < p < \infty$ with $p \not=2$, a fundamental question in the theory of Bochner–Riesz means is to determine the smallest critical exponent $\nu_{\mathrm{crit}}$ for which the Bochner-Riesz means $\sigma_R^\nu(B)$ induces a bounded operator on the Banach space $\L^p(\cal{M})$ with its norm uniformly bounded with respect to the parameter $R$ for any $\nu > \nu_{\mathrm{crit}}$. A substantial amount of research has been devoted to investigating the Bochner–Riesz means for various operators $B$. For further details, see \cite{COSY16}, \cite{CDHLY21}, \cite{DaC87}, \cite[Chapter 5]{Gra14b}, \cite{GOWWZ21}, \cite{JLR22}, \cite{Lai22}, \cite{LuY13}, \cite{Tha91}, \cite{Xu24} and the references therein.  


In the case of the Laplacian $B=-\Delta$ on the Hilbert space $\L^2(\R^n)$, the Bochner-Riesz means $\sigma_R^\nu(-\Delta)$ of order $\nu$ with respect to the parameter $R > 0$ is the classical Bochner-Riesz means introduced in \cite[Definition 5.2.1 p.~339]{Gra14b} and defined by
\begin{eqnarray}
\label{e1.1}
{S^{\nu}_{R}f}(x)
=\frac{1}{(2\pi)^n}\int_{\R^n} e^{\i x\cdot \xi} \left(1-{|\xi|^2\over R}\right)_+^{\nu} \hat{f}(\xi) \d\xi,  \quad \xi \in \R^n,
\end{eqnarray}
where $\hat{f}$ denotes the Fourier transform of the function $f$. It is worth noting that  the $\L^p$-boundedness of the operator $S^\nu_1$ implies by \cite[Exercise 5.2.1 p.~355]{Gra14b} the convergence $S^\nu_R f \to f$ as $R \to \infty$ in the Banach space $\L^p(\R^n)$ for any $f \in \L^p(\R^n)$.

Suppose that $1 < p < \infty$ with $p \not=2$. The fundamental problem known as the Bochner-Riesz conjecture is that the family $(S^\nu_R)_{R > 0}$ of operators induces a uniformly bounded family of operators on the Banach space $\L^p(\R^n)$ if and only if
\begin{eqnarray} 
\label{condition-Bochner-Riesz}
\nu
> \max\bigg\{ n\left|{1\over 2}-{1\over p}\right|-{1\over 2}, 0\bigg\}.
\end{eqnarray}
Herz \cite{Her54} established that the condition \eqref{condition-Bochner-Riesz} on $\nu$ is a necessary requirement for the $\L^p$-boundedness of $S^\nu$. Carleson and Sj\"olin \cite{CaS72} (see also \cite[Theorem 5.2.4 p.~342]{Gra14b}) later confirmed the conjecture in the case $n = 2$. However, for $n \geq 3$, the problem remains unresolved.

It is important to observe that the function $\sigma_R^\nu$ belongs to the space $\Hor^s_2(\R_+^*)$ for any $\nu > s  - \frac12$. 

\begin{prop}
\label{prop-Bochner-Riesz-is-Hor}
Let $s > \frac12$ and $\nu > s - \frac12$.
Then the function $\sigma_R^\nu$ belongs to the space $\Hor^s_2(\R_+^*)$ and we have
\[ 
\sup_{R > 0} \norm{\sigma_R^\nu}_{\Hor^s_2(\R_+^*)} < \infty.
\]
\end{prop}

\begin{proof}
We recall from \eqref{hor-2} that
\[ 
\norm{f}_{\Hor^s_2(\R_+^*)} = \sup_{t > 0} \norm{\eta f(t\,\cdot)}_{\W^{s,2}(\R)}, 
\]
where $\eta \co ]0,\infty[ \to \R$ is a fixed non-zero function of class $\mathrm{C}^\infty$ with compact support.
Moreover, as observed right after, we have 
\begin{equation*}
\label{}
\norm{\sigma_R^\nu}_{\Hor^s_2(\R_+^*)} 
= \norm{\sigma_R^\nu(R\,\cdot)}_{\Hor^s_2(\R_+^*)} 
= \norm{\sigma_{1}^\nu}_{\Hor^s_2(\R_+^*)}
\end{equation*}
for all $R > 0$, so that it suffices to prove that the function $\sigma_1^\nu$ has finite $\Hor^s_2(\R_+^*)$ norm.
Next, as observed in \cite[Proposition 4.8 (4), Proposition 4.11]{Kri09}, the mapping $f\mapsto f((\cdot)^{\frac12})$ is an isomorphism of the space $\Hor^s_2(\R_+^*)$, so it suffices to show that $\sigma^\nu(\lambda) \ov{\mathrm{def}}{=} (1-\lambda^2)^\nu_+$ belongs to the space $\Hor^s_2(\R_+^*)$.
We choose $\eta$ such that $\supp \eta \subseteq ]\frac12,2[$.
Let us consider the cases $t \leq \frac12,t \geq 2$ and $\frac12 < t < 2$ separately. Note that the support is $[0,\frac{1}{t}]$.

\paragraph{Case $t \leq \frac12$}
Then the function $\sigma^\nu(t\,\cdot)$ is non-smooth only at $\frac1t \geq 2$ and $\sigma^\nu(t\,\cdot)\eta$ belongs to the space $\mathrm{C}^\infty_c(]\frac12,2[)$.
This case is a non-critical case.
Namely, we can control the Sobolev norm, with any integer $k \in \N$ such that $k \geq s$, by
$\norm{\sigma^\nu(t\,\cdot)\eta}_{\W^{s,2}(\R)} \lesssim \norm{\sigma^\nu(t\,\cdot)\eta}_{\W^{k,2}(\R)}$. 
Now, for any integer $l \in \llbracket 0, k\rrbracket$, we have using Leibniz's rule for products
\begin{equation}
\label{inter-7890}
\frac{\d^l}{\d x^l} ( \sigma^\nu(tx) \eta(x) ) 
= \sum_{j=0}^l {l \choose j} t^j (\sigma^\nu)^{(j)}(tx) \eta^{(l-j)}(x).
\end{equation}
Note that the function $x \mapsto (\sigma^\nu)^{(j)}(tx)$ is bounded on $\supp \eta^{(l-j)} \subseteq ]\frac12,2[$ uniformly in $t \leq \frac12$ (since $tx$ remains in the interval $[0,1]$). Thus, we obtain the estimate
\[ 
\norm{x \mapsto \frac{\d^l}{\d x^l} (\sigma^\nu(tx) \eta(x))}_{\L^2(\R)} 
\ov{\eqref{inter-7890}}{\lesssim} \sum_{j=0}^l t^j \bnorm{(\sigma^\nu)^{(j)}(t\,\cdot)}_{\L^\infty(\supp \eta^{(l-j)})} 
\bnorm{\eta^{(l-j)}}_{\L^2(\R)} 
\lesssim 1.
\]
We finally deduce that
\[
\norm{\sigma^\nu(t\, \cdot)\eta}_{\W^{k,2}(\R)} 
\cong \sum_{l=0}^k \norm{\frac{\d^l}{\d x^l}(\sigma^\nu(t\, \cdot)\eta)}_{\L^2(\R)} 
\lesssim 1.
\]

\paragraph{Case $t \geq 2$}
Then $\supp \sigma^\nu(t\,\cdot)$ is a subset of $[0,\frac12]$ and $\sigma^\nu (t\,\cdot) \eta = 0$.

\paragraph{Case $\frac12 < t < 2$}
Note first the classical estimate
\begin{equation}
\label{homo-Sobolev}
\norm{f(t\,\cdot)}_{\W^{s,2}(\R)} 
\lesssim (t^{-\frac12} + t^{s-\frac12}) \norm{f}_{\W^{s,2}(\R)},
\end{equation}
which can be easily seen from the Sobolev norm expressed in terms of the Fourier transform of $f$ together with the formula $\widehat{f(t\,\cdot)}(\xi) = \frac1t \hat{f}(\frac{\xi}{t})$. Indeed, we have
\begin{align*}
\MoveEqLeft
\norm{f(t\, \cdot)}_{\W^{s,2}(\R)}^2 
\ov{\eqref{Sobolev-norm}}{=} \int_{\R} (1 + |\xi|^2)^s \big|\widehat{f(t\, \cdot)}(\xi)\big|^2 \d \xi 
= \int_{\R} (1 + |\xi|^2)^s \left|\frac{1}{t} \hat{f}\left(\frac{\xi}{t}\right)\right|^2 \d\xi \\
&= \frac{1}{t^2} \int_{\R} (1 + |\xi|^2)^s \left|\hat{f}\left(\frac{\xi}{t}\right)\right|^2 \d\xi.
\end{align*}
Using the change of variable $u = \frac{\xi}{t}$, we obtain
\[
\norm{f(t \, \cdot)}_{\W^{s,2}(\R)}^2 
= \frac{1}{t} \int_{\R} (1 + t^2 u^2)^s\, |\hat{f}(u)|^2 \d u.
\]
Using the inequality $
(1 + t^2 u^2)^s 
\leq (1 + u^2)^s  \max\{1, t^{2s}\}$, we get
\begin{align*}
\MoveEqLeft
\norm{f(t \, \cdot)}_{\W^{s,2}(\R)}^2 
\lesssim \frac{1}{t}  \max\{1, t^{2s}\} \int_{\R} (1 + u^2)^s\, |\hat{f}(u)|^2 \d u \\
&\ov{\eqref{Sobolev-norm}}{=} \max\big\{t^{-1}, t^{2s-1}\big\} \norm{f}_{\W^{s,2}(\R)}^2
\leq (t^{-1}+ t^{2s-1}) \norm{f}_{\W^{s,2}(\R)}^2.         
\end{align*}
Taking the square root, we obtain \eqref{homo-Sobolev}.

According to \cite[6.11 (a) p.~162]{Ste70} and \cite[Theorem 1 p.~222, Proposition p.~14]{RuSi96}, $\W^{s,2}(\R)$ is an algebra with respect to  pointwise multiplication, so that
\[ 
\norm{\sigma^\nu(t\,\cdot)\eta}_{\W^{s,2}(\R)} 
\lesssim \norm{\sigma^\nu(t\,\cdot)}_{\W^{s,2}(\R)} \norm{\eta}_{\W^{s,2}(\R)} 
\lesssim \norm{\sigma^\nu(t\,\cdot)}_{\W^{s,2}(\R)}. 
\]
Since $\frac12 < t < 2$ in our case, this further estimates to $\norm{\sigma^\nu}_{\W^{s,2}(\R)}$. It remains to prove that the latter quantity is finite. According to \cite[B.5 p.~578]{Gra14a},
\[ 
\widehat{\sigma^\nu}(\xi) 
= 2^{\nu} \Gamma(\nu+1) \frac{J_{\nu+\frac12}(|\xi|)}{|\xi|^{\nu+\frac12}},
\]
where $\Gamma$ denotes the gamma function and $J_{\nu+\frac12}$ denotes the Bessel function of order $\nu+\frac12$. Since $|J_{\nu+\frac12}(x)| \lesssim x^{\nu+\frac12}$ for any $x \in (0,1]$ by \cite[B.5 p.~579]{Gra14a}, the Fourier transform $\widehat{\sigma^\nu}(\xi)$ is bounded for $|\xi| \leq 1$. Moreover, since $\sup_{x \geq 0} x^{\frac12} |J_{\nu+\frac12}(x)| < \infty$ according to \cite[Exercise 24 p.~238]{AAR99}, we have the estimate $|\widehat{\sigma^\nu}(\xi)| \lesssim \frac1{|\xi|^{\nu+1}}$ for $|\xi| \geq 1$. Consequently, when $\xi \to \infty$ we have
\begin{equation*}
\label{}
|(1 + |\xi|^2)^{\frac{s}{2}}\widehat{\sigma^\nu}(\xi)|^2
\lesssim
\frac{1}{|\xi|^{2\nu-2s+2}}.
\end{equation*}
All in all, the function $(1 + |\xi|^2)^{\frac{s}{2}}\widehat{\sigma^\nu}(\xi)$ belongs to the space $\L^2(\R)$ for any $\nu > s - \frac12$ and the proof is complete.
\end{proof}

Proposition \ref{prop-Bochner-Riesz-is-Hor} provides a motivation for the H\"ormander functional calculus.
In this spirit, we obtain in Proposition \ref{prop-Riesz-means} a consequence of Corollary \ref{cor-Hormander-calculus-for-Theta-Weyl-tuple-on-LpM}, which gives the uniformly boundedness of some Bochner-Riesz means and the strong convergence to the identity as $R \to \infty$. 
The proof relies on the concept of a dyadic partition of unity from \cite[Definition 2.2 p.~3]{KrW16}, which we briefly recall below.

\begin{defi}
\label{defi-dyadic-partition-of-unity}
Let $\dot\varphi$ be a function of class $\mathrm{C}^\infty$ on $(0,\infty)$ whose compact support $\supp \dot\varphi$ is included in $[\frac12,2]$ and $\sum_{n = - \infty}^\infty \dot\varphi(2^{-n} t) = 1$ for all $t > 0$.
Then for any integer $n \in \Z$ we put $\dot\varphi_n \ov{\mathrm{def}}{=} \dot\varphi(2^{-n}\cdot)$ and we say that $(\dot\varphi_n)_{n \in \Z}$ is a dyadic partition of unity.
\end{defi}

We also recall the following result, which will be used in the proof. It is a convergence lemma for the H\"ormander functional calculus and a variant of~\cite[Proposition 3.4, p.~7]{KrW16}.
 Its proof follows from the fact that the Mikhlin norm below dominates the norm of the space $\mathcal{B}^\beta_{\infty,\infty}$, defined in \cite[Definition 4.6 p.~55]{Kri09}, there when $N > \beta$, together with Lemma \ref{lem-Hor-calculus-reflexive}.

\begin{lemma}
\label{lem-convergence-lemma-Hor-calculus}
Let $X$ be a reflexive Banach space. Consider a sectorial operator $A$ acting on $X$ admitting a bounded $\Hor^s_2(\R^*_+)$ functional calculus. Let $N$ be an integer with $N > s$. Consider a sequence $(f_n)$ of complex functions of class $\mathrm{C}^N$ on $[0,\infty)$ such that
\begin{align*}
(a) & \sup_{n} \sup \left\{ \left|\lambda^k\frac{\d^k f_n}{\d \lambda^k}(\lambda) \right| : \lambda > 0,\: k = 0,\ldots,N\right\}, \sup_n |f_n(0)| < \infty,\\
(b) & f_n(\lambda) \to f(\lambda) \text{ when $n \to \infty$ for any $\lambda \geq 0$}.
\end{align*}
Then the function $f|_{(0,\infty)}$ belongs to the space $\Hor^s_2(\R^*_+)$ and for all $x \in X$, we have $f_n(A)x \to f(A)x$ when $n \to \infty$.

In particular, if $(\dot\varphi_n)_{n \in \Z}$ is a dyadic partition of unity, then for all $x \in \overline{\Ran A}$, we have
\[ 
\sum_{n \in \Z} \dot\varphi_n(A)x = x \quad( \text{convergence in }X). 
\]
\end{lemma}

We are now in a position to prove the announced result.

\begin{prop}
\label{prop-Riesz-means}
Suppose that $1 < p < \infty$. Let $A = (A_1,\ldots,A_n)$ be a $\Theta$-Weyl tuple on a noncommutative $\L^p$-space associated with an injective von Neumann algebra $\cal{M}$ with separable predual, where $1 < p < \infty$. Let $\A -\alpha \Id$ be the closure of the operator $\sum_{k=1}^n A_k^2 - \alpha \Id$ as in Corollary \ref{cor-Hormander-calculus-for-Theta-Weyl-tuple-on-LpM}.
For any $\nu > \frac{n}{2}$, the Bochner-Riesz means $\sigma_R^\nu(\A - \alpha\Id)$ are uniformly bounded on $\L^p(\cal{M})$, i.e. 
\[
\sup_{R > 0} \bnorm{\sigma_R^\nu(\A-\alpha \Id)}_{\L^p(\cal{M}) \to \L^p(\cal{M})} 
< \infty. 
\]
Consequently, for any $f \in \L^p(\cal{M})$ we have $\sigma_R^\nu(\A-\alpha \Id)f \to f$ in the Banach space $\L^p(\cal{M})$ as $R \to \infty$.
\end{prop}

\begin{proof}
As we said previously the function $\sigma_R^\nu$ belongs to the space $\Hor^s_2(\R_+^*)$ for any $\nu > s  - \frac12$. Thus according to Corollary \ref{cor-Hormander-calculus-for-Theta-Weyl-tuple-on-LpM}, the operator $\sigma_R^\nu(\A - \alpha \Id)$ is bounded on the Banach space $\L^p(\cal{M})$ for any 
\begin{equation*}
\label{}
\nu > \frac{n}{2} + \frac12 - \frac12 = \frac{n}{2}.
\end{equation*}
Then since the $\Hor^s_2(\R_+^*)$ norm of \eqref{hor-2} is invariant under dilations $f \mapsto f(\cdot/R)$ as shown in \cite[Lemma 3.2 (3)]{KrW18}, we conclude the stated uniform boundedness. We turn to the strong convergence as $R \to \infty$. Let $(\dot\varphi_n)_{n \in \Z}$ be a dyadic partition of unity in the sense of Definition \ref{defi-dyadic-partition-of-unity}. For any integer $m \geq 0$, let 
\begin{equation*}
\label{}
\Phi_m 
\ov{\mathrm{def}}{=} \sum_{n = - m}^m \dot\varphi_n.
\end{equation*}
According to Lemma \ref{lem-convergence-lemma-Hor-calculus}, we have $\Phi_m(\A - \alpha \Id)f \to f$ as $m \to \infty$, for any $f \in \ovl{\Ran(\A-\alpha\Id)} \subseteq \L^p(\cal{M})$. Now, choose for $R \geq 4$ the positive integer $m_R$ such that 
\begin{equation*}
\label{}
\frac14 R \leq 2^{m_R+1} \leq \frac12 R.
\end{equation*}
Note that the support of the function $\Phi_{m_R}$ is contained in $[2^{-m_R-1}, 2^{m_R+1}]$. So the only point $0 < \lambda = R$ where the function $\sigma_R^\nu$ is non-smooth does not belong to this support.
We want to apply Lemma \ref{lem-convergence-lemma-Hor-calculus} to the spectral multipliers $\sigma_R^\nu \Phi_{m_R}$, with limit $1_{(0,\infty)}$ when $R \to \infty$.
Condition (b) of this result, $\sigma_R^\nu(\lambda) \Phi_{m_R}(\lambda) \to 1$ for all $\lambda > 0$ and $\sigma_R^\nu(0) \Phi_{m_R}(0) = 0$, is easily checked.
For condition (a), if $R \geq 4$, we note that for any integer $N > s$ we have
\begin{align*}
\MoveEqLeft
\max_{k=0,\ldots,N} \sup_{\lambda > 0} \left| \lambda^k \frac{\d^k}{\d \lambda^k}(\sigma_R^\nu \Phi_{m_R})(\lambda)\right| \\
& \lesssim \max_{k=0,\ldots,N} \sup_{\lambda \in [2^{-m_R-1},2^{m_R+1}]} \left|\lambda^k \frac{\d^k}{\d\lambda^k}(\sigma_R^\nu \Phi_{m_R})(\lambda)\right| 
\leq C.
\end{align*}
For any $f \in \ovl{\Ran(\A-\alpha\Id)}$, we have $1_{(0,\infty)}(A)f = f$ and consequently
\begin{align*}
\MoveEqLeft
\norm{\sigma_R^\nu(\A - \alpha \Id)f - f}_{\L^p(\cal{M})} \\
&\leq \norm{\sigma_R^\nu(\A-\alpha\Id)f -\sigma_R^\nu(\A - \alpha \Id)\Phi_{m_R}(\A  -\alpha\Id)f }_{\L^p(\cal{M})} \\
& +\norm{\sigma_R^\nu(\A - \alpha \Id)\Phi_{m_R}(\A - \alpha\Id)f - f}_{\L^p(\cal{M})} \\
& \lesssim \norm{f-\Phi_{m_R}(\A  -\alpha\Id)f}_{\L^p(\cal{M})} \\
&+\norm{\sigma_R^\nu(\A - \alpha \Id)\Phi_{m_R}(\A - \alpha\Id)f - 1_{(0,\infty)}(A)f}_{\L^p(\cal{M})}  \xra[R \to \infty]{} 0,
\end{align*}
according to Lemma \ref{lem-convergence-lemma-Hor-calculus} applied to both summands (note that the parameter $R$ belongs to a metrizable space).
Moreover, if $f \in \ker (\A-\alpha \Id)$, then $\sigma_R^\nu(\A - \alpha \Id)f = \sigma_R^\nu(0)f = f$, so that trivially $\sigma_R^\nu(\A - \alpha \Id)f \to f$ as $R \to \infty$.
The proposition now follows from the space decomposition $\L^p(\cal{M}) = \ker(\A-\alpha\Id) \oplus \overline{\Ran(\A-\alpha\Id)}$ provided by \eqref{decompo-reflexive}.
\end{proof}

\begin{remark} \normalfont
The argument applies to any operator that admits a bounded H\"ormander functional calculus of order strictly greater than $\frac{n}{2}+\frac{1}{2}$.
\end{remark}

\section{Functional calculus for the harmonic oscillator on the Moyal plane}
\label{sec-quantum-harmonic-oscillator}


Let $d \geq 1$ be an integer. Consider a skew-symmetric matrix $\Theta \in \M_d(\R)$. Recall that the $d$-dimensional quantum Euclidean space can be defined as the twisted group von Neumann algebra $\VN(\R^d,\sigma_\Theta)$, where $\sigma_\Theta \co \R^d \times \R^d \to \T$, $(s,t) \mapsto \e^{\frac{1}{2}\i\langle s, \Theta t\rangle}$ is the 2-cocycle defined in \eqref{def-cocycle-intro}. This von Neumann algebra, denoted $\R^d_\Theta$\label{def-quatum-Euclidean} or $\L^\infty(\R^d_\Theta)$, is generated by the unitary operators 
\begin{equation*}
\label{}
\lambda_{\Theta,s} \co \L^2(\R^d) \to \L^2(\R^d),
\end{equation*}
where $s \in \R^d$, defined by
\begin{equation}
\label{def-lambda-theta}
(\lambda_{\Theta,s} \xi)(t) 
\ov{\eqref{def-lambda-sigma}}{=} e^{\frac{1}{2}\i\langle s, \Theta t\rangle} \xi(t-s),\quad s,t \in \R^d, \xi \in \L^2(\R^d).
\end{equation}
We have the following relation
\begin{equation}
\label{equ-lem-commutation-relations-lambda-theta}
\lambda_{\Theta,s} \lambda_{\Theta,t}
\ov{\eqref{product-adjoint-twisted}}{=} e^{\frac{1}{2}\i \langle s , \Theta t \rangle} \lambda_{\Theta,s+t}, \quad s,t \in \R^d.
\end{equation}
This twisted group von Neumann algebra is also generated by the operators 
\begin{equation*}
\label{def-lambda-theta-f}
\lambda_\Theta(f) \co \L^2(\R^d) \to \L^2(\R^d),
\end{equation*}
where $f \in \L^1(\R^d)$, defined by 
\begin{equation*}
\label{}
\lambda_\Theta(f) \ov{\mathrm{def}}{=} \int_{\R^d} f(s) \lambda_{\Theta,s} \d s.
\end{equation*}
The Schwartz space $\cal{S}(\R^d_\Theta)$ is defined as the subspace of $\L^\infty(\R^d_\Theta)$ which is the image of the classical Schwartz space $\cal{S}(\R^d)$ under the map $\lambda_\Theta$, that is
\begin{equation*}
\label{Schwartz-Theta}
\cal{S}(\R^d_\Theta) 
\ov{\mathrm{def}}{=} \{ \lambda_\Theta(f) :\: f \in \cal{S}(\R^d) \}.
\end{equation*}
Recall that by  \cite[Proposition 2.4]{GJM22} and \cite[Proposition 1.1]{GJP21} the complex functional $\tau_\Theta \co \cal{S}(\R^d_\Theta) \to \C$, $\lambda_\Theta(f) \mapsto f(0)$ extends uniquely to a normal semifinite faithful trace on the von Neumann algebra $\L^\infty(\R^d_\Theta)$. This trace allows us to define the associated noncommutative $\L^p$-spaces $\L^p(\R^d_\Theta)$.

For any $t \in \R^d$, we consider the multiplication operator $\mathcal{T}_t \co \L^2(\R^d) \to \L^2(\R^d)$, $f \mapsto e^{\i \langle t , (\cdot) \rangle} f(\cdot)$\label{def-mathcalTt}.

\begin{prop}
\label{Prop-representation-quantum-translation}
For any $s,t \in \R^d$, we have
\begin{equation*}
\label{}
\mathcal{T}_t \lambda_{\Theta,s}\mathcal{T}_t^*
=e^{\i \la t,s \ra} \lambda_{\Theta,s}.
\end{equation*}
\end{prop}

\begin{proof}
For any $r,s,t \in \R^d$ and any $\xi \in \L^2(\R^d)$, we have
\begin{align*}
\MoveEqLeft
\big(\mathcal{T}_t \lambda_{\Theta,s}\mathcal{T}_t^*(\xi) \big)(r) 
= e^{\i \langle t, r \rangle} \big(\lambda_{\Theta,s} \mathcal{T}_t^*(\xi)\big)(r) 
\ov{\eqref{def-lambda-theta}}{=} e^{\i \langle t, r \rangle} e^{\frac{1}{2}\i\langle s, \Theta r\rangle} \mathcal{T}_t^*(\xi)(r-s) \\
& = e^{\i \langle t, r \rangle} e^{\frac{1}{2}\i\langle s, \Theta r\rangle} e^{-\i \langle t , r-s \rangle} \xi(r-s) \\
&= e^{\i \langle t, s \rangle} e^{\frac{1}{2}\i\langle s, \Theta r\rangle} \xi(r-s) 
 \ov{\eqref{def-lambda-theta}}{=} e^{\i \langle t, s \rangle} \lambda_{\Theta,s}(\xi)(r).
\end{align*}
\end{proof}

Consequently, we have a spatial $*$-automorphism $\mathscr{T}_t \co \L^\infty(\R^d_\Theta) \to \L^\infty(\R^d_\Theta)$\label{mathscr{T}_t}, $\lambda_{\Theta,s} \mapsto e^{\i \la t,s \ra} \lambda_{\Theta,s}$ of $\L^\infty(\R^d_\Theta)$, hence a weak* continuous group of $*$-automorphisms by \cite[p.~238]{Tak03}. Note that this map is trace preserving since
\begin{align*}
\MoveEqLeft
\tau_\Theta(\mathscr{T}_t(\lambda_\Theta(f)))
=\tau_\Theta\big(\lambda_\Theta(e^{\i \la t,\cdot \ra} f)\big)
=f(0) 
=\tau_\Theta(\lambda_\Theta(f)).
\end{align*}
In particular, we obtain a strongly continuous group $(\mathscr{T}_{t,p})_{t \in \R^d}$ of isometries on each noncommutative $\L^p$-space $\L^p(\R^d_\Theta)$. We denote by $\partial_k$\label{def-partialk} the generator of the strongly continuous group $(\mathscr{T}_{(0,\ldots,t,\ldots,0),p})_{t \in \R}$.
Note that for any $s \in \R^d$ and any $t \in \R$ we have
\begin{equation}
\label{quantum-derivations}
e^{t \partial_k}\lambda_{\Theta,s} 
= \mathscr{T}_{t e_k}\lambda_{\Theta,s} 
= e^{\i \langle t e_k,s \rangle} \lambda_{\Theta,s} 
= e^{\i t s_k} \lambda_{\Theta,s}.
\end{equation}

For any $1 \leq k \leq d$ and any $t \in \R$ we let 
\begin{equation*}
\label{}
u_k(t) 
\ov{\mathrm{def}}{=} \lambda_{\Theta,(0,0,\ldots,t,\ldots,0)}.
\end{equation*}
Then $(u_k(t))_{t \in \R}$ is a one-parameter strongly continuous group of unitaries of $\L^\infty(\R_\Theta^d)$. We denote by $x_k$\label{def-xk} the self-adjoint operator satisfying $u_k(t)=e^{\i tx_k }$ for any $t \in \R$. By \cite[Proposition 8.5 (1) p.~81]{JMX06}, the unbounded operator $x_j$ is affiliated with the von Neumann algebra $\L^\infty(\R_\Theta^d)$. Using \cite[Proposition 8.5 (2) p.~81]{JMX06}, we see that the family of maps $\L^p(\R_\Theta^d) \to \L^p(\R_\Theta^d)$, $y \mapsto e^{\i t x_k} y$ (for $t \in \R)$ defines a strongly continuous bounded group of operators on $\L^p(\R_\Theta^d)$. Since the operators $e^{\i t x_k}$ are unitaries, the group consists in fact of isometries on $\L^p(\R_\Theta^d)$. We denote by $X_k$\label{def-Xk} the unbounded operator acting on $\L^p(\R_\Theta^d)$ such that
\begin{equation}
\label{def-Xk}
e^{\i t X_k}(y)
=e^{\i t x_k} y,\quad y \in \L^p(\R_\Theta^d), t \in \R.
\end{equation}
Using again \cite[Proposition 8.5 (2) p.~81]{JMX06}, we note that the generator of the group $(e^{\i t X_k})_{t \in \R}$ is the left multiplication operator by $\i x_k$ in the sense of \cite[(8.3) p.~80]{JMX06}. Note that the formula 
\eqref{def-Xk} implies that
\begin{equation}
\label{def-Xk-ek}
\e^{\i t X_k}(y)
= \lambda_{\Theta,t e_k}y, \quad y \in \L^p(\R^d_\Theta), t \in \R
\end{equation}
where $e_k$ is the $k^{\textrm{th}}$-vector of the canonical basis of the space $\R^d$.


The groups of operators associated with noncommutative spatial variables and momentum variables satisfy the following commutation relations.

\begin{prop}
\label{prop-quantum-commutation-rules}
For any $s,t \in \R$ and any $j,k \in \{1, \ldots, d \}$, we have
\begin{align}
e^{\i sX_j} e^{\i tX_k} & = e^{\i st \Theta_{jk}} e^{\i tX_k} e^{\i sX_j} \label{equ-quantum-commutation-1}, \\
 e^{s \partial_j} e^{t \partial_k} & = e^{t \partial_k} e^{s \partial_j} \label{equ-quantum-commutation-2}, \\
e^{t \partial_k} e^{\i sX_j} &= e^{\i t s \delta_{j=k}} e^{\i sX_j} e^{t \partial_k}. \label{equ-quantum-commutation-3}
\end{align}
\end{prop}

\begin{proof}
We start with \eqref{equ-quantum-commutation-1}. For any $y \in \L^p(\R^d_\Theta)$, we have on the one hand
\begin{align*}
\MoveEqLeft
e^{\i s X_j} e^{\i t X_k}y
\ov{\eqref{def-Xk-ek}}{=} \lambda_{\Theta,se_j} \lambda_{\Theta,t e_k} y 
\overset{\eqref{equ-lem-commutation-relations-lambda-theta}}{=}
e^{\frac{1}{2}\i \langle s e_j, \Theta te_k\rangle} \lambda_{\Theta,s e_j+t e_k} y 
=e^{\frac{1}{2}\i st \Theta_{jk}} \lambda_{\Theta,s e_j+t e_k} y.
\end{align*}
On the other hand, we have
\begin{align*}
\MoveEqLeft
e^{\i st \Theta_{jk}} e^{\i t X_k} e^{\i s X_j} y
\ov{\eqref{def-Xk-ek}}{=} e^{\i st \Theta_{jk}} \lambda_{\Theta,t e_k}\lambda_{\Theta,s e_j} y 
\overset{\eqref{equ-lem-commutation-relations-lambda-theta}}{=} e^{\i st \Theta_{jk}} e^{\frac{1}{2}\i \langle t e_k, \Theta se_j\rangle} \lambda_{\Theta,t e_k+s e_j} y \\
&=e^{\i st \Theta_{jk}} e^{\frac{1}{2}\i st \Theta_{kj}} \lambda_{\Theta,t e_k+s e_j} y
=e^{\frac{1}{2}\i st \Theta_{jk}} \lambda_{\Theta,s e_j+t e_k} y.
\end{align*}
Let us turn to the equality \eqref{equ-quantum-commutation-2}. For any $r \in \R^d$ we have 
\begin{equation*}
\label{}
e^{s \partial_j} e^{t \partial_k} \lambda_{\Theta,r}
\ov{\eqref{quantum-derivations}}{=} e^{\i t r_k} e^{s \partial_j} \lambda_{\Theta,r}
\ov{\eqref{quantum-derivations}}{=} e^{\i t r_k} e^{\i t r_j} \lambda_{\Theta,r}
\ov{\eqref{quantum-derivations}}{=}e^{t \partial_k}e^{s \partial_j} \lambda_{\Theta,r},
\end{equation*}
from which we deduce \eqref{equ-quantum-commutation-2}. Finally, we shall prove \eqref{equ-quantum-commutation-3}. It suffices by density and linearity to check the commutation relations on each $\lambda_{\Theta,r}$ where $r \in \R^d$. Then on the one hand
\begin{align*}
\MoveEqLeft
e^{\i s X_j} e^{t \partial_k} \lambda_{\Theta,r} 
\ov{\eqref{quantum-derivations}}{=} e^{\i t r_k} e^{\i s X_j}  \lambda_{\Theta,r} 
\ov{\eqref{def-Xk-ek}}{=} e^{\i t r_k} \lambda_{\Theta, se_j} \lambda_{\Theta,r} 
\ov{\eqref{equ-lem-commutation-relations-lambda-theta}}{=} e^{\i t r_k} e^{\frac{1}{2}\i \langle se_j, \Theta r\rangle}  \lambda_{\Theta,se_j + r}.
\end{align*}
On the other hand, a last computation gives
\begin{align*}
\MoveEqLeft
e^{t \partial_k} e^{\i s X_j} \lambda_{\Theta,r} 
\ov{\eqref{def-Xk-ek}}{=} e^{t \partial_k} \lambda_{\Theta,s e_j} \lambda_{\Theta,r} \\
&\overset{\eqref{equ-lem-commutation-relations-lambda-theta}}{=} e^{\frac{1}{2}\i \langle s e_j, \Theta r\rangle}e^{t \partial_k}  \lambda_{\Theta,se_j + r} 
\ov{\eqref{quantum-derivations}}{=} e^{\frac{1}{2}\i \langle s e_j, \Theta r\rangle}  e^{\i t (r_k + s\delta_{k=j})} \lambda_{\Theta,s e_j+r}.
\end{align*}
\end{proof}

\begin{remark} \normalfont
\label{rem-non-com-var}
It is obvious that the relation \eqref{equ-quantum-commutation-1} is equivalent to the commutation rule 
\begin{equation*}
\label{}
e^{\i sx_j} e^{\i tx_k} = e^{\i st \Theta_{jk}} e^{\i tx_k} e^{\i sx_j} 
\end{equation*}
for any $s,t \in \R$ and any $j,k \in \{1, \ldots, d \}$. By Proposition \ref{prop-Kato} we deduce that each subspace $\dom(x_j x_k) \cap \dom(x_kx_j)$ is dense in the Banach space $\L^p(\R_\Theta^d)$ and the commutation relation 
\begin{equation*}
\label{non-com-var-bis}
\i[x_j, x_k] 
=\Theta_{jk}, \quad j,k \in \{1,\ldots,d\}. 
\end{equation*} 
of \eqref{non-com-var} is true on the subspace $\dom(x_j x_k) \cap \dom(x_k x_j)$.
\end{remark}

Now, we obtain our main result on the functional calculus of the harmonic oscillator on the Moyal-Groenewold plane.

\begin{thm}
\label{thm-Hormander-calculus-for-quantum-euclidean-space}
Let $d \in \N$. Consider some skew-symmetric matrix $\Theta \in \M_d(\R)$. Then the family $A = (X_1,\ldots,X_d,\frac{1}{\i}\partial_1,\ldots,\frac{1}{\i}\partial_d)$ defines a $\begin{pmatrix} 
\Theta & - \I_d \\ 
\I_d & 0 
\end{pmatrix}$-Weyl tuple on the Banach space $\L^p(\R^d_\Theta)$ for any $1 < p <\infty$. Consequently, the closure $\A - \alpha \Id$ of the operator 
\begin{equation*}
\label{}
-\sum_{k=1}^d \partial_k^2 + \sum_{k=1}^d X_k^2 - \alpha \Id
\end{equation*}
admits a bounded H\"ormander $\Hor^s_2(\R_+^*)$ functional calculus of order $s > d + \frac12$ on the space $\L^p(\R^d_\Theta)$, where $\alpha \geq 0$ is defined in \eqref{def-alpha}. Moreover, it admits a bounded $\H^\infty(\Sigma_\omega)$ functional calculus for any angle $\omega >0$ on the Banach space $\L^p(\R^d_\Theta)$, i.e.~we have $\omega_{\H^\infty}(\A -\alpha \Id)=0$.
\end{thm}

\begin{proof}
With Proposition \ref{prop-quantum-commutation-rules}, the verification of \eqref{Weyl-tuple} is obvious.
Then the H\"ormander functional calculus is a direct consequence of Corollary \ref{cor-Hormander-calculus-for-Theta-Weyl-tuple-on-LpM} and Corollary \ref{cor-Hormander-calculus-for-Theta-Weyl-tuple-on-LpM-bis}, observing that the von Neumann algebra $\R^d_\Theta$ is injective (see e.g.~\cite[Section 6]{LSZ20}).
\end{proof}

\begin{remark}\normalfont
\label{rem-interpolation-Moyal-plane}
Note that the family $A = (X_1,\ldots,X_d,\frac{1}{\i}\partial_1,\ldots,\frac{1}{\i}\partial_d)$ from Theorem \ref{thm-Hormander-calculus-for-quantum-euclidean-space} forms a consistent Weyl tuple on the Banach space $\L^p(\R^d_\Theta)$ for any $1 < p < \infty$ such that on the Hilbert space $\L^2(\R^d_\Theta)$, the operator $\A - \alpha \Id$ is self-adjoint. Thus Remark \ref{rem-interpolation} applies and we can improve the order of the H\"ormander functional calculus by means of complex interpolation. We get for $\A - \alpha\Id$ a bounded H\"ormander $\Hor^s_2(\R^*_+)$ functional calculus on the space $\L^p(\R^d_\Theta)$, where $1 < p < \infty$, with
\[ 
s > \begin{cases} (\frac2p - 1)(d + \frac12) + \frac1{p'} : \: p \leq 2 \\ (\frac2{p'}-1)(d+\frac12) + \frac1p : \: p \geq 2. \end{cases} 
\]
\end{remark}

\paragraph{Magnetic Weyl system}
   
The concept of a gauge-covariant Magnetic Weyl calculus, essential for analysing magnetic systems, was initially introduced in a preliminary form by M\"uller \cite{Mul99}. This framework gained a more solid foundation through the independent efforts of Mantoiu and Purice \cite{MaP04}. This mathematical framework proves to be invaluable in modelling both classical and quantum particles within the classical Euclidean space $\R^d$, when considering the impact of a magnetic field $B$. For a more comprehensive understanding, the survey \cite{MaP06}, along with the papers \cite{MPR05}, \cite{Lei11}, and \cite{LeL22} and their respective references, offer in-depth insights.

%

In this setting, the self-adjoint position operators $X_j$ and momentum operators $\frac{1}{\i}\partial^B_j$\label{def-partialBj} for magnetic Weyl calculus are characterized by the commutation relations 
\begin{align*}
\label{magnetic-Weyl-calculus}
[ X_j , X_k ] = 0 
,\quad \i[ X_j , \tfrac{1}{\i} \partial^B_k] = \delta_{j=k}, \quad 
\i\big[ \tfrac{1}{\i}\partial^B_j ,\tfrac{1}{\i}\partial^B_k\big] = B_{jk}, \quad 1 \leq j,k \leq d,
\end{align*}
see \cite[(2.26)]{MaP04}. Here 
\begin{equation*}
\label{}
B
=\sum_{j,k=1}^d B_{jk} \d x_j \wedge \d x_k
\end{equation*} 
is a \textit{constant} magnetic field on the Euclidean space $\R^d$. The difference to non-magnetic systems is that momenta along different directions no longer commute. Of course, we can state a version of Theorem \ref{thm-Hormander-calculus-for-quantum-euclidean-space} for the harmonic oscillator associated with this system using the matrix $\begin{pmatrix} 
0 & - \I_d \\ 
\I_d &  B
\end{pmatrix}$.

\paragraph{Magnetic Weyl system on the Moyal planes}
More generally, the authors of the paper \cite{GJM22} consider unbounded self-adjoint operators $\tfrac{1}{\i}\partial^B_1,\ldots,\tfrac{1}{\i}\partial^B_d$ acting on the space $\L^p(\R^d_{\Theta})$ satisfying the conditions 
\begin{equation*}
\label{}
\i\big[ \tfrac{1}{\i}\partial^B_j , \tfrac{1}{\i}\partial^B_k\big] = \delta_{j=k}
\end{equation*}
and 
\begin{equation*}
\label{}
\i\big[\tfrac{1}{\i}\partial^B_j , \tfrac{1}{\i}\partial^B_k \big] = B_{jk},
\end{equation*}
where $B$ is some skew-symmetric matrix of $\M_n(\R)$. Consequently, the set 
\begin{equation*}
\label{}
A 
= \big(X_1,\ldots,X_d,\tfrac{1}{\i}\partial^B_1,\ldots,\tfrac{1}{\i}\partial^B_d\big)
\end{equation*}
constitutes a $\begin{pmatrix} 
\Theta & - \I_d \\ 
\I_d & B
\end{pmatrix}$-Weyl tuple.
The clarification of the construction of these operators is beyond the scope of this paper. However, one can clearly state a version of Theorem \ref{thm-Hormander-calculus-for-quantum-euclidean-space} for the harmonic oscillator associated with such a system.

\chapter{Symbols}
\label{sec-symbols}

\begin{tabular}{p{2.75cm}p{7.25cm}p{1cm}}
$\Sigma_{\omega}$  & open sector of angle $2\omega$ & \pageref{equation-2-cocycle} \\
$\sigma$ & 2-cocycle &  \pageref{equation-2-cocycle}\\
$f \star_\Theta g$ & Moyal star product & \pageref{star-product}\\
$\sigma_\Theta$ & 2-cocycle on $\R^d$& \pageref{def-cocycle-intro} \\
$\norm{\cdot}_{\W^{s,2}(\R)}$ & Sobolev norm & \pageref{Sobolev-norm} \\
$\norm{\cdot}_{\Hor^s_2(\R_+^*)}$ & H\"ormander norm & \pageref{hor-2} \\
$S(\phi,m,\lambda,\Omega)$ & Grosse-Wulkenhaar action & \pageref{action-GW} \\
$\R^d_\Theta$& Quantum Euclidean spaces/Moyal-Groenewold planes& \pageref{def-quatum-Euclidean} \\
$\omega_{\sec}(A)$ & angle of sectoriality of the operator $A$& \pageref{equ-sectorial-operator} \\
$\omega_{\H^\infty}(A)$ & $\H^\infty$-angle of the operator $A$& \pageref{angle-Hinfty} \\
$\lambda_{\sigma,s}$ &  left regular $\sigma$-projective representation & \pageref{def-lambda-sigma} \\
$G_\sigma$ & Mackey group & \pageref{def-central-extension} \\
$\g_\sigma$ & Lie algebra of the Lie group $G_\sigma$& \pageref{def-gsigma} \\
$c_f$ &  2-cocycle $c_f$ associated to $f \co \g \to \mathbb{C}$ & \pageref{def-cf} \\
$\Theta^T$ & transpose of the matrix $\Theta$ & \pageref{transpose} \\
$\Theta^\uparrow$, $\Theta^\downarrow$& upper and lower triangle parts of the matrix $\Theta$& \pageref{upper} \\
$e^{\i t \cdot A}$ & some projective representation & \pageref{def-eitA} \\
$a(A)$ & operator provided by the $\Theta$-Weyl functional calculus & \pageref{Weyl-calculus} \\
$f *_\sigma g$ & twisted convolution & \pageref{twisted-convolution} \\
$f *_{\Theta} g$ & twisted convolution & \pageref{twisted-Rn} \\
$\norm{T}_{\reg,\L^p(\Omega) \to \L^p(\Omega')}$ & regular norm of $T \co \L^p(\Omega) \to \L^p(\Omega')$ & \pageref{Norm-reg-c} \\
$C_{f,\sigma}$ & twisted convolution operator & \pageref{def-Cfsigma} \\
$C_a^\Theta$ & twisted convolution operator & \pageref{convolution-bis} \\
$A_{\univ}^\Theta$ & universal $\Theta$-Weyl tuple& \pageref{A-univ} \\
$\mathbb{H}_\Theta$ & Mackey group & \pageref{def-Htheta} \\
$\tilde{\mathbb{H}}_\Theta$ & universal covering of the group $\mathbb{H}_\Theta$& \pageref{universal-covering} \\
$\exp_G$ & exponential map associated to the Lie group $G$ & \pageref{exponential-map} \\
$J_\mu$ & skew-symmetric endomorphism &\pageref{endomorphism}\\
$f_\mu$ & $\mu$-section of the partial Fourier transform of $f$ & \pageref{partial-Fourier} \\
$p_z^\Theta$ & some kernel & \pageref{def-pzTheta} \\
$J$ & some matrix & \pageref{def-de-J-bis}, \pageref{def-de-J} \\
$\omega_\mu$ & skew-symmetric bilinear form & \pageref{def-omega-mu} \\
\end{tabular}

\newpage

\begin{tabular}{p{2.75cm}p{7.25cm}p{1cm}}
$S(z)$ & $S(z)=\frac{z}{\sin(z)}$ & \pageref{functions-S-and-R} \\
$R(z)$& $R(z)=\frac{z}{\tan(z)}$ & \pageref{functions-S-and-R} \\
$K_{z,\mu}$ &some kernel & \pageref{prop-heatkernel-MMformula}\\
$\d\pi$ & infinitesimal generator of the strongly continuous group $(\pi(\exp_G(t a)))_{t \in \R}$ & \pageref{dpi}\\
$\alpha$ & $\alpha=\alpha_1+\cdots+\alpha_k$& \pageref{def-alpha} \\
$\cal{J}$ & matrix $\begin{pmatrix} 
0 & 1 \\ 
-1 & 0 
\end{pmatrix}$ &\pageref{def-cal-J} \\
$\A_\alpha$ & operator $(\A - \alpha \Id_{\L^p(\R^n)}) \ot \Id_X$ & \pageref{equ-A-alpha}\\
$h_z$ & some smooth even function & \pageref{majo-by-gz}\\
$g_z$ & some function & \pageref{majo-by-gz}\\
 $\L^p(\cal{M},\ell^\infty_I)$ & $\ell^\infty$-valued noncommutative $\L^p$-space & \pageref{LpMellI} \\
$\norm{{\sup_{i \in I}}^+ x_i}_p$ & norm of $\L^p(\cal{M},\ell^\infty_I)$ & \pageref{Norm-Lp-Linfty-Omega} \\
$\Phi_t$ & operator-valued Hardy-Littlewood operator & \pageref{def-Hardy} \\
$R_t^j$ & some average operator & \pageref{def-Rt} \\
$\TC_k$& integral operator by the kernel $k$& \pageref{integral-operator} \\
$\L^p(\Omega,X)$ & Bochner space & \pageref{Bochner-space} \\
$\type X$ & type of the Banach space $X$ & \pageref{cotype-inequalities} \\
$\cotype X$ & cotype of the Banach space $X$ & \pageref{cotype-inequalities} \\
$\norm{T}_{\cb, E \to F}$ & completely bounded norm of $T \co E \to F$ & \pageref{def-completely-bounded} \\
$\scr{L}\mu$ & Laplace transform of the bounded measure $\mu$& \pageref{Laplace-transform} \\
$\scr{L}\M$ & space of Laplace transforms of bounded measures & \pageref{def-LM} \\
$f(A)_{\HP}$ & operator provided by the Hille-Phillips functional calculus & \pageref{Hille-Phillips} \\
$\sigma_R^\nu$ & Bochner-Riesz means & \pageref{Bochner-means}\\
$\cal{L}_a$ & left multiplication operator by $a$& \pageref{Left-multiplication-operator} \\
$\cal{R}_b$ & right multiplication operator by $b$ & \pageref{def-right-mult} \\
$\Ad_{(a,b)}$ & closure of the operator $\cal{L}_a-\cal{R}_b$ & \pageref{7ad} \\
$\lambda_{\Theta,s}$ &operator given by the left regular $\sigma$-projective representation & \pageref{def-lambda-theta}\\
$\lambda_\Theta(f)$ &operator given by the left regular $\sigma$-projective representation & \pageref{def-lambda-theta-f} \\
$\cal{S}(\R^d_\Theta)$ & Schwartz space & \pageref{Schwartz-Theta} \\
$\mathcal{T}_t$ & operator $\mathcal{T}_t \co\L^2(\R^d) \to \L^2(\R^d)$, $f \mapsto e^{\i \langle t , (\cdot) \rangle} f(\cdot)$ & \pageref{def-mathcalTt} \\
$\mathscr{T}_t$ &$\mathscr{T}_t \co \L^\infty(\R^d_\Theta) \to \L^\infty(\R^d_\Theta)$\label{mathscr{T}_t}, $\lambda_{\Theta,s} \mapsto e^{\i \la t,s \ra} \lambda_{\Theta,s}$ & \pageref{mathscr{T}_t} \\
$\QWEP$& Kirhcberg's quotient weak expectation property & \pageref{QWEP}\\
$(A_1,\ldots,A_n)$ & $\Theta$-Weyl tuple & \pageref{defi-Weyl-tuple}
\end{tabular}

\newpage

\begin{tabular}{p{2.75cm}p{7.25cm}p{1cm}}
$\T$ & one-dimensional torus $\{z \in \mathbb{C} : |z|=1\}$ & \pageref{torus} \\
$\H^{\infty}(\Sigma_\theta)$ & algebra of all bounded analytic functions $f \co  \Sigma_\theta\to \C$ & \pageref{algebra-Hinfty} \\
$\H^{\infty}_{0}(\Sigma_\theta)$ & some subalgebra of $\H^{\infty}(\Sigma_\theta)$& \pageref{algebra-Hinfty0} \\
$f(A)$ & operator provided by the functional calculus & \pageref{2CauchySec}\\
$R(z,A)$ & Resolvent operator & \pageref{resolvent-operator}\\ 
$\norm{\cdot }_{\H^{\infty}(\Sigma_\theta)}$ & supremum norm on the sector $\Sigma_\theta$ &\pageref{norm-Hinfty} \\ 
$(\epsi_{k})_{k \geq 1}$ & sequence of independent Rademacher variables & \pageref{def-Rademacher} \\
$X_k$ & noncommutative spatial variable & \pageref{def-Xk}\\
$\frac{1}{\i}\partial_k$& momentum variable & \pageref{def-partialk}\\
$\frac{1}{\i}\partial^B_j$ & momentum variable & \pageref{def-partialBj}\\
$x_k$ & some self-adjoint operator & \pageref{def-xk}\\
$\rho$ & some function & \pageref{function-rho}\\
$\hat{f}$ & Fourier transform & \pageref{Fourier-transform} \\
$a \star_{\sigma} b$ & Moyal product & \pageref{def-Moyal-product-33} \\
$\lesssim$ &stands for an inequality up to a constant &\pageref{lesssim} \\
$\B(X)$ & algebra of bounded operators on $X$ & \pageref{bounded-BX}\\
$\W^{s,2}(\R)$ & inhomogeneous Sobolev space & \pageref{Sobolev-space} \\
$\mathcal{W}^{s,2}$ & Mellin-Sovolev space & \pageref{defi-mathcal-W}\\  
$\mathrm{C}^N_b(\R)$ & Banach space of functions $f \co \R \to \mathbb{C}$ of class $\mathrm{C}^N$ with bounded derivatives up to order $N$ & \pageref{CbRn}\\
$\Str_\theta$ & Strip $\{ z \in \C :\: |\Im z |< \theta \}$ & \pageref{strip}\\
$p^*$ & conjugate exponent $\frac{p}{p-1}$ of $p$ &  \\ 
$\Re z$ & real part of the complex number $z$ &  \\
$\Im z$& imaginary part of the complex number $z$  & \\
$\cal{S}(\R^d)$& Schwartz space & \\
$\cal{S}'(\R^d)$ & space of tempered distributions &  \\
$\UMD$& property of unconditionality for martingale differences &\\
$\Theta$ & real skew-symmetric matrix  & \\
$\Ran A$& Range of the operator $A$ & \\
\end{tabular}

\backmatter

%

\end{document}